\documentclass
[numbers=enddot,12pt,final,onecolumn,notitlepage,english]{scrartcl}%
\usepackage[headsepline,footsepline,manualmark]{scrlayer-scrpage}
\usepackage[all,cmtip]{xy}
\usepackage{amssymb}
\usepackage{amsmath}
\usepackage{amsthm}
\usepackage{framed}
\usepackage{comment}
\usepackage{color}
\usepackage[breaklinks=true]{hyperref}
\usepackage{hyperxmp}
\usepackage[sc]{mathpazo}
\usepackage[T1]{fontenc}
\usepackage{needspace}
\usepackage{tabls}
\usepackage{easytable}
\usepackage[type={CC}, modifier={zero}, version={1.0},]{doclicense}
\providecommand{\U}[1]{\protect\rule{.1in}{.1in}}
\theoremstyle{definition}
\newtheorem{theo}{Theorem}[section]
\newenvironment{theorem}[1][]
{\begin{theo}[#1]\begin{leftbar}}
{\end{leftbar}\end{theo}}
\newtheorem{lem}[theo]{Lemma}
\newenvironment{lemma}[1][]
{\begin{lem}[#1]\begin{leftbar}}
{\end{leftbar}\end{lem}}
\newtheorem{prop}[theo]{Proposition}
\newenvironment{proposition}[1][]
{\begin{prop}[#1]\begin{leftbar}}
{\end{leftbar}\end{prop}}
\newtheorem{defi}[theo]{Definition}
\newenvironment{definition}[1][]
{\begin{defi}[#1]\begin{leftbar}}
{\end{leftbar}\end{defi}}
\newtheorem{remk}[theo]{Remark}
\newenvironment{remark}[1][]
{\begin{remk}[#1]\begin{leftbar}}
{\end{leftbar}\end{remk}}
\newtheorem{coro}[theo]{Corollary}
\newenvironment{corollary}[1][]
{\begin{coro}[#1]\begin{leftbar}}
{\end{leftbar}\end{coro}}
\newtheorem{conv}[theo]{Convention}
\newenvironment{convention}[1][]
{\begin{conv}[#1]\begin{leftbar}}
{\end{leftbar}\end{conv}}
\newtheorem{quest}[theo]{Question}
\newenvironment{question}[1][]
{\begin{quest}[#1]\begin{leftbar}}
{\end{leftbar}\end{quest}}
\newtheorem{warn}[theo]{Warning}
\newenvironment{warning}[1][]
{\begin{warn}[#1]\begin{leftbar}}
{\end{leftbar}\end{warn}}
\newtheorem{conj}[theo]{Conjecture}

\newtheorem{exam}[theo]{Example}
\newenvironment{example}[1][]
{\begin{exam}[#1]\begin{leftbar}}
{\end{leftbar}\end{exam}}
\newenvironment{statement}{\begin{quote}}{\end{quote}}

\let\sumnonlimits\sum
\let\prodnonlimits\prod
\let\cupnonlimits\bigcup
\let\capnonlimits\bigcap
\renewcommand{\sum}{\sumnonlimits\limits}
\renewcommand{\prod}{\prodnonlimits\limits}
\renewcommand{\bigcup}{\cupnonlimits\limits}
\renewcommand{\bigcap}{\capnonlimits\limits}
\newenvironment{verlong}{}{}
\newenvironment{vershort}{}{}

\excludecomment{verlong}
\includecomment{vershort}
\excludecomment{noncompile}

\newcommand{\are}{\ar@{-}}
\newcommand{\sslash}{\mathbin{/\mkern-6mu/}}
\ihead{The pre-Pieri rules}
\ohead{page \thepage}
\begin{document}

\title{The pre-Pieri rules}
\author{Darij Grinberg}
\date{May 21, 2026}
\maketitle

\begin{abstract}
\textbf{Abstract.} Let $R$ be a commutative ring and $n\geq1$ and $p\geq0$ two
integers. Let $\mathbb{N}=\left\{  0,1,2,\ldots\right\}  $. Let $h_{k,\ i}$ be
an element of $R$ for all $k\in\mathbb{Z}$ and $i\in\left[  n\right]  $. For
any $\alpha\in\mathbb{Z}^{n}$, we define%
\[
t_{\alpha}:=\det\left(
\begin{array}
[c]{cccc}%
h_{\alpha_{1}+1,\ 1} & h_{\alpha_{1}+2,\ 1} & \cdots & h_{\alpha_{1}+n,\ 1}\\
h_{\alpha_{2}+1,\ 2} & h_{\alpha_{2}+2,\ 2} & \cdots & h_{\alpha_{2}+n,\ 2}\\
\vdots & \vdots & \ddots & \vdots\\
h_{\alpha_{n}+1,\ n} & h_{\alpha_{n}+2,\ n} & \cdots & h_{\alpha_{n}+n,\ n}%
\end{array}
\right)  \in R
\]
(where $\alpha_{i}$ denotes the $i$-th entry of $\alpha$). Then, we have the
identity%
\[
\sum_{\substack{\beta\in\mathbb{N}^{n};\\\left\vert \beta\right\vert
=p}}t_{\alpha+\beta}=\det\left(
\begin{array}
[c]{ccccc}%
h_{\alpha_{1}+1,\ 1} & h_{\alpha_{1}+2,\ 1} & \cdots & h_{\alpha_{1}+\left(
n-1\right)  ,\ 1} & h_{\alpha_{1}+\left(  n+p\right)  ,\ 1}\\
h_{\alpha_{2}+1,\ 2} & h_{\alpha_{2}+2,\ 2} & \cdots & h_{\alpha_{2}+\left(
n-1\right)  ,\ 2} & h_{\alpha_{2}+\left(  n+p\right)  ,\ 2}\\
\vdots & \vdots & \ddots & \vdots & \vdots\\
h_{\alpha_{n}+1,\ n} & h_{\alpha_{n}+2,\ n} & \cdots & h_{\alpha_{n}+\left(
n-1\right)  ,\ n} & h_{\alpha_{n}+\left(  n+p\right)  ,\ n}%
\end{array}
\right)
\]
(where $\alpha+\beta$ denotes the entrywise sum of the tuples $\alpha$ and
$\beta$). (The matrix on the right hand side here is the one from the
definition of $t_{\alpha}$, except for its last column, where the
\textquotedblleft$+n$\textquotedblright s have been replaced by
\textquotedblleft$+\left(  n+p\right)  $\textquotedblright s.) Furthermore, if
$p\leq n$, then
\[
\sum_{\substack{\beta\in\left\{  0,1\right\}  ^{n};\\\left\vert \beta
\right\vert =p}}t_{\alpha+\beta}=\det\left(
\begin{array}
[c]{cccc}%
h_{\alpha_{1}+\xi_{1},\ 1} & h_{\alpha_{1}+\xi_{2},\ 1} & \cdots &
h_{\alpha_{1}+\xi_{n},\ 1}\\
h_{\alpha_{2}+\xi_{1},\ 2} & h_{\alpha_{2}+\xi_{2},\ 2} & \cdots &
h_{\alpha_{2}+\xi_{n},\ 2}\\
\vdots & \vdots & \ddots & \vdots\\
h_{\alpha_{n}+\xi_{1},\ n} & h_{\alpha_{n}+\xi_{2},\ n} & \cdots &
h_{\alpha_{n}+\xi_{n},\ n}%
\end{array}
\right)  ,
\]
where $\xi=\left(  1,2,\ldots,n-p,n-p+2,n-p+3,\ldots,n+1\right)  $. We prove
these two identities (in a slightly more general setting, where $R$ is not
assumed commutative) and use them to derive some variants of the Pieri rule
found in the literature.

\textbf{Keywords:} determinantal identities, determinant, Pieri rules,
symmetric functions, Schur functions, immaculate functions.

\textbf{MSC classes (2020):} 15A15, 05E05, 15A24.

\end{abstract}
\tableofcontents

\section*{Introduction}

The Pieri rules in the theory of symmetric functions (see, e.g., \cite[Chapter
I, (5.16) and (5.17)]{Macdon95} or \cite[Theorem 7.15.7 and discussion before
Corollary 7.15.9]{Stanley-EC2}) give simple formulas for multiplying a Schur
function by a complete homogeneous or elementary symmetric function. One of
the simplest ways to state them (sidestepping the combinatorial background and
the geometric motivation) is as follows: We define the ring $\Lambda$ of
symmetric functions as a polynomial ring in countably many indeterminates
$h_{1},h_{2},h_{3},\ldots$ over some commutative ring (say, over $\mathbb{Z}%
$); we furthermore set $h_{0}:=1$ and $h_{i}:=0$ for all negative $i$. If
$\lambda=\left(  \lambda_{1},\lambda_{2},\lambda_{3},\ldots\right)  $ is an
integer partition\footnote{An \emph{integer partition} (or, for short, just
\emph{partition}) means a weakly decreasing sequence $\left(  \lambda
_{1},\lambda_{2},\lambda_{3},\ldots\right)  $ of nonnegative integers such
that all but finitely many $i>0$ satisfy $\lambda_{i}=0$.}, then the
corresponding \emph{Schur function} $s_{\lambda}\in\Lambda$ can be defined by
the formula%
\begin{equation}
s_{\lambda}=\det\left(  \left(  h_{\lambda_{i}-i+j}\right)  _{i,j\in\left[
n\right]  }\right)  , \label{eq.intro.slam=jt}%
\end{equation}
where $n$ is a nonnegative integer satisfying $\lambda_{n+1}=\lambda
_{n+2}=\lambda_{n+3}=\cdots=0$ (note that there are infinitely many possible
values for $n$, but they all give the same $s_{\lambda}$). If the partition
$\lambda$ has the form $\left(  1,1,\ldots,1,0,0,0,\ldots\right)  $ with $k$
many $1$'s, then the corresponding Schur function $s_{\lambda}$ is called
$e_{k}$. (Of course, these are not the usual definitions of $s_{\lambda}$ and
$e_{k}$; see \cite[Chapter I, (3.4) and (3.9)]{Macdon95} for their equivalence
to more standard definitions.) Now, the \emph{first Pieri rule} (\cite[Chapter
I, (5.16)]{Macdon95}) states that for each $p\in\mathbb{N}$ and each partition
$\lambda=\left(  \lambda_{1},\lambda_{2},\lambda_{3},\ldots\right)  $, we have%
\begin{equation}
s_{\lambda}h_{p}=\sum_{\mu}s_{\mu}, \label{eq.intro.pieri1}%
\end{equation}
where the sum ranges over all partitions $\mu=\left(  \mu_{1},\mu_{2},\mu
_{3},\ldots\right)  $ such that $\mu/\lambda$ is a \emph{horizontal }%
$p$\emph{-strip}\footnote{The notion of a \textquotedblleft horizontal
$p$-strip\textquotedblright\ and the notation $\mu/\lambda$ are best explained
in terms of Young diagrams. However, for our purposes, we can define them
algebraically: We say that \textquotedblleft$\mu/\lambda$ is a horizontal
$p$-strip\textquotedblright\ if and only if
\[
\mu_{1}\geq\lambda_{1}\geq\mu_{2}\geq\lambda_{2}\geq\mu_{3}\geq\lambda_{3}%
\geq\cdots\ \ \ \ \ \ \ \ \ \ \text{and }\sum_{i\geq1}\left(  \mu_{i}%
-\lambda_{i}\right)  =p.
\]
}. For example,%
\[
s_{\left(  2,1\right)  }h_{2}=s_{\left(  4,1\right)  }+s_{\left(  3,2\right)
}+s_{\left(  3,1,1\right)  }+s_{\left(  2,2,1\right)  },
\]
where we are using the standard convention of omitting zeroes from a partition
(i.e., we identify a partition $\lambda=\left(  \lambda_{1},\lambda
_{2},\lambda_{3},\ldots\right)  $ with the $n$-tuple $\left(  \lambda
_{1},\lambda_{2},\ldots,\lambda_{n}\right)  $ when $\lambda_{n+1}%
=\lambda_{n+2}=\lambda_{n+3}=\cdots=0$). The \emph{second Pieri rule}
(\cite[Chapter I, (5.17)]{Macdon95}) states that for each $p\in\mathbb{N}$ and
each partition $\lambda=\left(  \lambda_{1},\lambda_{2},\lambda_{3}%
,\ldots\right)  $, we have%
\begin{equation}
s_{\lambda}e_{p}=\sum_{\mu}s_{\mu}, \label{eq.intro.pieri2}%
\end{equation}
where the sum ranges over all partitions $\mu=\left(  \mu_{1},\mu_{2},\mu
_{3},\ldots\right)  $ such that $\mu/\lambda$ is a \emph{vertical }%
$p$\emph{-strip}\footnote{The notion of a \textquotedblleft vertical
$p$-strip\textquotedblright\ and the notation $\mu/\lambda$ are best explained
in terms of Young diagrams. However, for our purposes, we can define them
algebraically: We say that \textquotedblleft$\mu/\lambda$ is a vertical
$p$-strip\textquotedblright\ if and only if
\[
\left(  \mu_{i}-\lambda_{i}\in\left\{  0,1\right\}  \text{ for each }%
i\geq1\right)  \ \ \ \ \ \ \ \ \ \ \text{and }\sum_{i\geq1}\left(  \mu
_{i}-\lambda_{i}\right)  =p.
\]
}.

There is a less-known variant of the first Pieri rule (\ref{eq.intro.pieri1}),
which sometimes appears as a stepping stone to its proof (e.g., in
\cite[\S 2.4]{Tamvak13}). In order to state it, we agree to define a
\textquotedblleft Schur function\textquotedblright\ $s_{\lambda}$ by the
equality (\ref{eq.intro.slam=jt}) not just whenever $\lambda$ is a partition,
but also whenever $\lambda=\left(  \lambda_{1},\lambda_{2},\ldots,\lambda
_{n}\right)  $ is any $n$-tuple of integers (not necessarily nonnegative, not
necessarily weakly decreasing). This does not significantly extend the notion
of a \textquotedblleft Schur function\textquotedblright, since any such
$s_{\lambda}$ equals either $0$ or $\pm s_{\mu}$ for an (honest) partition
$\mu$. (This follows easily from basic properties of determinants.) Now, if
$\lambda=\left(  \lambda_{1},\lambda_{2},\ldots,\lambda_{n}\right)  $ is any
integer partition with $\lambda_{n}=0$, and if $p\in\mathbb{N}$ is arbitrary,
then%
\begin{equation}
s_{\lambda}h_{p}=\sum_{\mu}s_{\mu}, \label{eq.intro.pieri3}%
\end{equation}
where the sum now ranges over all $n$-tuples $\mu=\left(  \mu_{1},\mu
_{2},\ldots,\mu_{n}\right)  \in\mathbb{N}^{n}$ such that
\[
\left(  \mu_{i}\geq\lambda_{i}\text{ for each }i\right)  \text{ and }%
\sum_{i=1}^{n}\left(  \mu_{i}-\lambda_{i}\right)  =p.
\]
In other words, if $\lambda=\left(  \lambda_{1},\lambda_{2},\ldots,\lambda
_{n}\right)  $ is any integer partition with $\lambda_{n}=0$, and if
$p\in\mathbb{N}$ is arbitrary, then%
\begin{equation}
s_{\lambda}h_{p}=\sum_{\substack{\beta=\left(  \beta_{1},\beta_{2}%
,\ldots,\beta_{n}\right)  \in\mathbb{N}^{n};\\\beta_{1}+\beta_{2}+\cdots
+\beta_{n}=p}}s_{\lambda+\beta}, \label{eq.intro.pieri3b}%
\end{equation}
where $\lambda+\beta$ denotes the entrywise sum of the $n$-tuples $\lambda$
and $\beta$.

Note that each addend in the sum in (\ref{eq.intro.pieri1}) is also contained
in the sum in (\ref{eq.intro.pieri3}), but (usually) not the other way around:
An $n$-tuple $\mu$ appearing on the right hand side of (\ref{eq.intro.pieri3})
might fail to be a partition, and even if it is one, it may violate the
\textquotedblleft horizontal $p$-strip\textquotedblright\ condition (by
failing to satisfy $\lambda_{i}\geq\mu_{i+1}$ for some $i$). Some of these
extraneous addends in (\ref{eq.intro.pieri3}) are $0$, while others cancel
each other out. This cancellation argument appears (in some slightly different
contexts) in \cite[Lemma 2]{Tamvak13} and \cite[(2.1) vs. (2.2)]{LakTho07}.

The \textquotedblleft alternative first Pieri rule\textquotedblright%
\ (\ref{eq.intro.pieri3}) aka (\ref{eq.intro.pieri3b}) is itself not hard to
prove. In this note, we shall generalize it in multiple directions. The
ultimate generalization -- our Theorem \ref{thm.pre-pieri} -- we call the
\emph{first pre-Pieri rule}; it is an identity for row-determinants of
matrices over noncommutative rings. We will give an elementary proof of
Theorem \ref{thm.pre-pieri} (using simple combinatorics and manipulation of
sums\footnote{Our proof could be viewed as a sequence of sign-reversing
involutions, although we do not state it in such a form.}) and derive several
corollaries, which include not only (\ref{eq.intro.pieri3b}), but also a
noncommutative \textquotedblleft right-Pieri rule\textquotedblright\ for
immaculate functions due to Berg, Bergeron, Saliola, Serrano and Zabrocki
\cite[Theorem 3.5]{BBSSZ13} as well as a Pieri-like rule for Macdonald's
9th-variation Schur functions \cite[Proposition 3.9]{Fun}.\footnote{It should
also be possible to derive the \textquotedblleft uncancelled Pieri
rule\textquotedblright\ \cite[Theorem 11.7]{basisquot} from our first
pre-Pieri rule, but this will likely require more effort than it is worth.
(Note that \cite[Theorem 11.7]{basisquot} is not about symmetric functions,
but about symmetric polynomials in $k$ variables; on the other hand, unlike
(\ref{eq.intro.pieri3b}), there is no $\lambda_{n}=0$ requirement in
\cite[Theorem 11.7]{basisquot}. These differences are fairly substantial, and
it is not immediately obvious how to bridge them.)}

We will also show a \emph{second pre-Pieri rule}: an analogue of the first
pre-Pieri rule with a parallel retinue of corollaries. One such corollary is
an analogue of (\ref{eq.intro.pieri3b}) for $e_{p}$ instead of $h_{p}$; it
says that if $\lambda=\left(  \lambda_{1},\lambda_{2},\ldots,\lambda
_{n}\right)  $ is any partition, and if $p\in\left\{  0,1,\ldots,n\right\}  $
satisfies $\lambda_{n-p+1}=\lambda_{n-p+2}=\cdots=\lambda_{n}=0$, then%
\begin{equation}
s_{\lambda}e_{p}=\sum_{\substack{\beta=\left(  \beta_{1},\beta_{2}%
,\ldots,\beta_{n}\right)  \in\left\{  0,1\right\}  ^{n};\\\beta_{1}+\beta
_{2}+\cdots+\beta_{n}=p}}s_{\lambda+\beta}. \label{eq.intro.pieri4b}%
\end{equation}
This can be viewed as an alternative version of the second Pieri rule
(\ref{eq.intro.pieri2}), and indeed it is possible to obtain
(\ref{eq.intro.pieri2}) from (\ref{eq.intro.pieri4b}) by removing vanishing
addends. (Unlike for the first Pieri rule, cancellations are not required.)

Finally, we shall speculate on the existence of a \textquotedblleft pre-LR
rule\textquotedblright, which might include both pre-Pieri rules as particular cases.

\section{Notations}

Let us first introduce the notations that will be used throughout this note.

\begin{itemize}
\item Let $R$ be a ring (unital and associative, but not necessarily commutative).

\item Let $\mathbb{N}:=\left\{  0,1,2,\ldots\right\}  $ and $\mathbb{P}%
:=\left\{  1,2,3,\ldots\right\}  $.

\item For any $n\in\mathbb{N}$, we let $\left[  n\right]  $ denote the
$n$-element set $\left\{  1,2,\ldots,n\right\}  $.

\item If $\alpha$ is an $n$-tuple (for some $n\in\mathbb{N}$), and if
$i\in\left[  n\right]  $, then we let $\alpha_{i}$ denote the $i$-th entry of
$\alpha$ (so that $\alpha=\left(  \alpha_{1},\alpha_{2},\ldots,\alpha
_{n}\right)  $).

\item If $\alpha$ is an $n$-tuple of integers (for some $n\in\mathbb{N}$),
then we define $\left\vert \alpha\right\vert :=\alpha_{1}+\alpha_{2}%
+\cdots+\alpha_{n}$.

\item For any $n\in\mathbb{N}$, we let $S_{n}$ denote the $n$-th symmetric
group (i.e., the group of all permutations of the set $\left[  n\right]  $).

\item If $n\in\mathbb{N}$ and $\sigma\in S_{n}$, then we let $\left(
-1\right)  ^{\sigma}$ denote the sign of the permutation $\sigma$.

\item If $n\in\mathbb{N}$, and if we are given an element $a_{i,j}\in R$ for
each pair $\left(  i,j\right)  \in\left[  n\right]  \times\left[  n\right]  $,
then we let $\left(  a_{i,j}\right)  _{i,j\in\left[  n\right]  }$ denote the
$n\times n$-matrix whose $\left(  i,j\right)  $-th entry is $a_{i,j}$ for each
$\left(  i,j\right)  \in\left[  n\right]  \times\left[  n\right]  $. That is,
we let%
\[
\left(  a_{i,j}\right)  _{i,j\in\left[  n\right]  }:=\left(
\begin{array}
[c]{cccc}%
a_{1,1} & a_{1,2} & \cdots & a_{1,n}\\
a_{2,1} & a_{2,2} & \cdots & a_{2,n}\\
\vdots & \vdots & \ddots & \vdots\\
a_{n,1} & a_{n,2} & \cdots & a_{n,n}%
\end{array}
\right)  \in R^{n\times n}.
\]

\item If $A=\left(  a_{i,j}\right)  _{i,j\in\left[  n\right]  }$ is any
$n\times n$-matrix over $R$ (for some $n\in\mathbb{N}$), then we define an
element $\operatorname*{rowdet}A\in R$ by%
\[
\operatorname*{rowdet}A:=\sum_{\sigma\in S_{n}}\left(  -1\right)  ^{\sigma
}a_{1,\sigma\left(  1\right)  }a_{2,\sigma\left(  2\right)  }\cdots
a_{n,\sigma\left(  n\right)  }.
\]
This element is called the \emph{row-determinant} of $A$. When the ring $R$ is
commutative, this row-determinant $\operatorname*{rowdet}A$ is just the usual
determinant of $A$.

\item We regard the set $\mathbb{Z}^{n}$ as a $\mathbb{Z}$-module in the usual
way: i.e., we have%
\begin{align*}
\alpha+\beta &  =\left(  \alpha_{1}+\beta_{1},\ \ \alpha_{2}+\beta
_{2},\ \ \ldots,\ \ \alpha_{n}+\beta_{n}\right)
\ \ \ \ \ \ \ \ \ \ \text{and}\\
\alpha-\beta &  =\left(  \alpha_{1}-\beta_{1},\ \ \alpha_{2}-\beta
_{2},\ \ \ldots,\ \ \alpha_{n}-\beta_{n}\right)
\end{align*}
for any $\alpha\in\mathbb{Z}^{n}$ and $\beta\in\mathbb{Z}^{n}$. Thus,
$\alpha+\beta$ and $\alpha-\beta$ are defined for $\alpha\in\mathbb{N}^{n}$
and $\beta\in\mathbb{N}^{n}$ as well (since $\mathbb{N}^{n}$ is a subset of
$\mathbb{Z}^{n}$).
\end{itemize}

\section{The first pre-Pieri rule}

\subsection{The theorem}

We can now state our \textquotedblleft first pre-Pieri rule\textquotedblright%
\ in full generality:

\begin{theorem}
[first pre-Pieri rule]\label{thm.pre-pieri}Let $n\in\mathbb{P}$ and
$p\in\mathbb{N}$. Let $h_{k,\ i}$ be an element of $R$ for all $k\in
\mathbb{Z}$ and $i\in\left[  n\right]  $.

For any $\alpha\in\mathbb{Z}^{n}$, we define%
\[
t_{\alpha}:=\operatorname*{rowdet}\left(  \left(  h_{\alpha_{i}+j,\ i}\right)
_{i,j\in\left[  n\right]  }\right)  \in R.
\]

Let $\eta$ be the $n$-tuple
\[
\left(  1,2,\ldots,n\right)  +\left(  \underbrace{0,0,\ldots,0}_{n-1\text{
zeroes}},p\right)  =\left(  1,2,\ldots,n-1,n+p\right)  \in\mathbb{Z}^{n}.
\]

Let $\alpha\in\mathbb{Z}^{n}$. Then,%
\begin{equation}
\sum_{\substack{\beta\in\mathbb{N}^{n};\\\left\vert \beta\right\vert
=p}}t_{\alpha+\beta}=\operatorname*{rowdet}\left(  \left(  h_{\alpha_{i}%
+\eta_{j},\ i}\right)  _{i,j\in\left[  n\right]  }\right)  .
\label{eq.thm.pre-pieri.claim}%
\end{equation}

\end{theorem}

\begin{example}
For this example, set $n=2$ and $p=2$, and let $\alpha\in\mathbb{Z}^{2}$ be
arbitrary. Fix arbitrary elements $h_{k,\ i}\in R$ for all $k\in\mathbb{Z}$
and $i\in\left[  n\right]  $. Then, the $n$-tuple $\eta$ defined in Theorem
\ref{thm.pre-pieri} is $\left(  1,2\right)  +\left(  0,2\right)  =\left(
1,4\right)  $. Hence, (\ref{eq.thm.pre-pieri.claim}) says that%
\[
\sum_{\substack{\beta\in\mathbb{N}^{2};\\\left\vert \beta\right\vert
=2}}t_{\alpha+\beta}=\operatorname*{rowdet}\left(  \left(  h_{\alpha_{i}%
+\eta_{j},\ i}\right)  _{i,j\in\left[  2\right]  }\right)
=\operatorname*{rowdet}\left(
\begin{array}
[c]{cc}%
h_{\alpha_{1}+1,\ 1} & h_{\alpha_{1}+4,\ 1}\\
h_{\alpha_{2}+1,\ 2} & h_{\alpha_{2}+4,\ 2}%
\end{array}
\right)  .
\]
The left hand side of this equality can be rewritten as%
\begin{align*}
&  \sum_{\substack{\beta\in\mathbb{N}^{2};\\\left\vert \beta\right\vert
=2}}\ \ \underbrace{t_{\alpha+\beta}}_{\substack{=\operatorname*{rowdet}%
\left(  \left(  h_{\left(  \alpha+\beta\right)  _{i}+j,\ i}\right)
_{i,j\in\left[  2\right]  }\right)  \\\text{(by the definition of }%
t_{\alpha+\beta}\text{)}}}\\
&  =\sum_{\substack{\beta\in\mathbb{N}^{2};\\\left\vert \beta\right\vert
=2}}\operatorname*{rowdet}\left(  \left(  h_{\left(  \alpha+\beta\right)
_{i}+j,\ i}\right)  _{i,j\in\left[  2\right]  }\right) \\
&  =\sum_{\substack{\beta\in\mathbb{N}^{2};\\\left\vert \beta\right\vert
=2}}\operatorname*{rowdet}\left(  \left(  h_{\alpha_{i}+\beta_{i}%
+j,\ i}\right)  _{i,j\in\left[  2\right]  }\right) \\
&  \ \ \ \ \ \ \ \ \ \ \ \ \ \ \ \ \ \ \ \ \left(  \text{since }\left(
\alpha+\beta\right)  _{i}=\alpha_{i}+\beta_{i}\text{ for all }i\in\left[
2\right]  \right) \\
&  =\sum_{\substack{\beta\in\mathbb{N}^{2};\\\left\vert \beta\right\vert
=2}}\operatorname*{rowdet}\left(
\begin{array}
[c]{cc}%
h_{\alpha_{1}+\beta_{1}+1,\ 1} & h_{\alpha_{1}+\beta_{1}+2,\ 1}\\
h_{\alpha_{2}+\beta_{2}+1,\ 2} & h_{\alpha_{2}+\beta_{2}+2,\ 2}%
\end{array}
\right)
\end{align*}%
\begin{align*}
&  =\operatorname*{rowdet}\left(
\begin{array}
[c]{cc}%
h_{\alpha_{1}+2+1,\ 1} & h_{\alpha_{1}+2+2,\ 1}\\
h_{\alpha_{2}+0+1,\ 2} & h_{\alpha_{2}+0+2,\ 2}%
\end{array}
\right)  +\operatorname*{rowdet}\left(
\begin{array}
[c]{cc}%
h_{\alpha_{1}+1+1,\ 1} & h_{\alpha_{1}+1+2,\ 1}\\
h_{\alpha_{2}+1+1,\ 2} & h_{\alpha_{2}+1+2,\ 2}%
\end{array}
\right) \\
&  \ \ \ \ \ \ \ \ \ \ +\operatorname*{rowdet}\left(
\begin{array}
[c]{cc}%
h_{\alpha_{1}+0+1,\ 1} & h_{\alpha_{1}+0+2,\ 1}\\
h_{\alpha_{2}+2+1,\ 2} & h_{\alpha_{2}+2+2,\ 2}%
\end{array}
\right) \\
&  \ \ \ \ \ \ \ \ \ \ \ \ \ \ \ \ \ \ \ \ \left(
\begin{array}
[c]{c}%
\text{since there are exactly three }2\text{-tuples }\beta\in\mathbb{N}^{2}\\
\text{satisfying }\left\vert \beta\right\vert =2\text{, namely }\left(
2,0\right)  \text{, }\left(  1,1\right)  \text{ and }\left(  0,2\right)
\end{array}
\right) \\
&  =\operatorname*{rowdet}\left(
\begin{array}
[c]{cc}%
h_{\alpha_{1}+3,\ 1} & h_{\alpha_{1}+4,\ 1}\\
h_{\alpha_{2}+1,\ 2} & h_{\alpha_{2}+2,\ 2}%
\end{array}
\right)  +\operatorname*{rowdet}\left(
\begin{array}
[c]{cc}%
h_{\alpha_{1}+2,\ 1} & h_{\alpha_{1}+3,\ 1}\\
h_{\alpha_{2}+2,\ 2} & h_{\alpha_{2}+3,\ 2}%
\end{array}
\right) \\
&  \ \ \ \ \ \ \ \ \ \ +\operatorname*{rowdet}\left(
\begin{array}
[c]{cc}%
h_{\alpha_{1}+1,\ 1} & h_{\alpha_{1}+2,\ 1}\\
h_{\alpha_{2}+3,\ 2} & h_{\alpha_{2}+4,\ 2}%
\end{array}
\right)  .
\end{align*}
Therefore, (\ref{eq.thm.pre-pieri.claim}) rewrites as%
\begin{align*}
&  \operatorname*{rowdet}\left(
\begin{array}
[c]{cc}%
h_{\alpha_{1}+3,\ 1} & h_{\alpha_{1}+4,\ 1}\\
h_{\alpha_{2}+1,\ 2} & h_{\alpha_{2}+2,\ 2}%
\end{array}
\right)  +\operatorname*{rowdet}\left(
\begin{array}
[c]{cc}%
h_{\alpha_{1}+2,\ 1} & h_{\alpha_{1}+3,\ 1}\\
h_{\alpha_{2}+2,\ 2} & h_{\alpha_{2}+3,\ 2}%
\end{array}
\right) \\
&  \ \ \ \ \ \ \ \ \ \ +\operatorname*{rowdet}\left(
\begin{array}
[c]{cc}%
h_{\alpha_{1}+1,\ 1} & h_{\alpha_{1}+2,\ 1}\\
h_{\alpha_{2}+3,\ 2} & h_{\alpha_{2}+4,\ 2}%
\end{array}
\right) \\
&  =\operatorname*{rowdet}\left(
\begin{array}
[c]{cc}%
h_{\alpha_{1}+1,\ 1} & h_{\alpha_{1}+4,\ 1}\\
h_{\alpha_{2}+1,\ 2} & h_{\alpha_{2}+4,\ 2}%
\end{array}
\right)  .
\end{align*}
This is easy to check directly by expanding all four row-determinants.
\end{example}

\begin{verlong}
\begin{example}
For another example, set $n=3$ and $p=2$, and let $\alpha\in\mathbb{Z}^{3}$ be
arbitrary. Fix arbitrary elements $h_{k,\ i}\in R$ for all $k\in\mathbb{Z}$
and $i\in\left[  n\right]  $. Then, the $n$-tuple $\eta$ defined in Theorem
\ref{thm.pre-pieri} is $\left(  1,2,3\right)  +\left(  0,0,2\right)  =\left(
1,2,5\right)  $. Hence, (\ref{eq.thm.pre-pieri.claim}) says that%
\begin{align*}
\sum_{\substack{\beta\in\mathbb{N}^{3};\\\left\vert \beta\right\vert
=2}}t_{\alpha+\beta}  &  =\operatorname*{rowdet}\left(  \left(  h_{\alpha
_{i}+\eta_{j},\ i}\right)  _{i,j\in\left[  3\right]  }\right) \\
&  =\operatorname*{rowdet}\left(
\begin{array}
[c]{ccc}%
h_{\alpha_{1}+1,\ 1} & h_{\alpha_{1}+2,\ 1} & h_{\alpha_{1}+5,\ 1}\\
h_{\alpha_{2}+1,\ 2} & h_{\alpha_{2}+2,\ 2} & h_{\alpha_{2}+5,\ 2}\\
h_{\alpha_{3}+1,\ 3} & h_{\alpha_{3}+2,\ 3} & h_{\alpha_{3}+5,\ 3}%
\end{array}
\right)  .
\end{align*}
The left hand side of this equality can be rewritten as%
\begin{align*}
&  \sum_{\substack{\beta\in\mathbb{N}^{3};\\\left\vert \beta\right\vert
=2}}\ \ \underbrace{t_{\alpha+\beta}}_{\substack{=\operatorname*{rowdet}%
\left(  \left(  h_{\left(  \alpha+\beta\right)  _{i}+j,\ i}\right)
_{i,j\in\left[  3\right]  }\right)  \\\text{(by the definition of }%
t_{\alpha+\beta}\text{)}}}\\
&  =\sum_{\substack{\beta\in\mathbb{N}^{3};\\\left\vert \beta\right\vert
=2}}\operatorname*{rowdet}\left(  \left(  h_{\left(  \alpha+\beta\right)
_{i}+j,\ i}\right)  _{i,j\in\left[  3\right]  }\right) \\
&  =\sum_{\substack{\beta\in\mathbb{N}^{3};\\\left\vert \beta\right\vert
=2}}\operatorname*{rowdet}\left(  \left(  h_{\alpha_{i}+\beta_{i}%
+j,\ i}\right)  _{i,j\in\left[  3\right]  }\right) \\
&  \ \ \ \ \ \ \ \ \ \ \ \ \ \ \ \ \ \ \ \ \left(  \text{since }\left(
\alpha+\beta\right)  _{i}=\alpha_{i}+\beta_{i}\text{ for all }i\in\left[
3\right]  \right) \\
&  =\sum_{\substack{\beta\in\mathbb{N}^{3};\\\left\vert \beta\right\vert
=2}}\operatorname*{rowdet}\left(
\begin{array}
[c]{ccc}%
h_{\alpha_{1}+\beta_{1}+1,\ 1} & h_{\alpha_{1}+\beta_{1}+2,\ 1} &
h_{\alpha_{1}+\beta_{1}+3,\ 1}\\
h_{\alpha_{2}+\beta_{2}+1,\ 2} & h_{\alpha_{2}+\beta_{2}+2,\ 2} &
h_{\alpha_{2}+\beta_{2}+3,\ 2}\\
h_{\alpha_{3}+\beta_{3}+1,\ 3} & h_{\alpha_{3}+\beta_{3}+2,\ 3} &
h_{\alpha_{3}+\beta_{3}+3,\ 3}%
\end{array}
\right)  .
\end{align*}
This is a sum with $6$ addends, since the $3$-tuples $\beta\in\mathbb{N}^{3}$
satisfying $\left\vert \beta\right\vert =2$ are
\[
\left(  2,0,0\right)  ,\qquad\left(  0,2,0\right)  ,\qquad\left(
0,0,2\right)  ,\qquad\left(  1,1,0\right)  ,\qquad\left(  1,0,1\right)
,\qquad\left(  0,1,1\right)  .
\]
The second addend (corresponding to the $3$-tuple $\beta=\left(  0,2,0\right)
$) is%
\begin{align*}
&  \operatorname*{rowdet}\left(
\begin{array}
[c]{ccc}%
h_{\alpha_{1}+0+1,\ 1} & h_{\alpha_{1}+0+2,\ 1} & h_{\alpha_{1}+0+3,\ 1}\\
h_{\alpha_{2}+2+1,\ 2} & h_{\alpha_{2}+2+2,\ 2} & h_{\alpha_{2}+2+3,\ 2}\\
h_{\alpha_{3}+0+1,\ 3} & h_{\alpha_{3}+0+2,\ 3} & h_{\alpha_{3}+0+3,\ 3}%
\end{array}
\right) \\
=  &  \operatorname*{rowdet}\left(
\begin{array}
[c]{ccc}%
h_{\alpha_{1}+1,\ 1} & h_{\alpha_{1}+2,\ 1} & h_{\alpha_{1}+3,\ 1}\\
h_{\alpha_{2}+3,\ 2} & h_{\alpha_{2}+4,\ 2} & h_{\alpha_{2}+5,\ 2}\\
h_{\alpha_{3}+1,\ 3} & h_{\alpha_{3}+2,\ 3} & h_{\alpha_{3}+3,\ 3}%
\end{array}
\right)  .
\end{align*}

\end{example}
\end{verlong}

\subsection{The proof}

We shall now prepare for the proof of Theorem \ref{thm.pre-pieri} by
introducing some notations.

First, we introduce a right action of the symmetric group $S_{n}$ on the set
$\mathbb{Z}^{n}$ of all $n$-tuples of integers:

\begin{definition}
\label{def.etapi}Let $n\in\mathbb{N}$. Let $\eta\in\mathbb{Z}^{n}$ and
$\sigma\in S_{n}$. Then, we define $\eta\circ\sigma$ to be the $n$-tuple
$\left(  \eta_{\sigma\left(  1\right)  },\eta_{\sigma\left(  2\right)
},\ldots,\eta_{\sigma\left(  n\right)  }\right)  \in\mathbb{Z}^{n}$.
\end{definition}

Thus, the $n$-tuple $\eta\circ\sigma$ is obtained from $\eta$ by permuting the
entries using the permutation $\sigma$.

The following two properties of this right action are near-obvious:

\begin{proposition}
\label{prop.etapi.len}Let $n\in\mathbb{N}$. Let $\eta\in\mathbb{Z}^{n}$ and
$\sigma\in S_{n}$. Then,
\[
\left\vert \eta\circ\sigma\right\vert =\left\vert \eta\right\vert .
\]

\end{proposition}

\begin{verlong}
\begin{proof}
[Proof of Proposition \ref{prop.etapi.len}.]We have $\sigma\in S_{n}$. Thus,
$\sigma$ is a permutation of the set $\left[  n\right]  $ (since $S_{n}$ was
defined as the set of all permutations of the set $\left[  n\right]  $). In
other words, $\sigma$ is a bijection from $\left[  n\right]  $ to $\left[
n\right]  $. Therefore, $\eta_{\sigma\left(  1\right)  }+\eta_{\sigma\left(
2\right)  }+\cdots+\eta_{\sigma\left(  n\right)  }=\eta_{1}+\eta_{2}%
+\cdots+\eta_{n}$. However, Definition \ref{def.etapi} yields $\eta\circ
\sigma=\left(  \eta_{\sigma\left(  1\right)  },\eta_{\sigma\left(  2\right)
},\ldots,\eta_{\sigma\left(  n\right)  }\right)  $. Hence,%
\begin{align*}
\left\vert \eta\circ\sigma\right\vert  &  =\left\vert \left(  \eta
_{\sigma\left(  1\right)  },\eta_{\sigma\left(  2\right)  },\ldots
,\eta_{\sigma\left(  n\right)  }\right)  \right\vert =\eta_{\sigma\left(
1\right)  }+\eta_{\sigma\left(  2\right)  }+\cdots+\eta_{\sigma\left(
n\right)  }\\
&  \ \ \ \ \ \ \ \ \ \ \left(  \text{by the definition of }\left\vert \left(
\eta_{\sigma\left(  1\right)  },\eta_{\sigma\left(  2\right)  },\ldots
,\eta_{\sigma\left(  n\right)  }\right)  \right\vert \right) \\
&  =\eta_{1}+\eta_{2}+\cdots+\eta_{n}=\left\vert \eta\right\vert
\end{align*}
(since the definition of $\left\vert \eta\right\vert $ yields $\left\vert
\eta\right\vert =\eta_{1}+\eta_{2}+\cdots+\eta_{n}$). This proves Proposition
\ref{prop.etapi.len}.
\end{proof}
\end{verlong}

\begin{proposition}
\label{prop.etapi.dist}Let $n\in\mathbb{N}$. Let $\eta\in\mathbb{Z}^{n}$.
Assume that the $n$ numbers $\eta_{1},\eta_{2},\ldots,\eta_{n}$ are distinct.
Let $\sigma\in S_{n}$ and $\pi\in S_{n}$ be two distinct permutations. Then,
$\eta\circ\sigma\neq\eta\circ\pi$.
\end{proposition}

\begin{verlong}
\begin{proof}
[Proof of Proposition \ref{prop.etapi.dist}.]Both $\sigma$ and $\pi$ are
elements of $S_{n}$, and thus are permutations of $\left[  n\right]  $ (since
$S_{n}$ was defined as the set of all permutations of $\left[  n\right]  $).
We have assumed that $\sigma$ and $\pi$ are distinct. In other words,
$\sigma\neq\pi$. In other words, there exists an $i\in\left[  n\right]  $ such
that $\sigma\left(  i\right)  \neq\pi\left(  i\right)  $. Consider this $i$.
The elements $\sigma\left(  i\right)  $ and $\pi\left(  i\right)  $ of
$\left[  n\right]  $ are distinct (since $\sigma\left(  i\right)  \neq
\pi\left(  i\right)  $).

We have assumed that the $n$ numbers $\eta_{1},\eta_{2},\ldots,\eta_{n}$ are
distinct. In other words, if $p$ and $q$ are two distinct elements of $\left[
n\right]  $, then $\eta_{p}\neq\eta_{q}$. We can apply this to $p=\sigma
\left(  i\right)  $ and $q=\pi\left(  i\right)  $ (since $\sigma\left(
i\right)  $ and $\pi\left(  i\right)  $ are distinct), and thus obtain
$\eta_{\sigma\left(  i\right)  }\neq\eta_{\pi\left(  i\right)  }$.

However, Definition \ref{def.etapi} yields $\eta\circ\sigma=\left(
\eta_{\sigma\left(  1\right)  },\eta_{\sigma\left(  2\right)  },\ldots
,\eta_{\sigma\left(  n\right)  }\right)  $; thus, $\left(  \eta\circ
\sigma\right)  _{i}=\eta_{\sigma\left(  i\right)  }$. The same argument
(applied to $\pi$ instead of $\sigma$) yields $\left(  \eta\circ\pi\right)
_{i}=\eta_{\pi\left(  i\right)  }$. Thus, if we had $\eta\circ\sigma=\eta
\circ\pi$, then we would have%
\begin{align*}
\eta_{\sigma\left(  i\right)  }  &  =\left(  \underbrace{\eta\circ\sigma
}_{=\eta\circ\pi}\right)  _{i}\ \ \ \ \ \ \ \ \ \ \left(  \text{since }\left(
\eta\circ\sigma\right)  _{i}=\eta_{\sigma\left(  i\right)  }\right) \\
&  =\left(  \eta\circ\pi\right)  _{i}=\eta_{\pi\left(  i\right)  };
\end{align*}
but this would contradict $\eta_{\sigma\left(  i\right)  }\neq\eta_{\pi\left(
i\right)  }$. Hence, we cannot have $\eta\circ\sigma=\eta\circ\pi$. In other
words, we have $\eta\circ\sigma\neq\eta\circ\pi$. This proves Proposition
\ref{prop.etapi.dist}.
\end{proof}
\end{verlong}

Finally, we will use the \emph{Iverson bracket notation}:

\begin{definition}
\label{def.iverson}If $\mathcal{A}$ is a logical statement, then $\left[
\mathcal{A}\right]  $ means the \emph{truth value} of $\mathcal{A}$; this is
the integer $%
\begin{cases}
1, & \text{if }\mathcal{A}\text{ is true};\\
0, & \text{if }\mathcal{A}\text{ is false}.
\end{cases}
$
\end{definition}

For example, $\left[  2+2=4\right]  =1$ and $\left[  2+2=5\right]  =0$. The
following easy property of truth values can serve as a warm-up:

\begin{lemma}
\label{lem.iverson.geqk}Let $u\in\mathbb{Z}$ and $k\in\mathbb{Z}$ satisfy
$u\neq k$. Then, $\left[  u\geq k\right]  =\left[  u\geq k+1\right]  $.
\end{lemma}

\begin{verlong}
\begin{proof}
[Proof of Lemma \ref{lem.iverson.geqk}.]We are in one of the following two cases:

\textit{Case 1:} We have $u<k$.

\textit{Case 2:} We have $u\geq k$.\medskip

Let us first consider Case 1. In this case, we have $u<k$. Hence, we don't
have $u\geq k$. Thus, we have $\left[  u\geq k\right]  =0$. Furthermore, we
don't have $u\geq k+1$ (since we have $u<k<k+1$). Hence, we have $\left[
u\geq k+1\right]  =0$. Comparing this with $\left[  u\geq k\right]  =0$, we
obtain $\left[  u\geq k\right]  =\left[  u\geq k+1\right]  $. Thus, Lemma
\ref{lem.iverson.geqk} is proven in Case 1. \medskip

Let us now consider Case 2. In this case, we have $u\geq k$. Combining this
with $u\neq k$, we obtain $u>k$. Thus, $u\geq k+1$ (since $u$ and $k$ are
integers). Therefore, $\left[  u\geq k+1\right]  =1$. Moreover, $\left[  u\geq
k\right]  =1$ (since $u\geq k+1\geq k$). Comparing these two equalities, we
obtain $\left[  u\geq k\right]  =\left[  u\geq k+1\right]  $. Hence, Lemma
\ref{lem.iverson.geqk} is proven in Case 2. \medskip

We have now proven Lemma \ref{lem.iverson.geqk} in both Cases 1 and 2. Thus,
the proof of Lemma \ref{lem.iverson.geqk} is complete.
\end{proof}
\end{verlong}

\begin{verlong}
Next, we will use the following basic fact from the theory of finite sets:

\begin{lemma}
\label{lem.m-dist-objs}Let $m\in\mathbb{N}$. Let $u_{1},u_{2},\ldots,u_{m}$ be
any $m$ objects. Assume that $\left\vert \left\{  u_{1},u_{2},\ldots
,u_{m}\right\}  \right\vert =m$. Then, the $m$ objects $u_{1},u_{2}%
,\ldots,u_{m}$ are distinct.
\end{lemma}

\begin{proof}
[Proof of Lemma \ref{lem.m-dist-objs}.]The set $\left\{  u_{1},u_{2}%
,\ldots,u_{m}\right\}  $ has size $m$ (since $\left\vert \left\{  u_{1}%
,u_{2},\ldots,u_{m}\right\}  \right\vert =m$). If the $m$ objects $u_{1}%
,u_{2},\ldots,u_{m}$ were not distinct, then there would be fewer than $m$
distinct elements in the set $\left\{  u_{1},u_{2},\ldots,u_{m}\right\}  $;
therefore, the set $\left\{  u_{1},u_{2},\ldots,u_{m}\right\}  $ would have
size $<m$. But this would contradict the fact that this set has size $m$
(which is not $<m$). Thus, the $m$ objects $u_{1},u_{2},\ldots,u_{m}$ must be
distinct. This proves Lemma \ref{lem.m-dist-objs}.
\end{proof}
\end{verlong}

Besides the above generalities, our proof of Theorem \ref{thm.pre-pieri} will
rely on some more specific lemmas. The first is an easy exercise on the
pigeonhole principle:

\begin{lemma}
\label{lem.nu-eta}Let $n\in\mathbb{P}$. Let $\nu\in\mathbb{Z}^{n}$ and
$\eta\in\mathbb{Z}^{n}$ satisfy $\left\vert \nu\right\vert =\left\vert
\eta\right\vert $ and%
\[
\left\{  \eta_{1},\eta_{2},\ldots,\eta_{n-1}\right\}  \subseteq\left\{
\nu_{1},\nu_{2},\ldots,\nu_{n}\right\}  \ \ \ \ \ \ \ \ \ \ \text{and}%
\ \ \ \ \ \ \ \ \ \ \left\vert \left\{  \eta_{1},\eta_{2},\ldots,\eta
_{n-1}\right\}  \right\vert =n-1.
\]
Then, there exists some permutation $\pi\in S_{n}$ satisfying $\nu=\eta
\circ\pi$.
\end{lemma}

\begin{vershort}
\begin{proof}
[Proof of Lemma \ref{lem.nu-eta}.]The claim that we must prove can be restated
as \textquotedblleft the $n$-tuple $\nu$ is a permutation of $\eta
$\textquotedblright\ (that is, \textquotedblleft the $n$-tuple $\nu$ can be
obtained from $\eta$ by permuting the entries\textquotedblright). Thus, we can
permute the entries of $\nu$ without loss of generality (since neither the
truth of this claim, nor the integer $\left\vert \nu\right\vert $, nor the set
$\left\{  \nu_{1},\nu_{2},\ldots,\nu_{n}\right\}  $ change when we permute the
entries of $\nu$).

The $n-1$ numbers $\eta_{1},\eta_{2},\ldots,\eta_{n-1}$ are distinct (since
$\left\vert \left\{  \eta_{1},\eta_{2},\ldots,\eta_{n-1}\right\}  \right\vert
=n-1$). Furthermore, each of these $n-1$ numbers appears as an entry in the
$n$-tuple $\nu$ (since $\left\{  \eta_{1},\eta_{2},\ldots,\eta_{n-1}\right\}
\subseteq\left\{  \nu_{1},\nu_{2},\ldots,\nu_{n}\right\}  $). Since these
$n-1$ numbers are distinct, they must therefore appear as \textbf{distinct}
entries in $\nu$ (that is, no two of the $n-1$ numbers $\eta_{1},\eta
_{2},\ldots,\eta_{n-1}$ can appear in the same position of $\nu$). By
permuting the entries of $\nu$, we can therefore ensure that these $n-1$
numbers $\eta_{1},\eta_{2},\ldots,\eta_{n-1}$ are the \textbf{first }$n-1$
entries of $\nu$ in this very order; i.e., that we have%
\begin{equation}
\eta_{i}=\nu_{i}\ \ \ \ \ \ \ \ \ \ \text{for each }i\in\left[  n-1\right]  .
\label{pf.lem.nu-eta.short.wlog}%
\end{equation}
Thus, let us WLOG assume that (\ref{pf.lem.nu-eta.short.wlog}) holds (since we
can permute the entries of $\nu$ without loss of generality). Now, summing up
the equalities (\ref{pf.lem.nu-eta.short.wlog}) over all $i\in\left[
n-1\right]  $, we obtain $\eta_{1}+\eta_{2}+\cdots+\eta_{n-1}=\nu_{1}+\nu
_{2}+\cdots+\nu_{n-1}$. However,%
\begin{align*}
\left\vert \nu\right\vert  &  =\left\vert \eta\right\vert =\eta_{1}+\eta
_{2}+\cdots+\eta_{n}=\underbrace{\left(  \eta_{1}+\eta_{2}+\cdots+\eta
_{n-1}\right)  }_{=\nu_{1}+\nu_{2}+\cdots+\nu_{n-1}}+\eta_{n}\\
&  =\left(  \nu_{1}+\nu_{2}+\cdots+\nu_{n-1}\right)  +\eta_{n}.
\end{align*}
Comparing this with%
\[
\left\vert \nu\right\vert =\nu_{1}+\nu_{2}+\cdots+\nu_{n}=\left(  \nu_{1}%
+\nu_{2}+\cdots+\nu_{n-1}\right)  +\nu_{n},
\]
we obtain $\left(  \nu_{1}+\nu_{2}+\cdots+\nu_{n-1}\right)  +\eta_{n}=\left(
\nu_{1}+\nu_{2}+\cdots+\nu_{n-1}\right)  +\nu_{n}$. Cancelling $\nu_{1}%
+\nu_{2}+\cdots+\nu_{n-1}$, we obtain $\eta_{n}=\nu_{n}$. Thus, the equality
(\ref{pf.lem.nu-eta.short.wlog}) holds not only for each $i\in\left[
n-1\right]  $, but also for $i=n$. Hence, this equality holds for all
$i\in\left[  n\right]  $. In other words, we have $\eta=\nu$. Thus, the
$n$-tuple $\nu$ is a permutation of $\eta$. But this is precisely what we
needed to show. Thus, Lemma \ref{lem.nu-eta} is proven.
\end{proof}
\end{vershort}

\begin{verlong}
\begin{proof}
[Proof of Lemma \ref{lem.nu-eta}.]We know that $n$ is a positive integer
(since $n\in\mathbb{P}$). Hence, $n>0$ and $n-1\in\mathbb{N}$.

The definition of $\left[  n-1\right]  $ yields $\left[  n-1\right]  =\left\{
1,2,\ldots,n-1\right\}  $. The definition of $\left[  n\right]  $ yields
$\left[  n\right]  =\left\{  1,2,\ldots,n\right\}  $. Hence,%
\[
\left[  n\right]  \setminus\left\{  n\right\}  =\left\{  1,2,\ldots,n\right\}
\setminus\left\{  n\right\}  =\left\{  1,2,\ldots,n-1\right\}  =\left[
n-1\right]  .
\]
Thus, $\left[  n-1\right]  =\left[  n\right]  \setminus\left\{  n\right\}
\subseteq\left[  n\right]  $.

Now, recall that $n-1\in\mathbb{N}$ and $\left\vert \left\{  \eta_{1},\eta
_{2},\ldots,\eta_{n-1}\right\}  \right\vert =n-1$. Hence, Lemma
\ref{lem.m-dist-objs} (applied to $m=n-1$ and $u_{i}=\eta_{i}$) shows that the
$n-1$ objects $\eta_{1},\eta_{2},\ldots,\eta_{n-1}$ are distinct. In other
words, if $i$ and $j$ are two distinct elements of $\left[  n-1\right]  $,
then%
\begin{equation}
\eta_{i}\neq\eta_{j}. \label{pf.lem.nu-eta.etapneqetaq}%
\end{equation}

We define a map $f:\left[  n-1\right]  \rightarrow\left[  n\right]  $ as
follows: For each $i\in\left[  n-1\right]  $, we choose some $j\in\left[
n\right]  $ satisfying $\eta_{i}=\nu_{j}$ (indeed, such a $j$ exists, because
\begin{align*}
\eta_{i}  &  \in\left\{  \eta_{1},\eta_{2},\ldots,\eta_{n-1}\right\}
\ \ \ \ \ \ \ \ \ \ \left(  \text{since }i\in\left[  n-1\right]  =\left\{
1,2,\ldots,n-1\right\}  \right) \\
&  \subseteq\left\{  \nu_{1},\nu_{2},\ldots,\nu_{n}\right\}  =\left\{  \nu
_{j}\ \mid\ j\in\underbrace{\left\{  1,2,\ldots,n\right\}  }_{=\left[
n\right]  }\right\}  =\left\{  \nu_{j}\ \mid\ j\in\left[  n\right]  \right\}
\end{align*}
), and we define $f\left(  i\right)  $ to be this $j$. Thus, we have defined a
map $f:\left[  n-1\right]  \rightarrow\left[  n\right]  $.

For each $i\in\left[  n-1\right]  $, the image $f\left(  i\right)  $ of $i$ is
an element of $\left[  n\right]  $ and satisfies%
\begin{equation}
\eta_{i}=\nu_{f\left(  i\right)  } \label{pf.lem.nu-eta.def-f.2}%
\end{equation}
(since $f\left(  i\right)  $ is defined to be a $j\in\left[  n\right]  $
satisfying $\eta_{i}=\nu_{j}$).

The $n-1$ elements $f\left(  1\right)  ,f\left(  2\right)  ,\ldots,f\left(
n-1\right)  $ are distinct\footnote{\textit{Proof.} Let $i$ and $j$ be two
distinct elements of $\left[  n-1\right]  $. We shall show that $f\left(
i\right)  \neq f\left(  j\right)  $.
\par
Indeed, assume the contrary. Thus, $f\left(  i\right)  =f\left(  j\right)  $.
However, (\ref{pf.lem.nu-eta.def-f.2}) yields $\eta_{i}=\nu_{f\left(
i\right)  }=\nu_{f\left(  j\right)  }$ (since $f\left(  i\right)  =f\left(
j\right)  $). On the other hand, (\ref{pf.lem.nu-eta.def-f.2}) (applied to $j$
instead of $i$) yields $\eta_{j}=\nu_{f\left(  j\right)  }$. Comparing these
two equalities, we obtain $\eta_{i}=\eta_{j}$. However,
(\ref{pf.lem.nu-eta.etapneqetaq}) yields $\eta_{i}\neq\eta_{j}$. This
contradicts $\eta_{i}=\eta_{j}$. This contradiction shows that our assumption
was false. Hence, $f\left(  i\right)  \neq f\left(  j\right)  $.
\par
Forget that we fixed $i$ and $j$. We thus have shown that $f\left(  i\right)
\neq f\left(  j\right)  $ whenever $i$ and $j$ are two distinct elements of
$\left[  n-1\right]  $. In other words, the $n-1$ elements $f\left(  1\right)
,f\left(  2\right)  ,\ldots,f\left(  n-1\right)  $ are distinct.}. Hence,
\[
\left\vert \left\{  f\left(  1\right)  ,f\left(  2\right)  ,\ldots,f\left(
n-1\right)  \right\}  \right\vert =n-1.
\]
Clearly, $\left\{  f\left(  1\right)  ,f\left(  2\right)  ,\ldots,f\left(
n-1\right)  \right\}  $ is a subset of $\left[  n\right]  $ (since $f$ is a
map from $\left[  n-1\right]  $ to $\left[  n\right]  $). Thus,%
\begin{align*}
\left\vert \left[  n\right]  \setminus\left\{  f\left(  1\right)  ,f\left(
2\right)  ,\ldots,f\left(  n-1\right)  \right\}  \right\vert  &
=\underbrace{\left\vert \left[  n\right]  \right\vert }_{=n}%
-\underbrace{\left\vert \left\{  f\left(  1\right)  ,f\left(  2\right)
,\ldots,f\left(  n-1\right)  \right\}  \right\vert }_{=n-1}\\
&  =n-\left(  n-1\right)  =1.
\end{align*}
In other words, $\left[  n\right]  \setminus\left\{  f\left(  1\right)
,f\left(  2\right)  ,\ldots,f\left(  n-1\right)  \right\}  $ is a $1$-element
set. In other words,
\[
\left[  n\right]  \setminus\left\{  f\left(  1\right)  ,f\left(  2\right)
,\ldots,f\left(  n-1\right)  \right\}  =\left\{  p\right\}
\]
for some element $p$. Consider this $p$. We have $p\in\left\{  p\right\}
=\left[  n\right]  \setminus\left\{  f\left(  1\right)  ,f\left(  2\right)
,\ldots,f\left(  n-1\right)  \right\}  $. In other words,
\[
p\in\left[  n\right]  \ \ \ \ \ \ \ \ \ \ \text{and}%
\ \ \ \ \ \ \ \ \ \ p\notin\left\{  f\left(  1\right)  ,f\left(  2\right)
,\ldots,f\left(  n-1\right)  \right\}  .
\]

Now, for each $i\in\left[  n\right]  $, we have%
\begin{equation}%
\begin{cases}
f\left(  i\right)  , & \text{if }i\neq n;\\
p, & \text{if }i=n
\end{cases}
\ \ \in\left[  n\right]  \label{pf.lem.nu-eta.def-f.in-n}%
\end{equation}
\footnote{\textit{Proof of (\ref{pf.lem.nu-eta.def-f.in-n}):} Let $i\in\left[
n\right]  $. We must prove (\ref{pf.lem.nu-eta.def-f.in-n}). We are in one of
the following two cases:
\par
\textit{Case 1:} We have $i\neq n$.
\par
\textit{Case 2:} We have $i=n$.
\par
Let us first consider Case 1. In this case, we have $i\neq n$. Combining
$i\in\left[  n\right]  =\left\{  1,2,\ldots,n\right\}  $ with $i\neq n$, we
obtain%
\[
i\in\left\{  1,2,\ldots,n\right\}  \setminus\left\{  n\right\}  =\left\{
1,2,\ldots,n-1\right\}  =\left[  n-1\right]
\]
(since $\left[  n-1\right]  $ is defined as $\left\{  1,2,\ldots,n-1\right\}
$). Hence, $f\left(  i\right)  \in\left[  n\right]  $ (since $f$ is a map from
$\left[  n-1\right]  $ to $\left[  n\right]  $). Now, we have $i\neq n$ and
thus $%
\begin{cases}
f\left(  i\right)  , & \text{if }i\neq n;\\
p, & \text{if }i=n
\end{cases}
\ \ =f\left(  i\right)  \in\left[  n\right]  $. Thus,
(\ref{pf.lem.nu-eta.def-f.in-n}) is proved in Case 1.
\par
Let us now consider Case 2. In this case, we have $i=n$. Hence, $%
\begin{cases}
f\left(  i\right)  , & \text{if }i\neq n;\\
p, & \text{if }i=n
\end{cases}
\ \ =p\in\left[  n\right]  $. Thus, (\ref{pf.lem.nu-eta.def-f.in-n}) is proved
in Case 2.
\par
We have now proved (\ref{pf.lem.nu-eta.def-f.in-n}) in both Cases 1 and 2.
Thus, the proof of (\ref{pf.lem.nu-eta.def-f.in-n}) is complete.}. Hence, we
can define a map $g:\left[  n\right]  \rightarrow\left[  n\right]  $ by
setting%
\[
g\left(  i\right)  =%
\begin{cases}
f\left(  i\right)  , & \text{if }i\neq n;\\
p, & \text{if }i=n
\end{cases}
\ \ \ \ \ \ \ \ \ \ \text{for each }i\in\left[  n\right]  .
\]
Consider this map $g$.

The map $g$ is injective\footnote{\textit{Proof.} Let $i$ and $j$ be two
distinct elements of $\left[  n\right]  $. We shall show that $g\left(
i\right)  \neq g\left(  j\right)  $.
\par
Indeed, assume the contrary. Thus, $g\left(  i\right)  =g\left(  j\right)  $.
We are looking to find a contradiction.
\par
The situation is symmetric in $i$ and $j$; thus, we WLOG assume that $i\leq j$
(since otherwise, we can just swap $i$ with $j$). However, $i\neq j$ (since
$i$ and $j$ are distinct). Thus, $i<j$ (since $i\leq j$). On the other hand,
$j\in\left[  n\right]  =\left\{  1,2,\ldots,n\right\}  $, so that $j\leq n$.
Thus, $i<j\leq n$, so that $i\neq n$. Now, the definition of $g$ yields
\[
g\left(  i\right)  =%
\begin{cases}
f\left(  i\right)  , & \text{if }i\neq n;\\
p, & \text{if }i=n
\end{cases}
\ \ =f\left(  i\right)  \ \ \ \ \ \ \ \ \ \ \left(  \text{since }i\neq
n\right)  .
\]
\par
Combining $i\in\left[  n\right]  $ with $i\neq n$, we obtain $i\in\left[
n\right]  \setminus\left\{  n\right\}  =\left[  n-1\right]  =\left\{
1,2,\ldots,n-1\right\}  $. Note that%
\[
g\left(  i\right)  =f\left(  i\right)  \in\left\{  f\left(  1\right)
,f\left(  2\right)  ,\ldots,f\left(  n-1\right)  \right\}
\ \ \ \ \ \ \ \ \ \ \left(  \text{since }i\in\left\{  1,2,\ldots,n-1\right\}
\right)  .
\]
\par
If we had $j=n$, then we would have
\begin{align*}
g\left(  i\right)   &  =g\left(  j\right)  =%
\begin{cases}
f\left(  j\right)  , & \text{if }j\neq n;\\
p, & \text{if }j=n
\end{cases}
\ \ \ \ \ \ \ \ \ \ \left(  \text{by the definition of }g\right) \\
&  =p\ \ \ \ \ \ \ \ \ \ \left(  \text{since }j=n\right) \\
&  \notin\left\{  f\left(  1\right)  ,f\left(  2\right)  ,\ldots,f\left(
n-1\right)  \right\}  ,
\end{align*}
which would contradict $g\left(  i\right)  \in\left\{  f\left(  1\right)
,f\left(  2\right)  ,\ldots,f\left(  n-1\right)  \right\}  $. Hence, we cannot
have $j=n$. Thus, we have $j\neq n$. Combining $j\in\left[  n\right]  $ with
$j\neq n$, we obtain $j\in\left[  n\right]  \setminus\left\{  n\right\}
=\left[  n-1\right]  $. Thus, both $i$ and $j$ are elements of $\left[
n-1\right]  $ (since $i\in\left[  n-1\right]  $ and $j\in\left[  n-1\right]
$). These elements $i$ and $j$ are furthermore distinct (as we know). Hence,
we have $f\left(  i\right)  \neq f\left(  j\right)  $ (since the $n-1$
elements $f\left(  1\right)  ,f\left(  2\right)  ,\ldots,f\left(  n-1\right)
$ are distinct). However, from $g\left(  i\right)  =f\left(  i\right)  $, we
obtain%
\begin{align*}
f\left(  i\right)   &  =g\left(  i\right)  =g\left(  j\right)  =%
\begin{cases}
f\left(  j\right)  , & \text{if }j\neq n;\\
p, & \text{if }j=n
\end{cases}
\ \ \ \ \ \ \ \ \ \ \left(  \text{by the definition of }g\right) \\
&  =f\left(  j\right)  \ \ \ \ \ \ \ \ \ \ \left(  \text{since }j\neq
n\right)  ,
\end{align*}
and this contradicts $f\left(  i\right)  \neq f\left(  j\right)  $. This
contradiction shows that our assumption was false. Hence, $g\left(  i\right)
\neq g\left(  j\right)  $ is proven.
\par
Now, forget that we fixed $i$ and $j$. We thus have shown that if $i$ and $j$
are two distinct elements of $\left[  n\right]  $, then $g\left(  i\right)
\neq g\left(  j\right)  $. In other words, the map $g$ is injective.}.
However, a well-known fact from set theory (actually one of the versions of
the pigeonhole principle) says that any injective map $\phi$ from a finite set
$X$ to $X$ must be a permutation of $X$. Applying this to $X=\left[  n\right]
$ and $\phi=g$, we conclude that $g$ must be a permutation of $\left[
n\right]  $ (since $g$ is an injective map from the finite set $\left[
n\right]  $ to $\left[  n\right]  $). Therefore, the inverse $g^{-1}$ of $g$
is well-defined and also is a permutation of $\left[  n\right]  $. In other
words, $g^{-1}\in S_{n}$ (since $S_{n}$ is the set of all permutations of
$\left[  n\right]  $).

Next, we shall show that
\begin{equation}
\eta_{i}=\nu_{g\left(  i\right)  }\ \ \ \ \ \ \ \ \ \ \text{for each }%
i\in\left[  n-1\right]  . \label{pf.lem.nu-eta.nu-eta-g-1}%
\end{equation}

[\textit{Proof of (\ref{pf.lem.nu-eta.nu-eta-g-1}):} Let $i\in\left[
n-1\right]  $. Thus, $i\leq n-1<n$, so that $i\neq n$. However, $i\in\left[
n-1\right]  \subseteq\left[  n\right]  $. The definition of $g$ yields
$g\left(  i\right)  =%
\begin{cases}
f\left(  i\right)  , & \text{if }i\neq n;\\
p, & \text{if }i=n
\end{cases}
\ \ =f\left(  i\right)  $ (since $i\neq n$). On the other hand,
(\ref{pf.lem.nu-eta.def-f.2}) yields $\eta_{i}=\nu_{f\left(  i\right)  }$.
This rewrites as $\eta_{i}=\nu_{g\left(  i\right)  }$ (since $g\left(
i\right)  =f\left(  i\right)  $). This proves (\ref{pf.lem.nu-eta.nu-eta-g-1}%
).] \medskip

Now, recall that $g$ is a permutation of $\left[  n\right]  $. In other words,
$g$ is a bijection from $\left[  n\right]  $ to $\left[  n\right]  $. Hence,
we can substitute $i$ for $g\left(  i\right)  $ in the sum $\sum_{i\in\left[
n\right]  }\nu_{g\left(  i\right)  }$. We thus obtain%
\[
\sum_{i\in\left[  n\right]  }\nu_{g\left(  i\right)  }=\underbrace{\sum
_{i\in\left[  n\right]  }}_{=\sum_{i=1}^{n}}\nu_{i}=\sum_{i=1}^{n}\nu_{i}%
=\nu_{1}+\nu_{2}+\cdots+\nu_{n}=\left\vert \nu\right\vert
\]
(since $\left\vert \nu\right\vert $ was defined to be $\nu_{1}+\nu_{2}%
+\cdots+\nu_{n}$). Hence,%
\begin{align*}
\left\vert \nu\right\vert  &  =\underbrace{\sum_{i\in\left[  n\right]  }%
}_{=\sum_{i=1}^{n}}\nu_{g\left(  i\right)  }=\sum_{i=1}^{n}\nu_{g\left(
i\right)  }=\nu_{g\left(  n\right)  }+\sum_{i=1}^{n-1}\underbrace{\nu
_{g\left(  i\right)  }}_{\substack{=\eta_{i}\\\text{(by
(\ref{pf.lem.nu-eta.nu-eta-g-1}))}}}\ \ \ \ \ \ \ \ \ \ \left(
\begin{array}
[c]{c}%
\text{here, we have split off the}\\
\text{addend for }i=n\text{ from}\\
\text{the sum (since }n>0\text{)}%
\end{array}
\right) \\
&  =\nu_{g\left(  n\right)  }+\sum_{i=1}^{n-1}\eta_{i}.
\end{align*}
Comparing this with%
\begin{align*}
\left\vert \nu\right\vert  &  =\left\vert \eta\right\vert =\eta_{1}+\eta
_{2}+\cdots+\eta_{n}\ \ \ \ \ \ \ \ \ \ \left(  \text{by the definition of
}\left\vert \eta\right\vert \right) \\
&  =\sum_{i=1}^{n}\eta_{i}=\eta_{n}+\sum_{i=1}^{n-1}\eta_{i}%
\ \ \ \ \ \ \ \ \ \ \left(
\begin{array}
[c]{c}%
\text{here, we have split off }\\
\text{the addend for }i=n\text{ from}\\
\text{the sum (since }n>0\text{)}%
\end{array}
\right)  ,
\end{align*}
we obtain%
\[
\eta_{n}+\sum_{i=1}^{n-1}\eta_{i}=\nu_{g\left(  n\right)  }+\sum_{i=1}%
^{n-1}\eta_{i}.
\]
Subtracting $\sum_{i=1}^{n-1}\eta_{i}$ from both sides of this equality, we
obtain%
\begin{equation}
\eta_{n}=\nu_{g\left(  n\right)  }. \label{pf.lem.nu-eta.nu-eta-g-n}%
\end{equation}

Now, we can easily conclude that%
\begin{equation}
\eta_{i}=\nu_{g\left(  i\right)  }\ \ \ \ \ \ \ \ \ \ \text{for each }%
i\in\left[  n\right]  . \label{pf.lem.nu-eta.nu-eta-g}%
\end{equation}

[\textit{Proof of (\ref{pf.lem.nu-eta.nu-eta-g}):} Let $i\in\left[  n\right]
$. We then must show that $\eta_{i}=\nu_{g\left(  i\right)  }$. If $i=n$, then
this follows from (\ref{pf.lem.nu-eta.nu-eta-g-n}). Thus, for the rest of this
proof, we WLOG assume that $i\neq n$. Combining $i\in\left[  n\right]  $ and
$i\neq n$, we obtain $i\in\left[  n\right]  \setminus\left\{  n\right\}
=\left[  n-1\right]  $. Hence, (\ref{pf.lem.nu-eta.nu-eta-g-1}) yields
$\eta_{i}=\nu_{g\left(  i\right)  }$. This proves
(\ref{pf.lem.nu-eta.nu-eta-g}).] \medskip

Consequently, we obtain that%
\begin{equation}
\nu_{j}=\eta_{g^{-1}\left(  j\right)  }\ \ \ \ \ \ \ \ \ \ \text{for each
}j\in\left[  n\right]  . \label{pf.lem.nu-eta.eta-nu-g}%
\end{equation}

[\textit{Proof of (\ref{pf.lem.nu-eta.eta-nu-g}):} Let $j\in\left[  n\right]
$. Then, $g^{-1}\left(  j\right)  \in\left[  n\right]  $ (since $g^{-1}$ is a
map from $\left[  n\right]  $ to $\left[  n\right]  $). Hence,
(\ref{pf.lem.nu-eta.nu-eta-g}) (applied to $i=g^{-1}\left(  j\right)  $)
yields $\eta_{g^{-1}\left(  j\right)  }=\nu_{g\left(  g^{-1}\left(  j\right)
\right)  }=\nu_{j}$ (since $g\left(  g^{-1}\left(  j\right)  \right)  =j$). In
other words, $\nu_{j}=\eta_{g^{-1}\left(  j\right)  }$. This proves
(\ref{pf.lem.nu-eta.eta-nu-g}).] \medskip

We can rewrite (\ref{pf.lem.nu-eta.eta-nu-g}) as follows:%
\[
\left(  \nu_{1},\nu_{2},\ldots,\nu_{n}\right)  =\left(  \eta_{g^{-1}\left(
1\right)  },\eta_{g^{-1}\left(  2\right)  },\ldots,\eta_{g^{-1}\left(
n\right)  }\right)  .
\]
However, Definition \ref{def.etapi} yields%
\[
\eta\circ g^{-1}=\left(  \eta_{g^{-1}\left(  1\right)  },\eta_{g^{-1}\left(
2\right)  },\ldots,\eta_{g^{-1}\left(  n\right)  }\right)  .
\]
Comparing these two equalities, we obtain%
\[
\eta\circ g^{-1}=\left(  \nu_{1},\nu_{2},\ldots,\nu_{n}\right)  =\nu.
\]
In other words, $\nu=\eta\circ g^{-1}$. Hence, there exists some permutation
$\pi\in S_{n}$ satisfying $\nu=\eta\circ\pi$ (namely, $\pi=g^{-1}$). This
proves Lemma \ref{lem.nu-eta}.
\end{proof}
\end{verlong}

Our second lemma is about an integer determinant:

\begin{lemma}
\label{lem.nu-det}Let $n\in\mathbb{P}$ and $p\in\mathbb{N}$. Let $\eta$ be the
$n$-tuple
\[
\left(  1,2,\ldots,n\right)  +\left(  \underbrace{0,0,\ldots,0}_{n-1\text{
zeroes}},p\right)  =\left(  1,2,\ldots,n-1,n+p\right)  \in\mathbb{Z}^{n}.
\]
Let $\nu\in\mathbb{Z}^{n}$ be an $n$-tuple satisfying $\left\vert
\nu\right\vert =\left\vert \eta\right\vert $. Then,
\begin{equation}
\det\left(  \left(  \left[  \nu_{i}\geq j\right]  \right)  _{i,j\in\left[
n\right]  }\right)  =\sum_{\substack{\sigma\in S_{n};\\\nu=\eta\circ\sigma
}}\left(  -1\right)  ^{\sigma}. \label{eq.lem.nu-det.clm}%
\end{equation}

\end{lemma}

Note that the matrix $\left(  \left[  \nu_{i}\geq j\right]  \right)
_{i,j\in\left[  n\right]  }$ in (\ref{eq.lem.nu-det.clm}) is a matrix with
integer entries; thus, its determinant is a well-defined integer.

Before we prove Lemma \ref{lem.nu-det}, a remark is in order:

\begin{remark}
The sum $\sum_{\substack{\sigma\in S_{n};\\\nu=\eta\circ\sigma}}\left(
-1\right)  ^{\sigma}$ on the right hand side of (\ref{eq.lem.nu-det.clm})
always has either no addends or only one addend. (Indeed, it is easy to see
that the $n$-tuples $\eta\circ\sigma$ for different $\sigma\in S_{n}$ are
distinct; thus, no more than one of these $n$-tuples can equal $\nu$.) Thus,
this sum can be rewritten as
\[%
\begin{cases}
\left(  -1\right)  ^{\sigma}, & \text{if }\nu=\eta\circ\sigma\text{ for some
}\sigma\in S_{n};\\
0, & \text{otherwise.}%
\end{cases}
\]

\end{remark}

\begin{proof}
[Proof of Lemma \ref{lem.nu-det}.]The definition of $\eta$ yields%
\begin{equation}
\eta=\left(  1,2,\ldots,n\right)  +\left(  \underbrace{0,0,\ldots
,0}_{n-1\text{ zeroes}},p\right)  =\left(  1,2,\ldots,n-1,n+p\right)
.\nonumber
\end{equation}
Thus,%
\begin{equation}
\left(  \eta_{1},\eta_{2},\ldots,\eta_{n}\right)  =\eta=\left(  1,2,\ldots
,n-1,n+p\right)  . \label{pf.lem.nu-det.eta=2}%
\end{equation}
In other words, we have%
\begin{equation}
\eta_{i}=i\ \ \ \ \ \ \ \ \ \ \text{for each }i\in\left[  n-1\right]
\label{pf.lem.nu-det.etai=}%
\end{equation}
and%
\begin{equation}
\eta_{n}=n+p. \label{pf.lem.nu-det.etan=}%
\end{equation}

\begin{vershort}
\noindent It follows easily that the $n$ numbers $\eta_{1},\eta_{2}%
,\ldots,\eta_{n}$ are distinct (since $p\in\mathbb{N}$).
\end{vershort}

\begin{verlong}
We have $p\geq0$ (since $p\in\mathbb{N}$), so that $n+\underbrace{p}_{\geq
0}\geq n>n-1$. Thus, we have the following chain of inequalities:%
\[
1<2<\cdots<n-1<n+p
\]
(indeed, $n-1<n+p$ follows from $n+p>n-1$, whereas the rest of this chain of
inequalities is obvious). In view of (\ref{pf.lem.nu-det.eta=2}), we can
rewrite this chain of inequalities as follows:
\[
\eta_{1}<\eta_{2}<\cdots<\eta_{n}.
\]
Hence, the $n$ numbers $\eta_{1},\eta_{2},\ldots,\eta_{n}$ are distinct.
\end{verlong}

\begin{verlong}
We have $n-1\in\mathbb{N}$ (since $n\in\mathbb{P}$). Hence, $\left[
n-1\right]  =\left\{  1,2,\ldots,n-1\right\}  $ (by the definition of $\left[
n-1\right]  $). \medskip
\end{verlong}

We are in one of the following two cases:

\textit{Case 1:} We have $\left[  n-1\right]  \not \subseteq \left\{  \nu
_{1},\nu_{2},\ldots,\nu_{n}\right\}  $.

\textit{Case 2:} We have $\left[  n-1\right]  \subseteq\left\{  \nu_{1}%
,\nu_{2},\ldots,\nu_{n}\right\}  $. \medskip

\begin{vershort}
Let us first consider Case 1. In this case, we have $\left[  n-1\right]
\not \subseteq \left\{  \nu_{1},\nu_{2},\ldots,\nu_{n}\right\}  $. In other
words, there exists some $k\in\left[  n-1\right]  $ such that $k\notin\left\{
\nu_{1},\nu_{2},\ldots,\nu_{n}\right\}  $. Consider this $k$. For each
$i\in\left[  n\right]  $, we have $\nu_{i}\neq k$ (since $k\notin\left\{
\nu_{1},\nu_{2},\ldots,\nu_{n}\right\}  $) and therefore%
\[
\left[  \nu_{i}\geq k\right]  =\left[  \nu_{i}\geq k+1\right]
\]
(by Lemma \ref{lem.iverson.geqk}, applied to $u=\nu_{i}$). This shows that the
$k$-th and the $\left(  k+1\right)  $-st columns of the matrix $\left(
\left[  \nu_{i}\geq j\right]  \right)  _{i,j\in\left[  n\right]  }$ are equal.
Hence, this matrix $\left(  \left[  \nu_{i}\geq j\right]  \right)
_{i,j\in\left[  n\right]  }$ has two equal columns; therefore, the determinant
of this matrix is $0$. In other words,
\begin{equation}
\det\left(  \left(  \left[  \nu_{i}\geq j\right]  \right)  _{i,j\in\left[
n\right]  }\right)  =0. \label{pf.lem.nu-det.c1.short.det=0}%
\end{equation}

On the other hand, $k\in\left[  n-1\right]  $ entails $\eta_{k}=k$ (by
(\ref{pf.lem.nu-det.etai=})). Thus, the $n$-tuple $\eta$ contains the entry
$k$. However, the $n$-tuple $\nu$ does not (since $k\notin\left\{  \nu_{1}%
,\nu_{2},\ldots,\nu_{n}\right\}  $). Thus, the $n$-tuple $\nu$ is not a
permutation of the $n$-tuple $\eta$. In other words, there exists no
$\sigma\in S_{n}$ satisfying $\nu=\eta\circ\sigma$. Hence,%
\[
\sum_{\substack{\sigma\in S_{n};\\\nu=\eta\circ\sigma}}\left(  -1\right)
^{\sigma}=\left(  \text{empty sum}\right)  =0.
\]
Comparing this with (\ref{pf.lem.nu-det.c1.short.det=0}), we obtain
$\det\left(  \left(  \left[  \nu_{i}\geq j\right]  \right)  _{i,j\in\left[
n\right]  }\right)  =\sum_{\substack{\sigma\in S_{n};\\\nu=\eta\circ\sigma
}}\left(  -1\right)  ^{\sigma}$. Thus, Lemma \ref{lem.nu-det} is proved in
Case 1. \medskip
\end{vershort}

\begin{verlong}
Let us first consider Case 1. In this case, we have $\left[  n-1\right]
\not \subseteq \left\{  \nu_{1},\nu_{2},\ldots,\nu_{n}\right\}  $. In other
words, there exists some $k\in\left[  n-1\right]  $ such that $k\notin\left\{
\nu_{1},\nu_{2},\ldots,\nu_{n}\right\}  $. Consider this $k$. We have
$k\in\left[  n-1\right]  =\left\{  1,2,\ldots,n-1\right\}  $; thus, both $k$
and $k+1$ are elements of $\left\{  1,2,\ldots,n\right\}  $. Hence, the
$n\times n$-matrix $\left(  \left[  \nu_{i}\geq j\right]  \right)
_{i,j\in\left[  n\right]  }$ has both a $k$-th column and a $\left(
k+1\right)  $-st column.

Now, it is easy to see that
\begin{equation}
\left[  \nu_{i}\geq k\right]  =\left[  \nu_{i}\geq k+1\right]
\ \ \ \ \ \ \ \ \ \ \text{for each }i\in\left[  n\right]  .
\label{pf.lem.nu-det.c1.coleq}%
\end{equation}

[\textit{Proof of (\ref{pf.lem.nu-det.c1.coleq}):} Let $i\in\left[  n\right]
$. Then, $\nu_{i}\in\mathbb{Z}$ (since $\nu\in\mathbb{Z}^{n}$). If we had
$\nu_{i}=k$, then we would have $k=\nu_{i}\in\left\{  \nu_{1},\nu_{2}%
,\ldots,\nu_{n}\right\}  $, which would contradict $k\notin\left\{  \nu
_{1},\nu_{2},\ldots,\nu_{n}\right\}  $. Thus, we cannot have $\nu_{i}=k$.
Hence, we have $\nu_{i}\neq k$. Therefore, Lemma \ref{lem.iverson.geqk}
(applied to $u=\nu_{i}$) yields $\left[  \nu_{i}\geq k\right]  =\left[
\nu_{i}\geq k+1\right]  $. This proves (\ref{pf.lem.nu-det.c1.coleq}).]
\medskip

The equality (\ref{pf.lem.nu-det.c1.coleq}) shows that each entry in the
$k$-th column of the matrix $\left(  \left[  \nu_{i}\geq j\right]  \right)
_{i,j\in\left[  n\right]  }$ equals the corresponding entry in the $\left(
k+1\right)  $-st column of this matrix (because the $i$-th entry in the $k$-th
column is $\left[  \nu_{i}\geq k\right]  $, whereas the corresponding entry in
the $\left(  k+1\right)  $-st column is $\left[  \nu_{i}\geq k+1\right]  $).
In other words, the $k$-th and the $\left(  k+1\right)  $-st columns of the
matrix $\left(  \left[  \nu_{i}\geq j\right]  \right)  _{i,j\in\left[
n\right]  }$ are equal. Hence, this matrix $\left(  \left[  \nu_{i}\geq
j\right]  \right)  _{i,j\in\left[  n\right]  }$ has two equal columns.
However, it is well-known that if a square matrix\footnote{In this proof, the
word \textquotedblleft matrix\textquotedblright\ always means a matrix with
integer entries.} has two equal columns, then its determinant is $0$. Hence,
the determinant of the matrix $\left(  \left[  \nu_{i}\geq j\right]  \right)
_{i,j\in\left[  n\right]  }$ is $0$ (since this matrix has two equal columns).
In other words,
\begin{equation}
\det\left(  \left(  \left[  \nu_{i}\geq j\right]  \right)  _{i,j\in\left[
n\right]  }\right)  =0. \label{pf.lem.nu-det.c1.det=0}%
\end{equation}

On the other hand, there exists no $\sigma\in S_{n}$ satisfying $\nu=\eta
\circ\sigma$\ \ \ \ \footnote{\textit{Proof.} Let $\sigma\in S_{n}$ satisfy
$\nu=\eta\circ\sigma$. We shall obtain a contradiction.
\par
Indeed, we have $\sigma\in S_{n}$. In other words, $\sigma$ is a permutation
of $\left[  n\right]  $ (since $S_{n}$ was defined as the set of all
permutations of $\left[  n\right]  $). Hence,
\begin{align*}
\left\{  \eta_{\sigma\left(  1\right)  },\eta_{\sigma\left(  2\right)
},\ldots,\eta_{\sigma\left(  n\right)  }\right\}   &  =\left\{  \eta_{1}%
,\eta_{2},\ldots,\eta_{n}\right\} \\
&  =\left\{  1,2,\ldots,n-1,n+p\right\}  \ \ \ \ \ \ \ \ \ \ \left(  \text{by
(\ref{pf.lem.nu-det.eta=2})}\right) \\
&  \supseteq\left\{  1,2,\ldots,n-1\right\}  =\left[  n-1\right]
\end{align*}
(since $\left[  n-1\right]  =\left\{  1,2,\ldots,n-1\right\}  $) and therefore%
\begin{equation}
\left[  n-1\right]  \subseteq\left\{  \eta_{\sigma\left(  1\right)  }%
,\eta_{\sigma\left(  2\right)  },\ldots,\eta_{\sigma\left(  n\right)
}\right\}  . \label{pf.lem.nu-det.c1.fn3.1}%
\end{equation}
However, $\left(  \nu_{1},\nu_{2},\ldots,\nu_{n}\right)  =\nu=\eta\circ
\sigma=\left(  \eta_{\sigma\left(  1\right)  },\eta_{\sigma\left(  2\right)
},\ldots,\eta_{\sigma\left(  n\right)  }\right)  $ (by Definition
\ref{def.etapi}). Hence, $\left\{  \nu_{1},\nu_{2},\ldots,\nu_{n}\right\}
=\left\{  \eta_{\sigma\left(  1\right)  },\eta_{\sigma\left(  2\right)
},\ldots,\eta_{\sigma\left(  n\right)  }\right\}  $. In light of this, we can
rewrite (\ref{pf.lem.nu-det.c1.fn3.1}) as $\left[  n-1\right]  \subseteq
\left\{  \nu_{1},\nu_{2},\ldots,\nu_{n}\right\}  $. But this contradicts
$\left[  n-1\right]  \not \subseteq \left\{  \nu_{1},\nu_{2},\ldots,\nu
_{n}\right\}  $.
\par
Forget that we fixed $\sigma$. We thus have found a contradiction for each
$\sigma\in S_{n}$ satisfying $\nu=\eta\circ\sigma$. Hence, there exists no
such $\sigma$. Qed.}. Hence, the sum $\sum_{\substack{\sigma\in S_{n}%
;\\\nu=\eta\circ\sigma}}\left(  -1\right)  ^{\sigma}$ is empty. Thus,%
\[
\sum_{\substack{\sigma\in S_{n};\\\nu=\eta\circ\sigma}}\left(  -1\right)
^{\sigma}=\left(  \text{empty sum}\right)  =0.
\]
Comparing this with (\ref{pf.lem.nu-det.c1.det=0}), we obtain $\det\left(
\left(  \left[  \nu_{i}\geq j\right]  \right)  _{i,j\in\left[  n\right]
}\right)  =\sum_{\substack{\sigma\in S_{n};\\\nu=\eta\circ\sigma}}\left(
-1\right)  ^{\sigma}$. Thus, Lemma \ref{lem.nu-det} is proved in Case 1.
\medskip
\end{verlong}

Let us now consider Case 2. In this case, we have $\left[  n-1\right]
\subseteq\left\{  \nu_{1},\nu_{2},\ldots,\nu_{n}\right\}  $.

\begin{vershort}
However, (\ref{pf.lem.nu-det.etai=}) shows that
\[
\left\{  \eta_{1},\eta_{2},\ldots,\eta_{n-1}\right\}  =\left\{  1,2,\ldots
,n-1\right\}  =\left[  n-1\right]  \subseteq\left\{  \nu_{1},\nu_{2}%
,\ldots,\nu_{n}\right\}  .
\]
Moreover, from $\left\{  \eta_{1},\eta_{2},\ldots,\eta_{n-1}\right\}
=\left\{  1,2,\ldots,n-1\right\}  $, we obtain $\left\vert \left\{  \eta
_{1},\eta_{2},\ldots,\eta_{n-1}\right\}  \right\vert =\left\vert \left\{
1,2,\ldots,n-1\right\}  \right\vert =n-1$. Hence, Lemma \ref{lem.nu-eta}
yields that there exists some permutation $\pi\in S_{n}$ satisfying $\nu
=\eta\circ\pi$. Consider this $\pi$.
\end{vershort}

\begin{verlong}
However, (\ref{pf.lem.nu-det.etai=}) shows that $\left(  \eta_{1},\eta
_{2},\ldots,\eta_{n-1}\right)  =\left(  1,2,\ldots,n-1\right)  $. Therefore,
\[
\left\{  \eta_{1},\eta_{2},\ldots,\eta_{n-1}\right\}  =\left\{  1,2,\ldots
,n-1\right\}  =\left[  n-1\right]  \subseteq\left\{  \nu_{1},\nu_{2}%
,\ldots,\nu_{n}\right\}  .
\]
Moreover, from $\left\{  \eta_{1},\eta_{2},\ldots,\eta_{n-1}\right\}
=\left\{  1,2,\ldots,n-1\right\}  $, we obtain $\left\vert \left\{  \eta
_{1},\eta_{2},\ldots,\eta_{n-1}\right\}  \right\vert =\left\vert \left\{
1,2,\ldots,n-1\right\}  \right\vert =n-1$ (since $n-1\in\mathbb{N}$). Hence,
Lemma \ref{lem.nu-eta} yields that there exists some permutation $\pi\in
S_{n}$ satisfying $\nu=\eta\circ\pi$. Consider this $\pi$.
\end{verlong}

\begin{vershort}
Thus, $\pi$ is a permutation $\sigma\in S_{n}$ satisfying $\nu=\eta\circ
\sigma$. Furthermore, it is easy to see that $\pi$ is the \textbf{only} such
permutation $\sigma$ (because the $n$ numbers $\eta_{1},\eta_{2},\ldots
,\eta_{n}$ are distinct)\footnote{Here is the argument in more detail: Recall
that the $n$ numbers $\eta_{1},\eta_{2},\ldots,\eta_{n}$ are distinct.
Therefore, if $\sigma\in S_{n}$ is a permutation distinct from $\pi$, then
Proposition \ref{prop.etapi.dist} shows that $\eta\circ\sigma\neq\eta\circ
\pi=\nu$, so that $\nu\neq\eta\circ\sigma$. Hence, the only permutation
$\sigma\in S_{n}$ satisfying $\nu=\eta\circ\sigma$ is $\pi$.}. Hence, the sum
$\sum_{\substack{\sigma\in S_{n};\\\nu=\eta\circ\sigma}}\left(  -1\right)
^{\sigma}$ has only one addend, namely the addend for $\sigma=\pi$. Thus,
\begin{equation}
\sum_{\substack{\sigma\in S_{n};\\\nu=\eta\circ\sigma}}\left(  -1\right)
^{\sigma}=\left(  -1\right)  ^{\pi}. \label{pf.lem.nu-det.c2.short.sum=}%
\end{equation}

\end{vershort}

\begin{verlong}
It is easy to see that $\pi$ is the \textbf{only} permutation $\sigma\in
S_{n}$ satisfying $\nu=\eta\circ\sigma$\ \ \ \ \footnote{\textit{Proof.} It is
clear that $\pi$ is a permutation $\sigma\in S_{n}$ satisfying $\nu=\eta
\circ\sigma$ (because $\pi\in S_{n}$ is a permutation satisfying $\nu
=\eta\circ\pi$). It thus remains to show that $\pi$ is the \textbf{only} such
permutation. In other words, it remains to show that if $\sigma\in S_{n}$ is a
permutation satisfying $\nu=\eta\circ\sigma$, then $\sigma=\pi$. So let us
show this.
\par
Let $\sigma\in S_{n}$ be a permutation satisfying $\nu=\eta\circ\sigma$. Then,
$\eta\circ\sigma=\nu=\eta\circ\pi$.
\par
Recall that the $n$ numbers $\eta_{1},\eta_{2},\ldots,\eta_{n}$ are distinct.
Thus, Proposition \ref{prop.etapi.dist} shows that if the permutations
$\sigma$ and $\pi$ were distinct, then we would have $\eta\circ\sigma\neq
\eta\circ\pi$; but this would contradict $\eta\circ\sigma=\eta\circ\pi$.
Hence, $\sigma$ and $\pi$ cannot be distinct. In other words, we must have
$\sigma=\pi$.
\par
Forget that we fixed $\sigma$. We thus have shown that if $\sigma\in S_{n}$ is
a permutation satisfying $\nu=\eta\circ\sigma$, then $\sigma=\pi$. In other
words, $\pi$ is the \textbf{only} permutation $\sigma\in S_{n}$ satisfying
$\nu=\eta\circ\sigma$ (since we already know that $\pi$ is such a
permutation). Qed.}. Hence, the sum $\sum_{\substack{\sigma\in S_{n}%
;\\\nu=\eta\circ\sigma}}\left(  -1\right)  ^{\sigma}$ has only one addend,
namely the addend for $\sigma=\pi$. Thus,
\begin{equation}
\sum_{\substack{\sigma\in S_{n};\\\nu=\eta\circ\sigma}}\left(  -1\right)
^{\sigma}=\left(  -1\right)  ^{\pi}. \label{pf.lem.nu-det.c2.sum=}%
\end{equation}

\end{verlong}

We have $\nu=\eta\circ\pi=\left(  \eta_{\pi\left(  1\right)  },\eta
_{\pi\left(  2\right)  },\ldots,\eta_{\pi\left(  n\right)  }\right)  $ (by
Definition \ref{def.etapi}). Thus, for each $i\in\left[  n\right]  $, we have
\begin{equation}
\nu_{i}=\left(  \eta_{\pi\left(  1\right)  },\eta_{\pi\left(  2\right)
},\ldots,\eta_{\pi\left(  n\right)  }\right)  _{i}=\eta_{\pi\left(  i\right)
}. \label{pf.lem.nu-det.c2.nui=}%
\end{equation}

On the other hand, it is well-known (see, e.g., \cite[Corollary 6.4.15]{21s}
or \cite[Lemma 6.17 \textbf{(a)}]{detnotes}) that when the rows of a matrix
are permuted, then the determinant of this matrix gets multiplied by $\left(
-1\right)  ^{\tau}$, where $\tau$ is the permutation used to permute the rows.
In other words: If $\left(  a_{i,j}\right)  _{i,j\in\left[  n\right]  }$ is a
square matrix (with integer entries), and if $\tau\in S_{n}$ is a permutation,
then%
\[
\det\left(  \left(  a_{\tau\left(  i\right)  ,j}\right)  _{i,j\in\left[
n\right]  }\right)  =\left(  -1\right)  ^{\tau}\cdot\det\left(  \left(
a_{i,j}\right)  _{i,j\in\left[  n\right]  }\right)  .
\]
Applying this to $a_{i,j}=\left[  \eta_{i}\geq j\right]  $ and $\tau=\pi$, we
obtain%
\[
\det\left(  \left(  \left[  \eta_{\pi\left(  i\right)  }\geq j\right]
\right)  _{i,j\in\left[  n\right]  }\right)  =\left(  -1\right)  ^{\pi}%
\cdot\det\left(  \left(  \left[  \eta_{i}\geq j\right]  \right)
_{i,j\in\left[  n\right]  }\right)  .
\]
In view of (\ref{pf.lem.nu-det.c2.nui=}), we can rewrite this as%
\begin{equation}
\det\left(  \left(  \left[  \nu_{i}\geq j\right]  \right)  _{i,j\in\left[
n\right]  }\right)  =\left(  -1\right)  ^{\pi}\cdot\det\left(  \left(  \left[
\eta_{i}\geq j\right]  \right)  _{i,j\in\left[  n\right]  }\right)  .
\label{pf.lem.nu-det.c2.det=1}%
\end{equation}

\begin{vershort}
However, from the definition of $\eta$, we can easily see that every
$i,j\in\left[  n\right]  $ satisfying $i<j$ satisfy $\eta_{i}<j$ and therefore
$\left[  \eta_{i}\geq j\right]  =0$. In other words, the matrix $\left(
\left[  \eta_{i}\geq j\right]  \right)  _{i,j\in\left[  n\right]  }$ is
lower-triangular. Therefore, its determinant is the product $\prod_{i=1}%
^{n}\left[  \eta_{i}\geq i\right]  $ of its diagonal entries. In other words,%
\[
\det\left(  \left(  \left[  \eta_{i}\geq j\right]  \right)  _{i,j\in\left[
n\right]  }\right)  =\prod_{i=1}^{n}\underbrace{\left[  \eta_{i}\geq i\right]
}_{\substack{=1\\\text{(since it is easy}\\\text{to see that }\eta_{i}\geq
i\text{)}}}=\prod_{i=1}^{n}1=1.
\]
Thus, (\ref{pf.lem.nu-det.c2.det=1}) becomes%
\[
\det\left(  \left(  \left[  \nu_{i}\geq j\right]  \right)  _{i,j\in\left[
n\right]  }\right)  =\left(  -1\right)  ^{\pi}\cdot\underbrace{\det\left(
\left(  \left[  \eta_{i}\geq j\right]  \right)  _{i,j\in\left[  n\right]
}\right)  }_{=1}=\left(  -1\right)  ^{\pi}=\sum_{\substack{\sigma\in
S_{n};\\\nu=\eta\circ\sigma}}\left(  -1\right)  ^{\sigma}%
\]
(by (\ref{pf.lem.nu-det.c2.short.sum=})). Thus, Lemma \ref{lem.nu-det} is
proved in Case 2. \medskip
\end{vershort}

\begin{verlong}
However, we have $\left[  \eta_{i}\geq j\right]  =0$ for any $i,j\in\left[
n\right]  $ satisfying $i<j$\ \ \ \ \footnote{\textit{Proof.} Let
$i,j\in\left[  n\right]  $ satisfy $i<j$. We must prove that $\left[  \eta
_{i}\geq j\right]  =0$.
\par
We have $j\in\left[  n\right]  =\left\{  1,2,\ldots,n\right\}  $, so that
$j\leq n$. Thus, $i<j\leq n$, so that $i\neq n$. Combining $i\in\left[
n\right]  =\left\{  1,2,\ldots,n\right\}  $ with $i\neq n$, we obtain
$i\in\left\{  1,2,\ldots,n\right\}  \setminus\left\{  n\right\}  =\left\{
1,2,\ldots,n-1\right\}  =\left[  n-1\right]  $. Thus,
(\ref{pf.lem.nu-det.etai=}) yields $\eta_{i}=i$. Hence, $\eta_{i}=i<j$. Thus,
we don't have $\eta_{i}\geq j$. Hence, $\left[  \eta_{i}\geq j\right]  =0$,
qed.}. In other words, the matrix $\left(  \left[  \eta_{i}\geq j\right]
\right)  _{i,j\in\left[  n\right]  }$ is lower-triangular. Therefore, its
determinant is the product $\prod_{i=1}^{n}\left[  \eta_{i}\geq i\right]  $ of
its diagonal entries (because it is well-known that the determinant of any
lower-triangular matrix is the product of its diagonal entries). In other
words,%
\begin{equation}
\det\left(  \left(  \left[  \eta_{i}\geq j\right]  \right)  _{i,j\in\left[
n\right]  }\right)  =\prod_{i=1}^{n}\left[  \eta_{i}\geq i\right]  .
\label{pf.lem.nu-det.c2.det=2}%
\end{equation}

However, each $i\in\left[  n\right]  $ satisfies $\eta_{i}\geq i$%
\ \ \ \ \footnote{\textit{Proof.} Let $i\in\left[  n\right]  $. We must prove
that $\eta_{i}\geq i$. If $i\in\left[  n-1\right]  $, then this follows from
(\ref{pf.lem.nu-det.etai=}) (in fact, in this case, (\ref{pf.lem.nu-det.etai=}%
) yields $\eta_{i}=i\geq i$). Hence, for the rest of this proof, we WLOG
assume that we don't have $i\in\left[  n-1\right]  $. Thus, we have
$i\notin\left[  n-1\right]  $. Combining $i\in\left[  n\right]  $ with
$i\notin\left[  n-1\right]  $, we obtain%
\[
i\in\underbrace{\left[  n\right]  }_{=\left\{  1,2,\ldots,n\right\}
}\setminus\underbrace{\left[  n-1\right]  }_{=\left\{  1,2,\ldots,n-1\right\}
}=\left\{  1,2,\ldots,n\right\}  \setminus\left\{  1,2,\ldots,n-1\right\}
=\left\{  n\right\}  .
\]
Thus, $i=n$. Hence, $\eta_{i}=\eta_{n}=n+p$ (by (\ref{pf.lem.nu-det.etan=})).
However, $p\geq0$ (since $p\in\mathbb{N}$) and thus $n+p\geq n=i$ (since
$i=n$). Thus, $\eta_{i}=n+p\geq i$, qed.} and therefore
\begin{equation}
\left[  \eta_{i}\geq i\right]  =1. \label{pf.lem.nu-det.c2.det-diag}%
\end{equation}
Hence, (\ref{pf.lem.nu-det.c2.det=2}) becomes%
\[
\det\left(  \left(  \left[  \eta_{i}\geq j\right]  \right)  _{i,j\in\left[
n\right]  }\right)  =\prod_{i=1}^{n}\underbrace{\left[  \eta_{i}\geq i\right]
}_{\substack{=1\\\text{(by (\ref{pf.lem.nu-det.c2.det-diag}))}}}=\prod
_{i=1}^{n}1=1.
\]
Thus, (\ref{pf.lem.nu-det.c2.det=1}) becomes%
\[
\det\left(  \left(  \left[  \nu_{i}\geq j\right]  \right)  _{i,j\in\left[
n\right]  }\right)  =\left(  -1\right)  ^{\pi}\cdot\underbrace{\det\left(
\left(  \left[  \eta_{i}\geq j\right]  \right)  _{i,j\in\left[  n\right]
}\right)  }_{=1}=\left(  -1\right)  ^{\pi}=\sum_{\substack{\sigma\in
S_{n};\\\nu=\eta\circ\sigma}}\left(  -1\right)  ^{\sigma}%
\]
(by (\ref{pf.lem.nu-det.c2.sum=})). Thus, Lemma \ref{lem.nu-det} is proved in
Case 2. \medskip
\end{verlong}

We have now proved Lemma \ref{lem.nu-det} in both Cases 1 and 2. This
completes the proof of Lemma \ref{lem.nu-det}.
\end{proof}

Our third lemma is another expression for the same determinant as in Lemma
\ref{lem.nu-det}:

\begin{lemma}
\label{lem.nu-det2}Let $n\in\mathbb{N}$. For any permutation $\sigma\in S_{n}%
$, we let $\overline{\sigma}$ denote the $n$-tuple $\left(  \sigma\left(
1\right)  ,\sigma\left(  2\right)  ,\ldots,\sigma\left(  n\right)  \right)
\in\mathbb{Z}^{n}$.

Let $\nu\in\mathbb{Z}^{n}$ be an $n$-tuple. Then,%
\[
\det\left(  \left(  \left[  \nu_{i}\geq j\right]  \right)  _{i,j\in\left[
n\right]  }\right)  =\sum_{\substack{\sigma\in S_{n};\\\nu-\overline{\sigma
}\in\mathbb{N}^{n}}}\left(  -1\right)  ^{\sigma}.
\]

\end{lemma}

\begin{vershort}
\begin{proof}
[Proof of Lemma \ref{lem.nu-det2}.]If $\mathcal{A}$ and $\mathcal{B}$ are two
equivalent logical statements, then $\left[  \mathcal{A}\right]  =\left[
\mathcal{B}\right]  $. Thus, for each $\sigma\in S_{n}$, we have%
\begin{align}
&  \left[  \nu-\overline{\sigma}\in\mathbb{N}^{n}\right] \nonumber\\
&  =\left[  \left(  \nu-\overline{\sigma}\right)  _{i}\in\mathbb{N}\text{ for
each }i\in\left[  n\right]  \right] \nonumber\\
&  =\left[  \nu_{i}-\sigma\left(  i\right)  \in\mathbb{N}\text{ for each }%
i\in\left[  n\right]  \right] \nonumber\\
&  \ \ \ \ \ \ \ \ \ \ \ \ \ \ \ \ \ \ \ \ \left(
\begin{array}
[c]{c}%
\text{since each }i\in\left[  n\right]  \text{ satisfies }\left(
\nu-\overline{\sigma}\right)  _{i}=\nu_{i}-\overline{\sigma}_{i}=\nu
_{i}-\sigma\left(  i\right) \\
\text{(because the definition of }\overline{\sigma}\text{ yields }%
\overline{\sigma}_{i}=\sigma\left(  i\right)  \text{)}%
\end{array}
\right) \nonumber\\
&  =\left[  \nu_{i}\geq\sigma\left(  i\right)  \text{ for each }i\in\left[
n\right]  \right] \nonumber\\
&  \ \ \ \ \ \ \ \ \ \ \ \ \ \ \ \ \ \ \ \ \left(
\begin{array}
[c]{c}%
\text{since two integers }a\text{ and }b\text{ satisfy }a-b\in\mathbb{N}\\
\text{if and only if }a\geq b
\end{array}
\right) \nonumber\\
&  =\left[  \nu_{1}\geq\sigma\left(  1\right)  \text{ and }\nu_{2}\geq
\sigma\left(  2\right)  \text{ and }\cdots\text{ and }\nu_{n}\geq\sigma\left(
n\right)  \right] \nonumber\\
&  =\prod_{i=1}^{n}\left[  \nu_{i}\geq\sigma\left(  i\right)  \right]
\label{pf.lem.nu-det2.short.ive}%
\end{align}
(because for any $n$ logical statements $\mathcal{A}_{1},\mathcal{A}%
_{2},\ldots,\mathcal{A}_{n}$, we have \newline$\left[  \mathcal{A}_{1}\text{
and }\mathcal{A}_{2}\text{ and }\cdots\text{ and }\mathcal{A}_{n}\right]
=\prod_{i=1}^{n}\left[  \mathcal{A}_{i}\right]  $).

Now, the definition of the determinant of a matrix yields%
\begin{align}
\det\left(  \left(  \left[  \nu_{i}\geq j\right]  \right)  _{i,j\in\left[
n\right]  }\right)   &  =\sum_{\sigma\in S_{n}}\left(  -1\right)  ^{\sigma
}\underbrace{\prod_{i=1}^{n}\left[  \nu_{i}\geq\sigma\left(  i\right)
\right]  }_{\substack{=\left[  \nu-\overline{\sigma}\in\mathbb{N}^{n}\right]
\\\text{(by (\ref{pf.lem.nu-det2.short.ive}))}}}\nonumber\\
&  =\sum_{\sigma\in S_{n}}\left(  -1\right)  ^{\sigma}\left[  \nu
-\overline{\sigma}\in\mathbb{N}^{n}\right]  . \label{pf.lem.nu-det2.short.2}%
\end{align}

Recall that a truth value $\left[  \mathcal{A}\right]  $ is always either $1$
or $0$, depending on whether the statement $\mathcal{A}$ is true or false.
Hence, each addend of the sum on the right hand side of
(\ref{pf.lem.nu-det2.short.2}) is either of the form $\left(  -1\right)
^{\sigma}\cdot1$ or of the form $\left(  -1\right)  ^{\sigma}\cdot0$,
depending on whether the statement \textquotedblleft$\nu-\overline{\sigma}%
\in\mathbb{N}^{n}$\textquotedblright\ is true or false. The addends of the
form $\left(  -1\right)  ^{\sigma}\cdot1$ can be simplified to $\left(
-1\right)  ^{\sigma}$, whereas the addends of the form $\left(  -1\right)
^{\sigma}\cdot0$ can be discarded (since they are $0$). Thus, the sum
simplifies as follows:%
\[
\sum_{\sigma\in S_{n}}\left(  -1\right)  ^{\sigma}\left[  \nu-\overline
{\sigma}\in\mathbb{N}^{n}\right]  =\sum_{\substack{\sigma\in S_{n}%
;\\\nu-\overline{\sigma}\in\mathbb{N}^{n}}}\left(  -1\right)  ^{\sigma}.
\]
Combining this with (\ref{pf.lem.nu-det2.short.2}), we obtain
\[
\det\left(  \left(  \left[  \nu_{i}\geq j\right]  \right)  _{i,j\in\left[
n\right]  }\right)  =\sum_{\substack{\sigma\in S_{n};\\\nu-\overline{\sigma
}\in\mathbb{N}^{n}}}\left(  -1\right)  ^{\sigma}.
\]
This proves Lemma \ref{lem.nu-det2}.
\end{proof}
\end{vershort}

\begin{verlong}
\begin{proof}
[Proof of Lemma \ref{lem.nu-det2}.]We first observe that each $\sigma\in
S_{n}$ satisfies%
\begin{equation}
\left[  \nu-\overline{\sigma}\in\mathbb{N}^{n}\right]  =\prod_{i=1}^{n}\left[
\nu_{i}\geq\sigma\left(  i\right)  \right]  . \label{pf.lem.nu-det2.ives}%
\end{equation}

[\textit{Proof of (\ref{pf.lem.nu-det2.ives}):} Let $\sigma\in S_{n}$ be
arbitrary. We must prove the equality (\ref{pf.lem.nu-det2.ives}).

The definition of $\overline{\sigma}$ yields $\overline{\sigma}=\left(
\sigma\left(  1\right)  ,\sigma\left(  2\right)  ,\ldots,\sigma\left(
n\right)  \right)  $. Hence, for each $i\in\left[  n\right]  $, we have%
\begin{equation}
\overline{\sigma}_{i}=\sigma\left(  i\right)  .
\label{pf.pf.lem.nu-det2.ives.pf.sigbar}%
\end{equation}

We are in one of the following two cases:

\textit{Case 1:} We have $\nu-\overline{\sigma}\in\mathbb{N}^{n}$.

\textit{Case 2:} We have $\nu-\overline{\sigma}\notin\mathbb{N}^{n}$. \medskip

Let us first consider Case 1. In this case, we have $\nu-\overline{\sigma}%
\in\mathbb{N}^{n}$. Now, let $i\in\left[  n\right]  $. Then, $\left(
\nu-\overline{\sigma}\right)  _{i}\in\mathbb{N}$ (since $\nu-\overline{\sigma
}\in\mathbb{N}^{n}$). In view of
\[
\left(  \nu-\overline{\sigma}\right)  _{i}=\nu_{i}-\underbrace{\overline
{\sigma}_{i}}_{\substack{=\sigma\left(  i\right)  \\\text{(by
(\ref{pf.pf.lem.nu-det2.ives.pf.sigbar}))}}}=\nu_{i}-\sigma\left(  i\right)
,
\]
we can rewrite this as $\nu_{i}-\sigma\left(  i\right)  \in\mathbb{N}$. Hence,
$\nu_{i}-\sigma\left(  i\right)  \geq0$, so that $\nu_{i}\geq\sigma\left(
i\right)  $. Thus, $\left[  \nu_{i}\geq\sigma\left(  i\right)  \right]  =1$.

Forget that we fixed $i$. Thus, we have shown that $\left[  \nu_{i}\geq
\sigma\left(  i\right)  \right]  =1$ for each $i\in\left[  n\right]  $.
Multiplying these equalities over all $i\in\left[  n\right]  $, we obtain
$\prod_{i\in\left[  n\right]  }\left[  \nu_{i}\geq\sigma\left(  i\right)
\right]  =\prod_{i\in\left[  n\right]  }1=1$. On the other hand, $\left[
\nu-\overline{\sigma}\in\mathbb{N}^{n}\right]  =1$ (since $\nu-\overline
{\sigma}\in\mathbb{N}^{n}$). Comparing these two equalities, we obtain%
\[
\left[  \nu-\overline{\sigma}\in\mathbb{N}^{n}\right]  =\underbrace{\prod
_{i\in\left[  n\right]  }}_{=\prod_{i=1}^{n}}\left[  \nu_{i}\geq\sigma\left(
i\right)  \right]  =\prod_{i=1}^{n}\left[  \nu_{i}\geq\sigma\left(  i\right)
\right]  .
\]
Thus, (\ref{pf.lem.nu-det2.ives}) is proved in Case 1. \medskip

Let us now consider Case 2. In this case, we have $\nu-\overline{\sigma}%
\notin\mathbb{N}^{n}$. Hence, there exists some $j\in\left[  n\right]  $ such
that $\left(  \nu-\overline{\sigma}\right)  _{j}\notin\mathbb{N}$. Consider
this $j$. We have $\overline{\sigma}_{j}=\sigma\left(  j\right)  $ (by
(\ref{pf.pf.lem.nu-det2.ives.pf.sigbar}), applied to $i=j$). Now,
\[
\left(  \nu-\overline{\sigma}\right)  _{j}=\nu_{j}-\underbrace{\overline
{\sigma}_{j}}_{=\sigma\left(  j\right)  }=\nu_{j}-\sigma\left(  j\right)  ,
\]
so that $\nu_{j}-\sigma\left(  j\right)  =\left(  \nu-\overline{\sigma
}\right)  _{j}\notin\mathbb{N}$. Combining this with $\nu_{j}-\sigma\left(
j\right)  \in\mathbb{Z}$ (which is obvious, since $\nu_{j}$ and $\sigma\left(
j\right)  $ belong to $\mathbb{Z}$), we obtain $\nu_{j}-\sigma\left(
j\right)  \in\mathbb{Z}\setminus\mathbb{N}=\left\{  -1,-2,-3,\ldots\right\}
$. In other words, $\nu_{j}-\sigma\left(  j\right)  $ is a negative integer.
Hence $\nu_{j}-\sigma\left(  j\right)  <0$, so that $\nu_{j}<\sigma\left(
j\right)  $. Thus, we don't have $\nu_{j}\geq\sigma\left(  j\right)  $. Hence,
we have $\left[  \nu_{j}\geq\sigma\left(  j\right)  \right]  =0$.

Now, consider the product $\prod_{i=1}^{n}\left[  \nu_{i}\geq\sigma\left(
i\right)  \right]  $. The term $\left[  \nu_{j}\geq\sigma\left(  j\right)
\right]  $ is one of the factors of this product (namely, the factor for
$i=j$). Since this term is $0$ (because we have just shown that $\left[
\nu_{j}\geq\sigma\left(  j\right)  \right]  =0$), we thus conclude that one of
the factors of the product $\prod_{i=1}^{n}\left[  \nu_{i}\geq\sigma\left(
i\right)  \right]  $ is $0$. Therefore, the entire product is $0$. In other
words,
\begin{equation}
\prod_{i=1}^{n}\left[  \nu_{i}\geq\sigma\left(  i\right)  \right]  =0.
\label{pf.pf.lem.nu-det2.ives.pf.c2.prod0}%
\end{equation}
On the other hand, we do not have $\nu-\overline{\sigma}\in\mathbb{N}^{n}$
(since $\nu-\overline{\sigma}\notin\mathbb{N}^{n}$). Hence, $\left[
\nu-\overline{\sigma}\in\mathbb{N}^{n}\right]  =0$. Comparing this with
(\ref{pf.pf.lem.nu-det2.ives.pf.c2.prod0}), we obtain%
\[
\left[  \nu-\overline{\sigma}\in\mathbb{N}^{n}\right]  =\prod_{i=1}^{n}\left[
\nu_{i}\geq\sigma\left(  i\right)  \right]  .
\]
Thus, (\ref{pf.lem.nu-det2.ives}) is proved in Case 2. \medskip

We have now proved (\ref{pf.lem.nu-det2.ives}) in both Cases 1 and 2. Thus,
the proof of (\ref{pf.lem.nu-det2.ives}) is complete.] \medskip

Now, the definition of the determinant of a matrix yields%
\begin{align*}
&  \det\left(  \left(  \left[  \nu_{i}\geq j\right]  \right)  _{i,j\in\left[
n\right]  }\right) \\
&  =\sum_{\sigma\in S_{n}}\left(  -1\right)  ^{\sigma}\underbrace{\prod
_{i=1}^{n}\left[  \nu_{i}\geq\sigma\left(  i\right)  \right]  }%
_{\substack{=\left[  \nu-\overline{\sigma}\in\mathbb{N}^{n}\right]
\\\text{(by (\ref{pf.lem.nu-det2.ives}))}}}=\sum_{\sigma\in S_{n}}\left(
-1\right)  ^{\sigma}\left[  \nu-\overline{\sigma}\in\mathbb{N}^{n}\right] \\
&  =\sum_{\substack{\sigma\in S_{n};\\\nu-\overline{\sigma}\in\mathbb{N}^{n}%
}}\left(  -1\right)  ^{\sigma}\underbrace{\left[  \nu-\overline{\sigma}%
\in\mathbb{N}^{n}\right]  }_{\substack{=1\\\text{(since }\nu-\overline{\sigma
}\in\mathbb{N}^{n}\text{)}}}+\sum_{\substack{\sigma\in S_{n};\\\nu
-\overline{\sigma}\notin\mathbb{N}^{n}}}\left(  -1\right)  ^{\sigma
}\underbrace{\left[  \nu-\overline{\sigma}\in\mathbb{N}^{n}\right]
}_{\substack{=0\\\text{(since we don't}\\\text{have }\nu-\overline{\sigma}%
\in\mathbb{N}^{n}\\\text{(because }\nu-\overline{\sigma}\notin\mathbb{N}%
^{n}\text{))}}}\\
&  \ \ \ \ \ \ \ \ \ \ \ \ \ \ \ \ \ \ \ \ \left(
\begin{array}
[c]{c}%
\text{since each }\sigma\in S_{n}\text{ satisfies either }\nu-\overline
{\sigma}\in\mathbb{N}^{n}\\
\text{or }\nu-\overline{\sigma}\notin\mathbb{N}^{n}\text{ (but not both at the
same time)}%
\end{array}
\right) \\
&  =\sum_{\substack{\sigma\in S_{n};\\\nu-\overline{\sigma}\in\mathbb{N}^{n}%
}}\left(  -1\right)  ^{\sigma}+\underbrace{\sum_{\substack{\sigma\in
S_{n};\\\nu-\overline{\sigma}\notin\mathbb{N}^{n}}}\left(  -1\right)
^{\sigma}0}_{=0}=\sum_{\substack{\sigma\in S_{n};\\\nu-\overline{\sigma}%
\in\mathbb{N}^{n}}}\left(  -1\right)  ^{\sigma}.
\end{align*}
This proves Lemma \ref{lem.nu-det2}.
\end{proof}
\end{verlong}

Our fourth and last lemma is a trivial property of the $\mathbb{Z}$-module
$\mathbb{Z}^{n}$:

\begin{lemma}
\label{lem.additive}Let $\alpha\in\mathbb{Z}^{n}$ and $\beta\in\mathbb{Z}^{n}%
$. Then, $\left\vert \alpha+\beta\right\vert =\left\vert \alpha\right\vert
+\left\vert \beta\right\vert $.
\end{lemma}

\begin{verlong}
\begin{proof}
[Proof of Lemma \ref{lem.additive}.]The definition of $\left\vert
\alpha\right\vert $ yields $\left\vert \alpha\right\vert =\alpha_{1}%
+\alpha_{2}+\cdots+\alpha_{n}$. Similarly, $\left\vert \beta\right\vert
=\beta_{1}+\beta_{2}+\cdots+\beta_{n}$. Adding these two equalities together,
we obtain%
\begin{equation}
\left\vert \alpha\right\vert +\left\vert \beta\right\vert =\left(  \alpha
_{1}+\alpha_{2}+\cdots+\alpha_{n}\right)  +\left(  \beta_{1}+\beta_{2}%
+\cdots+\beta_{n}\right)  . \label{pf.lem.additive.1}%
\end{equation}

The definition of $\alpha+\beta$ yields $\alpha+\beta=\left(  \alpha_{1}%
+\beta_{1},\alpha_{2}+\beta_{2},\ldots,\alpha_{n}+\beta_{n}\right)  $. Hence,%
\begin{align*}
\left\vert \alpha+\beta\right\vert  &  =\left\vert \left(  \alpha_{1}%
+\beta_{1},\alpha_{2}+\beta_{2},\ldots,\alpha_{n}+\beta_{n}\right)
\right\vert \\
&  =\left(  \alpha_{1}+\beta_{1}\right)  +\left(  \alpha_{2}+\beta_{2}\right)
+\cdots+\left(  \alpha_{n}+\beta_{n}\right) \\
&  \ \ \ \ \ \ \ \ \ \ \ \ \ \ \ \ \ \ \ \ \left(  \text{by the definition of
}\left\vert \left(  \alpha_{1}+\beta_{1},\alpha_{2}+\beta_{2},\ldots
,\alpha_{n}+\beta_{n}\right)  \right\vert \right) \\
&  =\left(  \alpha_{1}+\alpha_{2}+\cdots+\alpha_{n}\right)  +\left(  \beta
_{1}+\beta_{2}+\cdots+\beta_{n}\right) \\
&  =\left\vert \alpha\right\vert +\left\vert \beta\right\vert
\ \ \ \ \ \ \ \ \ \ \left(  \text{by (\ref{pf.lem.additive.1})}\right)  .
\end{align*}
This proves Lemma \ref{lem.additive}.
\end{proof}
\end{verlong}

We are now ready to prove Theorem \ref{thm.pre-pieri}:

\begin{vershort}
\begin{proof}
[Proof of Theorem \ref{thm.pre-pieri}.]For any permutation $\sigma\in S_{n}$,
we let $\overline{\sigma}$ denote the $n$-tuple $\left(  \sigma\left(
1\right)  ,\sigma\left(  2\right)  ,\ldots,\sigma\left(  n\right)  \right)
\in\mathbb{Z}^{n}$. This $n$-tuple $\overline{\sigma}$ satisfies
\[
\left\vert \overline{\sigma}\right\vert =\sigma\left(  1\right)
+\sigma\left(  2\right)  +\cdots+\sigma\left(  n\right)  =1+2+\cdots+n
\]
and therefore
\begin{equation}
\left\vert \overline{\sigma}\right\vert +p=\left(  1+2+\cdots+n\right)
+p=\left\vert \eta\right\vert \label{pf.thm.pre-pieri.short.modsigbar1}%
\end{equation}
(since the definition of $\eta$ yields \newline$\left\vert \eta\right\vert
=1+2+\cdots+\left(  n-1\right)  +\left(  n+p\right)  =\left(  1+2+\cdots
+n\right)  +p$).

Let us now set $\prod_{i=1}^{n}a_{i}:=a_{1}a_{2}\cdots a_{n}$ for any
$a_{1},a_{2},\ldots,a_{n}\in R$. (This is, of course, the usual meaning of the
notation $\prod_{i=1}^{n}a_{i}$ when the ring $R$ is commutative; however, we
are now extending it to the case of arbitrary $R$.)

Using this notation, we can rewrite the definition of a row-determinant as
follows: If $\left(  a_{i,j}\right)  _{i,j\in\left[  n\right]  }\in R^{n\times
n}$ is any $n\times n$-matrix over $R$, then%
\begin{equation}
\operatorname*{rowdet}\left(  \left(  a_{i,j}\right)  _{i,j\in\left[
n\right]  }\right)  =\sum_{\sigma\in S_{n}}\left(  -1\right)  ^{\sigma}%
\prod_{i=1}^{n}a_{i,\sigma\left(  i\right)  }.
\label{pf.thm.pre-pieri.short.rowdetA=}%
\end{equation}

For each $\beta\in\mathbb{N}^{n}$, we have%
\begin{align}
t_{\alpha+\beta}  &  =\operatorname*{rowdet}\left(  \left(  h_{\left(
\alpha+\beta\right)  _{i}+j,\ i}\right)  _{i,j\in\left[  n\right]  }\right)
\ \ \ \ \ \ \ \ \ \ \left(  \text{by the definition of }t_{\alpha+\beta
}\right) \nonumber\\
&  =\sum_{\sigma\in S_{n}}\left(  -1\right)  ^{\sigma}\prod_{i=1}%
^{n}h_{\left(  \alpha+\beta\right)  _{i}+\sigma\left(  i\right)  ,\ i}
\label{pf.thm.pre-pieri.short.t1}%
\end{align}
(by (\ref{pf.thm.pre-pieri.short.rowdetA=}), applied to $a_{i,j}=h_{\left(
\alpha+\beta\right)  _{i}+j,\ i}$). However, for each $\beta\in\mathbb{N}^{n}%
$, each $\sigma\in S_{n}$ and each $i\in\left[  n\right]  $, we have%
\[
\underbrace{\left(  \alpha+\beta\right)  _{i}}_{=\alpha_{i}+\beta_{i}%
}+\underbrace{\sigma\left(  i\right)  }_{\substack{=\overline{\sigma}%
_{i}\\\text{(by the definition of }\overline{\sigma}\text{)}}}=\alpha
_{i}+\underbrace{\beta_{i}+\overline{\sigma}_{i}}_{=\left(  \beta
+\overline{\sigma}\right)  _{i}}=\alpha_{i}+\left(  \beta+\overline{\sigma
}\right)  _{i}.
\]
This allows us to rewrite (\ref{pf.thm.pre-pieri.short.t1}) as follows:%
\[
t_{\alpha+\beta}=\sum_{\sigma\in S_{n}}\left(  -1\right)  ^{\sigma}\prod
_{i=1}^{n}h_{\alpha_{i}+\left(  \beta+\overline{\sigma}\right)  _{i}%
,\ i}\ \ \ \ \ \ \ \ \ \ \text{for each }\beta\in\mathbb{N}^{n}.
\]
Summing these equalities over all $\beta\in\mathbb{N}^{n}$ satisfying
$\left\vert \beta\right\vert =p$, we obtain%
\begin{align}
\sum_{\substack{\beta\in\mathbb{N}^{n};\\\left\vert \beta\right\vert
=p}}t_{\alpha+\beta}  &  =\sum_{\substack{\beta\in\mathbb{N}^{n};\\\left\vert
\beta\right\vert =p}}\ \ \sum_{\sigma\in S_{n}}\left(  -1\right)  ^{\sigma
}\prod_{i=1}^{n}h_{\alpha_{i}+\left(  \beta+\overline{\sigma}\right)
_{i},\ i}\nonumber\\
&  =\sum_{\sigma\in S_{n}}\left(  -1\right)  ^{\sigma}\sum_{\substack{\beta
\in\mathbb{N}^{n};\\\left\vert \beta\right\vert =p}}\ \ \prod_{i=1}%
^{n}h_{\alpha_{i}+\left(  \beta+\overline{\sigma}\right)  _{i},\ i}.
\label{pf.thm.pre-pieri.short.sumt=1}%
\end{align}

Now, fix a permutation $\sigma\in S_{n}$. We shall rewrite the sum
$\sum_{\substack{\beta\in\mathbb{N}^{n};\\\left\vert \beta\right\vert
=p}}\ \ \prod_{i=1}^{n}h_{\alpha_{i}+\left(  \beta+\overline{\sigma}\right)
_{i},\ i}$ in terms of $\beta+\overline{\sigma}$. Indeed, we have the
following equality of summation signs:%
\begin{align*}
\sum_{\substack{\beta\in\mathbb{N}^{n};\\\left\vert \beta\right\vert =p}}  &
=\sum_{\substack{\beta\in\mathbb{N}^{n};\\\left\vert \beta\right\vert
+\left\vert \overline{\sigma}\right\vert =p+\left\vert \overline{\sigma
}\right\vert }}\ \ \ \ \ \ \ \ \ \ \left(
\begin{array}
[c]{c}%
\text{since the condition \textquotedblleft}\left\vert \beta\right\vert
=p\text{\textquotedblright\ is clearly}\\
\text{equivalent to \textquotedblleft}\left\vert \beta\right\vert +\left\vert
\overline{\sigma}\right\vert =p+\left\vert \overline{\sigma}\right\vert
\text{\textquotedblright}%
\end{array}
\right) \\
&  =\sum_{\substack{\beta\in\mathbb{N}^{n};\\\left\vert \beta+\overline
{\sigma}\right\vert =\left\vert \eta\right\vert }}\ \ \ \ \ \ \ \ \ \ \left(
\begin{array}
[c]{c}%
\text{since Lemma \ref{lem.additive} yields }\left\vert \beta\right\vert
+\left\vert \overline{\sigma}\right\vert =\left\vert \beta+\overline{\sigma
}\right\vert \text{,}\\
\text{and since }p+\left\vert \overline{\sigma}\right\vert =\left\vert
\overline{\sigma}\right\vert +p=\left\vert \eta\right\vert \text{ (by
(\ref{pf.thm.pre-pieri.short.modsigbar1}))}%
\end{array}
\right)  .
\end{align*}
Hence,%
\[
\sum_{\substack{\beta\in\mathbb{N}^{n};\\\left\vert \beta\right\vert
=p}}\ \ \prod_{i=1}^{n}h_{\alpha_{i}+\left(  \beta+\overline{\sigma}\right)
_{i},\ i}=\sum_{\substack{\beta\in\mathbb{N}^{n};\\\left\vert \beta
+\overline{\sigma}\right\vert =\left\vert \eta\right\vert }}\ \ \prod
_{i=1}^{n}h_{\alpha_{i}+\left(  \beta+\overline{\sigma}\right)  _{i},\ i}.
\]
However, $\mathbb{Z}^{n}$ is a group (under addition). Hence, when $\beta$
runs over $\mathbb{Z}^{n}$, the sum $\beta+\overline{\sigma}$ also runs over
$\mathbb{Z}^{n}$. (Formally speaking, this is saying that the map
\begin{align*}
\mathbb{Z}^{n}  &  \rightarrow\mathbb{Z}^{n},\\
\beta &  \mapsto\beta+\overline{\sigma}%
\end{align*}
is a bijection.) Thus, when $\beta$ runs over $\mathbb{N}^{n}$, the sum
$\beta+\overline{\sigma}$ runs over the set of all $\nu\in\mathbb{Z}^{n}$ that
satisfy $\nu-\overline{\sigma}\in\mathbb{N}^{n}$. Therefore, we can substitute
$\nu$ for $\beta+\overline{\sigma}$ in the sum $\sum_{\substack{\beta
\in\mathbb{N}^{n};\\\left\vert \beta+\overline{\sigma}\right\vert =\left\vert
\eta\right\vert }}\ \ \prod_{i=1}^{n}h_{\alpha_{i}+\left(  \beta
+\overline{\sigma}\right)  _{i},\ i}$. We thus obtain%
\[
\sum_{\substack{\beta\in\mathbb{N}^{n};\\\left\vert \beta+\overline{\sigma
}\right\vert =\left\vert \eta\right\vert }}\ \ \prod_{i=1}^{n}h_{\alpha
_{i}+\left(  \beta+\overline{\sigma}\right)  _{i},\ i}=\sum_{\substack{\nu
\in\mathbb{Z}^{n};\\\nu-\overline{\sigma}\in\mathbb{N}^{n};\\\left\vert
\nu\right\vert =\left\vert \eta\right\vert }}\ \ \prod_{i=1}^{n}h_{\alpha
_{i}+\nu_{i},\ i}.
\]
Hence, our above computation becomes%
\begin{align}
\sum_{\substack{\beta\in\mathbb{N}^{n};\\\left\vert \beta\right\vert
=p}}\ \ \prod_{i=1}^{n}h_{\alpha_{i}+\left(  \beta+\overline{\sigma}\right)
_{i},\ i}  &  =\sum_{\substack{\beta\in\mathbb{N}^{n};\\\left\vert
\beta+\overline{\sigma}\right\vert =\left\vert \eta\right\vert }%
}\ \ \prod_{i=1}^{n}h_{\alpha_{i}+\left(  \beta+\overline{\sigma}\right)
_{i},\ i}\nonumber\\
&  =\sum_{\substack{\nu\in\mathbb{Z}^{n};\\\nu-\overline{\sigma}\in
\mathbb{N}^{n};\\\left\vert \nu\right\vert =\left\vert \eta\right\vert
}}\ \ \prod_{i=1}^{n}h_{\alpha_{i}+\nu_{i},\ i}.
\label{pf.thm.pre-pieri.short.innersum}%
\end{align}

Forget that we fixed $\sigma$. We thus have proved
(\ref{pf.thm.pre-pieri.short.innersum}) for each permutation $\sigma\in S_{n}%
$. Now, (\ref{pf.thm.pre-pieri.short.sumt=1}) becomes%
\begin{align}
\sum_{\substack{\beta\in\mathbb{N}^{n};\\\left\vert \beta\right\vert
=p}}t_{\alpha+\beta}  &  =\sum_{\sigma\in S_{n}}\left(  -1\right)  ^{\sigma
}\sum_{\substack{\beta\in\mathbb{N}^{n};\\\left\vert \beta\right\vert
=p}}\ \ \prod_{i=1}^{n}h_{\alpha_{i}+\left(  \beta+\overline{\sigma}\right)
_{i},\ i}\nonumber\\
&  =\sum_{\sigma\in S_{n}}\left(  -1\right)  ^{\sigma}\sum_{\substack{\nu
\in\mathbb{Z}^{n};\\\nu-\overline{\sigma}\in\mathbb{N}^{n};\\\left\vert
\nu\right\vert =\left\vert \eta\right\vert }}\ \ \prod_{i=1}^{n}h_{\alpha
_{i}+\nu_{i},\ i}\ \ \ \ \ \ \ \ \ \ \left(  \text{by
(\ref{pf.thm.pre-pieri.short.innersum})}\right) \nonumber\\
&  =\underbrace{\sum_{\sigma\in S_{n}}\ \ \sum_{\substack{\nu\in\mathbb{Z}%
^{n};\\\nu-\overline{\sigma}\in\mathbb{N}^{n};\\\left\vert \nu\right\vert
=\left\vert \eta\right\vert }}}_{=\sum_{\substack{\nu\in\mathbb{Z}%
^{n};\\\left\vert \nu\right\vert =\left\vert \eta\right\vert }}\ \ \sum
_{\substack{\sigma\in S_{n};\\\nu-\overline{\sigma}\in\mathbb{N}^{n}}}}\left(
-1\right)  ^{\sigma}\prod_{i=1}^{n}h_{\alpha_{i}+\nu_{i},\ i}\nonumber\\
&  =\sum_{\substack{\nu\in\mathbb{Z}^{n};\\\left\vert \nu\right\vert
=\left\vert \eta\right\vert }}\ \ \sum_{\substack{\sigma\in S_{n}%
;\\\nu-\overline{\sigma}\in\mathbb{N}^{n}}}\left(  -1\right)  ^{\sigma}%
\prod_{i=1}^{n}h_{\alpha_{i}+\nu_{i},\ i}\nonumber\\
&  =\sum_{\substack{\nu\in\mathbb{Z}^{n};\\\left\vert \nu\right\vert
=\left\vert \eta\right\vert }}\underbrace{\left(  \sum_{\substack{\sigma\in
S_{n};\\\nu-\overline{\sigma}\in\mathbb{N}^{n}}}\left(  -1\right)  ^{\sigma
}\right)  }_{\substack{=\det\left(  \left(  \left[  \nu_{i}\geq j\right]
\right)  _{i,j\in\left[  n\right]  }\right)  \\\text{(by Lemma
\ref{lem.nu-det2})}}}\prod_{i=1}^{n}h_{\alpha_{i}+\nu_{i},\ i}\nonumber\\
&  =\sum_{\substack{\nu\in\mathbb{Z}^{n};\\\left\vert \nu\right\vert
=\left\vert \eta\right\vert }}\underbrace{\det\left(  \left(  \left[  \nu
_{i}\geq j\right]  \right)  _{i,j\in\left[  n\right]  }\right)  }%
_{\substack{=\sum_{\substack{\sigma\in S_{n};\\\nu=\eta\circ\sigma}}\left(
-1\right)  ^{\sigma}\\\text{(by Lemma \ref{lem.nu-det})}}}\cdot\prod_{i=1}%
^{n}h_{\alpha_{i}+\nu_{i},\ i}\nonumber\\
&  =\sum_{\substack{\nu\in\mathbb{Z}^{n};\\\left\vert \nu\right\vert
=\left\vert \eta\right\vert }}\left(  \sum_{\substack{\sigma\in S_{n}%
;\\\nu=\eta\circ\sigma}}\left(  -1\right)  ^{\sigma}\right)  \prod_{i=1}%
^{n}h_{\alpha_{i}+\nu_{i},\ i}\nonumber\\
&  =\sum_{\substack{\nu\in\mathbb{Z}^{n};\\\left\vert \nu\right\vert
=\left\vert \eta\right\vert }}\ \ \sum_{\substack{\sigma\in S_{n};\\\nu
=\eta\circ\sigma}}\left(  -1\right)  ^{\sigma}\prod_{i=1}^{n}h_{\alpha_{i}%
+\nu_{i},\ i}. \label{pf.thm.pre-pieri.short.2b}%
\end{align}

On the other hand,%
\begin{equation}
\operatorname*{rowdet}\left(  \left(  h_{\alpha_{i}+\eta_{j},\ i}\right)
_{i,j\in\left[  n\right]  }\right)  =\sum_{\sigma\in S_{n}}\left(  -1\right)
^{\sigma}\prod_{i=1}^{n}h_{\alpha_{i}+\eta_{\sigma\left(  i\right)  },\ i}
\label{pf.thm.pre-pieri.short.3a}%
\end{equation}
(by (\ref{pf.thm.pre-pieri.short.rowdetA=}), applied to $a_{i,j}=h_{\alpha
_{i}+\eta_{j},\ i}$). However, for each $\sigma\in S_{n}$, the $n$-tuple
$\eta\circ\sigma$ belongs to $\mathbb{Z}^{n}$ and satisfies $\left\vert
\eta\circ\sigma\right\vert =\left\vert \eta\right\vert $ (by Proposition
\ref{prop.etapi.len}). In other words, for each $\sigma\in S_{n}$, the
$n$-tuple $\eta\circ\sigma$ is a $\nu\in\mathbb{Z}^{n}$ satisfying $\left\vert
\nu\right\vert =\left\vert \eta\right\vert $. Hence, any sum ranging over all
$\sigma\in S_{n}$ can be split according to the value of $\eta\circ\sigma$. In
other words, we have the following equality of summation signs:%
\[
\sum_{\sigma\in S_{n}}=\sum_{\substack{\nu\in\mathbb{Z}^{n};\\\left\vert
\nu\right\vert =\left\vert \eta\right\vert }}\ \ \sum_{\substack{\sigma\in
S_{n};\\\eta\circ\sigma=\nu}}=\sum_{\substack{\nu\in\mathbb{Z}^{n}%
;\\\left\vert \nu\right\vert =\left\vert \eta\right\vert }}\ \ \sum
_{\substack{\sigma\in S_{n};\\\nu=\eta\circ\sigma}}.
\]
Thus, (\ref{pf.thm.pre-pieri.short.3a}) becomes%
\begin{align*}
\operatorname*{rowdet}\left(  \left(  h_{\alpha_{i}+\eta_{j},\ i}\right)
_{i,j\in\left[  n\right]  }\right)   &  =\underbrace{\sum_{\sigma\in S_{n}}%
}_{=\sum_{\substack{\nu\in\mathbb{Z}^{n};\\\left\vert \nu\right\vert
=\left\vert \eta\right\vert }}\ \ \sum_{\substack{\sigma\in S_{n};\\\nu
=\eta\circ\sigma}}}\left(  -1\right)  ^{\sigma}\prod_{i=1}^{n}h_{\alpha
_{i}+\eta_{\sigma\left(  i\right)  },\ i}\\
&  =\sum_{\substack{\nu\in\mathbb{Z}^{n};\\\left\vert \nu\right\vert
=\left\vert \eta\right\vert }}\ \ \sum_{\substack{\sigma\in S_{n};\\\nu
=\eta\circ\sigma}}\left(  -1\right)  ^{\sigma}\prod_{i=1}^{n}%
\underbrace{h_{\alpha_{i}+\eta_{\sigma\left(  i\right)  },\ i}}%
_{\substack{=h_{\alpha_{i}+\nu_{i},\ i}\\\text{(since }\nu=\eta\circ
\sigma\text{ entails }\nu_{i}=\eta_{\sigma\left(  i\right)  }\text{,}%
\\\text{so that }h_{\alpha_{i}+\nu_{i},\ i}=h_{\alpha_{i}+\eta_{\sigma\left(
i\right)  },\ i}\text{)}}}\\
&  =\sum_{\substack{\nu\in\mathbb{Z}^{n};\\\left\vert \nu\right\vert
=\left\vert \eta\right\vert }}\ \ \sum_{\substack{\sigma\in S_{n};\\\nu
=\eta\circ\sigma}}\left(  -1\right)  ^{\sigma}\prod_{i=1}^{n}h_{\alpha_{i}%
+\nu_{i},\ i}.
\end{align*}
Comparing this with (\ref{pf.thm.pre-pieri.short.2b}), we obtain%
\[
\sum_{\substack{\beta\in\mathbb{N}^{n};\\\left\vert \beta\right\vert
=p}}t_{\alpha+\beta}=\operatorname*{rowdet}\left(  \left(  h_{\alpha_{i}%
+\eta_{j},\ i}\right)  _{i,j\in\left[  n\right]  }\right)  .
\]
This proves Theorem \ref{thm.pre-pieri}.
\end{proof}
\end{vershort}

\begin{verlong}
\begin{proof}
[Proof of Theorem \ref{thm.pre-pieri}.]We have $n\in\mathbb{P}$. In other
words, $n$ is a positive integer. Thus, $n-1\in\mathbb{N}$ and $n>0$. The
definition of $\eta$ yields%
\[
\eta=\left(  1,2,\ldots,n\right)  +\left(  \underbrace{0,0,\ldots
,0}_{n-1\text{ zeroes}},p\right)  =\left(  1,2,\ldots,n-1,n+p\right)  .
\]
Thus,%
\begin{align}
\left\vert \eta\right\vert  &  =\left\vert \left(  1,2,\ldots,n-1,n+p\right)
\right\vert =1+2+\cdots+\left(  n-1\right)  +\left(  n+p\right) \nonumber\\
&  \ \ \ \ \ \ \ \ \ \ \ \ \ \ \ \ \ \ \ \ \left(  \text{by the definition of
}\left\vert \left(  1,2,\ldots,n-1,n+p\right)  \right\vert \right) \nonumber\\
&  =\underbrace{\left(  1+2+\cdots+\left(  n-1\right)  \right)  +n}%
_{=1+2+\cdots+n}+p\nonumber\\
&  =\left(  1+2+\cdots+n\right)  +p. \label{pf.thm.pre-pieri.sum-eta}%
\end{align}

For any permutation $\sigma\in S_{n}$, we let $\overline{\sigma}$ denote the
$n$-tuple $\left(  \sigma\left(  1\right)  ,\sigma\left(  2\right)
,\ldots,\sigma\left(  n\right)  \right)  \in\mathbb{Z}^{n}$. This $n$-tuple
$\overline{\sigma}$ satisfies $\left\vert \overline{\sigma}\right\vert
=1+2+\cdots+n$\ \ \ \ \footnote{\textit{Proof.} Let $\sigma\in S_{n}$. Thus,
$\sigma$ is a permutation of $\left[  n\right]  $ (since $S_{n}$ was defined
as the set of all permutations of $\left[  n\right]  $). In other words,
$\sigma$ is a bijection from $\left[  n\right]  $ to $\left[  n\right]  $.
However, the definition of $\overline{\sigma}$ yields $\overline{\sigma
}=\left(  \sigma\left(  1\right)  ,\sigma\left(  2\right)  ,\ldots
,\sigma\left(  n\right)  \right)  $. Thus,%
\begin{align*}
\left\vert \overline{\sigma}\right\vert  &  =\left\vert \left(  \sigma\left(
1\right)  ,\sigma\left(  2\right)  ,\ldots,\sigma\left(  n\right)  \right)
\right\vert =\sigma\left(  1\right)  +\sigma\left(  2\right)  +\cdots
+\sigma\left(  n\right) \\
&  \ \ \ \ \ \ \ \ \ \ \ \ \ \ \ \ \ \ \ \ \left(  \text{by the definition of
}\left\vert \left(  \sigma\left(  1\right)  ,\sigma\left(  2\right)
,\ldots,\sigma\left(  n\right)  \right)  \right\vert \right) \\
&  =\sum_{i\in\left[  n\right]  }\sigma\left(  i\right)  =\sum_{i\in\left[
n\right]  }i\ \ \ \ \ \ \ \ \ \ \left(
\begin{array}
[c]{c}%
\text{here, we have substituted }i\text{ for }\sigma\left(  i\right)  \text{
in the sum,}\\
\text{since the map }\sigma:\left[  n\right]  \rightarrow\left[  n\right]
\text{ is a bijection}%
\end{array}
\right) \\
&  =1+2+\cdots+n,
\end{align*}
qed.} and therefore
\begin{equation}
\underbrace{\left\vert \overline{\sigma}\right\vert }_{=1+2+\cdots
+n}+p=\left(  1+2+\cdots+n\right)  +p=\left\vert \eta\right\vert
\label{pf.thm.pre-pieri.modsigbar1}%
\end{equation}
(by (\ref{pf.thm.pre-pieri.sum-eta})). Moreover, for each $\sigma\in S_{n}$
and each $i\in\left[  n\right]  $, we have%
\begin{equation}
\overline{\sigma}_{i}=\sigma\left(  i\right)  \label{pf.thm.pre-pieri.sigbari}%
\end{equation}
(since $\overline{\sigma}=\left(  \sigma\left(  1\right)  ,\sigma\left(
2\right)  ,\ldots,\sigma\left(  n\right)  \right)  $).

Now, we claim the following:

\begin{statement}
\textit{Claim 1:} Let $\beta\in\mathbb{N}^{n}$ and $\sigma\in S_{n}$ be
arbitrary. Then, the statement \textquotedblleft$\left\vert \beta\right\vert
=p$\textquotedblright\ is equivalent to \textquotedblleft$\left\vert
\beta+\overline{\sigma}\right\vert =\left\vert \eta\right\vert $%
\textquotedblright.
\end{statement}

[\textit{Proof of Claim 1:} We have $\beta+\overline{\sigma}=\overline{\sigma
}+\beta$ (since $\mathbb{Z}^{n}$ is an abelian group). Thus, $\left\vert
\beta+\overline{\sigma}\right\vert =\left\vert \overline{\sigma}%
+\beta\right\vert =\left\vert \overline{\sigma}\right\vert +\left\vert
\beta\right\vert $ (by Lemma \ref{lem.additive}, applied to $\alpha
=\overline{\sigma}$). Therefore,%
\begin{align}
\underbrace{\left\vert \beta+\overline{\sigma}\right\vert }_{=\left\vert
\overline{\sigma}\right\vert +\left\vert \beta\right\vert }%
-\underbrace{\left\vert \eta\right\vert }_{\substack{=\left\vert
\overline{\sigma}\right\vert +p\\\text{(by (\ref{pf.thm.pre-pieri.modsigbar1}%
))}}}  &  =\left(  \left\vert \overline{\sigma}\right\vert +\left\vert
\beta\right\vert \right)  -\left(  \left\vert \overline{\sigma}\right\vert
+p\right) \nonumber\\
&  =\left\vert \beta\right\vert -p. \label{pf.thm.pre-pieri.modsigbar2}%
\end{align}
Hence, we have the following chain of logical equivalences:%
\begin{align*}
\left(  \left\vert \beta\right\vert =p\right)  \  &  \Longleftrightarrow
\ \left(  \underbrace{\left\vert \beta\right\vert -p}_{\substack{=\left\vert
\beta+\overline{\sigma}\right\vert -\left\vert \eta\right\vert \\\text{(by
(\ref{pf.thm.pre-pieri.modsigbar2}))}}}=0\right)  \ \ \Longleftrightarrow
\ \left(  \left\vert \beta+\overline{\sigma}\right\vert -\left\vert
\eta\right\vert =0\right) \\
&  \Longleftrightarrow\ \left(  \left\vert \beta+\overline{\sigma}\right\vert
=\left\vert \eta\right\vert \right)  .
\end{align*}
In other words, the statement \textquotedblleft$\left\vert \beta\right\vert
=p$\textquotedblright\ is equivalent to \textquotedblleft$\left\vert
\beta+\overline{\sigma}\right\vert =\left\vert \eta\right\vert $%
\textquotedblright. This proves Claim 1.] \medskip

Let us now set%
\begin{equation}
\prod_{i=1}^{n}a_{i}:=a_{1}a_{2}\cdots a_{n}
\label{pf.thm.pre-pieri.bottom-up-prod}%
\end{equation}
for any $a_{1},a_{2},\ldots,a_{n}\in R$. (This is, of course, the usual
meaning of the notation $\prod_{i=1}^{n}a_{i}$ when the ring $R$ is
commutative; however, we are now extending it to the case of arbitrary $R$ by
using the formula (\ref{pf.thm.pre-pieri.bottom-up-prod}).)

Thus, if $\left(  a_{i,j}\right)  _{i,j\in\left[  n\right]  }$ is any $n\times
n$-matrix over $R$, then the definition of $\operatorname*{rowdet}\left(
\left(  a_{i,j}\right)  _{i,j\in\left[  n\right]  }\right)  $ yields%
\begin{align}
\operatorname*{rowdet}\left(  \left(  a_{i,j}\right)  _{i,j\in\left[
n\right]  }\right)   &  =\sum_{\sigma\in S_{n}}\left(  -1\right)  ^{\sigma
}\underbrace{a_{1,\sigma\left(  1\right)  }a_{2,\sigma\left(  2\right)
}\cdots a_{n,\sigma\left(  n\right)  }}_{=\prod_{i=1}^{n}a_{i,\sigma\left(
i\right)  }}\nonumber\\
&  =\sum_{\sigma\in S_{n}}\left(  -1\right)  ^{\sigma}\prod_{i=1}%
^{n}a_{i,\sigma\left(  i\right)  }. \label{pf.thm.pre-pieri.rowdetA=}%
\end{align}

For each $\beta\in\mathbb{N}^{n}$, we have%
\begin{align}
t_{\alpha+\beta}  &  =\operatorname*{rowdet}\left(  \left(  h_{\left(
\alpha+\beta\right)  _{i}+j,\ i}\right)  _{i,j\in\left[  n\right]  }\right)
\ \ \ \ \ \ \ \ \ \ \left(  \text{by the definition of }t_{\alpha+\beta
}\right) \nonumber\\
&  =\operatorname*{rowdet}\left(  \left(  h_{\alpha_{i}+\beta_{i}%
+j,\ i}\right)  _{i,j\in\left[  n\right]  }\right) \nonumber\\
&  \ \ \ \ \ \ \ \ \ \ \ \ \ \ \ \ \ \ \ \ \left(  \text{since }\left(
\alpha+\beta\right)  _{i}=\alpha_{i}+\beta_{i}\text{ for each }i\in\left[
n\right]  \right) \nonumber\\
&  =\sum_{\sigma\in S_{n}}\left(  -1\right)  ^{\sigma}\prod_{i=1}^{n}%
h_{\alpha_{i}+\beta_{i}+\sigma\left(  i\right)  ,\ i}%
\label{pf.thm.pre-pieri.t1}\\
&  \ \ \ \ \ \ \ \ \ \ \ \ \ \ \ \ \ \ \ \ \left(  \text{by
(\ref{pf.thm.pre-pieri.rowdetA=}), applied to }a_{i,j}=h_{\alpha_{i}+\beta
_{i}+j,\ i}\right)  .\nonumber
\end{align}
However, for each $\beta\in\mathbb{N}^{n}$ and each $\sigma\in S_{n}$, we have%
\[
\beta_{i}+\underbrace{\sigma\left(  i\right)  }_{\substack{=\overline{\sigma
}_{i}\\\text{(by (\ref{pf.thm.pre-pieri.sigbari}))}}}=\beta_{i}+\overline
{\sigma}_{i}=\left(  \beta+\overline{\sigma}\right)  _{i}%
\]
and therefore%
\begin{equation}
h_{\alpha_{i}+\beta_{i}+\sigma\left(  i\right)  ,\ i}=h_{\alpha_{i}+\left(
\beta+\overline{\sigma}\right)  _{i},\ i}. \label{pf.thm.pre-pieri.h=h1}%
\end{equation}

Hence, for each $\beta\in\mathbb{N}^{n}$, we have the equality%
\begin{align*}
t_{\alpha+\beta}  &  =\sum_{\sigma\in S_{n}}\left(  -1\right)  ^{\sigma}%
\prod_{i=1}^{n}\underbrace{h_{\alpha_{i}+\beta_{i}+\sigma\left(  i\right)
,\ i}}_{\substack{=h_{\alpha_{i}+\left(  \beta+\overline{\sigma}\right)
_{i},\ i}\\\text{(by (\ref{pf.thm.pre-pieri.h=h1}))}}%
}\ \ \ \ \ \ \ \ \ \ \left(  \text{by (\ref{pf.thm.pre-pieri.t1})}\right) \\
&  =\sum_{\sigma\in S_{n}}\left(  -1\right)  ^{\sigma}\prod_{i=1}^{n}%
h_{\alpha_{i}+\left(  \beta+\overline{\sigma}\right)  _{i},\ i}.
\end{align*}
Summing these equalities over all $\beta\in\mathbb{N}^{n}$ satisfying
$\left\vert \beta\right\vert =p$, we obtain%
\begin{align}
\sum_{\substack{\beta\in\mathbb{N}^{n};\\\left\vert \beta\right\vert
=p}}t_{\alpha+\beta}  &  =\underbrace{\sum_{\substack{\beta\in\mathbb{N}%
^{n};\\\left\vert \beta\right\vert =p}}\ \ \sum_{\sigma\in S_{n}}}%
_{=\sum_{\sigma\in S_{n}}\ \ \sum_{\substack{\beta\in\mathbb{N}^{n}%
;\\\left\vert \beta\right\vert =p}}}\left(  -1\right)  ^{\sigma}\prod
_{i=1}^{n}h_{\alpha_{i}+\left(  \beta+\overline{\sigma}\right)  _{i}%
,\ i}\nonumber\\
&  =\sum_{\sigma\in S_{n}}\ \ \sum_{\substack{\beta\in\mathbb{N}%
^{n};\\\left\vert \beta\right\vert =p}}\left(  -1\right)  ^{\sigma}\prod
_{i=1}^{n}h_{\alpha_{i}+\left(  \beta+\overline{\sigma}\right)  _{i}%
,\ i}\nonumber\\
&  =\sum_{\sigma\in S_{n}}\left(  -1\right)  ^{\sigma}\sum_{\substack{\beta
\in\mathbb{N}^{n};\\\left\vert \beta\right\vert =p}}\ \ \prod_{i=1}%
^{n}h_{\alpha_{i}+\left(  \beta+\overline{\sigma}\right)  _{i},\ i}.
\label{pf.thm.pre-pieri.sumt=1}%
\end{align}

Now, fix a permutation $\sigma\in S_{n}$. We shall rewrite the sum
$\sum_{\substack{\beta\in\mathbb{N}^{n};\\\left\vert \beta\right\vert
=p}}\ \ \prod_{i=1}^{n}h_{\alpha_{i}+\left(  \beta+\overline{\sigma}\right)
_{i},\ i}$ in terms of $\beta+\overline{\sigma}$. Indeed, Claim 1 shows that
for any $n$-tuple $\beta\in\mathbb{N}^{n}$, the statement \textquotedblleft%
$\left\vert \beta\right\vert =p$\textquotedblright\ is equivalent to
\textquotedblleft$\left\vert \beta+\overline{\sigma}\right\vert =\left\vert
\eta\right\vert $\textquotedblright. Hence, we have the following equality of
summation signs:%
\begin{align*}
\sum_{\substack{\beta\in\mathbb{N}^{n};\\\left\vert \beta\right\vert =p}}  &
=\sum_{\substack{\beta\in\mathbb{N}^{n};\\\left\vert \beta+\overline{\sigma
}\right\vert =\left\vert \eta\right\vert }}=\sum_{\substack{\beta\in
\mathbb{Z}^{n};\\\beta\in\mathbb{N}^{n};\\\left\vert \beta+\overline{\sigma
}\right\vert =\left\vert \eta\right\vert }}\ \ \ \ \ \ \ \ \ \ \left(
\text{since }\mathbb{N}^{n}\subseteq\mathbb{Z}^{n}\right) \\
&  =\sum_{\substack{\beta\in\mathbb{Z}^{n};\\\left(  \beta+\overline{\sigma
}\right)  -\overline{\sigma}\in\mathbb{N}^{n};\\\left\vert \beta
+\overline{\sigma}\right\vert =\left\vert \eta\right\vert }%
}\ \ \ \ \ \ \ \ \ \ \left(  \text{since }\beta=\left(  \beta+\overline
{\sigma}\right)  -\overline{\sigma}\right)  .
\end{align*}
Thus,%
\[
\underbrace{\sum_{\substack{\beta\in\mathbb{N}^{n};\\\left\vert \beta
\right\vert =p}}}_{\substack{=\sum_{\substack{\beta\in\mathbb{Z}^{n};\\\left(
\beta+\overline{\sigma}\right)  -\overline{\sigma}\in\mathbb{N}^{n}%
;\\\left\vert \beta+\overline{\sigma}\right\vert =\left\vert \eta\right\vert
}}}}\ \ \prod_{i=1}^{n}h_{\alpha_{i}+\left(  \beta+\overline{\sigma}\right)
_{i},\ i}=\sum_{\substack{\beta\in\mathbb{Z}^{n};\\\left(  \beta
+\overline{\sigma}\right)  -\overline{\sigma}\in\mathbb{N}^{n};\\\left\vert
\beta+\overline{\sigma}\right\vert =\left\vert \eta\right\vert }%
}\ \ \prod_{i=1}^{n}h_{\alpha_{i}+\left(  \beta+\overline{\sigma}\right)
_{i},\ i}.
\]

However, $\mathbb{Z}^{n}$ is a group (under addition). Hence, for any $\rho
\in\mathbb{Z}^{n}$, the map%
\begin{align*}
\mathbb{Z}^{n}  &  \rightarrow\mathbb{Z}^{n},\\
\beta &  \mapsto\beta+\rho
\end{align*}
is a bijection (whose inverse is the map $\mathbb{Z}^{n}\rightarrow
\mathbb{Z}^{n},\ \nu\mapsto\nu-\rho$). Applying this to $\rho=\overline
{\sigma}$, we conclude that the map
\begin{align*}
\mathbb{Z}^{n}  &  \rightarrow\mathbb{Z}^{n},\\
\beta &  \mapsto\beta+\overline{\sigma}%
\end{align*}
is a bijection (whose inverse is the map $\mathbb{Z}^{n}\rightarrow
\mathbb{Z}^{n},\ \nu\mapsto\nu-\overline{\sigma}$). Therefore, we can
substitute $\nu$ for $\beta+\overline{\sigma}$ in the sum $\sum
_{\substack{\beta\in\mathbb{Z}^{n};\\\left(  \beta+\overline{\sigma}\right)
-\overline{\sigma}\in\mathbb{N}^{n};\\\left\vert \beta+\overline{\sigma
}\right\vert =\left\vert \eta\right\vert }}\ \ \prod_{i=1}^{n}h_{\alpha
_{i}+\left(  \beta+\overline{\sigma}\right)  _{i},\ i}$. We thus obtain%
\[
\sum_{\substack{\beta\in\mathbb{Z}^{n};\\\left(  \beta+\overline{\sigma
}\right)  -\overline{\sigma}\in\mathbb{N}^{n};\\\left\vert \beta
+\overline{\sigma}\right\vert =\left\vert \eta\right\vert }}\ \ \prod
_{i=1}^{n}h_{\alpha_{i}+\left(  \beta+\overline{\sigma}\right)  _{i},\ i}%
=\sum_{\substack{\nu\in\mathbb{Z}^{n};\\\nu-\overline{\sigma}\in\mathbb{N}%
^{n};\\\left\vert \nu\right\vert =\left\vert \eta\right\vert }}\ \ \prod
_{i=1}^{n}h_{\alpha_{i}+\nu_{i},\ i}.
\]
Hence, our above computation becomes%
\begin{align}
\sum_{\substack{\beta\in\mathbb{N}^{n};\\\left\vert \beta\right\vert
=p}}\ \ \prod_{i=1}^{n}h_{\alpha_{i}+\left(  \beta+\overline{\sigma}\right)
_{i},\ i}  &  =\sum_{\substack{\beta\in\mathbb{Z}^{n};\\\left(  \beta
+\overline{\sigma}\right)  -\overline{\sigma}\in\mathbb{N}^{n};\\\left\vert
\beta+\overline{\sigma}\right\vert =\left\vert \eta\right\vert }%
}\ \ \prod_{i=1}^{n}h_{\alpha_{i}+\left(  \beta+\overline{\sigma}\right)
_{i},\ i}\nonumber\\
&  =\sum_{\substack{\nu\in\mathbb{Z}^{n};\\\nu-\overline{\sigma}\in
\mathbb{N}^{n};\\\left\vert \nu\right\vert =\left\vert \eta\right\vert
}}\ \ \prod_{i=1}^{n}h_{\alpha_{i}+\nu_{i},\ i}.
\label{pf.thm.pre-pieri.innersum}%
\end{align}

Forget that we fixed $\sigma$. We thus have proved
(\ref{pf.thm.pre-pieri.innersum}) for each permutation $\sigma\in S_{n}$. Now,
(\ref{pf.thm.pre-pieri.sumt=1}) becomes%
\begin{align}
\sum_{\substack{\beta\in\mathbb{N}^{n};\\\left\vert \beta\right\vert
=p}}t_{\alpha+\beta}  &  =\sum_{\sigma\in S_{n}}\left(  -1\right)  ^{\sigma
}\underbrace{\sum_{\substack{\beta\in\mathbb{N}^{n};\\\left\vert
\beta\right\vert =p}}\ \ \prod_{i=1}^{n}h_{\alpha_{i}+\left(  \beta
+\overline{\sigma}\right)  _{i},\ i}}_{\substack{=\sum_{\substack{\nu
\in\mathbb{Z}^{n};\\\nu-\overline{\sigma}\in\mathbb{N}^{n};\\\left\vert
\nu\right\vert =\left\vert \eta\right\vert }}\ \ \prod_{i=1}^{n}h_{\alpha
_{i}+\nu_{i},\ i}\\\text{(by (\ref{pf.thm.pre-pieri.innersum}))}}}\nonumber\\
&  =\sum_{\sigma\in S_{n}}\left(  -1\right)  ^{\sigma}\sum_{\substack{\nu
\in\mathbb{Z}^{n};\\\nu-\overline{\sigma}\in\mathbb{N}^{n};\\\left\vert
\nu\right\vert =\left\vert \eta\right\vert }}\ \ \prod_{i=1}^{n}h_{\alpha
_{i}+\nu_{i},\ i}\nonumber\\
&  =\underbrace{\sum_{\sigma\in S_{n}}\ \ \sum_{\substack{\nu\in\mathbb{Z}%
^{n};\\\nu-\overline{\sigma}\in\mathbb{N}^{n};\\\left\vert \nu\right\vert
=\left\vert \eta\right\vert }}}_{=\sum_{\substack{\nu\in\mathbb{Z}%
^{n};\\\left\vert \nu\right\vert =\left\vert \eta\right\vert }}\ \ \sum
_{\substack{\sigma\in S_{n};\\\nu-\overline{\sigma}\in\mathbb{N}^{n}}}}\left(
-1\right)  ^{\sigma}\prod_{i=1}^{n}h_{\alpha_{i}+\nu_{i},\ i}\nonumber\\
&  =\sum_{\substack{\nu\in\mathbb{Z}^{n};\\\left\vert \nu\right\vert
=\left\vert \eta\right\vert }}\ \ \sum_{\substack{\sigma\in S_{n}%
;\\\nu-\overline{\sigma}\in\mathbb{N}^{n}}}\left(  -1\right)  ^{\sigma}%
\prod_{i=1}^{n}h_{\alpha_{i}+\nu_{i},\ i}\nonumber\\
&  =\sum_{\substack{\nu\in\mathbb{Z}^{n};\\\left\vert \nu\right\vert
=\left\vert \eta\right\vert }}\left(  \sum_{\substack{\sigma\in S_{n}%
;\\\nu-\overline{\sigma}\in\mathbb{N}^{n}}}\left(  -1\right)  ^{\sigma
}\right)  \prod_{i=1}^{n}h_{\alpha_{i}+\nu_{i},\ i}.
\label{pf.thm.pre-pieri.2}%
\end{align}

Now, if $\nu\in\mathbb{Z}^{n}$ is an $n$-tuple satisfying $\left\vert
\nu\right\vert =\left\vert \eta\right\vert $, then%
\begin{align}
\sum_{\substack{\sigma\in S_{n};\\\nu-\overline{\sigma}\in\mathbb{N}^{n}%
}}\left(  -1\right)  ^{\sigma}  &  =\det\left(  \left(  \left[  \nu_{i}\geq
j\right]  \right)  _{i,j\in\left[  n\right]  }\right)
\ \ \ \ \ \ \ \ \ \ \left(  \text{by Lemma \ref{lem.nu-det2}}\right)
\nonumber\\
&  =\sum_{\substack{\sigma\in S_{n};\\\nu=\eta\circ\sigma}}\left(  -1\right)
^{\sigma}\ \ \ \ \ \ \ \ \ \ \left(  \text{by Lemma \ref{lem.nu-det}}\right)
. \label{pf.thm.pre-pieri.sum=sum}%
\end{align}

Thus, (\ref{pf.thm.pre-pieri.2}) becomes%
\begin{align}
\sum_{\substack{\beta\in\mathbb{N}^{n};\\\left\vert \beta\right\vert
=p}}t_{\alpha+\beta}  &  =\sum_{\substack{\nu\in\mathbb{Z}^{n};\\\left\vert
\nu\right\vert =\left\vert \eta\right\vert }}\underbrace{\left(
\sum_{\substack{\sigma\in S_{n};\\\nu-\overline{\sigma}\in\mathbb{N}^{n}%
}}\left(  -1\right)  ^{\sigma}\right)  }_{\substack{=\sum_{\substack{\sigma\in
S_{n};\\\nu=\eta\circ\sigma}}\left(  -1\right)  ^{\sigma}\\\text{(by
(\ref{pf.thm.pre-pieri.sum=sum}))}}}\prod_{i=1}^{n}h_{\alpha_{i}+\nu_{i}%
,\ i}\nonumber\\
&  =\sum_{\substack{\nu\in\mathbb{Z}^{n};\\\left\vert \nu\right\vert
=\left\vert \eta\right\vert }}\left(  \sum_{\substack{\sigma\in S_{n}%
;\\\nu=\eta\circ\sigma}}\left(  -1\right)  ^{\sigma}\right)  \prod_{i=1}%
^{n}h_{\alpha_{i}+\nu_{i},\ i}\nonumber\\
&  =\sum_{\substack{\nu\in\mathbb{Z}^{n};\\\left\vert \nu\right\vert
=\left\vert \eta\right\vert }}\ \ \sum_{\substack{\sigma\in S_{n};\\\nu
=\eta\circ\sigma}}\left(  -1\right)  ^{\sigma}\prod_{i=1}^{n}h_{\alpha_{i}%
+\nu_{i},\ i}. \label{pf.thm.pre-pieri.2b}%
\end{align}

On the other hand,%
\begin{equation}
\operatorname*{rowdet}\left(  \left(  h_{\alpha_{i}+\eta_{j},\ i}\right)
_{i,j\in\left[  n\right]  }\right)  =\sum_{\sigma\in S_{n}}\left(  -1\right)
^{\sigma}\prod_{i=1}^{n}h_{\alpha_{i}+\eta_{\sigma\left(  i\right)  },\ i}
\label{pf.thm.pre-pieri.3a}%
\end{equation}
(by (\ref{pf.thm.pre-pieri.rowdetA=}), applied to $a_{i,j}=h_{\alpha_{i}%
+\eta_{j},\ i}$). However, for each $\sigma\in S_{n}$, the $n$-tuple
$\eta\circ\sigma$ belongs to $\mathbb{Z}^{n}$ and satisfies $\left\vert
\eta\circ\sigma\right\vert =\left\vert \eta\right\vert $ (by Proposition
\ref{prop.etapi.len}). In other words, for each $\sigma\in S_{n}$, the
$n$-tuple $\eta\circ\sigma$ is a $\nu\in\mathbb{Z}^{n}$ satisfying $\left\vert
\nu\right\vert =\left\vert \eta\right\vert $. Hence, any sum ranging over all
$\sigma\in S_{n}$ can be split according to the value of $\eta\circ\sigma$. In
other words, we have the following equality of summation signs:%
\[
\sum_{\sigma\in S_{n}}=\sum_{\substack{\nu\in\mathbb{Z}^{n};\\\left\vert
\nu\right\vert =\left\vert \eta\right\vert }}\ \ \underbrace{\sum
_{\substack{\sigma\in S_{n};\\\eta\circ\sigma=\nu}}}_{=\sum_{\substack{\sigma
\in S_{n};\\\nu=\eta\circ\sigma}}}=\sum_{\substack{\nu\in\mathbb{Z}%
^{n};\\\left\vert \nu\right\vert =\left\vert \eta\right\vert }}\ \ \sum
_{\substack{\sigma\in S_{n};\\\nu=\eta\circ\sigma}}.
\]
Thus, (\ref{pf.thm.pre-pieri.3a}) becomes%
\begin{align*}
\operatorname*{rowdet}\left(  \left(  h_{\alpha_{i}+\eta_{j},\ i}\right)
_{i,j\in\left[  n\right]  }\right)   &  =\underbrace{\sum_{\sigma\in S_{n}}%
}_{=\sum_{\substack{\nu\in\mathbb{Z}^{n};\\\left\vert \nu\right\vert
=\left\vert \eta\right\vert }}\ \ \sum_{\substack{\sigma\in S_{n};\\\nu
=\eta\circ\sigma}}}\left(  -1\right)  ^{\sigma}\prod_{i=1}^{n}h_{\alpha
_{i}+\eta_{\sigma\left(  i\right)  },\ i}\\
&  =\sum_{\substack{\nu\in\mathbb{Z}^{n};\\\left\vert \nu\right\vert
=\left\vert \eta\right\vert }}\ \ \sum_{\substack{\sigma\in S_{n};\\\nu
=\eta\circ\sigma}}\left(  -1\right)  ^{\sigma}\prod_{i=1}^{n}%
\underbrace{h_{\alpha_{i}+\eta_{\sigma\left(  i\right)  },\ i}}%
_{\substack{=h_{\alpha_{i}+\nu_{i},\ i}\\\text{(since }\eta_{\sigma\left(
i\right)  }=\nu_{i}\\\text{(because }\nu=\eta\circ\sigma=\left(  \eta
_{\sigma\left(  1\right)  },\eta_{\sigma\left(  2\right)  },\ldots
,\eta_{\sigma\left(  n\right)  }\right)  \\\text{(by Definition
\ref{def.etapi}) and therefore }\nu_{i}=\eta_{\sigma\left(  i\right)
}\text{))}}}\\
&  =\sum_{\substack{\nu\in\mathbb{Z}^{n};\\\left\vert \nu\right\vert
=\left\vert \eta\right\vert }}\ \ \sum_{\substack{\sigma\in S_{n};\\\nu
=\eta\circ\sigma}}\left(  -1\right)  ^{\sigma}\prod_{i=1}^{n}h_{\alpha_{i}%
+\nu_{i},\ i}.
\end{align*}
Comparing this with (\ref{pf.thm.pre-pieri.2b}), we obtain%
\[
\sum_{\substack{\beta\in\mathbb{N}^{n};\\\left\vert \beta\right\vert
=p}}t_{\alpha+\beta}=\operatorname*{rowdet}\left(  \left(  h_{\alpha_{i}%
+\eta_{j},\ i}\right)  _{i,j\in\left[  n\right]  }\right)  .
\]
This proves Theorem \ref{thm.pre-pieri}.
\end{proof}
\end{verlong}

\section{Corollaries}

\subsection{The $h_{k,\ n}=0$ case}

We shall now derive some corollaries from Theorem \ref{thm.pre-pieri} by
imposing some conditions on $R$ or on the $h_{k,\ i}$. We begin with the most
basic one, in which we force $h_{k,\ n}$ to be $0$ for all negative $k$:

\begin{corollary}
\label{cor.pre-pieri.2}Let $n\in\mathbb{P}$ and $p\in\mathbb{Z}$. Let
$h_{k,\ i}$ be an element of $R$ for all $k\in\mathbb{Z}$ and $i\in\left[
n\right]  $. Assume that%
\begin{equation}
h_{k,\ n}=0\ \ \ \ \ \ \ \ \ \ \text{for all }k<0.
\label{eq.cor.pre-pieri.2.ass0}%
\end{equation}

For any $\alpha\in\mathbb{Z}^{n}$, we define%
\[
t_{\alpha}:=\operatorname*{rowdet}\left(  \left(  h_{\alpha_{i}+j,\ i}\right)
_{i,j\in\left[  n\right]  }\right)  \in R.
\]

Let $\alpha\in\mathbb{Z}^{n}$ be such that $\alpha_{n}\leq-n$. Then,%
\[
\sum_{\substack{\beta\in\mathbb{N}^{n};\\\left\vert \beta\right\vert
=p}}t_{\alpha+\beta}=\operatorname*{rowdet}\left(  \left(  h_{\alpha
_{i}+j,\ i}\right)  _{i,j\in\left[  n-1\right]  }\right)  \cdot h_{\alpha
_{n}+n+p,\ n}.
\]

\end{corollary}

In order to derive this from Theorem \ref{thm.pre-pieri}, we shall need the
following formula for the row-determinant of a matrix whose last row is $0$
except for its rightmost entry:

\begin{lemma}
\label{lem.laplace.pre}Let $n$ be a positive integer. Let $A=\left(
a_{i,j}\right)  _{i,j\in\left[  n\right]  }\in R^{n\times n}$ be an $n\times
n$-matrix. Assume that%
\begin{equation}
a_{n,j}=0\ \ \ \ \ \ \ \ \ \ \text{for every }j\in\left\{  1,2,\ldots
,n-1\right\}  . \label{eq.lem.laplace.pre.ass}%
\end{equation}
Then,
\[
\operatorname*{rowdet}A=\operatorname*{rowdet}\left(  \left(  a_{i,j}\right)
_{i,j\in\left[  n-1\right]  }\right)  \cdot a_{n,n}.
\]

\end{lemma}

\begin{example}
For $n=3$ and $A=\left(
\begin{array}
[c]{ccc}%
a & b & c\\
d & e & f\\
0 & 0 & g
\end{array}
\right)  $, the claim of Lemma \ref{lem.laplace.pre} says that%
\[
\operatorname*{rowdet}\left(
\begin{array}
[c]{ccc}%
a & b & c\\
d & e & f\\
0 & 0 & g
\end{array}
\right)  =\operatorname*{rowdet}\left(
\begin{array}
[c]{cc}%
a & b\\
d & e
\end{array}
\right)  \cdot g.
\]
(The two zeroes in the third row of $A$ are necessary for Lemma
\ref{lem.laplace.pre} to be applicable, as they guarantee that the condition
(\ref{eq.lem.laplace.pre.ass}) is satisfied.)
\end{example}

\begin{vershort}
\begin{proof}
[Proof of Lemma \ref{lem.laplace.pre}.]This is a generalization of
\cite[Theorem 6.43]{detnotes} to the case of arbitrary $R$ (not necessarily
commutative). The proof given in \cite{detnotes} still works for this
generalization, as long as we keep in mind that the products are noncommutative.
\end{proof}
\end{vershort}

\begin{verlong}
\begin{proof}
[Proof of Lemma \ref{lem.laplace.pre}.]This is a generalization of
\cite[Theorem 6.43]{detnotes} to the case of arbitrary $R$ (not necessarily
commutative). The proof given in \cite{detnotes} still works for this
generalization, as long as we keep in mind that the products are noncommutative.

For the sake of completeness, here is this argument in detail, except for the
parts that can be copied from \cite[proof of Theorem 6.43]{detnotes} verbatim.

We shall first prove the following auxiliary claim (an analogue of \cite[Lemma
6.44]{detnotes}):

\begin{statement}
\textit{Claim 1:} We have
\[
\sum_{\substack{\sigma\in S_{n};\\\sigma\left(  n\right)  =n}}\left(
-1\right)  ^{\sigma}a_{1,\sigma\left(  1\right)  }a_{2,\sigma\left(  2\right)
}\cdots a_{n-1,\sigma\left(  n-1\right)  }=\operatorname*{rowdet}\left(
\left(  a_{i,j}\right)  _{i,j\in\left[  n-1\right]  }\right)  .
\]

\end{statement}

[\textit{Proof of Claim 1:} We define a subset $T$ of $S_{n}$ by%
\[
T=\left\{  \tau\in S_{n}\ \mid\ \tau\left(  n\right)  =n\right\}  .
\]

For every $\sigma\in S_{n-1}$, we define a map $\widehat{\sigma}:\left[
n\right]  \rightarrow\left[  n\right]  $ by setting%
\[
\widehat{\sigma}\left(  i\right)  =%
\begin{cases}
\sigma\left(  i\right)  , & \text{if }i<n;\\
n, & \text{if }i=n
\end{cases}
\ \ \ \ \ \ \ \ \ \ \text{for every }i\in\left[  n\right]  .
\]
It is easy to see (and has been proved in \cite[proof of Lemma 6.44]%
{detnotes}\footnote{A reader looking these arguments up in \cite[proof of
Lemma 6.44]{detnotes} should keep in mind that we are here using the notation
$\left[  n\right]  $ for what has been called $\left\{  1,2,\ldots,n\right\}
$ in \cite[proof of Lemma 6.44]{detnotes}.}) that this map $\widehat{\sigma}$
is well-defined and belongs to $T$. Thus, we can define a map $\Phi
:S_{n-1}\rightarrow T$ by setting%
\[
\Phi\left(  \sigma\right)  =\widehat{\sigma}\ \ \ \ \ \ \ \ \ \ \text{for
every }\sigma\in S_{n-1}.
\]
Consider this map $\Phi$. It is not hard to see (and has been proved in
\cite[proof of Lemma 6.44]{detnotes}) that the map $\Phi$ is a bijection.
Furthermore, every $\sigma\in S_{n-1}$ satisfies%
\begin{equation}
\left(  -1\right)  ^{\widehat{\sigma}}=\left(  -1\right)  ^{\sigma}
\label{pf.lem.laplace.pre.c1.pf.-1}%
\end{equation}
(again, this has been proved in \cite[proof of Lemma 6.44]{detnotes}).
Finally, every $\sigma\in S_{n-1}$ satisfies%
\begin{align}
&  a_{1,\widehat{\sigma}\left(  1\right)  }a_{2,\widehat{\sigma}\left(
2\right)  }\cdots a_{n-1,\widehat{\sigma}\left(  n-1\right)  }\nonumber\\
&  =a_{1,\sigma\left(  1\right)  }a_{2,\sigma\left(  2\right)  }\cdots
a_{n-1,\sigma\left(  n-1\right)  } \label{pf.lem.laplace.pre.c1.pf.prod}%
\end{align}
\footnote{\textit{Proof of (\ref{pf.lem.laplace.pre.c1.pf.prod}):} Let
$\sigma\in S_{n-1}$. Let $i\in\left[  n-1\right]  $. Then, $i\in\left[
n-1\right]  \subseteq\left[  n\right]  $. Moreover, from $i\in\left[
n-1\right]  =\left\{  1,2,\ldots,n-1\right\}  $, we obtain $i\leq n-1<n$. Now,
the definition of $\widehat{\sigma}$ yields
\[
\widehat{\sigma}\left(  i\right)  =%
\begin{cases}
\sigma\left(  i\right)  , & \text{if }i<n;\\
n, & \text{if }i=n
\end{cases}
\ \ =\sigma\left(  i\right)  \ \ \ \ \ \ \ \ \ \ \left(  \text{since
}i<n\right)  .
\]
Hence, $a_{i,\widehat{\sigma}\left(  i\right)  }=a_{i,\sigma\left(  i\right)
}$.
\par
Forget that we fixed $i$. We thus have proved that the equality
$a_{i,\widehat{\sigma}\left(  i\right)  }=a_{i,\sigma\left(  i\right)  }$
holds for each $i\in\left[  n-1\right]  $. Multiplying these equalities for
all $i\in\left[  n-1\right]  $, we obtain $a_{1,\widehat{\sigma}\left(
1\right)  }a_{2,\widehat{\sigma}\left(  2\right)  }\cdots
a_{n-1,\widehat{\sigma}\left(  n-1\right)  }=a_{1,\sigma\left(  1\right)
}a_{2,\sigma\left(  2\right)  }\cdots a_{n-1,\sigma\left(  n-1\right)  }$.
This proves (\ref{pf.lem.laplace.pre.c1.pf.prod}).}.

Now,%
\begin{align*}
&  \underbrace{\sum_{\substack{\sigma\in S_{n};\\\sigma\left(  n\right)  =n}%
}}_{\substack{=\sum_{\sigma\in\left\{  \tau\in S_{n}\ \mid\ \tau\left(
n\right)  =n\right\}  }=\sum_{\sigma\in T}\\\text{(since }\left\{  \tau\in
S_{n}\ \mid\ \tau\left(  n\right)  =n\right\}  =T\text{)}}}\left(  -1\right)
^{\sigma}a_{1,\sigma\left(  1\right)  }a_{2,\sigma\left(  2\right)  }\cdots
a_{n-1,\sigma\left(  n-1\right)  }\\
&  =\sum_{\sigma\in T}\left(  -1\right)  ^{\sigma}a_{1,\sigma\left(  1\right)
}a_{2,\sigma\left(  2\right)  }\cdots a_{n-1,\sigma\left(  n-1\right)  }\\
&  =\sum_{\sigma\in S_{n-1}}\underbrace{\left(  -1\right)  ^{\Phi\left(
\sigma\right)  }a_{1,\left(  \Phi\left(  \sigma\right)  \right)  \left(
1\right)  }a_{2,\left(  \Phi\left(  \sigma\right)  \right)  \left(  2\right)
}\cdots a_{n-1,\left(  \Phi\left(  \sigma\right)  \right)  \left(  n-1\right)
}}_{\substack{=\left(  -1\right)  ^{\widehat{\sigma}}a_{1,\widehat{\sigma
}\left(  1\right)  }a_{2,\widehat{\sigma}\left(  2\right)  }\cdots
a_{n-1,\widehat{\sigma}\left(  n-1\right)  }\\\text{(since }\Phi\left(
\sigma\right)  =\widehat{\sigma}\text{)}}}\\
&  \ \ \ \ \ \ \ \ \ \ \ \ \ \ \ \ \ \ \ \ \left(
\begin{array}
[c]{c}%
\text{here, we have substituted }\Phi\left(  \sigma\right)  \text{ for }%
\sigma\text{ in the sum,}\\
\text{since the map }\Phi:S_{n-1}\rightarrow T\text{ is a bijection}%
\end{array}
\right) \\
&  =\sum_{\sigma\in S_{n-1}}\underbrace{\left(  -1\right)  ^{\widehat{\sigma}%
}}_{\substack{=\left(  -1\right)  ^{\sigma}\\\text{(by
(\ref{pf.lem.laplace.pre.c1.pf.-1}))}}}\underbrace{a_{1,\widehat{\sigma
}\left(  1\right)  }a_{2,\widehat{\sigma}\left(  2\right)  }\cdots
a_{n-1,\widehat{\sigma}\left(  n-1\right)  }}_{\substack{=a_{1,\sigma\left(
1\right)  }a_{2,\sigma\left(  2\right)  }\cdots a_{n-1,\sigma\left(
n-1\right)  }\\\text{(by (\ref{pf.lem.laplace.pre.c1.pf.prod}))}}}\\
&  =\sum_{\sigma\in S_{n-1}}\left(  -1\right)  ^{\sigma}a_{1,\sigma\left(
1\right)  }a_{2,\sigma\left(  2\right)  }\cdots a_{n-1,\sigma\left(
n-1\right)  }.
\end{align*}
But the definition of the row-determinant of a matrix yields%
\[
\operatorname*{rowdet}\left(  \left(  a_{i,j}\right)  _{i,j\in\left[
n-1\right]  }\right)  =\sum_{\sigma\in S_{n-1}}\left(  -1\right)  ^{\sigma
}a_{1,\sigma\left(  1\right)  }a_{2,\sigma\left(  2\right)  }\cdots
a_{n-1,\sigma\left(  n-1\right)  }.
\]
Comparing these two equalities, we obtain%
\[
\sum_{\substack{\sigma\in S_{n};\\\sigma\left(  n\right)  =n}}\left(
-1\right)  ^{\sigma}a_{1,\sigma\left(  1\right)  }a_{2,\sigma\left(  2\right)
}\cdots a_{n-1,\sigma\left(  n-1\right)  }=\operatorname*{rowdet}\left(
\left(  a_{i,j}\right)  _{i,j\in\left[  n-1\right]  }\right)  .
\]
This proves Claim 1.] \medskip

Now, every permutation $\sigma\in S_{n}$ satisfying $\sigma\left(  n\right)
\neq n$ satisfies
\begin{equation}
a_{n,\sigma\left(  n\right)  }=0 \label{pf.lem.laplace.pre.1}%
\end{equation}
\footnote{\textit{Proof of (\ref{pf.lem.laplace.pre.1}):} Let $\sigma\in
S_{n}$ be a permutation satisfying $\sigma\left(  n\right)  \neq n$. Since
$\sigma\left(  n\right)  \in\left[  n\right]  =\left\{  1,2,\ldots,n\right\}
$ and $\sigma\left(  n\right)  \neq n$, we have $\sigma\left(  n\right)
\in\left\{  1,2,\ldots,n\right\}  \setminus\left\{  n\right\}  =\left\{
1,2,\ldots,n-1\right\}  $. Hence, (\ref{eq.lem.laplace.pre.ass}) (applied to
$j=\sigma\left(  n\right)  $) shows that $a_{n,\sigma\left(  n\right)  }=0$.
This proves (\ref{pf.lem.laplace.pre.1}).}.

Recall that $A=\left(  a_{i,j}\right)  _{i,j\in\left[  n\right]  }$. Hence,
the definition of the row-determinant of a matrix yields%
\begin{align*}
\operatorname*{rowdet}A  &  =\sum_{\sigma\in S_{n}}\left(  -1\right)
^{\sigma}\underbrace{a_{1,\sigma\left(  1\right)  }a_{2,\sigma\left(
2\right)  }\cdots a_{n,\sigma\left(  n\right)  }}_{=\left(  a_{1,\sigma\left(
1\right)  }a_{2,\sigma\left(  2\right)  }\cdots a_{n-1,\sigma\left(
n-1\right)  }\right)  a_{n,\sigma\left(  n\right)  }}\\
&  =\sum_{\sigma\in S_{n}}\left(  -1\right)  ^{\sigma}\left(  a_{1,\sigma
\left(  1\right)  }a_{2,\sigma\left(  2\right)  }\cdots a_{n-1,\sigma\left(
n-1\right)  }\right)  a_{n,\sigma\left(  n\right)  }\\
&  =\sum_{\substack{\sigma\in S_{n};\\\sigma\left(  n\right)  =n}}\left(
-1\right)  ^{\sigma}\left(  a_{1,\sigma\left(  1\right)  }a_{2,\sigma\left(
2\right)  }\cdots a_{n-1,\sigma\left(  n-1\right)  }\right)
\underbrace{a_{n,\sigma\left(  n\right)  }}_{\substack{=a_{n,n}\\\text{(since
}\sigma\left(  n\right)  =n\text{)}}}\\
&  \ \ \ \ \ \ \ \ \ \ +\sum_{\substack{\sigma\in S_{n};\\\sigma\left(
n\right)  \neq n}}\left(  -1\right)  ^{\sigma}\left(  a_{1,\sigma\left(
1\right)  }a_{2,\sigma\left(  2\right)  }\cdots a_{n-1,\sigma\left(
n-1\right)  }\right)  \underbrace{a_{n,\sigma\left(  n\right)  }%
}_{\substack{=0\\\text{(by (\ref{pf.lem.laplace.pre.1}))}}}\\
&  \ \ \ \ \ \ \ \ \ \ \ \ \ \ \ \ \ \ \ \ \left(
\begin{array}
[c]{c}%
\text{since every }\sigma\in S_{n}\text{ satisfies}\\
\text{either }\sigma\left(  n\right)  =n\text{ or }\sigma\left(  n\right)
\neq n\text{ (but not both)}%
\end{array}
\right) \\
&  =\sum_{\substack{\sigma\in S_{n};\\\sigma\left(  n\right)  =n}}\left(
-1\right)  ^{\sigma}\left(  a_{1,\sigma\left(  1\right)  }a_{2,\sigma\left(
2\right)  }\cdots a_{n-1,\sigma\left(  n-1\right)  }\right)  a_{n,n}\\
&  \ \ \ \ \ \ \ \ \ \ +\underbrace{\sum_{\substack{\sigma\in S_{n}%
;\\\sigma\left(  n\right)  \neq n}}\left(  -1\right)  ^{\sigma}\left(
a_{1,\sigma\left(  1\right)  }a_{2,\sigma\left(  2\right)  }\cdots
a_{n-1,\sigma\left(  n-1\right)  }\right)  0}_{=0}\\
&  =\sum_{\substack{\sigma\in S_{n};\\\sigma\left(  n\right)  =n}}\left(
-1\right)  ^{\sigma}\left(  a_{1,\sigma\left(  1\right)  }a_{2,\sigma\left(
2\right)  }\cdots a_{n-1,\sigma\left(  n-1\right)  }\right)  a_{n,n}\\
&  =\underbrace{\left(  \sum_{\substack{\sigma\in S_{n};\\\sigma\left(
n\right)  =n}}\left(  -1\right)  ^{\sigma}a_{1,\sigma\left(  1\right)
}a_{2,\sigma\left(  2\right)  }\cdots a_{n-1,\sigma\left(  n-1\right)
}\right)  }_{\substack{=\operatorname*{rowdet}\left(  \left(  a_{i,j}\right)
_{i,j\in\left[  n-1\right]  }\right)  \\\text{(by Claim 1)}}}\cdot a_{n,n}\\
&  =\operatorname*{rowdet}\left(  \left(  a_{i,j}\right)  _{i,j\in\left[
n-1\right]  }\right)  \cdot a_{n,n}.
\end{align*}
This proves Lemma \ref{lem.laplace.pre}.
\end{proof}
\end{verlong}

\begin{proof}
[Proof of Corollary \ref{cor.pre-pieri.2}.]If $p<0$, then Corollary
\ref{cor.pre-pieri.2} is easily seen to hold\footnotemark. Thus, for the rest
of this proof, we WLOG assume that we don't have $p<0$. Hence, we have
$p\geq0$. Hence, $p\in\mathbb{N}$ (since $p\in\mathbb{Z}$).

\begin{vershort}
\footnotetext{\textit{Proof.} Assume that $p<0$. Then, the sum $\sum
_{\substack{\beta\in\mathbb{N}^{n};\\\left\vert \beta\right\vert =p}%
}t_{\alpha+\beta}$ is empty and thus equals $0$. However, we have
$\underbrace{\alpha_{n}}_{\leq-n}+n+p\leq\left(  -n\right)  +n+p=p<0$, and
therefore (\ref{eq.cor.pre-pieri.2.ass0}) (applied to $k=\alpha_{n}+n+p$)
yields $h_{\alpha_{n}+n+p,\ n}=0$. Thus, the equality%
\[
\sum_{\substack{\beta\in\mathbb{N}^{n};\\\left\vert \beta\right\vert
=p}}t_{\alpha+\beta}=\operatorname*{rowdet}\left(  \left(  h_{\alpha
_{i}+j,\ i}\right)  _{i,j\in\left[  n-1\right]  }\right)  \cdot h_{\alpha
_{n}+n+p,\ n}%
\]
holds due to both of its sides being $0$. Thus, we have shown that Corollary
\ref{cor.pre-pieri.2} holds under the assumption that $p<0$.}
\end{vershort}

\begin{verlong}
\footnotetext{\textit{Proof.} Assume that $p<0$. Then, $\underbrace{\alpha
_{n}}_{\leq-n}+n+p\leq\left(  -n\right)  +n+p=p<0$. Thus,
(\ref{eq.cor.pre-pieri.2.ass0}) (applied to $k=\alpha_{n}+n+p$) yields
$h_{\alpha_{n}+n+p,\ n}=0$.
\par
However, if $\beta\in\mathbb{N}^{n}$, then
\begin{align*}
\left\vert \beta\right\vert  &  =\beta_{1}+\beta_{2}+\cdots+\beta
_{n}\ \ \ \ \ \ \ \ \ \ \left(  \text{by the definition of }\left\vert
\beta\right\vert \right) \\
&  =\sum_{i=1}^{n}\underbrace{\beta_{i}}_{\substack{\geq0\\\text{(since }%
\beta_{i}\in\mathbb{N}\\\text{(because }\beta\in\mathbb{N}^{n}\text{))}}%
}\geq\sum_{i=1}^{n}0=0>p\ \ \ \ \ \ \ \ \ \ \left(  \text{since }p<0\right)
\end{align*}
and therefore $\left\vert \beta\right\vert \neq p$. In other words, there
exists no $\beta\in\mathbb{N}^{n}$ satisfying $\left\vert \beta\right\vert
=p$. Hence, the sum $\sum_{\substack{\beta\in\mathbb{N}^{n};\\\left\vert
\beta\right\vert =p}}t_{\alpha+\beta}$ is empty. Thus,%
\[
\sum_{\substack{\beta\in\mathbb{N}^{n};\\\left\vert \beta\right\vert
=p}}t_{\alpha+\beta}=\left(  \text{empty sum}\right)  =0.
\]
Comparing this with%
\[
\operatorname*{rowdet}\left(  \left(  h_{\alpha_{i}+j,\ i}\right)
_{i,j\in\left[  n-1\right]  }\right)  \cdot\underbrace{h_{\alpha_{n}+n+p,\ n}%
}_{=0}=0,
\]
we obtain%
\[
\sum_{\substack{\beta\in\mathbb{N}^{n};\\\left\vert \beta\right\vert
=p}}t_{\alpha+\beta}=\operatorname*{rowdet}\left(  \left(  h_{\alpha
_{i}+j,\ i}\right)  _{i,j\in\left[  n-1\right]  }\right)  \cdot h_{\alpha
_{n}+n+p,\ n}.
\]
Thus, we have shown that Corollary \ref{cor.pre-pieri.2} holds under the
assumption that $p<0$.}
\end{verlong}

Define an $n$-tuple $\eta\in\mathbb{Z}^{n}$ as in Theorem \ref{thm.pre-pieri}.
Thus,%
\[
\left(  \eta_{1},\eta_{2},\ldots,\eta_{n}\right)  =\eta=\left(  1,2,\ldots
,n-1,n+p\right)  .
\]
In other words, we have%
\begin{equation}
\eta_{j}=j\ \ \ \ \ \ \ \ \ \ \text{for each }j\in\left[  n-1\right]
\label{pf.cor.pre-pieri.2.etai=i}%
\end{equation}
and%
\begin{equation}
\eta_{n}=n+p. \label{pf.cor.pre-pieri.2.etan=}%
\end{equation}

Now, Theorem \ref{thm.pre-pieri} yields%
\begin{equation}
\sum_{\substack{\beta\in\mathbb{N}^{n};\\\left\vert \beta\right\vert
=p}}t_{\alpha+\beta}=\operatorname*{rowdet}\left(  \left(  h_{\alpha_{i}%
+\eta_{j},\ i}\right)  _{i,j\in\left[  n\right]  }\right)  .
\label{pf.cor.pre-pieri.2.thm}%
\end{equation}

\begin{vershort}
However, it is easy to see that $h_{\alpha_{n}+\eta_{j},\ n}=0$ for every
$j\in\left\{  1,2,\ldots,n-1\right\}  $\ \ \ \ \footnote{\textit{Proof.} Let
$j\in\left\{  1,2,\ldots,n-1\right\}  $. We must show that $h_{\alpha_{n}%
+\eta_{j},\ n}=0$.
\par
We have $j\in\left\{  1,2,\ldots,n-1\right\}  =\left[  n-1\right]  $ and
therefore $\eta_{j}=j$ (by (\ref{pf.cor.pre-pieri.2.etai=i})). From
$j\in\left\{  1,2,\ldots,n-1\right\}  $, we also obtain $j\leq n-1<n$, so that
$\eta_{j}=j<n$ and thus $\alpha_{n}+\eta_{j}<\alpha_{n}+n\leq0$ (since
$\alpha_{n}\leq-n$). Therefore, (\ref{eq.cor.pre-pieri.2.ass0}) (applied to
$k=\alpha_{n}+\eta_{j}$) yields $h_{\alpha_{n}+\eta_{j},\ n}=0$, qed.}. Hence,
Lemma \ref{lem.laplace.pre} (applied to $\left(  h_{\alpha_{i}+\eta_{j}%
,\ i}\right)  _{i,j\in\left[  n\right]  }$ and $h_{\alpha_{i}+\eta_{j},\ i}$
instead of $A$ and $a_{i,j}$) yields%
\begin{align*}
\operatorname*{rowdet}\left(  \left(  h_{\alpha_{i}+\eta_{j},\ i}\right)
_{i,j\in\left[  n\right]  }\right)   &  =\operatorname*{rowdet}\left(  \left(
h_{\alpha_{i}+\eta_{j},\ i}\right)  _{i,j\in\left[  n-1\right]  }\right)
\cdot h_{\alpha_{n}+\eta_{n},\ n}\\
&  =\operatorname*{rowdet}\left(  \left(  h_{\alpha_{i}+j,\ i}\right)
_{i,j\in\left[  n-1\right]  }\right)  \cdot h_{\alpha_{n}+n+p,\ n}%
\end{align*}
(by (\ref{pf.cor.pre-pieri.2.etai=i}) and (\ref{pf.cor.pre-pieri.2.etan=})).
Hence, (\ref{pf.cor.pre-pieri.2.thm}) can be rewritten as%
\[
\sum_{\substack{\beta\in\mathbb{N}^{n};\\\left\vert \beta\right\vert
=p}}t_{\alpha+\beta}=\operatorname*{rowdet}\left(  \left(  h_{\alpha
_{i}+j,\ i}\right)  _{i,j\in\left[  n-1\right]  }\right)  \cdot h_{\alpha
_{n}+n+p,\ n}.
\]

\end{vershort}

\begin{verlong}
However, it is easy to see that $h_{\alpha_{n}+\eta_{j},\ n}=0$ for every
$j\in\left\{  1,2,\ldots,n-1\right\}  $\ \ \ \ \footnote{\textit{Proof.} Let
$j\in\left\{  1,2,\ldots,n-1\right\}  $. We must show that $h_{\alpha_{n}%
+\eta_{j},\ n}=0$.
\par
We have $j\in\left\{  1,2,\ldots,n-1\right\}  =\left[  n-1\right]  $ and
therefore $\eta_{j}=j$ (by (\ref{pf.cor.pre-pieri.2.etai=i})). From
$j\in\left\{  1,2,\ldots,n-1\right\}  $, we also obtain $j\leq n-1<n$, so that
$\underbrace{\alpha_{n}}_{\leq-n}+\underbrace{\eta_{j}}_{=j<n}<\left(
-n\right)  +n=0$. Therefore, (\ref{eq.cor.pre-pieri.2.ass0}) (applied to
$k=\alpha_{n}+\eta_{j}$) yields $h_{\alpha_{n}+\eta_{j},\ n}=0$, qed.}. Hence,
Lemma \ref{lem.laplace.pre} (applied to $\left(  h_{\alpha_{i}+\eta_{j}%
,\ i}\right)  _{i,j\in\left[  n\right]  }$ and $h_{\alpha_{i}+\eta_{j},\ i}$
instead of $A$ and $a_{i,j}$) yields%
\[
\operatorname*{rowdet}\left(  \left(  h_{\alpha_{i}+\eta_{j},\ i}\right)
_{i,j\in\left[  n\right]  }\right)  =\operatorname*{rowdet}\left(  \left(
h_{\alpha_{i}+\eta_{j},\ i}\right)  _{i,j\in\left[  n-1\right]  }\right)
\cdot h_{\alpha_{n}+\eta_{n},\ n}.
\]
Hence, (\ref{pf.cor.pre-pieri.2.thm}) becomes%
\begin{align}
\sum_{\substack{\beta\in\mathbb{N}^{n};\\\left\vert \beta\right\vert
=p}}t_{\alpha+\beta}  &  =\operatorname*{rowdet}\left(  \left(  h_{\alpha
_{i}+\eta_{j},\ i}\right)  _{i,j\in\left[  n\right]  }\right) \nonumber\\
&  =\operatorname*{rowdet}\left(  \left(  h_{\alpha_{i}+\eta_{j},\ i}\right)
_{i,j\in\left[  n-1\right]  }\right)  \cdot h_{\alpha_{n}+\eta_{n},\ n}.
\label{pf.cor.pre-pieri.2.3}%
\end{align}

However, each $j\in\left[  n-1\right]  $ satisfies $\eta_{j}=j$ (by
(\ref{pf.cor.pre-pieri.2.etai=i})). Thus, $h_{\alpha_{i}+\eta_{j}%
,\ i}=h_{\alpha_{i}+j,\ i}$ for every $i,j\in\left[  n-1\right]  $. In other
words,
\begin{equation}
\left(  h_{\alpha_{i}+\eta_{j},\ i}\right)  _{i,j\in\left[  n-1\right]
}=\left(  h_{\alpha_{i}+j,\ i}\right)  _{i,j\in\left[  n-1\right]  }.
\label{pf.cor.pre-pieri.2.rewr1}%
\end{equation}
Using (\ref{pf.cor.pre-pieri.2.rewr1}) and (\ref{pf.cor.pre-pieri.2.etan=}),
we can rewrite (\ref{pf.cor.pre-pieri.2.3}) as
\[
\sum_{\substack{\beta\in\mathbb{N}^{n};\\\left\vert \beta\right\vert
=p}}t_{\alpha+\beta}=\operatorname*{rowdet}\left(  \left(  h_{\alpha
_{i}+j,\ i}\right)  _{i,j\in\left[  n-1\right]  }\right)  \cdot h_{\alpha
_{n}+n+p,\ n}.
\]

\end{verlong}

\noindent Thus, Corollary \ref{cor.pre-pieri.2} is proved.
\end{proof}

\subsection{The commutative case}

\begin{corollary}
\label{cor.pre-pieri.2comm}Assume that $R$ is commutative. Let $n\in
\mathbb{P}$ and $p\in\mathbb{Z}$. Let $h_{k,\ i}$ be an element of $R$ for all
$k\in\mathbb{Z}$ and $i\in\left[  n\right]  $. Assume that%
\[
h_{k,\ n}=0\ \ \ \ \ \ \ \ \ \ \text{for all }k<0.
\]

For any $\alpha\in\mathbb{Z}^{n}$, we define%
\[
t_{\alpha}:=\det\left(  \left(  h_{\alpha_{i}+j,\ i}\right)  _{i,j\in\left[
n\right]  }\right)  \in R.
\]

Let $\alpha\in\mathbb{Z}^{n}$ be such that $\alpha_{n}\leq-n$. Then,%
\[
\sum_{\substack{\beta\in\mathbb{N}^{n};\\\left\vert \beta\right\vert
=p}}t_{\alpha+\beta}=\det\left(  \left(  h_{\alpha_{i}+j,\ i}\right)
_{i,j\in\left[  n-1\right]  }\right)  \cdot h_{\alpha_{n}+n+p,\ n}.
\]

\end{corollary}

\begin{proof}
[Proof of Corollary \ref{cor.pre-pieri.2comm}.]We have assumed that the ring
$R$ is commutative. Thus, every square matrix over $R$ has a well-defined
determinant. Furthermore, the row-determinant of any square matrix over $R$ is
the same as the determinant of this matrix (since our above definition of the
row-determinant clearly generalizes the standard definition of a determinant).
In other words, if $A$ is any square matrix over $R$, then%
\begin{equation}
\operatorname*{rowdet}A=\det A. \label{pf.cor.pre-pieri.2comm.rd=d}%
\end{equation}

\begin{vershort}
Hence, we can apply Corollary \ref{cor.pre-pieri.2}, replacing
\textquotedblleft$\operatorname*{rowdet}$\textquotedblright\ by
\textquotedblleft$\det$\textquotedblright\ throughout the statement. As a
result, we obtain precisely the claim of Corollary \ref{cor.pre-pieri.2comm}.
\qedhere

\end{vershort}

\begin{verlong}
Now, for any $\alpha\in\mathbb{Z}^{n}$, we have%
\begin{equation}
\operatorname*{rowdet}\left(  \left(  h_{\alpha_{i}+j,\ i}\right)
_{i,j\in\left[  n\right]  }\right)  =\det\left(  \left(  h_{\alpha_{i}%
+j,\ i}\right)  _{i,j\in\left[  n\right]  }\right)
\label{pf.cor.pre-pieri.2comm.1}%
\end{equation}
(by (\ref{pf.cor.pre-pieri.2comm.rd=d}), applied to $A=\left(  h_{\alpha
_{i}+j,\ i}\right)  _{i,j\in\left[  n\right]  }$). Hence, for any $\alpha
\in\mathbb{Z}^{n}$, we have%
\begin{align*}
t_{\alpha}  &  =\det\left(  \left(  h_{\alpha_{i}+j,\ i}\right)
_{i,j\in\left[  n\right]  }\right)  \ \ \ \ \ \ \ \ \ \ \left(  \text{by our
definition of }t_{\alpha}\right) \\
&  =\operatorname*{rowdet}\left(  \left(  h_{\alpha_{i}+j,\ i}\right)
_{i,j\in\left[  n\right]  }\right)  \ \ \ \ \ \ \ \ \ \ \left(  \text{by
(\ref{pf.cor.pre-pieri.2comm.1})}\right)  .
\end{align*}
But this is precisely how $t_{\alpha}$ was defined in Corollary
\ref{cor.pre-pieri.2}. Thus, our notations match the notations of Corollary
\ref{cor.pre-pieri.2}. Hence, we can apply Corollary \ref{cor.pre-pieri.2}. As
a result, we obtain%
\begin{align*}
\sum_{\substack{\beta\in\mathbb{N}^{n};\\\left\vert \beta\right\vert
=p}}t_{\alpha+\beta}  &  =\underbrace{\operatorname*{rowdet}\left(  \left(
h_{\alpha_{i}+j,\ i}\right)  _{i,j\in\left[  n-1\right]  }\right)
}_{\substack{=\det\left(  \left(  h_{\alpha_{i}+j,\ i}\right)  _{i,j\in\left[
n-1\right]  }\right)  \\\text{(by (\ref{pf.cor.pre-pieri.2comm.rd=d}), applied
to }A=\left(  h_{\alpha_{i}+j,\ i}\right)  _{i,j\in\left[  n-1\right]
}\text{)}}}\cdot h_{\alpha_{n}+n+p,\ n}\\
&  =\det\left(  \left(  h_{\alpha_{i}+j,\ i}\right)  _{i,j\in\left[
n-1\right]  }\right)  \cdot h_{\alpha_{n}+n+p,\ n}.
\end{align*}
This proves Corollary \ref{cor.pre-pieri.2comm}. \qedhere

\end{verlong}
\end{proof}

\subsection{A Schur-like reindexing}

Here are some more consequences of the first pre-Pieri rule:

\begin{corollary}
\label{cor.pre-pieri.4}Let $n\in\mathbb{P}$ and $p\in\mathbb{Z}$. Let
$h_{k,\ i}$ be an element of $R$ for all $k\in\mathbb{Z}$ and $i\in\left[
n\right]  $. Assume that%
\[
h_{k,\ n}=0\ \ \ \ \ \ \ \ \ \ \text{for all }k<0.
\]

For any $m\in\left\{  0,1,\ldots,n\right\}  $ and any $\lambda\in
\mathbb{Z}^{m}$, we define%
\[
s_{\lambda}:=\operatorname*{rowdet}\left(  \left(  h_{\lambda_{i}%
+j-i,\ i}\right)  _{i,j\in\left[  m\right]  }\right)  \in R.
\]

Fix an $n$-tuple $\mu\in\mathbb{Z}^{n}$ with $\mu_{n}=0$. Let $\overline{\mu
}=\left(  \mu_{1},\mu_{2},\ldots,\mu_{n-1}\right)  $. Then,%
\[
s_{\overline{\mu}}\cdot h_{p,\ n}=\sum_{\substack{\beta\in\mathbb{N}%
^{n};\\\left\vert \beta\right\vert =p}}s_{\mu+\beta}.
\]

\end{corollary}

\begin{proof}
[Proof of Corollary \ref{cor.pre-pieri.4}.]Define $t_{\alpha}\in R$ for each
$\alpha\in\mathbb{Z}^{n}$ as in Corollary \ref{cor.pre-pieri.2}.

Define an $n$-tuple $\alpha\in\mathbb{Z}^{n}$ by%
\[
\alpha=\left(  \mu_{1}-1,\ \mu_{2}-2,\ \ldots,\ \mu_{n}-n\right)  .
\]
Thus,%
\begin{equation}
\alpha_{i}=\mu_{i}-i\ \ \ \ \ \ \ \ \ \ \text{for each }i\in\left[  n\right]
. \label{pf.cor.pre-pieri.4.ali=}%
\end{equation}
Applying this to $i=n$, we obtain $\alpha_{n}=\underbrace{\mu_{n}}%
_{=0}-n=-n\leq-n$. Hence, Corollary \ref{cor.pre-pieri.2} yields
\begin{equation}
\sum_{\substack{\beta\in\mathbb{N}^{n};\\\left\vert \beta\right\vert
=p}}t_{\alpha+\beta}=\operatorname*{rowdet}\left(  \left(  h_{\alpha
_{i}+j,\ i}\right)  _{i,j\in\left[  n-1\right]  }\right)  \cdot h_{\alpha
_{n}+n+p,\ n}. \label{pf.cor.pre-pieri.4.old}%
\end{equation}

\begin{vershort}
However, using the definitions of $s_{\overline{\mu}}$ and $\alpha$, it is
easy to see that%
\[
\operatorname*{rowdet}\left(  \left(  h_{\alpha_{i}+j,\ i}\right)
_{i,j\in\left[  n-1\right]  }\right)  =s_{\overline{\mu}}.
\]
Furthermore, again using the definition of $\alpha$, we can easily check that%
\[
t_{\alpha+\beta}=s_{\mu+\beta}\ \ \ \ \ \ \ \ \ \ \text{for each }\beta
\in\mathbb{N}^{n}%
\]
(indeed, both $t_{\alpha+\beta}$ and $s_{\mu+\beta}$ are defined as
row-determinants of certain matrices, and a simple calculation of indices
shows that these two matrices are identical). Finally, we have
\[
\alpha_{n}+n+p=p\ \ \ \ \ \ \ \ \ \ \left(  \text{since }\alpha_{n}=-n\right)
.
\]
Using these three equalities, we can rewrite (\ref{pf.cor.pre-pieri.4.old}) as%
\[
\sum_{\substack{\beta\in\mathbb{N}^{n};\\\left\vert \beta\right\vert
=p}}s_{\mu+\beta}=s_{\overline{\mu}}\cdot h_{p,\ n}.
\]

\end{vershort}

\begin{verlong}
However, each $\beta\in\mathbb{N}^{n}$ satisfies
\begin{equation}
t_{\alpha+\beta}=s_{\mu+\beta} \label{pf.cor.pre-pieri.4.3}%
\end{equation}
\footnote{\textit{Proof.} Let $\beta\in\mathbb{N}^{n}$. For each
$i,j\in\left[  n\right]  $, we have%
\[
\underbrace{\left(  \alpha+\beta\right)  _{i}}_{=\alpha_{i}+\beta_{i}%
}+j=\underbrace{\alpha_{i}}_{\substack{=\mu_{i}-i\\\text{(by
(\ref{pf.cor.pre-pieri.4.ali=}))}}}+\beta_{i}+j=\mu_{i}-i+\beta_{i}%
+j=\underbrace{\mu_{i}+\beta_{i}}_{=\left(  \mu+\beta\right)  _{i}%
}+j-i=\left(  \mu+\beta\right)  _{i}+j-i
\]
and therefore
\[
h_{\left(  \alpha+\beta\right)  _{i}+j,\ i}=h_{\left(  \mu+\beta\right)
_{i}+j-i,\ i}.
\]
In other words, we have%
\[
\left(  h_{\left(  \alpha+\beta\right)  _{i}+j,\ i}\right)  _{i,j\in\left[
n\right]  }=\left(  h_{\left(  \mu+\beta\right)  _{i}+j-i,\ i}\right)
_{i,j\in\left[  n\right]  }.
\]
\par
The definition of $t_{\alpha+\beta}$ yields
\[
t_{\alpha+\beta}=\operatorname*{rowdet}\left(  \underbrace{\left(  h_{\left(
\alpha+\beta\right)  _{i}+j,\ i}\right)  _{i,j\in\left[  n\right]  }%
}_{=\left(  h_{\left(  \mu+\beta\right)  _{i}+j-i,\ i}\right)  _{i,j\in\left[
n\right]  }}\right)  =\operatorname*{rowdet}\left(  \left(  h_{\left(
\mu+\beta\right)  _{i}+j-i,\ i}\right)  _{i,j\in\left[  n\right]  }\right)
=s_{\mu+\beta}%
\]
(since the definition of $s_{\mu+\beta}$ yields $s_{\mu+\beta}%
=\operatorname*{rowdet}\left(  \left(  h_{\left(  \mu+\beta\right)
_{i}+j-i,\ i}\right)  _{i,j\in\left[  n\right]  }\right)  $). Qed.}
Furthermore, we have
\begin{equation}
\operatorname*{rowdet}\left(  \left(  h_{\alpha_{i}+j,\ i}\right)
_{i,j\in\left[  n-1\right]  }\right)  =s_{\overline{\mu}}
\label{pf.cor.pre-pieri.4.4}%
\end{equation}
\footnote{\textit{Proof.} We have $\overline{\mu}=\left(  \mu_{1},\mu
_{2},\ldots,\mu_{n-1}\right)  $; thus,
\begin{equation}
\overline{\mu}_{i}=\mu_{i}\ \ \ \ \ \ \ \ \ \ \text{for each }i\in\left[
n-1\right]  . \label{pf.cor.pre-pieri.4.4.pf.1}%
\end{equation}
\par
The definition of $s_{\overline{\mu}}$ yields%
\begin{equation}
s_{\overline{\mu}}=\operatorname*{rowdet}\left(  \left(  h_{\overline{\mu}%
_{i}+j-i,\ i}\right)  _{i,j\in\left[  n-1\right]  }\right)  .
\label{pf.cor.pre-pieri.4.4.pf.2}%
\end{equation}
\par
For each $i,j\in\left[  n-1\right]  $, we have
\[
\underbrace{\alpha_{i}}_{\substack{=\mu_{i}-i\\\text{(by
(\ref{pf.cor.pre-pieri.4.ali=}))}}}+j=\underbrace{\mu_{i}}%
_{\substack{=\overline{\mu}_{i}\\\text{(by (\ref{pf.cor.pre-pieri.4.4.pf.1}%
))}}}-i+j=\overline{\mu}_{i}-i+j=\overline{\mu}_{i}+j-i
\]
and thus $h_{\alpha_{i}+j,\ i}=h_{\overline{\mu}_{i}+j-i,\ i}$. In other
words, we have
\[
\left(  h_{\alpha_{i}+j,\ i}\right)  _{i,j\in\left[  n-1\right]  }=\left(
h_{\overline{\mu}_{i}+j-i,\ i}\right)  _{i,j\in\left[  n-1\right]  }.
\]
Hence,%
\[
\operatorname*{rowdet}\left(  \underbrace{\left(  h_{\alpha_{i}+j,\ i}\right)
_{i,j\in\left[  n-1\right]  }}_{=\left(  h_{\overline{\mu}_{i}+j-i,\ i}%
\right)  _{i,j\in\left[  n-1\right]  }}\right)  =\operatorname*{rowdet}\left(
\left(  h_{\overline{\mu}_{i}+j-i,\ i}\right)  _{i,j\in\left[  n-1\right]
}\right)  =s_{\overline{\mu}}%
\]
(by (\ref{pf.cor.pre-pieri.4.4.pf.2})).}. Finally, from $\alpha_{n}=-n$, we
obtain
\begin{equation}
\alpha_{n}+n+p=\left(  -n\right)  +n+p=p. \label{pf.cor.pre-pieri.4.5}%
\end{equation}

Using (\ref{pf.cor.pre-pieri.4.3}), (\ref{pf.cor.pre-pieri.4.4}) and
(\ref{pf.cor.pre-pieri.4.5}), we can rewrite (\ref{pf.cor.pre-pieri.4.old}) as%
\[
\sum_{\substack{\beta\in\mathbb{N}^{n};\\\left\vert \beta\right\vert
=p}}s_{\mu+\beta}=s_{\overline{\mu}}\cdot h_{p,\ n}.
\]

\end{verlong}

\noindent This proves Corollary \ref{cor.pre-pieri.4}.
\end{proof}

\begin{corollary}
\label{cor.pre-pieri.4comm}Assume that $R$ is commutative. Let $n\in
\mathbb{P}$ and $p\in\mathbb{Z}$. Let $h_{k,\ i}$ be an element of $R$ for all
$k\in\mathbb{Z}$ and $i\in\left[  n\right]  $. Assume that%
\[
h_{k,\ n}=0\ \ \ \ \ \ \ \ \ \ \text{for all }k<0.
\]

For any $m\in\left\{  0,1,\ldots,n\right\}  $ and any $\lambda\in
\mathbb{Z}^{m}$, we define%
\[
s_{\lambda}:=\det\left(  \left(  h_{\lambda_{i}+j-i,\ i}\right)
_{i,j\in\left[  m\right]  }\right)  \in R.
\]

Fix an $n$-tuple $\mu\in\mathbb{Z}^{n}$ with $\mu_{n}=0$. Let $\overline{\mu
}=\left(  \mu_{1},\mu_{2},\ldots,\mu_{n-1}\right)  $. Then,%
\begin{equation}
h_{p,\ n}\cdot s_{\overline{\mu}}=\sum_{\substack{\beta\in\mathbb{N}%
^{n};\\\left\vert \beta\right\vert =p}}s_{\mu+\beta}.
\label{eq.cor.pre-pieri.4comm.eq}%
\end{equation}

\end{corollary}

\begin{vershort}
\begin{proof}
[Proof of Corollary \ref{cor.pre-pieri.4comm}.]This can be derived from
Corollary \ref{cor.pre-pieri.4} in the same way as we derived Corollary
\ref{cor.pre-pieri.2comm} from Corollary \ref{cor.pre-pieri.2}.
\end{proof}
\end{vershort}

\begin{verlong}
\begin{proof}
[Proof of Corollary \ref{cor.pre-pieri.4comm}.]If $A$ is any square matrix
over $R$, then%
\begin{equation}
\operatorname*{rowdet}A=\det A. \label{pf.cor.pre-pieri.4comm.rd=d}%
\end{equation}
(Indeed, we have already shown this in our proof of Corollary
\ref{cor.pre-pieri.2comm}.) Now, for any $m\in\left\{  0,1,\ldots,n\right\}  $
and any $\lambda\in\mathbb{Z}^{m}$, we have%
\begin{equation}
\operatorname*{rowdet}\left(  \left(  h_{\lambda_{i}+j-i,\ i}\right)
_{i,j\in\left[  m\right]  }\right)  =\det\left(  \left(  h_{\lambda
_{i}+j-i,\ i}\right)  _{i,j\in\left[  m\right]  }\right)
\label{pf.cor.pre-pieri.4comm.1}%
\end{equation}
(by (\ref{pf.cor.pre-pieri.4comm.rd=d}), applied to $A=\left(  h_{\lambda
_{i}+j-i,\ i}\right)  _{i,j\in\left[  m\right]  }$). Hence, for any
$m\in\left\{  0,1,\ldots,n\right\}  $ and any $\lambda\in\mathbb{Z}^{m}$, we
have%
\begin{align*}
s_{\lambda}  &  =\det\left(  \left(  h_{\lambda_{i}+j-i,\ i}\right)
_{i,j\in\left[  m\right]  }\right)  \ \ \ \ \ \ \ \ \ \ \left(  \text{by our
definition of }s_{\lambda}\right) \\
&  =\operatorname*{rowdet}\left(  \left(  h_{\lambda_{i}+j-i,\ i}\right)
_{i,j\in\left[  m\right]  }\right)  \ \ \ \ \ \ \ \ \ \ \left(  \text{by
(\ref{pf.cor.pre-pieri.4comm.1})}\right)  .
\end{align*}
But this is precisely how $s_{\lambda}$ was defined in Corollary
\ref{cor.pre-pieri.4}. Thus, our notations match the notations of Corollary
\ref{cor.pre-pieri.4}. Hence, we can apply Corollary \ref{cor.pre-pieri.4}. As
a result, we obtain%
\[
s_{\overline{\mu}}\cdot h_{p,\ n}=\sum_{\substack{\beta\in\mathbb{N}%
^{n};\\\left\vert \beta\right\vert =p}}s_{\mu+\beta}.
\]
Hence,%
\[
\sum_{\substack{\beta\in\mathbb{N}^{n};\\\left\vert \beta\right\vert
=p}}s_{\mu+\beta}=s_{\overline{\mu}}\cdot h_{p,\ n}=h_{p,\ n}\cdot
s_{\overline{\mu}}.
\]
This proves Corollary \ref{cor.pre-pieri.4comm}.
\end{proof}
\end{verlong}

We note that (\ref{eq.intro.pieri3b}) is the particular case of Corollary
\ref{cor.pre-pieri.4comm} for $R=\Lambda$, $h_{k,\ i}=h_{k}$ and $\mu=\lambda$.

\subsection{Recovering Fun's rule}

Let us now explain how \cite[Proposition 3.9]{Fun} follows from Corollary
\ref{cor.pre-pieri.4comm}. Here is the claim of \cite[Proposition 3.9]{Fun}
with slightly modified notations:

\begin{proposition}
\label{prop.fun}Assume that $R$ is commutative. Let $\ell\in\mathbb{N}$,
$p\in\mathbb{Z}$ and $q\in\mathbb{Z}$. Let $g_{k,\ j}$ be an element of $R$
for all $k,j\in\mathbb{Z}$. Assume that%
\begin{equation}
g_{k,\ j}=0\ \ \ \ \ \ \ \ \ \ \text{for all }k<0\text{ and }j\in\mathbb{Z}.
\label{eq.prop.fun.ass=0}%
\end{equation}

For any $m\in\mathbb{N}$ and any $\mu\in\mathbb{Z}^{m}$ and $\beta
\in\mathbb{Z}^{m}$, we define%
\[
s_{\mu,\ \beta}:=\det\left(  \left(  g_{\mu_{i}+j-i,\ \beta_{i}+j-i}\right)
_{i,j\in\left[  m\right]  }\right)  \in R.
\]

Let $\mu\in\mathbb{Z}^{\ell+1}$ be an $\left(  \ell+1\right)  $-tuple
satisfying $\mu_{\ell+1}=0$. Let $\overline{\mu}$ be the $\ell$-tuple $\left(
\mu_{1},\mu_{2},\ldots,\mu_{\ell}\right)  \in\mathbb{Z}^{\ell}$.

Let $\beta\in\mathbb{Z}^{\ell}$ be an $\ell$-tuple, and let $\beta^{\prime}%
\in\mathbb{Z}^{\ell+1}$ be the $\left(  \ell+1\right)  $-tuple $\left(
\beta_{1},\beta_{2},\ldots,\beta_{\ell},q-p\right)  $.

Then,%
\[
g_{p,\ q}\cdot s_{\overline{\mu},\ \beta}=\sum_{\substack{\delta\in
\mathbb{N}^{\ell+1};\\\left\vert \delta\right\vert =p}}s_{\mu+\delta
,\ \beta^{\prime}+\delta}.
\]

\end{proposition}

To be precise, Proposition \ref{prop.fun} does not only differ from
\cite[Proposition 3.9]{Fun} in the notations (what we call $q$ and $g_{k,\ j}$
corresponds to $\beta_{e}$ and $h_{k,\ j}$ in \cite[Proposition 3.9]{Fun};
furthermore, the summation index $\delta$ in our sum is called $\sigma$ in
\cite[Proposition 3.9]{Fun}), but is also slightly more general (our $\mu$ can
be any $\left(  \ell+1\right)  $-tuple in $\mathbb{Z}^{\ell+1}$ satisfying
$\mu_{\ell+1}=0$ rather than just a partition of length $\ell$; our
$g_{k,\ j}$ are not required to satisfy $g_{0,\ j}=1$; our $p$ is not assumed
to be positive). Note that our definition of $s_{\mu,\ \beta}$ in Proposition
\ref{prop.fun} is equivalent to the one in \cite{Fun}, because of
\cite[(3.3)]{Fun}.

As promised, we can now easily derive Proposition \ref{prop.fun} from
Corollary \ref{cor.pre-pieri.4comm}:

\begin{vershort}
\begin{proof}
[Proof of Proposition \ref{prop.fun}.]Set $n:=\ell+1$. Set%
\[
h_{k,\ i}:=g_{k,\ \beta_{i}^{\prime}-\mu_{i}+k}\ \ \ \ \ \ \ \ \ \ \text{for
every }k\in\mathbb{Z}\text{ and }i\in\left[  n\right]  .
\]
Then, (\ref{eq.prop.fun.ass=0}) entails that $h_{k,\ n}=0$ for all $k<0$.
Furthermore, using $\mu_{n}=\mu_{\ell+1}=0$ and $\beta_{n}^{\prime}%
=\beta_{\ell+1}^{\prime}=q-p$, we obtain $h_{p,\ n}=g_{p,\ \left(  q-p\right)
-0+p}=g_{p,\ q}$.

For any $m\in\left\{  0,1,\ldots,n\right\}  $ and any $\lambda\in
\mathbb{Z}^{m}$, we define $s_{\lambda}\in R$ as in Corollary
\ref{cor.pre-pieri.4comm}. Thus, (\ref{eq.cor.pre-pieri.4comm.eq}) (with the
summation index $\beta$ renamed as $\delta$) yields%
\begin{equation}
h_{p,\ n}\cdot s_{\overline{\mu}}=\sum_{\substack{\delta\in\mathbb{N}%
^{n};\\\left\vert \delta\right\vert =p}}s_{\mu+\delta}.
\label{pf.prop.fun.short.cor}%
\end{equation}

However, let us recall that $h_{p,\ n}=g_{p,\ q}$ and $n=\ell+1$; furthermore,
it is easy to see that $s_{\overline{\mu}}=s_{\overline{\mu},\ \beta}$ and%
\[
s_{\mu+\delta}=s_{\mu+\delta,\ \beta^{\prime}+\delta}%
\ \ \ \ \ \ \ \ \ \ \text{for each }\delta\in\mathbb{N}^{\ell+1}.
\]
Thus, (\ref{pf.prop.fun.short.cor}) rewrites as
\[
g_{p,\ q}\cdot s_{\overline{\mu},\ \beta}=\sum_{\substack{\delta\in
\mathbb{N}^{\ell+1};\\\left\vert \delta\right\vert =p}}s_{\mu+\delta
,\ \beta^{\prime}+\delta}.
\]
This proves Proposition \ref{prop.fun}.
\end{proof}
\end{vershort}

\begin{verlong}
\begin{proof}
[Proof of Proposition \ref{prop.fun}.]We have $\overline{\mu}=\left(  \mu
_{1},\mu_{2},\ldots,\mu_{\ell}\right)  $. Thus,%
\begin{equation}
\overline{\mu}_{i}=\mu_{i}\ \ \ \ \ \ \ \ \ \ \text{for each }i\in\left[
\ell\right]  . \label{pf.prop.fun.mubari}%
\end{equation}
From $\beta^{\prime}=\left(  \beta_{1},\beta_{2},\ldots,\beta_{\ell
},q-p\right)  $, we obtain%
\begin{equation}
\beta_{i}^{\prime}=\beta_{i}\ \ \ \ \ \ \ \ \ \ \text{for each }i\in\left[
\ell\right]  \label{pf.prop.fun.bprimei}%
\end{equation}
and $\beta_{\ell+1}^{\prime}=q-p$.

Set
\begin{equation}
n:=\ell+1. \label{pf.prop.fun.0}%
\end{equation}
Set%
\[
h_{k,\ i}:=g_{k,\ \beta_{i}^{\prime}-\mu_{i}+k}\ \ \ \ \ \ \ \ \ \ \text{for
every }k\in\mathbb{Z}\text{ and }i\in\left[  n\right]  .
\]
Then, we have $h_{k,\ n}=0$ for all $k<0$\ \ \ \ \footnote{\textit{Proof.} Let
$k<0$ be an integer. Then, the definition of $h_{k,\ n}$ yields $h_{k,\ n}%
=g_{k,\ \beta_{n}^{\prime}-\mu_{n}+k}=0$ (by (\ref{eq.prop.fun.ass=0}),
applied to $j=\beta_{n}^{\prime}-\mu_{n}+k$). Qed.}. Furthermore, from
$n=\ell+1$, we obtain $\mu_{n}=\mu_{\ell+1}=0$ and $\beta_{n}^{\prime}%
=\beta_{\ell+1}^{\prime}=q-p$; therefore, $\underbrace{\beta_{n}^{\prime}%
}_{=q-p}-\underbrace{\mu_{n}}_{=0}+p=\left(  q-p\right)  -0+p=q$. Now, the
definition of $h_{p,\ n}$ yields%
\begin{equation}
h_{p,\ n}=g_{p,\ \beta_{n}^{\prime}-\mu_{n}+p}=g_{p,\ q} \label{pf.prop.fun.1}%
\end{equation}
(since $\beta_{n}^{\prime}-\mu_{n}+p=q$).

For any $m\in\left\{  0,1,\ldots,n\right\}  $ and any $\lambda\in
\mathbb{Z}^{m}$, we define $s_{\lambda}\in R$ as in Corollary
\ref{cor.pre-pieri.4comm}. Thus, (\ref{eq.cor.pre-pieri.4comm.eq}) (with the
summation index $\beta$ renamed as $\delta$) yields%
\begin{equation}
h_{p,\ n}\cdot s_{\overline{\mu}}=\sum_{\substack{\delta\in\mathbb{N}%
^{n};\\\left\vert \delta\right\vert =p}}s_{\mu+\delta}.
\label{pf.prop.fun.cor}%
\end{equation}

However, it is easy to see that%
\begin{equation}
s_{\overline{\mu}}=s_{\overline{\mu},\ \beta} \label{pf.prop.fun.2}%
\end{equation}
\footnote{\textit{Proof of (\ref{pf.prop.fun.2}):} For each $i,j\in\left[
\ell\right]  $, we have%
\begin{align}
h_{\overline{\mu}_{i}+j-i,\ i}  &  =h_{\mu_{i}+j-i,\ i}%
\ \ \ \ \ \ \ \ \ \ \left(  \text{since (\ref{pf.prop.fun.mubari}) yields
}\overline{\mu}_{i}=\mu_{i}\right) \nonumber\\
&  =g_{\mu_{i}+j-i,\ \beta_{i}^{\prime}-\mu_{i}+\left(  \mu_{i}+j-i\right)
}\ \ \ \ \ \ \ \ \ \ \left(  \text{by the definition of }h_{\mu_{i}%
+j-i,\ i}\right) \nonumber\\
&  =g_{\mu_{i}+j-i,\ \beta_{i}+j-i}\nonumber\\
&  \ \ \ \ \ \ \ \ \ \ \left(  \text{since }\beta_{i}^{\prime}-\mu_{i}+\left(
\mu_{i}+j-i\right)  =\underbrace{\beta_{i}^{\prime}}_{\substack{=\beta
_{i}\\\text{(by (\ref{pf.prop.fun.bprimei}))}}}+j-i=\beta_{i}+j-i\right)
\nonumber\\
&  =g_{\overline{\mu}_{i}+j-i,\ \beta_{i}+j-i} \label{pf.prop.fun.2.pf.2}%
\end{align}
(since (\ref{pf.prop.fun.mubari}) yields $\mu_{i}=\overline{\mu}_{i}$).
\par
The definition of $s_{\overline{\mu}}$ yields%
\[
s_{\overline{\mu}}=\det\left(  \left(  \underbrace{h_{\overline{\mu}%
_{i}+j-i,\ i}}_{\substack{=g_{\overline{\mu}_{i}+j-i,\ \beta_{i}%
+j-i}\\\text{(by (\ref{pf.prop.fun.2.pf.2}))}}}\right)  _{i,j\in\left[
\ell\right]  }\right)  =\det\left(  \left(  g_{\overline{\mu}_{i}%
+j-i,\ \beta_{i}+j-i}\right)  _{i,j\in\left[  \ell\right]  }\right)  .
\]
On the other hand, the definition of $s_{\overline{\mu},\ \beta}$ yields%
\[
s_{\overline{\mu},\ \beta}=\det\left(  \left(  g_{\overline{\mu}%
_{i}+j-i,\ \beta_{i}+j-i}\right)  _{i,j\in\left[  \ell\right]  }\right)  .
\]
Comparing these two equalities, we obtain $s_{\overline{\mu}}=s_{\overline
{\mu},\ \beta}$, qed.}. Furthermore,
\begin{equation}
s_{\mu+\delta}=s_{\mu+\delta,\ \beta^{\prime}+\delta}%
\ \ \ \ \ \ \ \ \ \ \text{for each }\delta\in\mathbb{N}^{\ell+1}
\label{pf.prop.fun.3}%
\end{equation}
\footnote{\textit{Proof of (\ref{pf.prop.fun.3}):} Let $\delta\in
\mathbb{N}^{\ell+1}$.
\par
For each $i,j\in\left[  \ell+1\right]  $, we have%
\begin{align}
h_{\left(  \mu+\delta\right)  _{i}+j-i,\ i}  &  =h_{\mu_{i}+\delta
_{i}+j-i,\ i}\ \ \ \ \ \ \ \ \ \ \left(  \text{since }\left(  \mu
+\delta\right)  _{i}=\mu_{i}+\delta_{i}\right) \nonumber\\
&  =g_{\mu_{i}+\delta_{i}+j-i,\ \beta_{i}^{\prime}-\mu_{i}+\left(  \mu
_{i}+\delta_{i}+j-i\right)  }\ \ \ \ \ \ \ \ \ \ \left(  \text{by the
definition of }h_{\mu_{i}+\delta_{i}+j-i,\ i}\right) \nonumber\\
&  =g_{\mu_{i}+\delta_{i}+j-i,\ \left(  \beta^{\prime}+\delta\right)
_{i}+j-i}\nonumber\\
&  \ \ \ \ \ \ \ \ \ \ \left(  \text{since }\beta_{i}^{\prime}-\mu_{i}+\left(
\mu_{i}+\delta_{i}+j-i\right)  =\underbrace{\beta_{i}^{\prime}+\delta_{i}%
}_{=\left(  \beta^{\prime}+\delta\right)  _{i}}+j-i=\left(  \beta^{\prime
}+\delta\right)  _{i}+j-i\right) \nonumber\\
&  =g_{\left(  \mu+\delta\right)  _{i}+j-i,\ \left(  \beta^{\prime}%
+\delta\right)  _{i}+j-i} \label{pf.prop.fun.3.pf.2}%
\end{align}
(since $\mu_{i}+\delta_{i}=\left(  \mu+\delta\right)  _{i}$).
\par
The definition of $s_{\mu+\delta}$ yields%
\[
s_{\mu+\delta}=\det\left(  \left(  \underbrace{h_{\left(  \mu+\delta\right)
_{i}+j-i,\ i}}_{\substack{=g_{\left(  \mu+\delta\right)  _{i}+j-i,\ \left(
\beta^{\prime}+\delta\right)  _{i}+j-i}\\\text{(by (\ref{pf.prop.fun.3.pf.2}%
))}}}\right)  _{i,j\in\left[  \ell+1\right]  }\right)  =\det\left(  \left(
g_{\left(  \mu+\delta\right)  _{i}+j-i,\ \left(  \beta^{\prime}+\delta\right)
_{i}+j-i}\right)  _{i,j\in\left[  \ell+1\right]  }\right)  .
\]
On the other hand, the definition of $s_{\mu+\delta,\ \beta^{\prime}+\delta}$
yields%
\[
s_{\mu+\delta,\ \beta^{\prime}+\delta}=\det\left(  \left(  g_{\left(
\mu+\delta\right)  _{i}+j-i,\ \left(  \beta^{\prime}+\delta\right)  _{i}%
+j-i}\right)  _{i,j\in\left[  \ell+1\right]  }\right)  .
\]
Comparing these two equalities, we obtain $s_{\mu+\delta}=s_{\mu
+\delta,\ \beta^{\prime}+\delta}$. This proves (\ref{pf.prop.fun.3}).}.

Using (\ref{pf.prop.fun.0}), (\ref{pf.prop.fun.1}), (\ref{pf.prop.fun.2}) and
(\ref{pf.prop.fun.3}), we can rewrite (\ref{pf.prop.fun.cor}) as%
\[
g_{p,\ q}\cdot s_{\overline{\mu},\ \beta}=\sum_{\substack{\delta\in
\mathbb{N}^{\ell+1};\\\left\vert \delta\right\vert =p}}s_{\mu+\delta
,\ \beta^{\prime}+\delta}.
\]
This proves Proposition \ref{prop.fun}.
\end{proof}
\end{verlong}

\subsection{Recovering the immaculate Pieri rule}

Next, we shall exhibit \cite[Theorem 3.5]{BBSSZ13} as a consequence of
Corollary \ref{cor.pre-pieri.4}. To this aim, we will briefly introduce the
relevant parts of the scene of \cite[Theorem 3.5]{BBSSZ13}. We fix a
commutative ring $\mathbf{k}$, and we let $\operatorname*{NSym}$ be the
algebra of noncommutative polynomials in countably many variables $H_{1}%
,H_{2},H_{3},\ldots$ over $\mathbf{k}$. We also set $H_{0}:=1$ and $H_{k}:=0$
for all $k<0$. For every $m\in\mathbb{N}$ and every $\alpha\in\mathbb{Z}^{m}$,
we set%
\[
H_{\alpha}:=H_{\alpha_{1}}H_{\alpha_{2}}\cdots H_{\alpha_{m}}\in
\operatorname*{NSym}.
\]
For every $m\in\mathbb{N}$ and every $\alpha\in\mathbb{Z}^{m}$, we set
\[
\mathfrak{S}_{\alpha}:=\sum_{\sigma\in S_{m}}\left(  -1\right)  ^{\sigma
}H_{\left(  \alpha_{1}+\sigma\left(  1\right)  -1,\ \alpha_{2}+\sigma\left(
2\right)  -2,\ \ldots,\ \alpha_{m}+\sigma\left(  m\right)  -m\right)  }%
\in\operatorname*{NSym}.
\]
(This is not how $\mathfrak{S}_{\alpha}$ is defined in \cite{BBSSZ13}, but it
is an equivalent definition, because \cite[Theorem 3.27]{BBSSZ13} shows that
the $\mathfrak{S}_{\alpha}$ from \cite{BBSSZ13} is given by the same formula.)
Now, \cite[Theorem 3.5]{BBSSZ13} (the \textquotedblleft right-Pieri rule for
multiplication by $H_{s}$\textquotedblright\ in the terminology of
\cite{BBSSZ13}) states the following:

\begin{proposition}
\label{prop.immac}Let $n\in\mathbb{N}$ and $\alpha\in\mathbb{Z}^{n}$. Let
$s\in\mathbb{Z}$. Then,%
\[
\mathfrak{S}_{\alpha}H_{s}=\sum_{\substack{\beta\in\mathbb{Z}^{n+1}%
;\\\left\vert \beta\right\vert =\left\vert \alpha\right\vert +s;\\\alpha
_{i}\leq\beta_{i}\text{ for all }i\in\left[  n\right]  ;\\0\leq\beta_{n+1}%
}}\mathfrak{S}_{\beta}.
\]

\end{proposition}

(The conditions under the summation sign here are essentially what is denoted
by \textquotedblleft$\alpha\subset_{s}\beta$\textquotedblright\ in
\cite{BBSSZ13}.)

To be fully precise, Proposition \ref{prop.immac} is a bit more general than
\cite[Theorem 3.5]{BBSSZ13}, since $\alpha$ is assumed to be a composition
(i.e., a tuple of positive integers) in \cite[Theorem 3.5]{BBSSZ13}, while we
are only assuming that $\alpha\in\mathbb{Z}^{n}$.

\begin{proof}
[Proof of Proposition \ref{prop.immac}.]Let $R$ be the ring
$\operatorname*{NSym}$. Set%
\[
h_{k,\ i}:=H_{k}\in R\ \ \ \ \ \ \ \ \ \ \text{for each }k\in\mathbb{Z}\text{
and }i\in\mathbb{Z}.
\]
Thus, $h_{s,\ n+1}=H_{s}$.

For any $m\in\left\{  0,1,\ldots,n\right\}  $ and any $\lambda\in
\mathbb{Z}^{m}$, we define $s_{\lambda}\in R$ as in Corollary
\ref{cor.pre-pieri.4}. (This $s_{\lambda}$ has nothing to do with the integer
$s$.)

\begin{vershort}
Let $\mu\in\mathbb{Z}^{n+1}$ be the $\left(  n+1\right)  $-tuple $\left(
\alpha_{1},\alpha_{2},\ldots,\alpha_{n},0\right)  $. Thus, $\mu_{n+1}=0$ and
$\left\vert \mu\right\vert =\left\vert \alpha\right\vert $.
\end{vershort}

\begin{verlong}
Let $\mu\in\mathbb{Z}^{n+1}$ be the $\left(  n+1\right)  $-tuple $\left(
\alpha_{1},\alpha_{2},\ldots,\alpha_{n},0\right)  $. Thus, $\mu_{n+1}=0$ and
$\left\vert \mu\right\vert =\left\vert \alpha\right\vert $%
\ \ \ \ \footnote{\textit{Proof.} From $\mu=\left(  \alpha_{1},\alpha
_{2},\ldots,\alpha_{n},0\right)  $, we obtain%
\begin{align*}
\left\vert \mu\right\vert  &  =\left\vert \left(  \alpha_{1},\alpha_{2}%
,\ldots,\alpha_{n},0\right)  \right\vert =\alpha_{1}+\alpha_{2}+\cdots
+\alpha_{n}+0\ \ \ \ \ \ \ \ \ \ \left(  \text{by the definition of
}\left\vert \left(  \alpha_{1},\alpha_{2},\ldots,\alpha_{n},0\right)
\right\vert \right) \\
&  =\alpha_{1}+\alpha_{2}+\cdots+\alpha_{n}=\left\vert \alpha\right\vert
\ \ \ \ \ \ \ \ \ \ \left(  \text{since }\left\vert \alpha\right\vert \text{
is defined to be }\alpha_{1}+\alpha_{2}+\cdots+\alpha_{n}\right)  .
\end{align*}
}. Moreover, we have%
\[
\left(  \mu_{1},\mu_{2},\ldots,\mu_{n+1}\right)  =\mu=\left(  \alpha
_{1},\alpha_{2},\ldots,\alpha_{n},0\right)  .
\]
Thus,%
\begin{equation}
\mu_{i}=\alpha_{i}\ \ \ \ \ \ \ \ \ \ \text{for each }i\in\left[  n\right]  .
\label{pf.prop.immac.mui=}%
\end{equation}

\end{verlong}

\begin{vershort}
Let $\overline{\mu}=\left(  \mu_{1},\mu_{2},\ldots,\mu_{n}\right)
\in\mathbb{Z}^{n}$. Thus, $\overline{\mu}=\alpha$.
\end{vershort}

\begin{verlong}
Let $\overline{\mu}=\left(  \mu_{1},\mu_{2},\ldots,\mu_{n}\right)
\in\mathbb{Z}^{n}$. Thus, $\overline{\mu}=\alpha$%
\ \ \ \ \footnote{\textit{Proof.} We just have seen that%
\[
\mu_{i}=\alpha_{i}\ \ \ \ \ \ \ \ \ \ \text{for each }i\in\left[  n\right]  .
\]
In other words, $\left(  \mu_{1},\mu_{2},\ldots,\mu_{n}\right)  =\left(
\alpha_{1},\alpha_{2},\ldots,\alpha_{n}\right)  $. Hence, $\overline{\mu
}=\left(  \mu_{1},\mu_{2},\ldots,\mu_{n}\right)  =\left(  \alpha_{1}%
,\alpha_{2},\ldots,\alpha_{n}\right)  =\alpha$ (since $\alpha\in\mathbb{Z}%
^{n}$).}.
\end{verlong}

For each $k<0$, we have
\begin{align*}
h_{k,\ n+1}  &  =H_{k}\ \ \ \ \ \ \ \ \ \ \left(  \text{by the definition of
}h_{k,\ n+1}\right) \\
&  =0\ \ \ \ \ \ \ \ \ \ \left(  \text{since }k<0\right)  .
\end{align*}
Thus, we have shown that $h_{k,\ n+1}=0$ for all $k<0$. Hence, Corollary
\ref{cor.pre-pieri.4} (applied to $n+1$ and $s$ instead of $n$ and $p$) yields%
\[
s_{\overline{\mu}}\cdot h_{s,\ n+1}=\sum_{\substack{\beta\in\mathbb{N}%
^{n+1};\\\left\vert \beta\right\vert =s}}s_{\mu+\beta}.
\]
In view of $h_{s,\ n+1}=H_{s}$ and $\overline{\mu}=\alpha$, we can rewrite
this as%
\begin{equation}
s_{\alpha}\cdot H_{s}=\sum_{\substack{\beta\in\mathbb{N}^{n+1};\\\left\vert
\beta\right\vert =s}}s_{\mu+\beta}. \label{pf.prop.immac.1}%
\end{equation}

\begin{vershort}
It is easy to see (by comparing the definitions of $\mathfrak{S}_{\lambda}$
and $s_{\lambda}$) that%
\begin{equation}
\mathfrak{S}_{\lambda}=s_{\lambda}\ \ \ \ \ \ \ \ \ \ \text{for any }%
m\in\mathbb{N}\text{ and any }\lambda\in\mathbb{Z}^{m}.
\label{pf.prop.immac.short.S=s}%
\end{equation}

\end{vershort}

\begin{verlong}
It is easy to see that%
\begin{equation}
\mathfrak{S}_{\lambda}=s_{\lambda}\ \ \ \ \ \ \ \ \ \ \text{for any }%
m\in\mathbb{N}\text{ and any }\lambda\in\mathbb{Z}^{m}.
\label{pf.prop.immac.S=s}%
\end{equation}

[\textit{Proof of (\ref{pf.prop.immac.S=s}):} Let $m\in\mathbb{N}$ and
$\lambda\in\mathbb{Z}^{m}$. The definition of $\mathfrak{S}_{\lambda}$ yields%
\begin{align*}
\mathfrak{S}_{\lambda}  &  =\sum_{\sigma\in S_{m}}\left(  -1\right)  ^{\sigma
}\underbrace{H_{\left(  \lambda_{1}+\sigma\left(  1\right)  -1,\ \lambda
_{2}+\sigma\left(  2\right)  -2,\ \ldots,\ \lambda_{m}+\sigma\left(  m\right)
-m\right)  }}_{\substack{=H_{\lambda_{1}+\sigma\left(  1\right)  -1}%
H_{\lambda_{2}+\sigma\left(  2\right)  -2}\cdots H_{\lambda_{m}+\sigma\left(
m\right)  -m}\\\text{(by the definition of }H_{\left(  \lambda_{1}%
+\sigma\left(  1\right)  -1,\ \lambda_{2}+\sigma\left(  2\right)
-2,\ \ldots,\ \lambda_{m}+\sigma\left(  m\right)  -m\right)  }\text{)}}}\\
&  =\sum_{\sigma\in S_{m}}\left(  -1\right)  ^{\sigma}H_{\lambda_{1}%
+\sigma\left(  1\right)  -1}H_{\lambda_{2}+\sigma\left(  2\right)  -2}\cdots
H_{\lambda_{m}+\sigma\left(  m\right)  -m}.
\end{align*}
On the other hand, the definition of $s_{\lambda}$ yields%
\begin{align*}
s_{\lambda}  &  =\operatorname*{rowdet}\left(  \left(  \underbrace{h_{\lambda
_{i}+j-i,\ i}}_{\substack{=H_{\lambda_{i}+j-i}\\\text{(by the definition of
}h_{\lambda_{i}+j-i,\ i}\text{)}}}\right)  _{i,j\in\left[  m\right]  }\right)
=\operatorname*{rowdet}\left(  \left(  H_{\lambda_{i}+j-i}\right)
_{i,j\in\left[  m\right]  }\right) \\
&  =\sum_{\sigma\in S_{m}}\left(  -1\right)  ^{\sigma}H_{\lambda_{1}%
+\sigma\left(  1\right)  -1}H_{\lambda_{2}+\sigma\left(  2\right)  -2}\cdots
H_{\lambda_{m}+\sigma\left(  m\right)  -m}%
\end{align*}
(by the definition of a row-determinant). Comparing these two equalities, we
obtain $\mathfrak{S}_{\lambda}=s_{\lambda}$. This proves
(\ref{pf.prop.immac.S=s}).] \medskip
\end{verlong}

\begin{vershort}
Hence, (\ref{pf.prop.immac.1}) rewrites as
\begin{equation}
\mathfrak{S}_{\alpha}\cdot H_{s}=\sum_{\substack{\beta\in\mathbb{N}%
^{n+1};\\\left\vert \beta\right\vert =s}}\mathfrak{S}_{\mu+\beta}.
\label{pf.prop.immac.short.2}%
\end{equation}
Thus,
\[
\mathfrak{S}_{\alpha}\cdot H_{s}=\sum_{\substack{\beta\in\mathbb{N}%
^{n+1};\\\left\vert \beta\right\vert =s}}\mathfrak{S}_{\mu+\beta}%
=\sum_{\substack{\gamma\in\mathbb{Z}^{n+1};\\\left\vert \gamma\right\vert
=\left\vert \mu\right\vert +s;\\\mu_{i}\leq\gamma_{i}\text{ for all }%
i\in\left[  n+1\right]  }}\mathfrak{S}_{\gamma}%
\]
(here, we have substituted $\gamma$ for $\mu+\beta$ in the sum, noticing that
the conditions under the new summation sign are precisely the conditions that
guarantee that $\gamma$ has the form $\mu+\beta$ for some $\beta\in
\mathbb{N}^{n+1}$ satisfying $\left\vert \beta\right\vert =s$). Hence,%
\[
\mathfrak{S}_{\alpha}\cdot H_{s}=\sum_{\substack{\gamma\in\mathbb{Z}%
^{n+1};\\\left\vert \gamma\right\vert =\left\vert \mu\right\vert +s;\\\mu
_{i}\leq\gamma_{i}\text{ for all }i\in\left[  n+1\right]  }}\mathfrak{S}%
_{\gamma}=\sum_{\substack{\gamma\in\mathbb{Z}^{n+1};\\\left\vert
\gamma\right\vert =\left\vert \alpha\right\vert +s;\\\alpha_{i}\leq\gamma
_{i}\text{ for all }i\in\left[  n\right]  ;\\0\leq\gamma_{n+1}}}\mathfrak{S}%
_{\gamma}%
\]
(here, we have rewritten the conditions under the summation sign in an
equivalent fashion, using the facts that $\mu=\left(  \alpha_{1},\alpha
_{2},\ldots,\alpha_{n},0\right)  $ and $\left\vert \mu\right\vert =\left\vert
\alpha\right\vert $). Renaming the summation index $\gamma$ as $\beta$, we
obtain precisely the claim of Proposition \ref{prop.immac}. \qedhere

\end{vershort}

\begin{verlong}
Now, (\ref{pf.prop.immac.S=s}) (applied to $m=n$ and $\lambda=\alpha$) yields
$\mathfrak{S}_{\alpha}=s_{\alpha}$. Moreover, for each $\beta\in
\mathbb{N}^{n+1}$, we have $\mathfrak{S}_{\mu+\beta}=s_{\mu+\beta}$ (by
(\ref{pf.prop.immac.S=s}), applied to $m=n+1$ and $\lambda=\mu+\beta$). In
view of these two equalities, we can rewrite (\ref{pf.prop.immac.1}) as%
\begin{equation}
\mathfrak{S}_{\alpha}\cdot H_{s}=\sum_{\substack{\beta\in\mathbb{N}%
^{n+1};\\\left\vert \beta\right\vert =s}}\mathfrak{S}_{\mu+\beta}.
\label{pf.prop.immac.2}%
\end{equation}

Now, let $W$ be the set of all $\left(  n+1\right)  $-tuples $\gamma
\in\mathbb{Z}^{n+1}$ satisfying $\left\vert \gamma\right\vert =\left\vert
\alpha\right\vert +s$ and $\left(  \alpha_{i}\leq\gamma_{i}\text{ for all
}i\in\left[  n\right]  \right)  $ and $0\leq\gamma_{n+1}$. Thus,%
\begin{equation}
\sum_{\gamma\in W}\mathfrak{S}_{\gamma}=\sum_{\substack{\gamma\in
\mathbb{Z}^{n+1};\\\left\vert \gamma\right\vert =\left\vert \alpha\right\vert
+s;\\\alpha_{i}\leq\gamma_{i}\text{ for all }i\in\left[  n\right]
;\\0\leq\gamma_{n+1}}}\mathfrak{S}_{\gamma}. \label{pf.prop.immac.3}%
\end{equation}

Now, for each $\beta\in\mathbb{N}^{n+1}$ satisfying $\left\vert \beta
\right\vert =s$, we have $\mu+\beta\in W$\ \ \ \ \footnote{\textit{Proof.} Let
$\beta\in\mathbb{N}^{n+1}$ satisfy $\left\vert \beta\right\vert =s$. We must
prove that $\mu+\beta\in W$.
\par
We have $\mu+\beta\in\mathbb{Z}^{n+1}$ for obvious reasons. Moreover, Lemma
\ref{lem.additive} (applied to $n+1$ and $\mu$ instead of $n$ and $\alpha$)
yields%
\[
\left\vert \mu+\beta\right\vert =\underbrace{\left\vert \mu\right\vert
}_{=\left\vert \alpha\right\vert }+\underbrace{\left\vert \beta\right\vert
}_{=s}=\left\vert \alpha\right\vert +s.
\]
Furthermore, for each $i\in\left[  n\right]  $, we have%
\begin{align*}
\left(  \mu+\beta\right)  _{i}  &  =\mu_{i}+\beta_{i}\geq\mu_{i}%
\ \ \ \ \ \ \ \ \ \ \left(  \text{since }\beta_{i}\geq0\text{ (because }%
\beta_{i}\in\mathbb{N}\text{ (since }\beta\in\mathbb{N}^{n+1}\text{))}\right)
\\
&  =\alpha_{i}\ \ \ \ \ \ \ \ \ \ \left(  \text{by (\ref{pf.prop.immac.mui=}%
)}\right)  ,
\end{align*}
so that $\alpha_{i}\leq\left(  \mu+\beta\right)  _{i}$. In other words, we
have $\alpha_{i}\leq\left(  \mu+\beta\right)  _{i}$ for all $i\in\left[
n\right]  $. Finally, we have%
\[
\left(  \mu+\beta\right)  _{n+1}=\underbrace{\mu_{n+1}}_{=0}+\beta_{n+1}%
=\beta_{n+1}\geq0\ \ \ \ \ \ \ \ \ \ \left(  \text{since }\beta_{n+1}%
\in\mathbb{N}\text{ (because }\beta\in\mathbb{N}^{n+1}\text{)}\right)  ,
\]
so that $0\leq\left(  \mu+\beta\right)  _{n+1}$.
\par
We now have proved that $\mu+\beta$ is an $\left(  n+1\right)  $-tuple in
$\mathbb{Z}^{n+1}$ that satisfies $\left\vert \mu+\beta\right\vert =\left\vert
\alpha\right\vert +s$ and $\left(  \alpha_{i}\leq\left(  \mu+\beta\right)
_{i}\text{ for all }i\in\left[  n\right]  \right)  $ and $0\leq\left(
\mu+\beta\right)  _{n+1}$. In other words, $\mu+\beta$ is an $\left(
n+1\right)  $-tuple $\gamma\in\mathbb{Z}^{n+1}$ satisfying $\left\vert
\gamma\right\vert =\left\vert \alpha\right\vert +s$ and $\left(  \alpha
_{i}\leq\gamma_{i}\text{ for all }i\in\left[  n\right]  \right)  $ and
$0\leq\gamma_{n+1}$. In other words, $\mu+\beta\in W$ (since $W$ is the set of
all such $\left(  n+1\right)  $-tuples $\gamma$). Qed.}. Hence, we can define
a map%
\begin{align*}
f:\left\{  \beta\in\mathbb{N}^{n+1}\ \mid\ \left\vert \beta\right\vert
=s\right\}   &  \rightarrow W,\\
\beta &  \mapsto\mu+\beta.
\end{align*}
Consider this map $f$.

On the other hand, for each $\delta\in W$, we have $\delta-\mu\in\left\{
\beta\in\mathbb{N}^{n+1}\ \mid\ \left\vert \beta\right\vert =s\right\}
$\ \ \ \ \footnote{\textit{Proof.} Let $\delta\in W$. In other words, $\delta$
is an $\left(  n+1\right)  $-tuple $\gamma\in\mathbb{Z}^{n+1}$ satisfying
$\left\vert \gamma\right\vert =\left\vert \alpha\right\vert +s$ and $\left(
\alpha_{i}\leq\gamma_{i}\text{ for all }i\in\left[  n\right]  \right)  $ and
$0\leq\gamma_{n+1}$ (since $W$ is the set of all such $\left(  n+1\right)
$-tuples $\gamma$). In other words, $\delta$ is an $\left(  n+1\right)
$-tuple in $\mathbb{Z}^{n+1}$ that satisfies $\left\vert \delta\right\vert
=\left\vert \alpha\right\vert +s$ and $\left(  \alpha_{i}\leq\delta_{i}\text{
for all }i\in\left[  n\right]  \right)  $ and $0\leq\delta_{n+1}$. In
particular, we therefore have%
\begin{equation}
\alpha_{i}\leq\delta_{i}\ \ \ \ \ \ \ \ \ \ \text{for all }i\in\left[
n\right]  . \label{pf.prop.immac.fn5.pf.1}%
\end{equation}
\par
Now, $\delta=\left(  \delta-\mu\right)  +\mu$, so that%
\[
\left\vert \delta\right\vert =\left\vert \left(  \delta-\mu\right)
+\mu\right\vert =\left\vert \delta-\mu\right\vert +\left\vert \mu\right\vert
\]
(by Lemma \ref{lem.additive}, applied to $n+1$, $\delta-\mu$ and $\mu$ instead
of $n$, $\alpha$ and $\beta$). Thus,%
\[
\left\vert \delta-\mu\right\vert =\underbrace{\left\vert \delta\right\vert
}_{=\left\vert \alpha\right\vert +s}-\underbrace{\left\vert \mu\right\vert
}_{=\left\vert \alpha\right\vert }=\left\vert \alpha\right\vert +s-\left\vert
\alpha\right\vert =s.
\]
\par
Let $i\in\left[  n+1\right]  $. We shall now prove that $\left(  \delta
-\mu\right)  _{i}\geq0$. Indeed, if $i=n+1$, then we have%
\[
\left(  \delta-\mu\right)  _{i}=\left(  \delta-\mu\right)  _{n+1}=\delta
_{n+1}-\underbrace{\mu_{n+1}}_{=0}=\delta_{n+1}\geq
0\ \ \ \ \ \ \ \ \ \ \left(  \text{since }0\leq\delta_{n+1}\right)  .
\]
Hence, $\left(  \delta-\mu\right)  _{i}\geq0$ is proved in the case when
$i=n+1$. Thus, for the rest of this proof of $\left(  \delta-\mu\right)
_{i}\geq0$, we WLOG assume that $i\neq n+1$. Combining $i\in\left[
n+1\right]  $ with $i\neq n+1$, we obtain $i\in\left[  n+1\right]
\setminus\left\{  n+1\right\}  =\left[  n\right]  $. Hence,%
\[
\left(  \delta-\mu\right)  _{i}=\delta_{i}-\underbrace{\mu_{i}}%
_{\substack{=\alpha_{i}\\\text{(by (\ref{pf.prop.immac.mui=}))}}}=\delta
_{i}-\alpha_{i}\geq0
\]
(by (\ref{pf.prop.immac.fn5.pf.1})). Thus, the proof of $\left(  \delta
-\mu\right)  _{i}\geq0$ is complete.
\par
Now, from $\left(  \delta-\mu\right)  _{i}\geq0$, we obtain $\left(
\delta-\mu\right)  _{i}\in\mathbb{N}$ (since $\left(  \delta-\mu\right)  _{i}$
is an integer).
\par
Forget that we fixed $i$. We thus have shown that $\left(  \delta-\mu\right)
_{i}\in\mathbb{N}$ for each $i\in\left[  n+1\right]  $. In other words,
$\delta-\mu\in\mathbb{N}^{n+1}$.
\par
We now know that $\delta-\mu\in\mathbb{N}^{n+1}$ and $\left\vert \delta
-\mu\right\vert =s$. In other words, $\delta-\mu$ is a $\beta\in
\mathbb{N}^{n+1}$ satisfying $\left\vert \beta\right\vert =s$. In other words,
$\delta-\mu\in\left\{  \beta\in\mathbb{N}^{n+1}\ \mid\ \left\vert
\beta\right\vert =s\right\}  $. Qed.}. Thus, we can define a map%
\begin{align*}
g:W  &  \rightarrow\left\{  \beta\in\mathbb{N}^{n+1}\ \mid\ \left\vert
\beta\right\vert =s\right\}  ,\\
\delta &  \mapsto\delta-\mu.
\end{align*}
Consider this map $g$.

We have $f\circ g=\operatorname*{id}$\ \ \ \ \footnote{\textit{Proof.} Let
$\delta\in W$. Then, the definition of $g$ yields $g\left(  \delta\right)
=\delta-\mu$. However,%
\begin{align*}
\left(  f\circ g\right)  \left(  \delta\right)   &  =f\left(  g\left(
\delta\right)  \right)  =\mu+\underbrace{g\left(  \delta\right)  }%
_{=\delta-\mu}\ \ \ \ \ \ \ \ \ \ \left(  \text{by the definition of }f\right)
\\
&  =\mu+\left(  \delta-\mu\right)  =\delta=\operatorname*{id}\left(
\delta\right)  .
\end{align*}
\par
Forget that we fixed $\delta$. We thus have shown that $\left(  f\circ
g\right)  \left(  \delta\right)  =\operatorname*{id}\left(  \delta\right)  $
for each $\delta\in W$. In other words, $f\circ g=\operatorname*{id}$.} and
$g\circ f=\operatorname*{id}$\ \ \ \ \footnote{\textit{Proof.} Let $\gamma
\in\left\{  \beta\in\mathbb{N}^{n+1}\ \mid\ \left\vert \beta\right\vert
=s\right\}  $. Then, the definition of $f$ yields $f\left(  \gamma\right)
=\mu+\gamma$. However,%
\begin{align*}
\left(  g\circ f\right)  \left(  \gamma\right)   &  =g\left(  f\left(
\gamma\right)  \right)  =\underbrace{f\left(  \gamma\right)  }_{=\mu+\gamma
}-\mu\ \ \ \ \ \ \ \ \ \ \left(  \text{by the definition of }g\right) \\
&  =\left(  \mu+\gamma\right)  -\mu=\gamma=\operatorname*{id}\left(
\gamma\right)  .
\end{align*}
\par
Forget that we fixed $\gamma$. We thus have shown that $\left(  g\circ
f\right)  \left(  \gamma\right)  =\operatorname*{id}\left(  \gamma\right)  $
for each $\gamma\in\left\{  \beta\in\mathbb{N}^{n+1}\ \mid\ \left\vert
\beta\right\vert =s\right\}  $. In other words, $g\circ f=\operatorname*{id}%
$.}. Thus, the maps $f$ and $g$ are mutually inverse. Hence, the map $f$ is
invertible, i.e., is a bijection. In other words, the map%
\begin{align*}
\left\{  \beta\in\mathbb{N}^{n+1}\ \mid\ \left\vert \beta\right\vert
=s\right\}   &  \rightarrow W,\\
\beta &  \mapsto\mu+\beta
\end{align*}
is a bijection (since the map $f$ is precisely this map). Hence, we can
substitute $\mu+\beta$ for $\gamma$ in the sum $\sum_{\gamma\in W}%
\mathfrak{S}_{\gamma}$. Thus, we obtain%
\[
\sum_{\gamma\in W}\mathfrak{S}_{\gamma}=\sum_{\substack{\beta\in
\mathbb{N}^{n+1};\\\left\vert \beta\right\vert =s}}\mathfrak{S}_{\mu+\beta}.
\]
Comparing this with (\ref{pf.prop.immac.2}), we obtain%
\begin{align*}
\mathfrak{S}_{\alpha}\cdot H_{s}  &  =\sum_{\gamma\in W}\mathfrak{S}_{\gamma
}=\sum_{\substack{\gamma\in\mathbb{Z}^{n+1};\\\left\vert \gamma\right\vert
=\left\vert \alpha\right\vert +s;\\\alpha_{i}\leq\gamma_{i}\text{ for all
}i\in\left[  n\right]  ;\\0\leq\gamma_{n+1}}}\mathfrak{S}_{\gamma
}\ \ \ \ \ \ \ \ \ \ \left(  \text{by (\ref{pf.prop.immac.3})}\right) \\
&  =\sum_{\substack{\beta\in\mathbb{Z}^{n+1};\\\left\vert \beta\right\vert
=\left\vert \alpha\right\vert +s;\\\alpha_{i}\leq\beta_{i}\text{ for all }%
i\in\left[  n\right]  ;\\0\leq\beta_{n+1}}}\mathfrak{S}_{\beta}%
\end{align*}
(here, we have renamed the summation index $\gamma$ as $\beta$). This proves
Proposition \ref{prop.immac}.
\end{verlong}
\end{proof}

\section{The second pre-Pieri rule}

The \textquotedblleft second pre-Pieri rule\textquotedblright\ structurally
resembles the first, but involves a sum over (a subset of) $\left\{
0,1\right\}  ^{n}$ instead of $\mathbb{N}^{n}$. This is, of course, analogous
to the relationship between the elementary symmetric functions and the
complete homogeneous symmetric functions, or the relationship between sets and
multisets, or various other \textquotedblleft combinatorial
reciprocities\textquotedblright.

\subsection{The theorem}

Before we state the second pre-Pieri rule, we observe that $\left\{
0,1\right\}  ^{n}\subseteq\mathbb{Z}^{n}$ for each $n\in\mathbb{N}$.

\begin{theorem}
[second pre-Pieri rule]\label{thm.pre-pieri2}Let $n\in\mathbb{N}$ and
$p\in\left\{  0,1,\ldots,n\right\}  $. Let $h_{k,\ i}$ be an element of $R$
for all $k\in\mathbb{Z}$ and $i\in\left[  n\right]  $.

For any $\alpha\in\mathbb{Z}^{n}$, we define%
\[
t_{\alpha}:=\operatorname*{rowdet}\left(  \left(  h_{\alpha_{i}+j,\ i}\right)
_{i,j\in\left[  n\right]  }\right)  \in R.
\]

Let $\xi$ be the $n$-tuple
\begin{align*}
&  \left(  1,2,\ldots,n\right)  +\left(  \underbrace{0,0,\ldots,0}_{n-p\text{
zeroes}},\underbrace{1,1,\ldots,1}_{p\text{ ones}}\right) \\
&  =\left(  1,2,\ldots,n-p,n-p+2,n-p+3,\ldots,n+1\right)  \in\mathbb{Z}^{n}.
\end{align*}

Let $\alpha\in\mathbb{Z}^{n}$. Then,%
\begin{equation}
\sum_{\substack{\beta\in\left\{  0,1\right\}  ^{n};\\\left\vert \beta
\right\vert =p}}t_{\alpha+\beta}=\operatorname*{rowdet}\left(  \left(
h_{\alpha_{i}+\xi_{j},\ i}\right)  _{i,j\in\left[  n\right]  }\right)  .
\label{eq.thm.pre-pieri2.claim}%
\end{equation}

\end{theorem}

\begin{example}
For this example, set $n=3$ and $p=2$, and let $\alpha\in\mathbb{Z}^{3}$ be
arbitrary. Fix arbitrary elements $h_{k,\ i}\in R$ for all $k\in\mathbb{Z}$
and $i\in\left[  n\right]  $. Then, the $n$-tuple $\xi$ defined in Theorem
\ref{thm.pre-pieri2} is $\left(  1,2,3\right)  +\left(  0,1,1\right)  =\left(
1,3,4\right)  $. Hence, (\ref{eq.thm.pre-pieri2.claim}) says that%
\[
\sum_{\substack{\beta\in\left\{  0,1\right\}  ^{3};\\\left\vert \beta
\right\vert =2}}t_{\alpha+\beta}=\operatorname*{rowdet}\left(  \left(
h_{\alpha_{i}+\xi_{j},\ i}\right)  _{i,j\in\left[  3\right]  }\right)
=\operatorname*{rowdet}\left(
\begin{array}
[c]{ccc}%
h_{\alpha_{1}+1,\ 1} & h_{\alpha_{1}+3,\ 1} & h_{\alpha_{1}+4,\ 1}\\
h_{\alpha_{2}+1,\ 2} & h_{\alpha_{2}+3,\ 2} & h_{\alpha_{2}+4,\ 2}\\
h_{\alpha_{3}+1,\ 3} & h_{\alpha_{3}+3,\ 3} & h_{\alpha_{3}+4,\ 3}%
\end{array}
\right)  .
\]
The left hand side of this equality can be rewritten as%
\begin{align*}
&  \sum_{\substack{\beta\in\left\{  0,1\right\}  ^{3};\\\left\vert
\beta\right\vert =2}}\ \ \underbrace{t_{\alpha+\beta}}%
_{\substack{=\operatorname*{rowdet}\left(  \left(  h_{\left(  \alpha
+\beta\right)  _{i}+j,\ i}\right)  _{i,j\in\left[  3\right]  }\right)
\\\text{(by the definition of }t_{\alpha+\beta}\text{)}}}\\
&  =\sum_{\substack{\beta\in\left\{  0,1\right\}  ^{3};\\\left\vert
\beta\right\vert =2}}\operatorname*{rowdet}\left(  \left(  h_{\left(
\alpha+\beta\right)  _{i}+j,\ i}\right)  _{i,j\in\left[  3\right]  }\right) \\
&  =\sum_{\substack{\beta\in\left\{  0,1\right\}  ^{3};\\\left\vert
\beta\right\vert =2}}\operatorname*{rowdet}\left(  \left(  h_{\alpha_{i}%
+\beta_{i}+j,\ i}\right)  _{i,j\in\left[  3\right]  }\right) \\
&  \ \ \ \ \ \ \ \ \ \ \ \ \ \ \ \ \ \ \ \ \left(  \text{since }\left(
\alpha+\beta\right)  _{i}=\alpha_{i}+\beta_{i}\text{ for all }i\in\left[
3\right]  \right) \\
&  =\sum_{\substack{\beta\in\left\{  0,1\right\}  ^{3};\\\left\vert
\beta\right\vert =2}}\operatorname*{rowdet}\left(
\begin{array}
[c]{ccc}%
h_{\alpha_{1}+\beta_{1}+1,\ 1} & h_{\alpha_{1}+\beta_{1}+2,\ 1} &
h_{\alpha_{1}+\beta_{1}+3,\ 1}\\
h_{\alpha_{2}+\beta_{2}+1,\ 2} & h_{\alpha_{2}+\beta_{2}+2,\ 2} &
h_{\alpha_{2}+\beta_{2}+3,\ 2}\\
h_{\alpha_{3}+\beta_{3}+1,\ 3} & h_{\alpha_{3}+\beta_{3}+2,\ 3} &
h_{\alpha_{3}+\beta_{3}+3,\ 3}%
\end{array}
\right)
\end{align*}%
\begin{align*}
&  =\operatorname*{rowdet}\left(
\begin{array}
[c]{ccc}%
h_{\alpha_{1}+1+1,\ 1} & h_{\alpha_{1}+1+2,\ 1} & h_{\alpha_{1}+1+3,\ 1}\\
h_{\alpha_{2}+1+1,\ 2} & h_{\alpha_{2}+1+2,\ 2} & h_{\alpha_{2}+1+3,\ 2}\\
h_{\alpha_{3}+0+1,\ 3} & h_{\alpha_{3}+0+2,\ 3} & h_{\alpha_{3}+0+3,\ 3}%
\end{array}
\right) \\
&  \ \ \ \ \ \ \ \ \ \ +\operatorname*{rowdet}\left(
\begin{array}
[c]{ccc}%
h_{\alpha_{1}+1+1,\ 1} & h_{\alpha_{1}+1+2,\ 1} & h_{\alpha_{1}+1+3,\ 1}\\
h_{\alpha_{2}+0+1,\ 2} & h_{\alpha_{2}+0+2,\ 2} & h_{\alpha_{2}+0+3,\ 2}\\
h_{\alpha_{3}+1+1,\ 3} & h_{\alpha_{3}+1+2,\ 3} & h_{\alpha_{3}+1+3,\ 3}%
\end{array}
\right) \\
&  \ \ \ \ \ \ \ \ \ \ +\operatorname*{rowdet}\left(
\begin{array}
[c]{ccc}%
h_{\alpha_{1}+0+1,\ 1} & h_{\alpha_{1}+0+2,\ 1} & h_{\alpha_{1}+0+3,\ 1}\\
h_{\alpha_{2}+1+1,\ 2} & h_{\alpha_{2}+1+2,\ 2} & h_{\alpha_{2}+1+3,\ 2}\\
h_{\alpha_{3}+1+1,\ 3} & h_{\alpha_{3}+1+2,\ 3} & h_{\alpha_{3}+1+3,\ 3}%
\end{array}
\right) \\
&  \ \ \ \ \ \ \ \ \ \ \ \ \ \ \ \ \ \ \ \ \left(
\begin{array}
[c]{c}%
\text{since there are exactly three }2\text{-tuples }\beta\in\left\{
0,1\right\}  ^{3}\\
\text{satisfying }\left\vert \beta\right\vert =2\text{, namely }\left(
1,1,0\right)  \text{, }\left(  1,0,1\right)  \text{ and }\left(  0,1,1\right)
\end{array}
\right) \\
&  =\operatorname*{rowdet}\left(
\begin{array}
[c]{ccc}%
h_{\alpha_{1}+2,\ 1} & h_{\alpha_{1}+3,\ 1} & h_{\alpha_{1}+4,\ 1}\\
h_{\alpha_{2}+2,\ 2} & h_{\alpha_{2}+3,\ 2} & h_{\alpha_{2}+4,\ 2}\\
h_{\alpha_{3}+1,\ 3} & h_{\alpha_{3}+2,\ 3} & h_{\alpha_{3}+3,\ 3}%
\end{array}
\right) \\
&  \ \ \ \ \ \ \ \ \ \ +\operatorname*{rowdet}\left(
\begin{array}
[c]{ccc}%
h_{\alpha_{1}+2,\ 1} & h_{\alpha_{1}+3,\ 1} & h_{\alpha_{1}+4,\ 1}\\
h_{\alpha_{2}+1,\ 2} & h_{\alpha_{2}+2,\ 2} & h_{\alpha_{2}+3,\ 2}\\
h_{\alpha_{3}+2,\ 3} & h_{\alpha_{3}+3,\ 3} & h_{\alpha_{3}+4,\ 3}%
\end{array}
\right) \\
&  \ \ \ \ \ \ \ \ \ \ +\operatorname*{rowdet}\left(
\begin{array}
[c]{ccc}%
h_{\alpha_{1}+1,\ 1} & h_{\alpha_{1}+2,\ 1} & h_{\alpha_{1}+3,\ 1}\\
h_{\alpha_{2}+2,\ 2} & h_{\alpha_{2}+3,\ 2} & h_{\alpha_{2}+4,\ 2}\\
h_{\alpha_{3}+2,\ 3} & h_{\alpha_{3}+3,\ 3} & h_{\alpha_{3}+4,\ 3}%
\end{array}
\right)  .
\end{align*}
Therefore, (\ref{eq.thm.pre-pieri2.claim}) rewrites as%
\begin{align*}
&  \operatorname*{rowdet}\left(
\begin{array}
[c]{ccc}%
h_{\alpha_{1}+2,\ 1} & h_{\alpha_{1}+3,\ 1} & h_{\alpha_{1}+4,\ 1}\\
h_{\alpha_{2}+2,\ 2} & h_{\alpha_{2}+3,\ 2} & h_{\alpha_{2}+4,\ 2}\\
h_{\alpha_{3}+1,\ 3} & h_{\alpha_{3}+2,\ 3} & h_{\alpha_{3}+3,\ 3}%
\end{array}
\right) \\
&  \ \ \ \ \ \ \ \ \ \ +\operatorname*{rowdet}\left(
\begin{array}
[c]{ccc}%
h_{\alpha_{1}+2,\ 1} & h_{\alpha_{1}+3,\ 1} & h_{\alpha_{1}+4,\ 1}\\
h_{\alpha_{2}+1,\ 2} & h_{\alpha_{2}+2,\ 2} & h_{\alpha_{2}+3,\ 2}\\
h_{\alpha_{3}+2,\ 3} & h_{\alpha_{3}+3,\ 3} & h_{\alpha_{3}+4,\ 3}%
\end{array}
\right) \\
&  \ \ \ \ \ \ \ \ \ \ +\operatorname*{rowdet}\left(
\begin{array}
[c]{ccc}%
h_{\alpha_{1}+1,\ 1} & h_{\alpha_{1}+2,\ 1} & h_{\alpha_{1}+3,\ 1}\\
h_{\alpha_{2}+2,\ 2} & h_{\alpha_{2}+3,\ 2} & h_{\alpha_{2}+4,\ 2}\\
h_{\alpha_{3}+2,\ 3} & h_{\alpha_{3}+3,\ 3} & h_{\alpha_{3}+4,\ 3}%
\end{array}
\right) \\
&  =\operatorname*{rowdet}\left(
\begin{array}
[c]{ccc}%
h_{\alpha_{1}+1,\ 1} & h_{\alpha_{1}+3,\ 1} & h_{\alpha_{1}+4,\ 1}\\
h_{\alpha_{2}+1,\ 2} & h_{\alpha_{2}+3,\ 2} & h_{\alpha_{2}+4,\ 2}\\
h_{\alpha_{3}+1,\ 2} & h_{\alpha_{3}+3,\ 2} & h_{\alpha_{3}+4,\ 2}%
\end{array}
\right)  .
\end{align*}
This is easy to check directly by expanding all four row-determinants.
\end{example}

\subsection{The proof}

Our proof of the second pre-Pieri rule will be similar to that of the first.

Again, we will use Definition \ref{def.etapi} and Definition \ref{def.iverson}%
. Again, several lemmas will be used. The first is an analogue of Lemma
\ref{lem.nu-eta}:

\begin{lemma}
\label{lem.nu-xi}Let $n\in\mathbb{N}$. Let $q\in\mathbb{Z}$ and $\nu
\in\mathbb{Z}^{n}$ and $\xi\in\mathbb{Z}^{n}$ satisfy $\left\vert
\nu\right\vert =\left\vert \xi\right\vert $ and%
\[
\left\{  \nu_{1},\nu_{2},\ldots,\nu_{n}\right\}  \subseteq\left\{  \xi_{1}%
,\xi_{2},\ldots,\xi_{n}\right\}  \cup\left\{  q\right\}
\ \ \ \ \ \ \ \ \ \ \text{and}\ \ \ \ \ \ \ \ \ \ \left\vert \left\{  \nu
_{1},\nu_{2},\ldots,\nu_{n}\right\}  \right\vert =n.
\]
Then, there exists some permutation $\pi\in S_{n}$ satisfying $\nu=\xi\circ
\pi$.
\end{lemma}

\begin{vershort}
\begin{proof}
[Proof of Lemma \ref{lem.nu-xi}.]Set $\xi_{n+1}:=q$. Thus, $\left\{  \nu
_{1},\nu_{2},\ldots,\nu_{n}\right\}  \subseteq\left\{  \xi_{1},\xi_{2}%
,\ldots,\xi_{n}\right\}  \cup\left\{  q\right\}  $ rewrites as
\begin{equation}
\left\{  \nu_{1},\nu_{2},\ldots,\nu_{n}\right\}  \subseteq\left\{  \xi_{1}%
,\xi_{2},\ldots,\xi_{n+1}\right\}  . \label{pf.lem.nu-xi.short.1}%
\end{equation}
The $n$ numbers $\nu_{1},\nu_{2},\ldots,\nu_{n}$ are distinct (since
$\left\vert \left\{  \nu_{1},\nu_{2},\ldots,\nu_{n}\right\}  \right\vert =n$).
Furthermore, each of these $n$ numbers appears in the $\left(  n+1\right)
$-tuple $\left(  \xi_{1},\xi_{2},\ldots,\xi_{n+1}\right)  $ (by
(\ref{pf.lem.nu-xi.short.1})). Since these $n$ numbers are distinct, they must
therefore appear as $n$ \textbf{distinct} entries in this $\left(  n+1\right)
$-tuple. Thus, they must be the entries $\xi_{1},\xi_{2},\ldots,\xi_{p-1}%
,\xi_{p+1},\xi_{p+2},\ldots,\xi_{n+1}$ in some order, where $p$ is some
element of $\left[  n+1\right]  $. Consider this $p$.

Now,%
\[
\left\vert \nu\right\vert =\left\vert \xi\right\vert =\xi_{1}+\xi_{2}%
+\cdots+\xi_{n}=\left(  \xi_{1}+\xi_{2}+\cdots+\xi_{n+1}\right)  -\xi_{n+1}.
\]
Comparing this with%
\begin{align*}
\left\vert \nu\right\vert  &  =\nu_{1}+\nu_{2}+\cdots+\nu_{n}=\xi_{1}+\xi
_{2}+\cdots+\xi_{p-1}+\xi_{p+1}+\xi_{p+2}+\cdots+\xi_{n+1}\\
&  \ \ \ \ \ \ \ \ \ \ \ \ \ \ \ \ \ \ \ \ \left(
\begin{array}
[c]{c}%
\text{since the }n\text{ numbers }\nu_{1},\nu_{2},\ldots,\nu_{n}\text{ are}\\
\text{the numbers }\xi_{1},\xi_{2},\ldots,\xi_{p-1},\xi_{p+1},\xi_{p+2}%
,\ldots,\xi_{n+1}\text{ in some order}%
\end{array}
\right) \\
&  =\left(  \xi_{1}+\xi_{2}+\cdots+\xi_{n+1}\right)  -\xi_{p},
\end{align*}
we obtain $\left(  \xi_{1}+\xi_{2}+\cdots+\xi_{n+1}\right)  -\xi_{n+1}=\left(
\xi_{1}+\xi_{2}+\cdots+\xi_{n+1}\right)  -\xi_{p}$. In other words, $\xi
_{p}=\xi_{n+1}$. Hence, the numbers $\xi_{1},\xi_{2},\ldots,\xi_{p-1}%
,\xi_{p+1},\xi_{p+2},\ldots,\xi_{n+1}$ are precisely the numbers $\xi_{1}%
,\xi_{2},\ldots,\xi_{n}$ (up to order). Since the $n$ numbers $\nu_{1},\nu
_{2},\ldots,\nu_{n}$ are the numbers $\xi_{1},\xi_{2},\ldots,\xi_{p-1}%
,\xi_{p+1},\xi_{p+2},\ldots,\xi_{n+1}$ in some order, we thus conclude that
the $n$ numbers $\nu_{1},\nu_{2},\ldots,\nu_{n}$ are the numbers $\xi_{1}%
,\xi_{2},\ldots,\xi_{n}$ in some order. In other words, the $n$-tuple $\nu$ is
obtained from $\xi$ by permuting the entries. This proves Lemma
\ref{lem.nu-xi}.
\end{proof}
\end{vershort}

\begin{verlong}
\begin{proof}
[Proof of Lemma \ref{lem.nu-xi}.]The definition of $\left[  n\right]  $ yields
$\left[  n\right]  =\left\{  1,2,\ldots,n\right\}  $. The definition of
$\left[  n+1\right]  $ yields $\left[  n+1\right]  =\left\{  1,2,\ldots
,n+1\right\}  $. Hence,%
\[
\left[  n+1\right]  \setminus\left\{  n+1\right\}  =\left\{  1,2,\ldots
,n+1\right\}  \setminus\left\{  n+1\right\}  =\left\{  1,2,\ldots,n\right\}
=\left[  n\right]  .
\]
Thus, $\left[  n\right]  =\left[  n+1\right]  \setminus\left\{  n+1\right\}
\subseteq\left[  n+1\right]  $.

The definition of $\left\vert \xi\right\vert $ yields
\begin{equation}
\left\vert \xi\right\vert =\xi_{1}+\xi_{2}+\cdots+\xi_{n}=\sum_{i=1}^{n}%
\xi_{i}. \label{pf.lem.nu-xi.sum-xi}%
\end{equation}
The definition of $\left\vert \nu\right\vert $ yields
\begin{equation}
\left\vert \nu\right\vert =\nu_{1}+\nu_{2}+\cdots+\nu_{n}=\sum_{i=1}^{n}%
\nu_{i}. \label{pf.lem.nu-xi.sum-nu}%
\end{equation}

We extend the $n$-tuple $\xi=\left(  \xi_{1},\xi_{2},\ldots,\xi_{n}\right)  $
to an $\left(  n+1\right)  $-tuple $\left(  \xi_{1},\xi_{2},\ldots,\xi
_{n+1}\right)  \in\mathbb{Z}^{n+1}$ by setting $\xi_{n+1}:=q$. Then,%
\begin{align}
\left\{  \xi_{1},\xi_{2},\ldots,\xi_{n+1}\right\}   &  =\left\{  \xi_{1}%
,\xi_{2},\ldots,\xi_{n},\xi_{n+1}\right\}  =\left\{  \xi_{1},\xi_{2}%
,\ldots,\xi_{n}\right\}  \cup\left\{  \underbrace{\xi_{n+1}}_{=q}\right\}
\nonumber\\
&  =\left\{  \xi_{1},\xi_{2},\ldots,\xi_{n}\right\}  \cup\left\{  q\right\}  .
\label{pf.lem.nu-xi.1}%
\end{align}
Now,
\begin{align}
\left\{  \nu_{1},\nu_{2},\ldots,\nu_{n}\right\}   &  \subseteq\left\{  \xi
_{1},\xi_{2},\ldots,\xi_{n}\right\}  \cup\left\{  q\right\} \nonumber\\
&  =\left\{  \xi_{1},\xi_{2},\ldots,\xi_{n+1}\right\}
\ \ \ \ \ \ \ \ \ \ \left(  \text{by (\ref{pf.lem.nu-xi.1})}\right)  .
\label{pf.lem.nu-xi.2}%
\end{align}

Now, recall that $n\in\mathbb{N}$ and $\left\vert \left\{  \nu_{1},\nu
_{2},\ldots,\nu_{n}\right\}  \right\vert =n$. Hence, Lemma
\ref{lem.m-dist-objs} (applied to $m=n$ and $u_{i}=\nu_{i}$) shows that the
$n$ objects $\nu_{1},\nu_{2},\ldots,\nu_{n}$ are distinct. In other words, if
$i$ and $j$ are two distinct elements of $\left[  n\right]  $, then%
\begin{equation}
\nu_{i}\neq\nu_{j}. \label{pf.lem.nu-xi.xipneqetaq}%
\end{equation}

We define a map $f:\left[  n\right]  \rightarrow\left[  n+1\right]  $ as
follows: For each $i\in\left[  n\right]  $, we choose some $j\in\left[
n+1\right]  $ satisfying $\nu_{i}=\xi_{j}$ (indeed, such a $j$ exists,
because
\begin{align*}
\nu_{i}  &  \in\left\{  \nu_{1},\nu_{2},\ldots,\nu_{n}\right\}
\ \ \ \ \ \ \ \ \ \ \left(  \text{since }i\in\left[  n\right]  =\left\{
1,2,\ldots,n\right\}  \right) \\
&  \subseteq\left\{  \xi_{1},\xi_{2},\ldots,\xi_{n+1}\right\}
\ \ \ \ \ \ \ \ \ \ \left(  \text{by (\ref{pf.lem.nu-xi.2})}\right) \\
&  =\left\{  \xi_{j}\ \mid\ j\in\underbrace{\left\{  1,2,\ldots,n+1\right\}
}_{=\left[  n+1\right]  }\right\}  =\left\{  \xi_{j}\ \mid\ j\in\left[
n+1\right]  \right\}
\end{align*}
), and we define $f\left(  i\right)  $ to be this $j$. Thus, we have defined a
map $f:\left[  n\right]  \rightarrow\left[  n+1\right]  $.

For each $i\in\left[  n\right]  $, the image $f\left(  i\right)  $ of $i$ is
an element of $\left[  n+1\right]  $ and satisfies%
\begin{equation}
\nu_{i}=\xi_{f\left(  i\right)  } \label{pf.lem.nu-xi.def-f.2}%
\end{equation}
(since $f\left(  i\right)  $ is defined to be a $j\in\left[  n+1\right]  $
satisfying $\nu_{i}=\xi_{j}$).

The $n$ elements $f\left(  1\right)  ,f\left(  2\right)  ,\ldots,f\left(
n\right)  $ are distinct\footnote{\textit{Proof.} Let $i$ and $j$ be two
distinct elements of $\left[  n\right]  $. We shall show that $f\left(
i\right)  \neq f\left(  j\right)  $.
\par
Indeed, assume the contrary. Thus, $f\left(  i\right)  =f\left(  j\right)  $.
However, (\ref{pf.lem.nu-xi.def-f.2}) yields $\nu_{i}=\xi_{f\left(  i\right)
}=\xi_{f\left(  j\right)  }$ (since $f\left(  i\right)  =f\left(  j\right)
$). On the other hand, (\ref{pf.lem.nu-xi.def-f.2}) (applied to $j$ instead of
$i$) yields $\nu_{j}=\xi_{f\left(  j\right)  }$. Comparing these two
equalities, we obtain $\nu_{i}=\nu_{j}$. However,
(\ref{pf.lem.nu-xi.xipneqetaq}) yields $\nu_{i}\neq\nu_{j}$. This contradicts
$\nu_{i}=\nu_{j}$. This contradiction shows that our assumption was false.
Hence, $f\left(  i\right)  \neq f\left(  j\right)  $.
\par
Forget that we fixed $i$ and $j$. We thus have shown that $f\left(  i\right)
\neq f\left(  j\right)  $ whenever $i$ and $j$ are two distinct elements of
$\left[  n\right]  $. In other words, the $n$ elements $f\left(  1\right)
,f\left(  2\right)  ,\ldots,f\left(  n\right)  $ are distinct.}. Hence,
\[
\left\vert \left\{  f\left(  1\right)  ,f\left(  2\right)  ,\ldots,f\left(
n\right)  \right\}  \right\vert =n.
\]
Clearly, $\left\{  f\left(  1\right)  ,f\left(  2\right)  ,\ldots,f\left(
n\right)  \right\}  $ is a subset of $\left[  n+1\right]  $ (since $f$ is a
map from $\left[  n\right]  $ to $\left[  n+1\right]  $). Thus,%
\begin{align*}
\left\vert \left[  n+1\right]  \setminus\left\{  f\left(  1\right)  ,f\left(
2\right)  ,\ldots,f\left(  n\right)  \right\}  \right\vert  &
=\underbrace{\left\vert \left[  n+1\right]  \right\vert }%
_{\substack{=n+1\\\text{(since }\left[  n+1\right]  =\left\{  1,2,\ldots
,n+1\right\}  \text{)}}}-\underbrace{\left\vert \left\{  f\left(  1\right)
,f\left(  2\right)  ,\ldots,f\left(  n\right)  \right\}  \right\vert }_{=n}\\
&  =\left(  n+1\right)  -n=1.
\end{align*}
In other words, $\left[  n+1\right]  \setminus\left\{  f\left(  1\right)
,f\left(  2\right)  ,\ldots,f\left(  n\right)  \right\}  $ is a $1$-element
set. In other words,
\[
\left[  n+1\right]  \setminus\left\{  f\left(  1\right)  ,f\left(  2\right)
,\ldots,f\left(  n\right)  \right\}  =\left\{  p\right\}
\ \ \ \ \ \ \ \ \ \ \text{for some element }p.
\]
Consider this $p$. We have $p\in\left\{  p\right\}  =\left[  n+1\right]
\setminus\left\{  f\left(  1\right)  ,f\left(  2\right)  ,\ldots,f\left(
n\right)  \right\}  $. In other words,
\[
p\in\left[  n+1\right]  \ \ \ \ \ \ \ \ \ \ \text{and}%
\ \ \ \ \ \ \ \ \ \ p\notin\left\{  f\left(  1\right)  ,f\left(  2\right)
,\ldots,f\left(  n\right)  \right\}  .
\]

As we have seen, the $n$ elements $f\left(  1\right)  ,f\left(  2\right)
,\ldots,f\left(  n\right)  $ are distinct. Hence, the $n$-tuple $\left(
f\left(  1\right)  ,f\left(  2\right)  ,\ldots,f\left(  n\right)  \right)  $
is a list of all $n$ elements of the set $\left\{  f\left(  1\right)
,f\left(  2\right)  ,\ldots,f\left(  n\right)  \right\}  $ with no repetition.
Therefore,%
\begin{align}
\sum_{i\in\left\{  f\left(  1\right)  ,f\left(  2\right)  ,\ldots,f\left(
n\right)  \right\}  }\xi_{i}  &  =\xi_{f\left(  1\right)  }+\xi_{f\left(
2\right)  }+\cdots+\xi_{f\left(  n\right)  }=\sum_{i=1}^{n}\underbrace{\xi
_{f\left(  i\right)  }}_{\substack{=\nu_{i}\\\text{(by
(\ref{pf.lem.nu-xi.def-f.2}))}}}=\sum_{i=1}^{n}\nu_{i}\nonumber\\
&  =\left\vert \nu\right\vert \ \ \ \ \ \ \ \ \ \ \left(  \text{by
(\ref{pf.lem.nu-xi.sum-nu})}\right)  . \label{pf.lem.nu-xi.sum-nu2}%
\end{align}
However,
\begin{align*}
\sum_{i\in\left[  n+1\right]  }\xi_{i}  &  =\underbrace{\sum_{\substack{i\in
\left[  n+1\right]  ;\\i\in\left\{  f\left(  1\right)  ,f\left(  2\right)
,\ldots,f\left(  n\right)  \right\}  }}}_{\substack{=\sum_{i\in\left\{
f\left(  1\right)  ,f\left(  2\right)  ,\ldots,f\left(  n\right)  \right\}
}\\\text{(since }\left\{  f\left(  1\right)  ,f\left(  2\right)
,\ldots,f\left(  n\right)  \right\}  \\\text{is a subset of }\left[
n+1\right]  \text{)}}}\xi_{i}+\underbrace{\sum_{\substack{i\in\left[
n+1\right]  ;\\i\notin\left\{  f\left(  1\right)  ,f\left(  2\right)
,\ldots,f\left(  n\right)  \right\}  }}}_{\substack{=\sum_{i\in\left[
n+1\right]  \setminus\left\{  f\left(  1\right)  ,f\left(  2\right)
,\ldots,f\left(  n\right)  \right\}  }\\=\sum_{i\in\left\{  p\right\}
}\\\text{(since }\left[  n+1\right]  \setminus\left\{  f\left(  1\right)
,f\left(  2\right)  ,\ldots,f\left(  n\right)  \right\}  =\left\{  p\right\}
\text{)}}}\xi_{i}\\
&  \ \ \ \ \ \ \ \ \ \ \ \ \ \ \ \ \ \ \ \ \left(
\begin{array}
[c]{c}%
\text{since each }i\in\left[  n+1\right]  \text{ satisfies}\\
\text{either }i\in\left\{  f\left(  1\right)  ,f\left(  2\right)
,\ldots,f\left(  n\right)  \right\} \\
\text{or }i\notin\left\{  f\left(  1\right)  ,f\left(  2\right)
,\ldots,f\left(  n\right)  \right\}  \text{ (but not both)}%
\end{array}
\right) \\
&  =\underbrace{\sum_{i\in\left\{  f\left(  1\right)  ,f\left(  2\right)
,\ldots,f\left(  n\right)  \right\}  }\xi_{i}}_{\substack{=\left\vert
\nu\right\vert \\\text{(by (\ref{pf.lem.nu-xi.sum-nu2}))}}}+\underbrace{\sum
_{i\in\left\{  p\right\}  }\xi_{i}}_{=\xi_{p}}=\underbrace{\left\vert
\nu\right\vert }_{=\left\vert \xi\right\vert }+\xi_{p}=\left\vert
\xi\right\vert +\xi_{p}.
\end{align*}
Therefore,%
\begin{align*}
\left\vert \xi\right\vert +\xi_{p}  &  =\underbrace{\sum_{i\in\left[
n+1\right]  }}_{=\sum_{i=1}^{n+1}}\xi_{i}=\sum_{i=1}^{n+1}\xi_{i}=\xi_{1}%
+\xi_{2}+\cdots+\xi_{n+1}\\
&  =\underbrace{\left(  \xi_{1}+\xi_{2}+\cdots+\xi_{n}\right)  }%
_{\substack{=\sum_{i=1}^{n}\xi_{i}=\left\vert \xi\right\vert \\\text{(by
(\ref{pf.lem.nu-xi.sum-xi}))}}}+\xi_{n+1}=\left\vert \xi\right\vert +\xi
_{n+1}.
\end{align*}
Subtracting $\left\vert \xi\right\vert $ from both sides of this equality, we
obtain%
\begin{equation}
\xi_{p}=\xi_{n+1}. \label{pf.lem.nu-xi.xip=q}%
\end{equation}

Using this equality, it is now easy to see that
\begin{equation}
\left\{  \nu_{1},\nu_{2},\ldots,\nu_{n}\right\}  \subseteq\left\{  \xi_{1}%
,\xi_{2},\ldots,\xi_{n}\right\}  . \label{pf.lem.nu-xi.nu-in-xi}%
\end{equation}

[\textit{Proof of (\ref{pf.lem.nu-xi.nu-in-xi}):} Let $i\in\left\{
1,2,\ldots,n\right\}  $. We shall show that $\nu_{i}\in\left\{  \xi_{1}%
,\xi_{2},\ldots,\xi_{n}\right\}  $.

We have $i\in\left\{  1,2,\ldots,n\right\}  =\left[  n\right]  $, so that
$\nu_{i}=\xi_{f\left(  i\right)  }$ (by (\ref{pf.lem.nu-xi.def-f.2})). We are
in one of the following two cases:

\textit{Case 1:} We have $f\left(  i\right)  =n+1$.

\textit{Case 2:} We have $f\left(  i\right)  \neq n+1$.

Let us first consider Case 1. In this case, we have $f\left(  i\right)  =n+1$.
Hence, $\xi_{f\left(  i\right)  }=\xi_{n+1}$. However, from $f\left(
i\right)  =n+1$, we obtain $n+1=f\left(  i\right)  \in\left\{  f\left(
1\right)  ,f\left(  2\right)  ,\ldots,f\left(  n\right)  \right\}  $ (since
$i\in\left\{  1,2,\ldots,n\right\}  $). Thus, we cannot have $p=n+1$ (because
if we had $p=n+1$, then we would have $p=n+1\in\left\{  f\left(  1\right)
,f\left(  2\right)  ,\ldots,f\left(  n\right)  \right\}  $, which would
contradict $p\notin\left\{  f\left(  1\right)  ,f\left(  2\right)
,\ldots,f\left(  n\right)  \right\}  $). In other words, we have $p\neq n+1$.
Combining $p\in\left[  n+1\right]  $ with $p\neq n+1$, we obtain $p\in\left[
n+1\right]  \setminus\left\{  n+1\right\}  =\left[  n\right]  =\left\{
1,2,\ldots,n\right\}  $. Hence, $\xi_{p}\in\left\{  \xi_{1},\xi_{2},\ldots
,\xi_{n}\right\}  $. Now,
\begin{align*}
\nu_{i}  &  =\xi_{f\left(  i\right)  }=\xi_{n+1}=\xi_{p}%
\ \ \ \ \ \ \ \ \ \ \left(  \text{by (\ref{pf.lem.nu-xi.xip=q})}\right) \\
&  \in\left\{  \xi_{1},\xi_{2},\ldots,\xi_{n}\right\}  .
\end{align*}
Thus, $\nu_{i}\in\left\{  \xi_{1},\xi_{2},\ldots,\xi_{n}\right\}  $ is proved
in Case 1.

Next, let us consider Case 2. In this case, we have $f\left(  i\right)  \neq
n+1$. Combining $f\left(  i\right)  \in\left[  n+1\right]  $ with $f\left(
i\right)  \neq n+1$, we obtain $f\left(  i\right)  \in\left[  n+1\right]
\setminus\left\{  n+1\right\}  =\left[  n\right]  =\left\{  1,2,\ldots
,n\right\}  $. Hence, $\xi_{f\left(  i\right)  }\in\left\{  \xi_{1},\xi
_{2},\ldots,\xi_{n}\right\}  $. Now,
\[
\nu_{i}=\xi_{f\left(  i\right)  }\in\left\{  \xi_{1},\xi_{2},\ldots,\xi
_{n}\right\}  .
\]
Thus, $\nu_{i}\in\left\{  \xi_{1},\xi_{2},\ldots,\xi_{n}\right\}  $ is proved
in Case 2.

We have now proved $\nu_{i}\in\left\{  \xi_{1},\xi_{2},\ldots,\xi_{n}\right\}
$ in both Cases 1 and 2. Hence, $\nu_{i}\in\left\{  \xi_{1},\xi_{2},\ldots
,\xi_{n}\right\}  $ always holds.

Now, forget that we fixed $i$. We thus have shown that $\nu_{i}\in\left\{
\xi_{1},\xi_{2},\ldots,\xi_{n}\right\}  $ for each $i\in\left\{
1,2,\ldots,n\right\}  $. In other words, all $n$ numbers $\nu_{1},\nu
_{2},\ldots,\nu_{n}$ are elements of $\left\{  \xi_{1},\xi_{2},\ldots,\xi
_{n}\right\}  $. In other words, $\left\{  \nu_{1},\nu_{2},\ldots,\nu
_{n}\right\}  \subseteq\left\{  \xi_{1},\xi_{2},\ldots,\xi_{n}\right\}  $.
This proves (\ref{pf.lem.nu-xi.nu-in-xi}).] \medskip

Now, we define a map $g:\left[  n\right]  \rightarrow\left[  n\right]  $ as
follows: For each $i\in\left[  n\right]  $, we choose some $j\in\left[
n\right]  $ satisfying $\nu_{i}=\xi_{j}$ (indeed, such a $j$ exists, because
\begin{align*}
\nu_{i}  &  \in\left\{  \nu_{1},\nu_{2},\ldots,\nu_{n}\right\}
\ \ \ \ \ \ \ \ \ \ \left(  \text{since }i\in\left[  n\right]  =\left\{
1,2,\ldots,n\right\}  \right) \\
&  \subseteq\left\{  \xi_{1},\xi_{2},\ldots,\xi_{n}\right\}
\ \ \ \ \ \ \ \ \ \ \left(  \text{by (\ref{pf.lem.nu-xi.nu-in-xi})}\right) \\
&  =\left\{  \xi_{j}\ \mid\ j\in\underbrace{\left\{  1,2,\ldots,n\right\}
}_{=\left[  n\right]  }\right\}  =\left\{  \xi_{j}\ \mid\ j\in\left[
n\right]  \right\}
\end{align*}
), and we define $g\left(  i\right)  $ to be this $j$. Thus, we have defined a
map $g:\left[  n\right]  \rightarrow\left[  n\right]  $.

For each $i\in\left[  n\right]  $, the image $g\left(  i\right)  $ of $i$ is
an element of $\left[  n\right]  $ and satisfies%
\begin{equation}
\nu_{i}=\xi_{g\left(  i\right)  } \label{pf.lem.nu-xi.def-g.2}%
\end{equation}
(since $g\left(  i\right)  $ is defined to be a $j\in\left[  n\right]  $
satisfying $\nu_{i}=\xi_{j}$).

The map $g$ is injective\footnote{\textit{Proof.} Let $i$ and $j$ be two
distinct elements of $\left[  n\right]  $. We shall show that $g\left(
i\right)  \neq g\left(  j\right)  $.
\par
Indeed, assume the contrary. Thus, $g\left(  i\right)  =g\left(  j\right)  $.
However, (\ref{pf.lem.nu-xi.def-g.2}) yields $\nu_{i}=\xi_{g\left(  i\right)
}=\xi_{g\left(  j\right)  }$ (since $g\left(  i\right)  =g\left(  j\right)
$). On the other hand, (\ref{pf.lem.nu-xi.def-g.2}) (applied to $j$ instead of
$i$) yields $\nu_{j}=\xi_{g\left(  j\right)  }$. Comparing these two
equalities, we obtain $\nu_{i}=\nu_{j}$. However,
(\ref{pf.lem.nu-xi.xipneqetaq}) yields $\nu_{i}\neq\nu_{j}$. This contradicts
$\nu_{i}=\nu_{j}$. This contradiction shows that our assumption was false.
Hence, $g\left(  i\right)  \neq g\left(  j\right)  $.
\par
Forget that we fixed $i$ and $j$. We thus have shown that $g\left(  i\right)
\neq g\left(  j\right)  $ whenever $i$ and $j$ are two distinct elements of
$\left[  n\right]  $. In other words, the map $g$ is injective.}. However, a
well-known fact from set theory (actually one of the versions of the
pigeonhole principle) says that any injective map $\phi$ from a finite set $X$
to $X$ must be a permutation of $X$. Applying this to $X=\left[  n\right]  $
and $\phi=g$, we conclude that $g$ must be a permutation of $\left[  n\right]
$ (since $g$ is an injective map from the finite set $\left[  n\right]  $ to
$\left[  n\right]  $). In other words, $g\in S_{n}$ (since $S_{n}$ is the set
of all permutations of $\left[  n\right]  $).

From (\ref{pf.lem.nu-xi.def-g.2}), we know that $\nu_{i}=\xi_{g\left(
i\right)  }$ for each $i\in\left[  n\right]  $. In other words,%
\[
\left(  \nu_{1},\nu_{2},\ldots,\nu_{n}\right)  =\left(  \xi_{g\left(
1\right)  },\xi_{g\left(  2\right)  },\ldots,\xi_{g\left(  n\right)  }\right)
.
\]
However, Definition \ref{def.etapi} yields%
\[
\xi\circ g=\left(  \xi_{g\left(  1\right)  },\xi_{g\left(  2\right)  }%
,\ldots,\xi_{g\left(  n\right)  }\right)  .
\]
Comparing these two equalities, we obtain%
\[
\xi\circ g=\left(  \nu_{1},\nu_{2},\ldots,\nu_{n}\right)  =\nu.
\]
In other words, $\nu=\xi\circ g$. Hence, there exists some permutation $\pi\in
S_{n}$ satisfying $\nu=\xi\circ\pi$ (namely, $\pi=g$). This proves Lemma
\ref{lem.nu-xi}.
\end{proof}
\end{verlong}

We shall furthermore use the following notation:

\begin{definition}
\label{def.cover01}Let $u$ and $v$ be two integers. We write \textquotedblleft%
$u\trianglerighteq v$\textquotedblright\ if and only if $u-v\in\left\{
0,1\right\}  $.
\end{definition}

Thus, for example, $2\trianglerighteq2$ and $2\trianglerighteq1$, but we don't
have $2\trianglerighteq0$.

We need the following simple lemma:

\begin{lemma}
\label{lem.xi-unique}Let $n\in\mathbb{N}$ and $p\in\left\{  0,1,\ldots
,n\right\}  $. Let $\xi$ be the $n$-tuple
\begin{align*}
&  \left(  1,2,\ldots,n\right)  +\left(  \underbrace{0,0,\ldots,0}_{n-p\text{
zeroes}},\underbrace{1,1,\ldots,1}_{p\text{ ones}}\right) \\
&  =\left(  1,2,\ldots,n-p,n-p+2,n-p+3,\ldots,n+1\right)  \in\mathbb{Z}^{n}.
\end{align*}
Let $\sigma\in S_{n}$ be a permutation such that $\sigma\neq\operatorname*{id}%
$ (where $\operatorname*{id}$ denotes the identity map $\left[  n\right]
\rightarrow\left[  n\right]  $). Then,%
\[
\prod_{i=1}^{n}\left[  \xi_{i}\trianglerighteq\sigma\left(  i\right)  \right]
=0.
\]

\end{lemma}

\begin{vershort}
\begin{proof}
[Proof of Lemma \ref{lem.xi-unique}.]The definition of $\xi$ shows that%
\begin{equation}
\xi_{i}=i\ \ \ \ \ \ \ \ \ \ \text{for each }i\in\left\{  1,2,\ldots
,n-p\right\}  \label{pf.lem.xi-unique.short.xii=}%
\end{equation}
and%
\begin{equation}
\xi_{i}=i+1\ \ \ \ \ \ \ \ \ \ \text{for each }i\in\left\{  n-p+1,n-p+2,\ldots
,n\right\}  . \label{pf.lem.xi-unique.short.xin=}%
\end{equation}

We assumed that $\sigma\neq\operatorname*{id}$. Hence, there exists some
$i\in\left[  n\right]  $ such that $\sigma\left(  i\right)  \neq i$. Let $a$
be the \textbf{smallest} such $i$, and let $b$ be the \textbf{largest} such
$i$. Then, $\sigma\left(  a\right)  >a$\ \ \ \ \footnote{\textit{Proof.} We
have $\sigma\left(  a\right)  \neq a$ (since $a$ is an $i\in\left[  n\right]
$ such that $\sigma\left(  i\right)  \neq i$). Thus, $\sigma\left(
\sigma\left(  a\right)  \right)  \neq\sigma\left(  a\right)  $ (since $\sigma$
is a permutation and therefore injective). Hence, $\sigma\left(  a\right)  $
is an $i\in\left[  n\right]  $ such that $\sigma\left(  i\right)  \neq i$.
Since $a$ is the \textbf{smallest} such $i$, we thus conclude that
$\sigma\left(  a\right)  \geq a$. Hence, $\sigma\left(  a\right)  >a$ (since
$\sigma\left(  a\right)  \neq a$).} and $\sigma\left(  b\right)
<b$\ \ \ \ \footnote{The proof of this is similar to the proof we just gave
for $\sigma\left(  a\right)  >a$.} and $a\leq b$\ \ \ \ \footnote{since $a$ is
the \textbf{smallest} $i\in\left[  n\right]  $ such that $\sigma\left(
i\right)  \neq i$, while $b$ is the \textbf{largest} such $i$}.

Now, we are in one of the following two cases:

\textit{Case 1:} We have $a\leq n-p$.

\textit{Case 2:} We have $a>n-p$.

Let us first consider Case 1. In this case, we have $a\leq n-p$. Thus,
$a\in\left\{  1,2,\ldots,n-p\right\}  $. Hence,
(\ref{pf.lem.xi-unique.short.xii=}) (applied to $i=a$) yields $\xi
_{a}=a<\sigma\left(  a\right)  $ (since $\sigma\left(  a\right)  >a$), so that
$\xi_{a}-\sigma\left(  a\right)  <0$. Hence, we do not have $\xi
_{a}\trianglerighteq\sigma\left(  a\right)  $ (since $\xi_{a}\trianglerighteq
\sigma\left(  a\right)  $ would mean that $\xi_{a}-\sigma\left(  a\right)
\in\left\{  0,1\right\}  $, which would contradict $\xi_{a}-\sigma\left(
a\right)  <0$). In other words, we have $\left[  \xi_{a}\trianglerighteq
\sigma\left(  a\right)  \right]  =0$.

Therefore, $\prod_{i=1}^{n}\left[  \xi_{i}\trianglerighteq\sigma\left(
i\right)  \right]  =0$ (since $\left[  \xi_{a}\trianglerighteq\sigma\left(
a\right)  \right]  $ is one of the factors of the product $\prod_{i=1}%
^{n}\left[  \xi_{i}\trianglerighteq\sigma\left(  i\right)  \right]  $). Thus,
Lemma \ref{lem.xi-unique} is proved in Case 1.

Let us next consider Case 2. In this case, we have $a>n-p$. Thus, $b\geq
a>n-p$, so that $b\in\left\{  n-p+1,n-p+2,\ldots,n\right\}  $. Hence,
(\ref{pf.lem.xi-unique.short.xin=}) (applied to $i=b$) yields $\xi_{b}=b+1$.
However, recall that $\sigma\left(  b\right)  <b$. Thus, $\underbrace{\xi_{b}%
}_{=b+1}-\underbrace{\sigma\left(  b\right)  }_{<b}>\left(  b+1\right)  -b=1$.
Therefore, we do not have $\xi_{b}\trianglerighteq\sigma\left(  b\right)  $
(because $\xi_{b}\trianglerighteq\sigma\left(  b\right)  $ would mean that
$\xi_{b}-\sigma\left(  b\right)  \in\left\{  0,1\right\}  $, which would
contradict $\xi_{b}-\sigma\left(  b\right)  >1$). In other words, we have
$\left[  \xi_{b}\trianglerighteq\sigma\left(  b\right)  \right]  =0$.

Therefore, $\prod_{i=1}^{n}\left[  \xi_{i}\trianglerighteq\sigma\left(
i\right)  \right]  =0$ (since $\left[  \xi_{b}\trianglerighteq\sigma\left(
b\right)  \right]  $ is one of the factors of the product $\prod_{i=1}%
^{n}\left[  \xi_{i}\trianglerighteq\sigma\left(  i\right)  \right]  $). Thus,
Lemma \ref{lem.xi-unique} is proved in Case 2.

We have now proved Lemma \ref{lem.xi-unique} in both Cases 1 and 2. Hence, the
proof of Lemma \ref{lem.xi-unique} is complete.
\end{proof}
\end{vershort}

\begin{verlong}
\begin{proof}
[Proof of Lemma \ref{lem.xi-unique}.]The definition of $\xi$ yields%
\begin{align*}
\xi &  =\left(  1,2,\ldots,n\right)  +\left(  \underbrace{0,0,\ldots
,0}_{n-p\text{ zeroes}},\underbrace{1,1,\ldots,1}_{p\text{ ones}}\right) \\
&  =\left(  1,2,\ldots,n-p,n-p+2,n-p+3,\ldots,n+1\right)  .
\end{align*}
Thus,%
\begin{align}
&  \left(  \xi_{1},\xi_{2},\ldots,\xi_{n}\right) \nonumber\\
&  =\xi=\left(  1,2,\ldots,n-p,n-p+2,n-p+3,\ldots,n+1\right)  .
\label{pf.lem.xi-unique.xi=2}%
\end{align}
In other words, we have%
\begin{equation}
\xi_{i}=i\ \ \ \ \ \ \ \ \ \ \text{for each }i\in\left\{  1,2,\ldots
,n-p\right\}  \label{pf.lem.xi-unique.xii=}%
\end{equation}
and%
\begin{equation}
\xi_{i}=i+1\ \ \ \ \ \ \ \ \ \ \text{for each }i\in\left\{  n-p+1,n-p+2,\ldots
,n\right\}  . \label{pf.lem.xi-unique.xin=}%
\end{equation}

We define a subset $Z$ of $\left[  n\right]  $ by%
\[
Z=\left\{  i\in\left[  n\right]  \ \mid\ \sigma\left(  i\right)  \neq
i\right\}  .
\]
This set $Z$ is nonempty\footnote{\textit{Proof.} Assume the contrary. Thus,
$Z$ is empty. In other words, $Z=\varnothing$. In other words, $\left\{
i\in\left[  n\right]  \ \mid\ \sigma\left(  i\right)  \neq i\right\}
=\varnothing$ (since $Z=\left\{  i\in\left[  n\right]  \ \mid\ \sigma\left(
i\right)  \neq i\right\}  $). In other words, there exists no $i\in\left[
n\right]  $ satisfying $\sigma\left(  i\right)  \neq i$. In other words, each
$i\in\left[  n\right]  $ satisfies $\sigma\left(  i\right)  =i$. Hence, each
$i\in\left[  n\right]  $ satisfies $\sigma\left(  i\right)
=i=\operatorname*{id}\left(  i\right)  $. In other words, we have
$\sigma=\operatorname*{id}$ (since both $\sigma$ and $\operatorname*{id}$ are
maps from $\left[  n\right]  $ to $\left[  n\right]  $). This contradicts
$\sigma\neq\operatorname*{id}$. This contradiction shows that our assumption
was false, qed.} and finite\footnote{since it is a subset of the finite set
$\left[  n\right]  $ (by its definition)}. Hence, this set $Z$ has a smallest
element and a largest element. Let us denote these two elements by $a$ and
$b$, respectively. Thus, $a$ is the smallest element of $Z$. In other words,
$a\in Z$ and%
\begin{equation}
a\leq z\ \ \ \ \ \ \ \ \ \ \text{for each }z\in Z.
\label{pf.lem.xi-unique.amin}%
\end{equation}
Furthermore, $b$ is the largest element of $Z$. In other words, $b\in Z$ and%
\begin{equation}
b\geq z\ \ \ \ \ \ \ \ \ \ \text{for each }z\in Z.
\label{pf.lem.xi-unique.bmax}%
\end{equation}
Applying (\ref{pf.lem.xi-unique.bmax}) to $z=a$, we obtain $b\geq a$ (since
$a\in Z$). Now, we are in one of the following two cases:

\textit{Case 1:} We have $a\leq n-p$.

\textit{Case 2:} We have $a>n-p$.

Let us first consider Case 1. In this case, we have $a\leq n-p$. Also, $a\in
Z\subseteq\left[  n\right]  =\left\{  1,2,\ldots,n\right\}  $ (by the
definition of $\left[  n\right]  $), so that $a\geq1$. Combining this with
$a\leq n-p$, we obtain $a\in\left\{  1,2,\ldots,n-p\right\}  $. Hence,
(\ref{pf.lem.xi-unique.xii=}) (applied to $i=a$) yields $\xi_{a}=a$. However,
it is easy to see that $\sigma\left(  a\right)  >a$%
\ \ \ \ \footnote{\textit{Proof.} We have $a\in Z=\left\{  i\in\left[
n\right]  \ \mid\ \sigma\left(  i\right)  \neq i\right\}  $. In other words,
$a$ is an $i\in\left[  n\right]  $ satisfying $\sigma\left(  i\right)  \neq
i$. In other words, $a$ is an element of $\left[  n\right]  $ and satisfies
$\sigma\left(  a\right)  \neq a$.
\par
The map $\sigma$ belongs to $S_{n}$, and thus is a permutation of $\left[
n\right]  $ (since $S_{n}$ is the set of all permutations of $\left[
n\right]  $). Hence, this map is bijective, therefore injective. In other
words, if $i$ and $j$ are two elements of $\left[  n\right]  $ satisfying
$i\neq j$, then $\sigma\left(  i\right)  \neq\sigma\left(  j\right)  $.
Applying this to $i=\sigma\left(  a\right)  $ and $j=a$, we obtain
$\sigma\left(  \sigma\left(  a\right)  \right)  \neq\sigma\left(  a\right)  $
(since $\sigma\left(  a\right)  \neq a$). Thus, $\sigma\left(  a\right)  $ is
an $i\in\left[  n\right]  $ satisfying $\sigma\left(  i\right)  \neq i$. In
other words, $\sigma\left(  a\right)  \in\left\{  i\in\left[  n\right]
\ \mid\ \sigma\left(  i\right)  \neq i\right\}  $. In other words,
$\sigma\left(  a\right)  \in Z$ (since $Z=\left\{  i\in\left[  n\right]
\ \mid\ \sigma\left(  i\right)  \neq i\right\}  $). Hence,
(\ref{pf.lem.xi-unique.amin}) (applied to $z=\sigma\left(  a\right)  $) yields
$a\leq\sigma\left(  a\right)  $. In other words, $\sigma\left(  a\right)  \geq
a$. Combining this with $\sigma\left(  a\right)  \neq a$, we obtain
$\sigma\left(  a\right)  >a$.}. Hence, $\xi_{a}=a<\sigma\left(  a\right)  $
(since $\sigma\left(  a\right)  >a$), so that $\xi_{a}-\sigma\left(  a\right)
<0$. Therefore, we do not have $\xi_{a}-\sigma\left(  a\right)  \in\left\{
0,1\right\}  $ (because if we had $\xi_{a}-\sigma\left(  a\right)  \in\left\{
0,1\right\}  $, then we would have $\xi_{a}-\sigma\left(  a\right)
\in\left\{  0,1\right\}  \subseteq\mathbb{N}$ and therefore $\xi_{a}%
-\sigma\left(  a\right)  \geq0$, which would contradict $\xi_{a}-\sigma\left(
a\right)  <0$). However, the statement \textquotedblleft$\xi_{a}%
\trianglerighteq\sigma\left(  a\right)  $\textquotedblright\ is equivalent to
\textquotedblleft$\xi_{a}-\sigma\left(  a\right)  \in\left\{  0,1\right\}
$\textquotedblright\ (by Definition \ref{def.cover01}). Hence, we do not have
$\xi_{a}\trianglerighteq\sigma\left(  a\right)  $ (since we do not have
$\xi_{a}-\sigma\left(  a\right)  \in\left\{  0,1\right\}  $). In other words,
we have $\left[  \xi_{a}\trianglerighteq\sigma\left(  a\right)  \right]  =0$.

Now, recall that $a\in\left\{  1,2,\ldots,n\right\}  $. Thus, the term
$\left[  \xi_{a}\trianglerighteq\sigma\left(  a\right)  \right]  $ is one of
the factors of the product $\prod_{i=1}^{n}\left[  \xi_{i}\trianglerighteq
\sigma\left(  i\right)  \right]  $ (namely, the factor for $i=a$). Since this
term is $0$ (because we have just shown that $\left[  \xi_{a}\trianglerighteq
\sigma\left(  a\right)  \right]  =0$), we thus conclude that one of the
factors of the product $\prod_{i=1}^{n}\left[  \xi_{i}\trianglerighteq
\sigma\left(  i\right)  \right]  $ is $0$. Therefore, the entire product is
$0$. In other words, we have $\prod_{i=1}^{n}\left[  \xi_{i}\trianglerighteq
\sigma\left(  i\right)  \right]  =0$. Thus, Lemma \ref{lem.xi-unique} is
proved in Case 1.

Let us next consider Case 2. In this case, we have $a>n-p$. Thus, $b\geq
a>n-p$, so that $b\geq n-p+1$ (since $b$ and $n-p$ are integers). Also, $b\in
Z\subseteq\left[  n\right]  =\left\{  1,2,\ldots,n\right\}  $ (by the
definition of $\left[  n\right]  $), so that $b\leq n$. Combining this with
$b\geq n-p+1$, we obtain $b\in\left\{  n-p+1,n-p+2,\ldots,n\right\}  $. Hence,
(\ref{pf.lem.xi-unique.xin=}) (applied to $i=b$) yields $\xi_{b}=b+1$.
However, it is easy to see that $\sigma\left(  b\right)  <b$%
\ \ \ \ \footnote{\textit{Proof.} We have $b\in Z=\left\{  i\in\left[
n\right]  \ \mid\ \sigma\left(  i\right)  \neq i\right\}  $. In other words,
$b$ is an $i\in\left[  n\right]  $ satisfying $\sigma\left(  i\right)  \neq
i$. In other words, $b$ is an element of $\left[  n\right]  $ and satisfies
$\sigma\left(  b\right)  \neq b$.
\par
The map $\sigma$ belongs to $S_{n}$, and thus is a permutation of $\left[
n\right]  $ (since $S_{n}$ is the set of all permutations of $\left[
n\right]  $). Hence, this map is bijective, therefore injective. In other
words, if $i$ and $j$ are two elements of $\left[  n\right]  $ satisfying
$i\neq j$, then $\sigma\left(  i\right)  \neq\sigma\left(  j\right)  $.
Applying this to $i=\sigma\left(  b\right)  $ and $j=b$, we obtain
$\sigma\left(  \sigma\left(  b\right)  \right)  \neq\sigma\left(  b\right)  $
(since $\sigma\left(  b\right)  \neq b$). Thus, $\sigma\left(  b\right)  $ is
an $i\in\left[  n\right]  $ satisfying $\sigma\left(  i\right)  \neq i$. In
other words, $\sigma\left(  b\right)  \in\left\{  i\in\left[  n\right]
\ \mid\ \sigma\left(  i\right)  \neq i\right\}  $. In other words,
$\sigma\left(  b\right)  \in Z$ (since $Z=\left\{  i\in\left[  n\right]
\ \mid\ \sigma\left(  i\right)  \neq i\right\}  $). Hence,
(\ref{pf.lem.xi-unique.bmax}) (applied to $z=\sigma\left(  b\right)  $) yields
$b\geq\sigma\left(  b\right)  $. In other words, $\sigma\left(  b\right)  \leq
b$. Combining this with $\sigma\left(  b\right)  \neq b$, we obtain
$\sigma\left(  b\right)  <b$.}. Thus, $\underbrace{\xi_{b}}_{=b+1}%
-\underbrace{\sigma\left(  b\right)  }_{<b}>\left(  b+1\right)  -b=1$.
Therefore, we do not have $\xi_{b}-\sigma\left(  b\right)  \in\left\{
0,1\right\}  $ (because if we had $\xi_{b}-\sigma\left(  b\right)  \in\left\{
0,1\right\}  $, then we would have $\xi_{b}-\sigma\left(  b\right)  \leq1$
(because any element of the set $\left\{  0,1\right\}  $ is $\leq1$), which
would contradict $\xi_{b}-\sigma\left(  b\right)  >1$). However, the statement
\textquotedblleft$\xi_{b}\trianglerighteq\sigma\left(  b\right)
$\textquotedblright\ is equivalent to \textquotedblleft$\xi_{b}-\sigma\left(
b\right)  \in\left\{  0,1\right\}  $\textquotedblright\ (by Definition
\ref{def.cover01}). Hence, we do not have $\xi_{b}\trianglerighteq
\sigma\left(  b\right)  $ (since we do not have $\xi_{b}-\sigma\left(
b\right)  \in\left\{  0,1\right\}  $). In other words, we have $\left[
\xi_{b}\trianglerighteq\sigma\left(  b\right)  \right]  =0$.

Now, recall that $b\in\left\{  1,2,\ldots,n\right\}  $. Thus, the term
$\left[  \xi_{b}\trianglerighteq\sigma\left(  b\right)  \right]  $ is one of
the factors of the product $\prod_{i=1}^{n}\left[  \xi_{i}\trianglerighteq
\sigma\left(  i\right)  \right]  $ (namely, the factor for $i=b$). Since this
term is $0$ (because we have just shown that $\left[  \xi_{b}\trianglerighteq
\sigma\left(  b\right)  \right]  =0$), we thus conclude that one of the
factors of the product $\prod_{i=1}^{n}\left[  \xi_{i}\trianglerighteq
\sigma\left(  i\right)  \right]  $ is $0$. Therefore, the entire product is
$0$. In other words, we have $\prod_{i=1}^{n}\left[  \xi_{i}\trianglerighteq
\sigma\left(  i\right)  \right]  =0$. Thus, Lemma \ref{lem.xi-unique} is
proved in Case 2.

We have now proved Lemma \ref{lem.xi-unique} in both Cases 1 and 2. Hence, the
proof of Lemma \ref{lem.xi-unique} is complete.
\end{proof}
\end{verlong}

The following lemma is an analogue of Lemma \ref{lem.nu-det}:

\begin{lemma}
\label{lem.xi-det}Let $n\in\mathbb{N}$ and $p\in\left\{  0,1,\ldots,n\right\}
$. Let $\xi$ be the $n$-tuple
\begin{align*}
&  \left(  1,2,\ldots,n\right)  +\left(  \underbrace{0,0,\ldots,0}_{n-p\text{
zeroes}},\underbrace{1,1,\ldots,1}_{p\text{ ones}}\right) \\
&  =\left(  1,2,\ldots,n-p,n-p+2,n-p+3,\ldots,n+1\right)  \in\mathbb{Z}^{n}.
\end{align*}
Let $\nu\in\mathbb{Z}^{n}$ be an $n$-tuple satisfying $\left\vert
\nu\right\vert =\left\vert \xi\right\vert $. Then,
\begin{equation}
\det\left(  \left(  \left[  \nu_{i}\trianglerighteq j\right]  \right)
_{i,j\in\left[  n\right]  }\right)  =\sum_{\substack{\sigma\in S_{n};\\\nu
=\xi\circ\sigma}}\left(  -1\right)  ^{\sigma}. \label{eq.lem.xi-det.clm}%
\end{equation}

\end{lemma}

Note that the matrix $\left(  \left[  \nu_{i}\trianglerighteq j\right]
\right)  _{i,j\in\left[  n\right]  }$ in (\ref{eq.lem.xi-det.clm}) is a matrix
with integer entries; thus, its determinant is a well-defined integer.

Before we prove Lemma \ref{lem.xi-det}, a remark is in order:

\begin{remark}
The sum $\sum_{\substack{\sigma\in S_{n};\\\nu=\xi\circ\sigma}}\left(
-1\right)  ^{\sigma}$ on the right hand side of (\ref{eq.lem.xi-det.clm})
always has either no addends or only one addend. (Indeed, it is easy to see
that the $n$-tuples $\xi\circ\sigma$ for different $\sigma\in S_{n}$ are
distinct; thus, no more than one of these $n$-tuples can equal $\nu$.) Thus,
this sum can be rewritten as
\[%
\begin{cases}
\left(  -1\right)  ^{\sigma}, & \text{if }\nu=\xi\circ\sigma\text{ for some
}\sigma\in S_{n};\\
0, & \text{otherwise.}%
\end{cases}
\]

\end{remark}

\begin{proof}
[Proof of Lemma \ref{lem.xi-det}.]The definition of $\xi$ yields%
\begin{align*}
\xi &  =\left(  1,2,\ldots,n\right)  +\left(  \underbrace{0,0,\ldots
,0}_{n-p\text{ zeroes}},\underbrace{1,1,\ldots,1}_{p\text{ ones}}\right) \\
&  =\left(  1,2,\ldots,n-p,n-p+2,n-p+3,\ldots,n+1\right)  .
\end{align*}
Thus,%
\begin{align}
&  \left(  \xi_{1},\xi_{2},\ldots,\xi_{n}\right) \nonumber\\
&  =\xi=\left(  1,2,\ldots,n-p,n-p+2,n-p+3,\ldots,n+1\right)  .
\label{pf.lem.xi-det.xi=2}%
\end{align}
In other words, we have%
\begin{equation}
\xi_{i}=i\ \ \ \ \ \ \ \ \ \ \text{for each }i\in\left\{  1,2,\ldots
,n-p\right\}  \label{pf.lem.xi-det.xii=}%
\end{equation}
and%
\begin{equation}
\xi_{i}=i+1\ \ \ \ \ \ \ \ \ \ \text{for each }i\in\left\{  n-p+1,n-p+2,\ldots
,n\right\}  . \label{pf.lem.xi-det.xin=}%
\end{equation}

\begin{vershort}
\noindent It follows easily that the $n$ numbers $\xi_{1},\xi_{2},\ldots
,\xi_{n}$ are distinct. Therefore, we have $\left\vert \left\{  \xi_{1}%
,\xi_{2},\ldots,\xi_{n}\right\}  \right\vert =n$.
\end{vershort}

\begin{verlong}
Clearly, the following two chains of inequalities hold:%
\[
1<2<\cdots<n-p\ \ \ \ \ \ \ \ \ \ \text{and}%
\ \ \ \ \ \ \ \ \ \ n-p+2<n-p+3<\cdots<n+1.
\]
Since $n-p<n-p+2$, we can splice these two chains together to obtain a single
chain of inequalities:%
\[
1<2<\cdots<n-p<n-p+2<n-p+3<\cdots<n+1.
\]
In view of (\ref{pf.lem.xi-det.xi=2}), we can rewrite this chain of
inequalities as follows:
\[
\xi_{1}<\xi_{2}<\cdots<\xi_{n}.
\]
Hence, the $n$ numbers $\xi_{1},\xi_{2},\ldots,\xi_{n}$ are distinct.
Therefore, $\left\vert \left\{  \xi_{1},\xi_{2},\ldots,\xi_{n}\right\}
\right\vert =n$.
\end{verlong}

\begin{verlong}
We have $\left[  n\right]  =\left\{  1,2,\ldots,n\right\}  $ (by the
definition of $\left[  n\right]  $) and $\left[  n+1\right]  =\left\{
1,2,\ldots,n+1\right\}  $ (by the definition of $\left[  n+1\right]  $).
Hence, $\left[  n\right]  =\left\{  1,2,\ldots,n\right\}  \subseteq\left\{
1,2,\ldots,n+1\right\}  =\left[  n+1\right]  $.\medskip
\end{verlong}

We are in one of the following three cases:

\textit{Case 1:} We have $\left\{  \nu_{1},\nu_{2},\ldots,\nu_{n}\right\}
\not \subseteq \left[  n+1\right]  $.

\textit{Case 2:} We have $\left\vert \left\{  \nu_{1},\nu_{2},\ldots,\nu
_{n}\right\}  \right\vert \neq n$.

\textit{Case 3:} We have neither $\left\{  \nu_{1},\nu_{2},\ldots,\nu
_{n}\right\}  \not \subseteq \left[  n+1\right]  $ nor $\left\vert \left\{
\nu_{1},\nu_{2},\ldots,\nu_{n}\right\}  \right\vert \neq n$.\medskip

\begin{vershort}
Let us first consider Case 1. In this case, we have $\left\{  \nu_{1},\nu
_{2},\ldots,\nu_{n}\right\}  \not \subseteq \left[  n+1\right]  $. In other
words, there exists some $k\in\left[  n\right]  $ such that $\nu_{k}%
\notin\left[  n+1\right]  $. Consider this $k$. Then, $\left[  \nu
_{k}\trianglerighteq j\right]  =0$ for each $j\in\left[  n\right]  $ (since
$\nu_{k}\trianglerighteq j$ would entail $\nu_{k}\in\left\{  j,j+1\right\}
\subseteq\left[  n+1\right]  $, contradicting $\nu_{k}\notin\left[
n+1\right]  $). Hence, the matrix $\left(  \left[  \nu_{i}\trianglerighteq
j\right]  \right)  _{i,j\in\left[  n\right]  }$ has a zero row (namely, the
$k$-th row); therefore, the determinant of this matrix is $0$. In other
words,
\begin{equation}
\det\left(  \left(  \left[  \nu_{i}\trianglerighteq j\right]  \right)
_{i,j\in\left[  n\right]  }\right)  =0. \label{pf.lem.xi-det.c1.short.det=0}%
\end{equation}

On the other hand, the $n$-tuple $\nu$ contains the entry $\nu_{k}$
(obviously), whereas the $n$-tuple $\xi$ does not (since $\nu_{k}\notin\left[
n+1\right]  $, but all entries of $\xi$ belong to $\left[  n+1\right]  $).
Thus, the $n$-tuple $\nu$ is not a permutation of the $n$-tuple $\xi$. In
other words, there exists no $\sigma\in S_{n}$ satisfying $\nu=\xi\circ\sigma
$. Hence,%
\[
\sum_{\substack{\sigma\in S_{n};\\\nu=\xi\circ\sigma}}\left(  -1\right)
^{\sigma}=\left(  \text{empty sum}\right)  =0.
\]
Comparing this with (\ref{pf.lem.xi-det.c1.short.det=0}), we obtain
$\det\left(  \left(  \left[  \nu_{i}\trianglerighteq j\right]  \right)
_{i,j\in\left[  n\right]  }\right)  =\sum_{\substack{\sigma\in S_{n};\\\nu
=\xi\circ\sigma}}\left(  -1\right)  ^{\sigma}$. Thus, Lemma \ref{lem.xi-det}
is proved in Case 1. \medskip
\end{vershort}

\begin{verlong}
Let us first consider Case 1. In this case, we have $\left\{  \nu_{1},\nu
_{2},\ldots,\nu_{n}\right\}  \not \subseteq \left[  n+1\right]  $. In other
words, there exists some $w\in\left\{  \nu_{1},\nu_{2},\ldots,\nu_{n}\right\}
$ such that $w\notin\left[  n+1\right]  $. Consider this $w$. We have
$w\in\left\{  \nu_{1},\nu_{2},\ldots,\nu_{n}\right\}  $; in other words, there
exists some $k\in\left\{  1,2,\ldots,n\right\}  $ such that $w=\nu_{k}$.
Consider this $k$. Hence, $\nu_{k}=w\notin\left[  n+1\right]  $.

Now, it is easy to see that
\begin{equation}
\left[  \nu_{k}\trianglerighteq j\right]  =0\ \ \ \ \ \ \ \ \ \ \text{for each
}j\in\left[  n\right]  . \label{pf.lem.xi-det.c1.coleq}%
\end{equation}

[\textit{Proof of (\ref{pf.lem.xi-det.c1.coleq}):} Let $j\in\left[  n\right]
$. Then, $j\in\left[  n\right]  =\left\{  1,2,\ldots,n\right\}  $, so that
$j+1\in\left\{  2,3,\ldots,n+1\right\}  \subseteq\left\{  1,2,\ldots
,n+1\right\}  =\left[  n+1\right]  $. It is easy to see that we do not have
$\nu_{k}-j\in\left\{  0,1\right\}  $\ \ \ \ \footnote{\textit{Proof.} Assume
the contrary. Thus, $\nu_{k}-j\in\left\{  0,1\right\}  $. In other words,
$\nu_{k}-j=0$ or $\nu_{k}-j=1$. Hence, we are in one of the following two
cases:
\par
\textit{Case 1:} We have $\nu_{k}-j=0$.
\par
\textit{Case 2:} We have $\nu_{k}-j=1$.
\par
Let us consider Case 1 first. In this case, we have $\nu_{k}-j=0$. Thus,
$\nu_{k}=j\in\left[  n\right]  \subseteq\left[  n+1\right]  $. This
contradicts $\nu_{k}\notin\left[  n+1\right]  $. Thus, we have obtained a
contradiction in Case 1.
\par
Let us now consider Case 2. In this case, we have $\nu_{k}-j=1$. Thus,
$\nu_{k}=j+1\in\left[  n+1\right]  $. This contradicts $\nu_{k}\notin\left[
n+1\right]  $. Thus, we have obtained a contradiction in Case 2.
\par
We have now found a contradiction in each of the two Cases 1 and 2. Thus, we
always have a contradiction. This shows that our assumption was false, qed.}.
However, the statement \textquotedblleft$\nu_{k}\trianglerighteq
j$\textquotedblright\ is equivalent to \textquotedblleft$\nu_{k}-j\in\left\{
0,1\right\}  $\textquotedblright\ (by Definition \ref{def.cover01}). Hence, we
do not have $\nu_{k}\trianglerighteq j$ (since we do not have $\nu_{k}%
-j\in\left\{  0,1\right\}  $). In other words, we have $\left[  \nu
_{k}\trianglerighteq j\right]  =0$. This proves (\ref{pf.lem.xi-det.c1.coleq}%
).] \medskip

The equality (\ref{pf.lem.xi-det.c1.coleq}) shows that each entry in the
$k$-th row of the matrix $\left(  \left[  \nu_{i}\trianglerighteq j\right]
\right)  _{i,j\in\left[  n\right]  }$ equals $0$ (because the $j$-th entry in
the $k$-th row of this matrix is $\left[  \nu_{k}\trianglerighteq j\right]
$). In other words, the $k$-th row of the matrix $\left(  \left[  \nu
_{i}\trianglerighteq j\right]  \right)  _{i,j\in\left[  n\right]  }$ is zero.
Hence, this matrix $\left(  \left[  \nu_{i}\trianglerighteq j\right]  \right)
_{i,j\in\left[  n\right]  }$ has a zero row. However, it is well-known that if
a square matrix\footnote{In this proof, the word \textquotedblleft
matrix\textquotedblright\ always means a matrix with integer entries.} has a
zero row, then its determinant is $0$. Hence, the determinant of the matrix
$\left(  \left[  \nu_{i}\trianglerighteq j\right]  \right)  _{i,j\in\left[
n\right]  }$ is $0$ (since this matrix has a zero row). In other words,
\begin{equation}
\det\left(  \left(  \left[  \nu_{i}\trianglerighteq j\right]  \right)
_{i,j\in\left[  n\right]  }\right)  =0. \label{pf.lem.xi-det.c1.det=0}%
\end{equation}

On the other hand, there exists no $\sigma\in S_{n}$ satisfying $\nu=\xi
\circ\sigma$\ \ \ \ \footnote{\textit{Proof.} Let $\sigma\in S_{n}$ satisfy
$\nu=\xi\circ\sigma$. We shall obtain a contradiction.
\par
Indeed, we have $\sigma\in S_{n}$. In other words, $\sigma$ is a permutation
of $\left[  n\right]  $ (since $S_{n}$ was defined as the set of all
permutations of $\left[  n\right]  $). Hence,
\begin{align*}
\left\{  \xi_{\sigma\left(  1\right)  },\xi_{\sigma\left(  2\right)  }%
,\ldots,\xi_{\sigma\left(  n\right)  }\right\}   &  =\left\{  \xi_{1},\xi
_{2},\ldots,\xi_{n}\right\} \\
&  =\left\{  1,2,\ldots,n-p,n-p+2,n-p+3,\ldots,n+1\right\}
\ \ \ \ \ \ \ \ \ \ \left(  \text{by (\ref{pf.lem.xi-det.xi=2})}\right) \\
&  =\underbrace{\left\{  1,2,\ldots,n+1\right\}  }_{=\left[  n+1\right]
}\setminus\left\{  n-p+1\right\}  =\left[  n+1\right]  \setminus\left\{
n-p+1\right\}  \subseteq\left[  n+1\right]  .
\end{align*}
However, $\left(  \nu_{1},\nu_{2},\ldots,\nu_{n}\right)  =\nu=\xi\circ
\sigma=\left(  \xi_{\sigma\left(  1\right)  },\xi_{\sigma\left(  2\right)
},\ldots,\xi_{\sigma\left(  n\right)  }\right)  $ (by Definition
\ref{def.etapi}). Hence,
\[
\left\{  \nu_{1},\nu_{2},\ldots,\nu_{n}\right\}  =\left\{  \xi_{\sigma\left(
1\right)  },\xi_{\sigma\left(  2\right)  },\ldots,\xi_{\sigma\left(  n\right)
}\right\}  \subseteq\left[  n+1\right]  .
\]
But this contradicts $\left\{  \nu_{1},\nu_{2},\ldots,\nu_{n}\right\}
\not \subseteq \left[  n+1\right]  $.
\par
Forget that we fixed $\sigma$. We thus have found a contradiction for each
$\sigma\in S_{n}$ satisfying $\nu=\xi\circ\sigma$. Hence, there exists no such
$\sigma$. Qed.}. Hence, the sum $\sum_{\substack{\sigma\in S_{n};\\\nu
=\xi\circ\sigma}}\left(  -1\right)  ^{\sigma}$ is empty. Thus,%
\[
\sum_{\substack{\sigma\in S_{n};\\\nu=\xi\circ\sigma}}\left(  -1\right)
^{\sigma}=\left(  \text{empty sum}\right)  =0.
\]
Comparing this with (\ref{pf.lem.xi-det.c1.det=0}), we obtain $\det\left(
\left(  \left[  \nu_{i}\trianglerighteq j\right]  \right)  _{i,j\in\left[
n\right]  }\right)  =\sum_{\substack{\sigma\in S_{n};\\\nu=\xi\circ\sigma
}}\left(  -1\right)  ^{\sigma}$. Thus, Lemma \ref{lem.xi-det} is proved in
Case 1. \medskip
\end{verlong}

\begin{vershort}
Let us next consider Case 2. In this case, we have $\left\vert \left\{
\nu_{1},\nu_{2},\ldots,\nu_{n}\right\}  \right\vert \neq n$. Thus, two of the
numbers $\nu_{1},\nu_{2},\ldots,\nu_{n}$ are equal. Hence, the corresponding
two rows of the matrix $\left(  \left[  \nu_{i}\trianglerighteq j\right]
\right)  _{i,j\in\left[  n\right]  }$ are equal. Thus, this matrix $\left(
\left[  \nu_{i}\trianglerighteq j\right]  \right)  _{i,j\in\left[  n\right]
}$ has two equal rows, and therefore its determinant is $0$. In other words,
\begin{equation}
\det\left(  \left(  \left[  \nu_{i}\trianglerighteq j\right]  \right)
_{i,j\in\left[  n\right]  }\right)  =0. \label{pf.lem.xi-det.c2.short.det=0}%
\end{equation}
On the other hand, the $n$-tuple $\nu$ contains two equal entries (since two
of the numbers $\nu_{1},\nu_{2},\ldots,\nu_{n}$ are equal), whereas the
$n$-tuple $\xi$ does not (since $\xi_{1},\xi_{2},\ldots,\xi_{n}$ are
distinct). Thus, the $n$-tuple $\nu$ is not a permutation of the $n$-tuple
$\xi$. In other words, there exists no $\sigma\in S_{n}$ satisfying $\nu
=\xi\circ\sigma$. Hence,%
\[
\sum_{\substack{\sigma\in S_{n};\\\nu=\xi\circ\sigma}}\left(  -1\right)
^{\sigma}=\left(  \text{empty sum}\right)  =0.
\]
Comparing this with (\ref{pf.lem.xi-det.c2.short.det=0}), we obtain
$\det\left(  \left(  \left[  \nu_{i}\trianglerighteq j\right]  \right)
_{i,j\in\left[  n\right]  }\right)  =\sum_{\substack{\sigma\in S_{n};\\\nu
=\xi\circ\sigma}}\left(  -1\right)  ^{\sigma}$. Thus, Lemma \ref{lem.xi-det}
is proved in Case 2. \medskip
\end{vershort}

\begin{verlong}
Let us next consider Case 2. In this case, we have $\left\vert \left\{
\nu_{1},\nu_{2},\ldots,\nu_{n}\right\}  \right\vert \neq n$. Hence, the $n$
numbers $\nu_{1},\nu_{2},\ldots,\nu_{n}$ cannot be distinct (because if they
were distinct, then we would have $\left\vert \left\{  \nu_{1},\nu_{2}%
,\ldots,\nu_{n}\right\}  \right\vert =n$, which would contradict $\left\vert
\left\{  \nu_{1},\nu_{2},\ldots,\nu_{n}\right\}  \right\vert \neq n$). In
other words, two of these $n$ numbers are equal. In other words, there exist
two distinct elements $r$ and $s$ of $\left[  n\right]  $ such that $\nu
_{r}=\nu_{s}$. Consider these $r$ and $s$. Now, for each $j\in\left[
n\right]  $, we have%
\[
\left[  \nu_{r}\trianglerighteq j\right]  =\left[  \nu_{s}\trianglerighteq
j\right]  \ \ \ \ \ \ \ \ \ \ \left(  \text{since }\nu_{r}=\nu_{s}\right)  .
\]
This shows that each entry in the $r$-th row of the matrix $\left(  \left[
\nu_{i}\trianglerighteq j\right]  \right)  _{i,j\in\left[  n\right]  }$ equals
the corresponding entry in the $s$-th row of this matrix (because the $j$-th
entry in the $r$-th row of this matrix is $\left[  \nu_{r}\trianglerighteq
j\right]  $, whereas the corresponding entry in the $s$-th row is $\left[
\nu_{s}\trianglerighteq j\right]  $). In other words, the $r$-th row of the
matrix $\left(  \left[  \nu_{i}\trianglerighteq j\right]  \right)
_{i,j\in\left[  n\right]  }$ equals the $s$-th row of this matrix. Hence, this
matrix $\left(  \left[  \nu_{i}\trianglerighteq j\right]  \right)
_{i,j\in\left[  n\right]  }$ has two equal rows (since $r$ and $s$ are
distinct). However, it is well-known that if a square matrix\footnote{In this
proof, the word \textquotedblleft matrix\textquotedblright\ always means a
matrix with integer entries.} has two equal rows, then its determinant is $0$.
Hence, the determinant of the matrix $\left(  \left[  \nu_{i}\trianglerighteq
j\right]  \right)  _{i,j\in\left[  n\right]  }$ is $0$ (since this matrix has
two equal rows). In other words,
\begin{equation}
\det\left(  \left(  \left[  \nu_{i}\trianglerighteq j\right]  \right)
_{i,j\in\left[  n\right]  }\right)  =0. \label{pf.lem.xi-det.c2.det=0}%
\end{equation}

On the other hand, there exists no $\sigma\in S_{n}$ satisfying $\nu=\xi
\circ\sigma$\ \ \ \ \footnote{\textit{Proof.} Let $\sigma\in S_{n}$ satisfy
$\nu=\xi\circ\sigma$. We shall obtain a contradiction.
\par
Indeed, we have $\sigma\in S_{n}$. In other words, $\sigma$ is a permutation
of $\left[  n\right]  $ (since $S_{n}$ was defined as the set of all
permutations of $\left[  n\right]  $). Hence, $\left\{  \xi_{\sigma\left(
1\right)  },\xi_{\sigma\left(  2\right)  },\ldots,\xi_{\sigma\left(  n\right)
}\right\}  =\left\{  \xi_{1},\xi_{2},\ldots,\xi_{n}\right\}  $.
\par
However, $\left(  \nu_{1},\nu_{2},\ldots,\nu_{n}\right)  =\nu=\xi\circ
\sigma=\left(  \xi_{\sigma\left(  1\right)  },\xi_{\sigma\left(  2\right)
},\ldots,\xi_{\sigma\left(  n\right)  }\right)  $ (by Definition
\ref{def.etapi}). Hence,
\[
\left\{  \nu_{1},\nu_{2},\ldots,\nu_{n}\right\}  =\left\{  \xi_{\sigma\left(
1\right)  },\xi_{\sigma\left(  2\right)  },\ldots,\xi_{\sigma\left(  n\right)
}\right\}  =\left\{  \xi_{1},\xi_{2},\ldots,\xi_{n}\right\}  .
\]
Thus, $\left\vert \left\{  \nu_{1},\nu_{2},\ldots,\nu_{n}\right\}  \right\vert
=\left\vert \left\{  \xi_{1},\xi_{2},\ldots,\xi_{n}\right\}  \right\vert =n$.
This contradicts $\left\vert \left\{  \nu_{1},\nu_{2},\ldots,\nu_{n}\right\}
\right\vert \neq n$.
\par
Forget that we fixed $\sigma$. We thus have found a contradiction for each
$\sigma\in S_{n}$ satisfying $\nu=\xi\circ\sigma$. Hence, there exists no such
$\sigma$. Qed.}. Hence, the sum $\sum_{\substack{\sigma\in S_{n};\\\nu
=\xi\circ\sigma}}\left(  -1\right)  ^{\sigma}$ is empty. Thus,%
\[
\sum_{\substack{\sigma\in S_{n};\\\nu=\xi\circ\sigma}}\left(  -1\right)
^{\sigma}=\left(  \text{empty sum}\right)  =0.
\]
Comparing this with (\ref{pf.lem.xi-det.c2.det=0}), we obtain $\det\left(
\left(  \left[  \nu_{i}\trianglerighteq j\right]  \right)  _{i,j\in\left[
n\right]  }\right)  =\sum_{\substack{\sigma\in S_{n};\\\nu=\xi\circ\sigma
}}\left(  -1\right)  ^{\sigma}$. Thus, Lemma \ref{lem.xi-det} is proved in
Case 2. \medskip
\end{verlong}

Finally, let us consider Case 3. In this case, we have neither $\left\{
\nu_{1},\nu_{2},\ldots,\nu_{n}\right\}  \not \subseteq \left[  n+1\right]  $
nor $\left\vert \left\{  \nu_{1},\nu_{2},\ldots,\nu_{n}\right\}  \right\vert
\neq n$. In other words, we have $\left\{  \nu_{1},\nu_{2},\ldots,\nu
_{n}\right\}  \subseteq\left[  n+1\right]  $ and $\left\vert \left\{  \nu
_{1},\nu_{2},\ldots,\nu_{n}\right\}  \right\vert =n$.

\begin{vershort}
Note that%
\[
\left\{  \nu_{1},\nu_{2},\ldots,\nu_{n}\right\}  \subseteq\left[  n+1\right]
=\left\{  \xi_{1},\xi_{2},\ldots,\xi_{n}\right\}  \cup\left\{  n-p+1\right\}
\]
(since (\ref{pf.lem.xi-det.xi=2}) shows that $\xi_{1},\xi_{2},\ldots,\xi_{n}$
are precisely the numbers $1,2,\ldots,n+1$ except for $n-p+1$). Thus, Lemma
\ref{lem.nu-xi} (applied to $q=n-p+1$) yields that there exists some
permutation $\pi\in S_{n}$ satisfying $\nu=\xi\circ\pi$. Consider this $\pi$.
\end{vershort}

\begin{verlong}
Set $q:=n-p+1$. Thus, $q\in\mathbb{Z}$. From (\ref{pf.lem.xi-det.xi=2}), we
obtain%
\begin{align*}
\left\{  \xi_{1},\xi_{2},\ldots,\xi_{n}\right\}   &  =\left\{  1,2,\ldots
,n-p,n-p+2,n-p+3,\ldots,n+1\right\} \\
&  =\underbrace{\left\{  1,2,\ldots,n+1\right\}  }_{=\left[  n+1\right]
}\setminus\left\{  \underbrace{n-p+1}_{=q}\right\}  =\left[  n+1\right]
\setminus\left\{  q\right\}  .
\end{align*}
Hence,%
\begin{align}
\underbrace{\left\{  \xi_{1},\xi_{2},\ldots,\xi_{n}\right\}  }_{=\left[
n+1\right]  \setminus\left\{  q\right\}  }\cup\left\{  q\right\}   &  =\left(
\left[  n+1\right]  \setminus\left\{  q\right\}  \right)  \cup\left\{
q\right\} \nonumber\\
&  =\left[  n+1\right]  \cup\left\{  q\right\}  . \label{pf.lem.xi-det.c3.1}%
\end{align}
Now,
\[
\left\{  \nu_{1},\nu_{2},\ldots,\nu_{n}\right\}  \subseteq\left[  n+1\right]
\subseteq\left[  n+1\right]  \cup\left\{  q\right\}  =\left\{  \xi_{1},\xi
_{2},\ldots,\xi_{n}\right\}  \cup\left\{  q\right\}
\]
(by (\ref{pf.lem.xi-det.c3.1})). Thus, Lemma \ref{lem.nu-xi} yields that there
exists some permutation $\pi\in S_{n}$ satisfying $\nu=\xi\circ\pi$. Consider
this $\pi$.
\end{verlong}

\begin{vershort}
Thus, $\pi$ is a permutation $\sigma\in S_{n}$ satisfying $\nu=\xi\circ\sigma
$. Furthermore, it is easy to see that $\pi$ is the \textbf{only} such
permutation $\sigma$ (because the $n$ numbers $\xi_{1},\xi_{2},\ldots,\xi_{n}$
are distinct)\footnote{Here is the argument in more detail: Recall that the
$n$ numbers $\xi_{1},\xi_{2},\ldots,\xi_{n}$ are distinct. Therefore, if
$\sigma\in S_{n}$ is a permutation distinct from $\pi$, then Proposition
\ref{prop.etapi.dist} (applied to $\eta=\xi$) shows that $\xi\circ\sigma
\neq\xi\circ\pi=\nu$, so that $\nu\neq\xi\circ\sigma$. Hence, the only
permutation $\sigma\in S_{n}$ satisfying $\nu=\xi\circ\sigma$ is $\pi$.}.
Hence, the sum $\sum_{\substack{\sigma\in S_{n};\\\nu=\xi\circ\sigma}}\left(
-1\right)  ^{\sigma}$ has only one addend, namely the addend for $\sigma=\pi$.
Thus,
\begin{equation}
\sum_{\substack{\sigma\in S_{n};\\\nu=\xi\circ\sigma}}\left(  -1\right)
^{\sigma}=\left(  -1\right)  ^{\pi}. \label{pf.lem.xi-det.c3.short.sum=}%
\end{equation}

\end{vershort}

\begin{verlong}
It is easy to see that $\pi$ is the \textbf{only} permutation $\sigma\in
S_{n}$ satisfying $\nu=\xi\circ\sigma$\ \ \ \ \footnote{\textit{Proof.} It is
clear that $\pi$ is a permutation $\sigma\in S_{n}$ satisfying $\nu=\xi
\circ\sigma$ (because $\pi\in S_{n}$ is a permutation satisfying $\nu=\xi
\circ\pi$). It thus remains to show that $\pi$ is the \textbf{only} such
permutation. In other words, it remains to show that if $\sigma\in S_{n}$ is a
permutation satisfying $\nu=\xi\circ\sigma$, then $\sigma=\pi$. So let us show
this.
\par
Let $\sigma\in S_{n}$ be a permutation satisfying $\nu=\xi\circ\sigma$. Then,
$\xi\circ\sigma=\nu=\xi\circ\pi$.
\par
Recall that the $n$ numbers $\xi_{1},\xi_{2},\ldots,\xi_{n}$ are distinct.
Thus, Proposition \ref{prop.etapi.dist} (applied to $\eta=\xi$) shows that if
the permutations $\sigma$ and $\pi$ were distinct, then we would have
$\xi\circ\sigma\neq\xi\circ\pi$; but this would contradict $\xi\circ\sigma
=\xi\circ\pi$. Hence, $\sigma$ and $\pi$ cannot be distinct. In other words,
we must have $\sigma=\pi$.
\par
Forget that we fixed $\sigma$. We thus have shown that if $\sigma\in S_{n}$ is
a permutation satisfying $\nu=\xi\circ\sigma$, then $\sigma=\pi$. In other
words, $\pi$ is the \textbf{only} permutation $\sigma\in S_{n}$ satisfying
$\nu=\xi\circ\sigma$ (since we already know that $\pi$ is such a permutation).
Qed.}. Hence, the sum $\sum_{\substack{\sigma\in S_{n};\\\nu=\xi\circ\sigma
}}\left(  -1\right)  ^{\sigma}$ has only one addend, namely the addend for
$\sigma=\pi$. Thus,
\begin{equation}
\sum_{\substack{\sigma\in S_{n};\\\nu=\xi\circ\sigma}}\left(  -1\right)
^{\sigma}=\left(  -1\right)  ^{\pi}. \label{pf.lem.xi-det.c3.sum=}%
\end{equation}

\end{verlong}

\begin{vershort}
From $\nu=\xi\circ\pi$, we easily obtain%
\begin{equation}
\det\left(  \left(  \left[  \nu_{i}\trianglerighteq j\right]  \right)
_{i,j\in\left[  n\right]  }\right)  =\left(  -1\right)  ^{\pi}\cdot\det\left(
\left(  \left[  \xi_{i}\trianglerighteq j\right]  \right)  _{i,j\in\left[
n\right]  }\right)  . \label{pf.lem.xi-det.short.c3.det=1}%
\end{equation}
(Indeed, this can be proved just as we showed (\ref{pf.lem.nu-det.c2.det=1}),
except that $\eta$ and the $\geq$ sign are replaced by $\xi$ and the
$\trianglerighteq$ sign.)
\end{vershort}

\begin{verlong}
We have $\nu=\xi\circ\pi=\left(  \xi_{\pi\left(  1\right)  },\xi_{\pi\left(
2\right)  },\ldots,\xi_{\pi\left(  n\right)  }\right)  $ (by Definition
\ref{def.etapi}). Thus, for each $i\in\left[  n\right]  $, we have
\begin{equation}
\nu_{i}=\left(  \xi_{\pi\left(  1\right)  },\xi_{\pi\left(  2\right)  }%
,\ldots,\xi_{\pi\left(  n\right)  }\right)  _{i}=\xi_{\pi\left(  i\right)  }.
\label{pf.lem.xi-det.c3.nui=}%
\end{equation}

On the other hand, it is well-known (see, e.g., \cite[Corollary 6.4.15]{21s}
or \cite[Lemma 6.17 \textbf{(a)}]{detnotes}) that when the rows of a matrix
are permuted, then the determinant of this matrix gets multiplied by $\left(
-1\right)  ^{\tau}$, where $\tau$ is the permutation used to permute the rows.
In other words: If $\left(  a_{i,j}\right)  _{i,j\in\left[  n\right]  }$ is a
square matrix (with integer entries), and if $\tau\in S_{n}$ is a permutation,
then%
\[
\det\left(  \left(  a_{\tau\left(  i\right)  ,j}\right)  _{i,j\in\left[
n\right]  }\right)  =\left(  -1\right)  ^{\tau}\cdot\det\left(  \left(
a_{i,j}\right)  _{i,j\in\left[  n\right]  }\right)  .
\]
Applying this to $a_{i,j}=\left[  \xi_{i}\trianglerighteq j\right]  $ and
$\tau=\pi$, we obtain%
\[
\det\left(  \left(  \left[  \xi_{\pi\left(  i\right)  }\trianglerighteq
j\right]  \right)  _{i,j\in\left[  n\right]  }\right)  =\left(  -1\right)
^{\pi}\cdot\det\left(  \left(  \left[  \xi_{i}\trianglerighteq j\right]
\right)  _{i,j\in\left[  n\right]  }\right)  .
\]
In view of (\ref{pf.lem.xi-det.c3.nui=}), we can rewrite this as%
\begin{equation}
\det\left(  \left(  \left[  \nu_{i}\trianglerighteq j\right]  \right)
_{i,j\in\left[  n\right]  }\right)  =\left(  -1\right)  ^{\pi}\cdot\det\left(
\left(  \left[  \xi_{i}\trianglerighteq j\right]  \right)  _{i,j\in\left[
n\right]  }\right)  . \label{pf.lem.xi-det.c3.det=1}%
\end{equation}

\end{verlong}

\begin{vershort}
However, the definition of the determinant of a matrix yields%
\begin{align*}
\det\left(  \left(  \left[  \xi_{i}\trianglerighteq j\right]  \right)
_{i,j\in\left[  n\right]  }\right)   &  =\sum_{\sigma\in S_{n}}\left(
-1\right)  ^{\sigma}\prod_{i=1}^{n}\left[  \xi_{i}\trianglerighteq
\sigma\left(  i\right)  \right] \\
&  =\underbrace{\left(  -1\right)  ^{\operatorname*{id}}}_{=1}\prod_{i=1}%
^{n}\left[  \xi_{i}\trianglerighteq\underbrace{\operatorname*{id}\left(
i\right)  }_{=i}\right]  +\sum_{\substack{\sigma\in S_{n};\\\sigma
\neq\operatorname*{id}}}\left(  -1\right)  ^{\sigma}\underbrace{\prod
_{i=1}^{n}\left[  \xi_{i}\trianglerighteq\sigma\left(  i\right)  \right]
}_{\substack{=0\\\text{(by Lemma \ref{lem.xi-unique})}}}\\
&  \ \ \ \ \ \ \ \ \ \ \ \ \ \ \ \ \ \ \ \ \left(
\begin{array}
[c]{c}%
\text{here, we have split off the addend}\\
\text{for }\sigma=\operatorname*{id}\text{ from the sum (since }%
\operatorname*{id}\in S_{n}\text{)}%
\end{array}
\right) \\
&  =\prod_{i=1}^{n}\underbrace{\left[  \xi_{i}\trianglerighteq i\right]
}_{\substack{=1\\\text{(this follows easily}\\\text{from
(\ref{pf.lem.xi-det.xii=}) and (\ref{pf.lem.xi-det.xin=}))}}}+\underbrace{\sum
_{\substack{\sigma\in S_{n};\\\sigma\neq\operatorname*{id}}}\left(  -1\right)
^{\sigma}0}_{=0}=\prod_{i=1}^{n}1=1.
\end{align*}
Thus, (\ref{pf.lem.xi-det.short.c3.det=1}) becomes%
\[
\det\left(  \left(  \left[  \nu_{i}\trianglerighteq j\right]  \right)
_{i,j\in\left[  n\right]  }\right)  =\left(  -1\right)  ^{\pi}\cdot
\underbrace{\det\left(  \left(  \left[  \xi_{i}\trianglerighteq j\right]
\right)  _{i,j\in\left[  n\right]  }\right)  }_{=1}=\left(  -1\right)  ^{\pi
}=\sum_{\substack{\sigma\in S_{n};\\\nu=\xi\circ\sigma}}\left(  -1\right)
^{\sigma}%
\]
(by (\ref{pf.lem.xi-det.c3.short.sum=})). Thus, Lemma \ref{lem.xi-det} is
proved in Case 3. \medskip
\end{vershort}

\begin{verlong}
However, we have
\begin{equation}
\left[  \xi_{i}\trianglerighteq i\right]  =1\ \ \ \ \ \ \ \ \ \ \text{for each
}i\in\left\{  1,2,\ldots,n\right\}  \label{pf.lem.xi-det.c3.diagelt}%
\end{equation}
\footnote{\textit{Proof of (\ref{pf.lem.xi-det.c3.diagelt}):} Let
$i\in\left\{  1,2,\ldots,n\right\}  $. We must prove that $\left[  \xi
_{i}\trianglerighteq i\right]  =1$. We are in one of the following two cases:
\par
\textit{Case 1:} We have $i\leq n-p$.
\par
\textit{Case 2:} We have $i>n-p$.
\par
Let us first consider Case 1. In this case, we have $i\leq n-p$. Combining
this with $i\geq1$ (which follows from $i\in\left\{  1,2,\ldots,n\right\}  $),
we obtain $i\in\left\{  1,2,\ldots,n-p\right\}  $. Thus,
(\ref{pf.lem.xi-det.xii=}) yields $\xi_{i}=i$. Hence, $\xi_{i}-i=0\in\left\{
0,1\right\}  $. In other words, $\xi_{i}\trianglerighteq i$ (because
Definition \ref{def.cover01} shows that the statement \textquotedblleft%
$\xi_{i}\trianglerighteq i$\textquotedblright\ is equivalent to
\textquotedblleft$\xi_{i}-i\in\left\{  0,1\right\}  $\textquotedblright).
Therefore, $\left[  \xi_{i}\trianglerighteq i\right]  =1$. Thus,
(\ref{pf.lem.xi-det.c3.diagelt}) is proved in Case 1.
\par
Let us next consider Case 2. In this case, we have $i>n-p$. Hence, $i\geq
n-p+1$ (since $i$ and $n-p$ are integers). Combining this with $i\leq n$
(which follows from $i\in\left\{  1,2,\ldots,n\right\}  $), we obtain
$i\in\left\{  n-p+1,n-p+2,\ldots,n\right\}  $. Thus, (\ref{pf.lem.xi-det.xin=}%
) yields $\xi_{i}=i+1$. Hence, $\xi_{i}-i=1\in\left\{  0,1\right\}  $. In
other words, $\xi_{i}\trianglerighteq i$ (because Definition \ref{def.cover01}
shows that the statement \textquotedblleft$\xi_{i}\trianglerighteq
i$\textquotedblright\ is equivalent to \textquotedblleft$\xi_{i}-i\in\left\{
0,1\right\}  $\textquotedblright). Therefore, $\left[  \xi_{i}\trianglerighteq
i\right]  =1$. Thus, (\ref{pf.lem.xi-det.c3.diagelt}) is proved in Case 2.
\par
We have now proved (\ref{pf.lem.xi-det.c3.diagelt}) in both Cases 1 and 2.
Thus, the proof of (\ref{pf.lem.xi-det.c3.diagelt}) is complete.}.

Now, consider the identity permutation $\operatorname*{id}:\left[  n\right]
\rightarrow\left[  n\right]  $ of the set $\left[  n\right]  $. The definition
of the determinant of a matrix yields%
\begin{align*}
\det\left(  \left(  \left[  \xi_{i}\trianglerighteq j\right]  \right)
_{i,j\in\left[  n\right]  }\right)   &  =\sum_{\sigma\in S_{n}}\left(
-1\right)  ^{\sigma}\prod_{i=1}^{n}\left[  \xi_{i}\trianglerighteq
\sigma\left(  i\right)  \right] \\
&  =\underbrace{\left(  -1\right)  ^{\operatorname*{id}}}_{=1}\prod_{i=1}%
^{n}\left[  \xi_{i}\trianglerighteq\underbrace{\operatorname*{id}\left(
i\right)  }_{=i}\right]  +\sum_{\substack{\sigma\in S_{n};\\\sigma
\neq\operatorname*{id}}}\left(  -1\right)  ^{\sigma}\underbrace{\prod
_{i=1}^{n}\left[  \xi_{i}\trianglerighteq\sigma\left(  i\right)  \right]
}_{\substack{=0\\\text{(by Lemma \ref{lem.xi-unique})}}}\\
&  \ \ \ \ \ \ \ \ \ \ \ \ \ \ \ \ \ \ \ \ \left(
\begin{array}
[c]{c}%
\text{here, we have split off the addend}\\
\text{for }\sigma=\operatorname*{id}\text{ from the sum (since }%
\operatorname*{id}\in S_{n}\text{)}%
\end{array}
\right) \\
&  =\prod_{i=1}^{n}\underbrace{\left[  \xi_{i}\trianglerighteq i\right]
}_{\substack{=1\\\text{(by (\ref{pf.lem.xi-det.c3.diagelt}))}}%
}+\underbrace{\sum_{\substack{\sigma\in S_{n};\\\sigma\neq\operatorname*{id}%
}}\left(  -1\right)  ^{\sigma}0}_{=0}=\prod_{i=1}^{n}1=1.
\end{align*}
Thus, (\ref{pf.lem.xi-det.c3.det=1}) becomes%
\[
\det\left(  \left(  \left[  \nu_{i}\trianglerighteq j\right]  \right)
_{i,j\in\left[  n\right]  }\right)  =\left(  -1\right)  ^{\pi}\cdot
\underbrace{\det\left(  \left(  \left[  \xi_{i}\trianglerighteq j\right]
\right)  _{i,j\in\left[  n\right]  }\right)  }_{=1}=\left(  -1\right)  ^{\pi
}=\sum_{\substack{\sigma\in S_{n};\\\nu=\xi\circ\sigma}}\left(  -1\right)
^{\sigma}%
\]
(by (\ref{pf.lem.xi-det.c3.sum=})). Thus, Lemma \ref{lem.xi-det} is proved in
Case 3. \medskip
\end{verlong}

We have now proved Lemma \ref{lem.xi-det} in all three Cases 1, 2 and 3. This
completes the proof of Lemma \ref{lem.xi-det}.
\end{proof}

Our next lemma is an analogue to Lemma \ref{lem.nu-det2}:

\begin{lemma}
\label{lem.xi-det2}Let $n\in\mathbb{N}$. For any permutation $\sigma\in S_{n}%
$, we let $\overline{\sigma}$ denote the $n$-tuple $\left(  \sigma\left(
1\right)  ,\sigma\left(  2\right)  ,\ldots,\sigma\left(  n\right)  \right)
\in\mathbb{Z}^{n}$.

Let $\nu\in\mathbb{Z}^{n}$ be an $n$-tuple. Then,%
\[
\det\left(  \left(  \left[  \nu_{i}\trianglerighteq j\right]  \right)
_{i,j\in\left[  n\right]  }\right)  =\sum_{\substack{\sigma\in S_{n}%
;\\\nu-\overline{\sigma}\in\left\{  0,1\right\}  ^{n}}}\left(  -1\right)
^{\sigma}.
\]

\end{lemma}

\begin{vershort}
\begin{proof}
[Proof of Lemma \ref{lem.xi-det2}.]This proof is similar to the above proof of
Lemma \ref{lem.nu-det2}; we leave the necessary changes to the reader.
\end{proof}
\end{vershort}

\begin{verlong}
The following proof of Lemma \ref{lem.xi-det2} is strongly similar to the
above proof of Lemma \ref{lem.nu-det2}; we nevertheless provide it here in
full for the sake of completeness.

\begin{proof}
[Proof of Lemma \ref{lem.xi-det2}.]We first observe that each $\sigma\in
S_{n}$ satisfies%
\begin{equation}
\left[  \nu-\overline{\sigma}\in\left\{  0,1\right\}  ^{n}\right]
=\prod_{i=1}^{n}\left[  \nu_{i}\trianglerighteq\sigma\left(  i\right)
\right]  . \label{pf.lem.xi-det2.ives}%
\end{equation}

[\textit{Proof of (\ref{pf.lem.xi-det2.ives}):} Let $\sigma\in S_{n}$ be
arbitrary. We must prove the equality (\ref{pf.lem.xi-det2.ives}).

The definition of $\overline{\sigma}$ yields $\overline{\sigma}=\left(
\sigma\left(  1\right)  ,\sigma\left(  2\right)  ,\ldots,\sigma\left(
n\right)  \right)  $. Hence, for each $i\in\left[  n\right]  $, we have%
\begin{equation}
\overline{\sigma}_{i}=\sigma\left(  i\right)  .
\label{pf.pf.lem.xi-det2.ives.pf.sigbar}%
\end{equation}

We are in one of the following two cases:

\textit{Case 1:} We have $\nu-\overline{\sigma}\in\left\{  0,1\right\}  ^{n}$.

\textit{Case 2:} We have $\nu-\overline{\sigma}\notin\left\{  0,1\right\}
^{n}$. \medskip

Let us first consider Case 1. In this case, we have $\nu-\overline{\sigma}%
\in\left\{  0,1\right\}  ^{n}$. Now, let $i\in\left[  n\right]  $. Then,
$\left(  \nu-\overline{\sigma}\right)  _{i}\in\left\{  0,1\right\}  $ (since
$\nu-\overline{\sigma}\in\left\{  0,1\right\}  ^{n}$). In view of
\[
\left(  \nu-\overline{\sigma}\right)  _{i}=\nu_{i}-\underbrace{\overline
{\sigma}_{i}}_{\substack{=\sigma\left(  i\right)  \\\text{(by
(\ref{pf.pf.lem.xi-det2.ives.pf.sigbar}))}}}=\nu_{i}-\sigma\left(  i\right)
,
\]
we can rewrite this as $\nu_{i}-\sigma\left(  i\right)  \in\left\{
0,1\right\}  $. In other words, $\nu_{i}\trianglerighteq\sigma\left(
i\right)  $ (since Definition \ref{def.cover01} says that the statement
\textquotedblleft$\nu_{i}\trianglerighteq\sigma\left(  i\right)
$\textquotedblright\ is equivalent to \textquotedblleft$\nu_{i}-\sigma\left(
i\right)  \in\left\{  0,1\right\}  $\textquotedblright). Thus, $\left[
\nu_{i}\trianglerighteq\sigma\left(  i\right)  \right]  =1$.

Forget that we fixed $i$. Thus, we have shown that $\left[  \nu_{i}%
\trianglerighteq\sigma\left(  i\right)  \right]  =1$ for each $i\in\left[
n\right]  $. Multiplying these equalities over all $i\in\left[  n\right]  $,
we obtain $\prod_{i\in\left[  n\right]  }\left[  \nu_{i}\trianglerighteq
\sigma\left(  i\right)  \right]  =\prod_{i\in\left[  n\right]  }1=1$. On the
other hand, $\left[  \nu-\overline{\sigma}\in\left\{  0,1\right\}
^{n}\right]  =1$ (since $\nu-\overline{\sigma}\in\left\{  0,1\right\}  ^{n}$).
Comparing these two equalities, we obtain%
\[
\left[  \nu-\overline{\sigma}\in\left\{  0,1\right\}  ^{n}\right]
=\underbrace{\prod_{i\in\left[  n\right]  }}_{=\prod_{i=1}^{n}}\left[  \nu
_{i}\trianglerighteq\sigma\left(  i\right)  \right]  =\prod_{i=1}^{n}\left[
\nu_{i}\trianglerighteq\sigma\left(  i\right)  \right]  .
\]
Thus, (\ref{pf.lem.xi-det2.ives}) is proved in Case 1. \medskip

Let us now consider Case 2. In this case, we have $\nu-\overline{\sigma}%
\notin\left\{  0,1\right\}  ^{n}$. Hence, there exists some $j\in\left[
n\right]  $ such that $\left(  \nu-\overline{\sigma}\right)  _{j}%
\notin\left\{  0,1\right\}  $. Consider this $j$. We have $\overline{\sigma
}_{j}=\sigma\left(  j\right)  $ (by (\ref{pf.pf.lem.xi-det2.ives.pf.sigbar}),
applied to $i=j$). Now,
\[
\left(  \nu-\overline{\sigma}\right)  _{j}=\nu_{j}-\underbrace{\overline
{\sigma}_{j}}_{=\sigma\left(  j\right)  }=\nu_{j}-\sigma\left(  j\right)  ,
\]
so that $\nu_{j}-\sigma\left(  j\right)  =\left(  \nu-\overline{\sigma
}\right)  _{j}\notin\left\{  0,1\right\}  $. In other words, we don't have
$\nu_{j}-\sigma\left(  j\right)  \in\left\{  0,1\right\}  $. In other words,
we don't have $\nu_{j}\trianglerighteq\sigma\left(  j\right)  $ (since
Definition \ref{def.cover01} says that the statement \textquotedblleft$\nu
_{j}\trianglerighteq\sigma\left(  j\right)  $\textquotedblright\ is equivalent
to \textquotedblleft$\nu_{j}-\sigma\left(  j\right)  \in\left\{  0,1\right\}
$\textquotedblright). Thus, $\left[  \nu_{j}\trianglerighteq\sigma\left(
j\right)  \right]  =0$.

Now, consider the product $\prod_{i=1}^{n}\left[  \nu_{i}\trianglerighteq
\sigma\left(  i\right)  \right]  $. The term $\left[  \nu_{j}\trianglerighteq
\sigma\left(  j\right)  \right]  $ is one of the factors of this product
(namely, the factor for $i=j$). Since this term is $0$ (because we have just
shown that $\left[  \nu_{j}\trianglerighteq\sigma\left(  j\right)  \right]
=0$), we thus conclude that one of the factors of the product $\prod_{i=1}%
^{n}\left[  \nu_{i}\trianglerighteq\sigma\left(  i\right)  \right]  $ is $0$.
Therefore, the entire product is $0$. In other words,
\begin{equation}
\prod_{i=1}^{n}\left[  \nu_{i}\trianglerighteq\sigma\left(  i\right)  \right]
=0. \label{pf.pf.lem.xi-det2.ives.pf.c2.prod0}%
\end{equation}
On the other hand, we do not have $\nu-\overline{\sigma}\in\left\{
0,1\right\}  ^{n}$ (since $\nu-\overline{\sigma}\notin\left\{  0,1\right\}
^{n}$). Hence, $\left[  \nu-\overline{\sigma}\in\left\{  0,1\right\}
^{n}\right]  =0$. Comparing this with
(\ref{pf.pf.lem.xi-det2.ives.pf.c2.prod0}), we obtain%
\[
\left[  \nu-\overline{\sigma}\in\left\{  0,1\right\}  ^{n}\right]
=\prod_{i=1}^{n}\left[  \nu_{i}\trianglerighteq\sigma\left(  i\right)
\right]  .
\]
Thus, (\ref{pf.lem.xi-det2.ives}) is proved in Case 2. \medskip

We have now proved (\ref{pf.lem.xi-det2.ives}) in both Cases 1 and 2. Thus,
the proof of (\ref{pf.lem.xi-det2.ives}) is complete.] \medskip

Now, the definition of the determinant of a matrix yields%
\begin{align*}
&  \det\left(  \left(  \left[  \nu_{i}\trianglerighteq j\right]  \right)
_{i,j\in\left[  n\right]  }\right) \\
&  =\sum_{\sigma\in S_{n}}\left(  -1\right)  ^{\sigma}\underbrace{\prod
_{i=1}^{n}\left[  \nu_{i}\trianglerighteq\sigma\left(  i\right)  \right]
}_{\substack{=\left[  \nu-\overline{\sigma}\in\left\{  0,1\right\}
^{n}\right]  \\\text{(by (\ref{pf.lem.xi-det2.ives}))}}}=\sum_{\sigma\in
S_{n}}\left(  -1\right)  ^{\sigma}\left[  \nu-\overline{\sigma}\in\left\{
0,1\right\}  ^{n}\right] \\
&  =\sum_{\substack{\sigma\in S_{n};\\\nu-\overline{\sigma}\in\left\{
0,1\right\}  ^{n}}}\left(  -1\right)  ^{\sigma}\underbrace{\left[
\nu-\overline{\sigma}\in\left\{  0,1\right\}  ^{n}\right]  }%
_{\substack{=1\\\text{(since }\nu-\overline{\sigma}\in\left\{  0,1\right\}
^{n}\text{)}}}+\sum_{\substack{\sigma\in S_{n};\\\nu-\overline{\sigma}%
\notin\left\{  0,1\right\}  ^{n}}}\left(  -1\right)  ^{\sigma}%
\underbrace{\left[  \nu-\overline{\sigma}\in\left\{  0,1\right\}  ^{n}\right]
}_{\substack{=0\\\text{(since we don't}\\\text{have }\nu-\overline{\sigma}%
\in\left\{  0,1\right\}  ^{n}\\\text{(because }\nu-\overline{\sigma}%
\notin\left\{  0,1\right\}  ^{n}\text{))}}}\\
&  \ \ \ \ \ \ \ \ \ \ \ \ \ \ \ \ \ \ \ \ \left(
\begin{array}
[c]{c}%
\text{since each }\sigma\in S_{n}\text{ satisfies either }\nu-\overline
{\sigma}\in\left\{  0,1\right\}  ^{n}\\
\text{or }\nu-\overline{\sigma}\notin\left\{  0,1\right\}  ^{n}\text{ (but not
both at the same time)}%
\end{array}
\right) \\
&  =\sum_{\substack{\sigma\in S_{n};\\\nu-\overline{\sigma}\in\left\{
0,1\right\}  ^{n}}}\left(  -1\right)  ^{\sigma}+\underbrace{\sum
_{\substack{\sigma\in S_{n};\\\nu-\overline{\sigma}\notin\left\{  0,1\right\}
^{n}}}\left(  -1\right)  ^{\sigma}0}_{=0}=\sum_{\substack{\sigma\in
S_{n};\\\nu-\overline{\sigma}\in\left\{  0,1\right\}  ^{n}}}\left(  -1\right)
^{\sigma}.
\end{align*}
This proves Lemma \ref{lem.xi-det2}.
\end{proof}
\end{verlong}

We can now prove Theorem \ref{thm.pre-pieri2}:

\begin{vershort}
\begin{proof}
[Proof of Theorem \ref{thm.pre-pieri2}.]The first step is to check that
$\left\vert \xi\right\vert =\left(  1+2+\cdots+n\right)  +p$. The rest of the
proof is a straightforward modification of the above proof of Theorem
\ref{thm.pre-pieri} (Lemmas \ref{lem.xi-det} and \ref{lem.xi-det2} have to be
applied instead of Lemmas \ref{lem.nu-det} and \ref{lem.nu-det2}). We leave
the details to the reader.
\end{proof}
\end{vershort}

\begin{verlong}
\begin{proof}
[Proof of Theorem \ref{thm.pre-pieri2}.]The definition of $\xi$ yields%
\begin{equation}
\xi=\left(  1,2,\ldots,n\right)  +\left(  \underbrace{0,0,\ldots,0}_{n-p\text{
zeroes}},\underbrace{1,1,\ldots,1}_{p\text{ ones}}\right)  .
\label{pf.thm.pre-pieri2.1}%
\end{equation}
Thus,%
\begin{align*}
\left\vert \xi\right\vert  &  =\left\vert \left(  1,2,\ldots,n\right)
+\left(  \underbrace{0,0,\ldots,0}_{n-p\text{ zeroes}},\underbrace{1,1,\ldots
,1}_{p\text{ ones}}\right)  \right\vert \\
&  =\left\vert \left(  1,2,\ldots,n\right)  \right\vert +\left\vert \left(
\underbrace{0,0,\ldots,0}_{n-p\text{ zeroes}},\underbrace{1,1,\ldots
,1}_{p\text{ ones}}\right)  \right\vert
\end{align*}
(by Lemma \ref{lem.additive}, applied to $\alpha=\left(  1,2,\ldots,n\right)
$ and $\beta=\left(  \underbrace{0,0,\ldots,0}_{n-p\text{ zeroes}%
},\underbrace{1,1,\ldots,1}_{p\text{ ones}}\right)  $). In view of%
\[
\left\vert \left(  1,2,\ldots,n\right)  \right\vert =1+2+\cdots
+n\ \ \ \ \ \ \ \ \ \ \left(  \text{by the definition of }\left\vert \left(
1,2,\ldots,n\right)  \right\vert \right)
\]
and%
\begin{align*}
&  \left\vert \left(  \underbrace{0,0,\ldots,0}_{n-p\text{ zeroes}%
},\underbrace{1,1,\ldots,1}_{p\text{ ones}}\right)  \right\vert \\
&  =\underbrace{0+0+\cdots+0}_{n-p\text{ zeroes}}+\underbrace{1+1+\cdots
+1}_{p\text{ ones}}\\
&  \ \ \ \ \ \ \ \ \ \ \ \ \ \ \ \ \ \ \ \ \left(  \text{by the definition of
}\left\vert \left(  \underbrace{0,0,\ldots,0}_{n-p\text{ zeroes}%
},\underbrace{1,1,\ldots,1}_{p\text{ ones}}\right)  \right\vert \right) \\
&  =\left(  n-p\right)  \cdot0+p\cdot1=p,
\end{align*}
this rewrites as%
\begin{equation}
\left\vert \xi\right\vert =\left(  1+2+\cdots+n\right)  +p.
\label{pf.thm.pre-pieri2.sum-xi}%
\end{equation}

For any permutation $\sigma\in S_{n}$, we let $\overline{\sigma}$ denote the
$n$-tuple $\left(  \sigma\left(  1\right)  ,\sigma\left(  2\right)
,\ldots,\sigma\left(  n\right)  \right)  \in\mathbb{Z}^{n}$. This $n$-tuple
$\overline{\sigma}$ satisfies $\left\vert \overline{\sigma}\right\vert
=1+2+\cdots+n$\ \ \ \ \footnote{\textit{Proof.} Let $\sigma\in S_{n}$. Thus,
$\sigma$ is a permutation of $\left[  n\right]  $ (since $S_{n}$ was defined
as the set of all permutations of $\left[  n\right]  $). In other words,
$\sigma$ is a bijection from $\left[  n\right]  $ to $\left[  n\right]  $.
However, the definition of $\overline{\sigma}$ yields $\overline{\sigma
}=\left(  \sigma\left(  1\right)  ,\sigma\left(  2\right)  ,\ldots
,\sigma\left(  n\right)  \right)  $. Thus,%
\begin{align*}
\left\vert \overline{\sigma}\right\vert  &  =\left\vert \left(  \sigma\left(
1\right)  ,\sigma\left(  2\right)  ,\ldots,\sigma\left(  n\right)  \right)
\right\vert =\sigma\left(  1\right)  +\sigma\left(  2\right)  +\cdots
+\sigma\left(  n\right) \\
&  \ \ \ \ \ \ \ \ \ \ \ \ \ \ \ \ \ \ \ \ \left(  \text{by the definition of
}\left\vert \left(  \sigma\left(  1\right)  ,\sigma\left(  2\right)
,\ldots,\sigma\left(  n\right)  \right)  \right\vert \right) \\
&  =\sum_{i\in\left[  n\right]  }\sigma\left(  i\right)  =\sum_{i\in\left[
n\right]  }i\ \ \ \ \ \ \ \ \ \ \left(
\begin{array}
[c]{c}%
\text{here, we have substituted }i\text{ for }\sigma\left(  i\right)  \text{
in the sum,}\\
\text{since the map }\sigma:\left[  n\right]  \rightarrow\left[  n\right]
\text{ is a bijection}%
\end{array}
\right) \\
&  =1+2+\cdots+n,
\end{align*}
qed.} and therefore
\begin{equation}
\underbrace{\left\vert \overline{\sigma}\right\vert }_{=1+2+\cdots
+n}+p=\left(  1+2+\cdots+n\right)  +p=\left\vert \xi\right\vert
\label{pf.thm.pre-pieri2.modsigbar1}%
\end{equation}
(by (\ref{pf.thm.pre-pieri2.sum-xi})). Moreover, for each $\sigma\in S_{n}$
and each $i\in\left[  n\right]  $, we have%
\begin{equation}
\overline{\sigma}_{i}=\sigma\left(  i\right)
\label{pf.thm.pre-pieri2.sigbari}%
\end{equation}
(since $\overline{\sigma}=\left(  \sigma\left(  1\right)  ,\sigma\left(
2\right)  ,\ldots,\sigma\left(  n\right)  \right)  $).

Now, we claim the following:

\begin{statement}
\textit{Claim 1:} Let $\beta\in\left\{  0,1\right\}  ^{n}$ and $\sigma\in
S_{n}$ be arbitrary. Then, the statement \textquotedblleft$\left\vert
\beta\right\vert =p$\textquotedblright\ is equivalent to \textquotedblleft%
$\left\vert \beta+\overline{\sigma}\right\vert =\left\vert \xi\right\vert
$\textquotedblright.
\end{statement}

[\textit{Proof of Claim 1:} We have $\beta+\overline{\sigma}=\overline{\sigma
}+\beta$ (since $\mathbb{Z}^{n}$ is an abelian group). Thus, $\left\vert
\beta+\overline{\sigma}\right\vert =\left\vert \overline{\sigma}%
+\beta\right\vert =\left\vert \overline{\sigma}\right\vert +\left\vert
\beta\right\vert $ (by Lemma \ref{lem.additive}, applied to $\alpha
=\overline{\sigma}$). Therefore,%
\begin{align}
\underbrace{\left\vert \beta+\overline{\sigma}\right\vert }_{=\left\vert
\overline{\sigma}\right\vert +\left\vert \beta\right\vert }%
-\underbrace{\left\vert \xi\right\vert }_{\substack{=\left\vert \overline
{\sigma}\right\vert +p\\\text{(by (\ref{pf.thm.pre-pieri2.modsigbar1}))}}}  &
=\left(  \left\vert \overline{\sigma}\right\vert +\left\vert \beta\right\vert
\right)  -\left(  \left\vert \overline{\sigma}\right\vert +p\right)
\nonumber\\
&  =\left\vert \beta\right\vert -p. \label{pf.thm.pre-pieri2.modsigbar2}%
\end{align}
Hence, we have the following chain of logical equivalences:%
\begin{align*}
\left(  \left\vert \beta\right\vert =p\right)  \  &  \Longleftrightarrow
\ \left(  \underbrace{\left\vert \beta\right\vert -p}_{\substack{=\left\vert
\beta+\overline{\sigma}\right\vert -\left\vert \xi\right\vert \\\text{(by
(\ref{pf.thm.pre-pieri2.modsigbar2}))}}}=0\right)  \ \ \Longleftrightarrow
\ \left(  \left\vert \beta+\overline{\sigma}\right\vert -\left\vert
\xi\right\vert =0\right) \\
&  \Longleftrightarrow\ \left(  \left\vert \beta+\overline{\sigma}\right\vert
=\left\vert \xi\right\vert \right)  .
\end{align*}
In other words, the statement \textquotedblleft$\left\vert \beta\right\vert
=p$\textquotedblright\ is equivalent to \textquotedblleft$\left\vert
\beta+\overline{\sigma}\right\vert =\left\vert \xi\right\vert $%
\textquotedblright. This proves Claim 1.] \medskip

Let us now set%
\begin{equation}
\prod_{i=1}^{n}a_{i}:=a_{1}a_{2}\cdots a_{n}
\label{pf.thm.pre-pieri2.bottom-up-prod}%
\end{equation}
for any $a_{1},a_{2},\ldots,a_{n}\in R$. (This is, of course, the usual
meaning of the notation $\prod_{i=1}^{n}a_{i}$ when the ring $R$ is
commutative; however, we are now extending it to the case of arbitrary $R$ by
using the formula (\ref{pf.thm.pre-pieri2.bottom-up-prod}).)

Thus, if $\left(  a_{i,j}\right)  _{i,j\in\left[  n\right]  }$ is any $n\times
n$-matrix over $R$, then the definition of $\operatorname*{rowdet}\left(
\left(  a_{i,j}\right)  _{i,j\in\left[  n\right]  }\right)  $ yields%
\begin{align}
\operatorname*{rowdet}\left(  \left(  a_{i,j}\right)  _{i,j\in\left[
n\right]  }\right)   &  =\sum_{\sigma\in S_{n}}\left(  -1\right)  ^{\sigma
}\underbrace{a_{1,\sigma\left(  1\right)  }a_{2,\sigma\left(  2\right)
}\cdots a_{n,\sigma\left(  n\right)  }}_{=\prod_{i=1}^{n}a_{i,\sigma\left(
i\right)  }}\nonumber\\
&  =\sum_{\sigma\in S_{n}}\left(  -1\right)  ^{\sigma}\prod_{i=1}%
^{n}a_{i,\sigma\left(  i\right)  }. \label{pf.thm.pre-pieri2.rowdetA=}%
\end{align}

For each $\beta\in\left\{  0,1\right\}  ^{n}$, we have%
\begin{align}
t_{\alpha+\beta}  &  =\operatorname*{rowdet}\left(  \left(  h_{\left(
\alpha+\beta\right)  _{i}+j,\ i}\right)  _{i,j\in\left[  n\right]  }\right)
\ \ \ \ \ \ \ \ \ \ \left(  \text{by the definition of }t_{\alpha+\beta
}\right) \nonumber\\
&  =\operatorname*{rowdet}\left(  \left(  h_{\alpha_{i}+\beta_{i}%
+j,\ i}\right)  _{i,j\in\left[  n\right]  }\right) \nonumber\\
&  \ \ \ \ \ \ \ \ \ \ \ \ \ \ \ \ \ \ \ \ \left(  \text{since }\left(
\alpha+\beta\right)  _{i}=\alpha_{i}+\beta_{i}\text{ for each }i\in\left[
n\right]  \right) \nonumber\\
&  =\sum_{\sigma\in S_{n}}\left(  -1\right)  ^{\sigma}\prod_{i=1}^{n}%
h_{\alpha_{i}+\beta_{i}+\sigma\left(  i\right)  ,\ i}%
\label{pf.thm.pre-pieri2.t1}\\
&  \ \ \ \ \ \ \ \ \ \ \ \ \ \ \ \ \ \ \ \ \left(  \text{by
(\ref{pf.thm.pre-pieri2.rowdetA=}), applied to }a_{i,j}=h_{\alpha_{i}%
+\beta_{i}+j,\ i}\right)  .\nonumber
\end{align}
However, for each $\beta\in\left\{  0,1\right\}  ^{n}$ and each $\sigma\in
S_{n}$, we have%
\[
\beta_{i}+\underbrace{\sigma\left(  i\right)  }_{\substack{=\overline{\sigma
}_{i}\\\text{(by (\ref{pf.thm.pre-pieri2.sigbari}))}}}=\beta_{i}%
+\overline{\sigma}_{i}=\left(  \beta+\overline{\sigma}\right)  _{i}%
\]
and therefore%
\begin{equation}
h_{\alpha_{i}+\beta_{i}+\sigma\left(  i\right)  ,\ i}=h_{\alpha_{i}+\left(
\beta+\overline{\sigma}\right)  _{i},\ i}. \label{pf.thm.pre-pieri2.h=h1}%
\end{equation}

Hence, for each $\beta\in\left\{  0,1\right\}  ^{n}$, we have the equality%
\begin{align*}
t_{\alpha+\beta}  &  =\sum_{\sigma\in S_{n}}\left(  -1\right)  ^{\sigma}%
\prod_{i=1}^{n}\underbrace{h_{\alpha_{i}+\beta_{i}+\sigma\left(  i\right)
,\ i}}_{\substack{=h_{\alpha_{i}+\left(  \beta+\overline{\sigma}\right)
_{i},\ i}\\\text{(by (\ref{pf.thm.pre-pieri2.h=h1}))}}%
}\ \ \ \ \ \ \ \ \ \ \left(  \text{by (\ref{pf.thm.pre-pieri2.t1})}\right) \\
&  =\sum_{\sigma\in S_{n}}\left(  -1\right)  ^{\sigma}\prod_{i=1}^{n}%
h_{\alpha_{i}+\left(  \beta+\overline{\sigma}\right)  _{i},\ i}.
\end{align*}
Summing these equalities over all $\beta\in\left\{  0,1\right\}  ^{n}$
satisfying $\left\vert \beta\right\vert =p$, we obtain%
\begin{align}
\sum_{\substack{\beta\in\left\{  0,1\right\}  ^{n};\\\left\vert \beta
\right\vert =p}}t_{\alpha+\beta}  &  =\underbrace{\sum_{\substack{\beta
\in\left\{  0,1\right\}  ^{n};\\\left\vert \beta\right\vert =p}}\ \ \sum
_{\sigma\in S_{n}}}_{=\sum_{\sigma\in S_{n}}\ \ \sum_{\substack{\beta
\in\left\{  0,1\right\}  ^{n};\\\left\vert \beta\right\vert =p}}}\left(
-1\right)  ^{\sigma}\prod_{i=1}^{n}h_{\alpha_{i}+\left(  \beta+\overline
{\sigma}\right)  _{i},\ i}\nonumber\\
&  =\sum_{\sigma\in S_{n}}\ \ \sum_{\substack{\beta\in\left\{  0,1\right\}
^{n};\\\left\vert \beta\right\vert =p}}\left(  -1\right)  ^{\sigma}\prod
_{i=1}^{n}h_{\alpha_{i}+\left(  \beta+\overline{\sigma}\right)  _{i}%
,\ i}\nonumber\\
&  =\sum_{\sigma\in S_{n}}\left(  -1\right)  ^{\sigma}\sum_{\substack{\beta
\in\left\{  0,1\right\}  ^{n};\\\left\vert \beta\right\vert =p}}\ \ \prod
_{i=1}^{n}h_{\alpha_{i}+\left(  \beta+\overline{\sigma}\right)  _{i},\ i}.
\label{pf.thm.pre-pieri2.sumt=1}%
\end{align}

Now, fix a permutation $\sigma\in S_{n}$. We shall rewrite the sum
$\sum_{\substack{\beta\in\left\{  0,1\right\}  ^{n};\\\left\vert
\beta\right\vert =p}}\ \ \prod_{i=1}^{n}h_{\alpha_{i}+\left(  \beta
+\overline{\sigma}\right)  _{i},\ i}$ in terms of $\beta+\overline{\sigma}$.
Indeed, Claim 1 shows that for any $n$-tuple $\beta\in\left\{  0,1\right\}
^{n}$, the statement \textquotedblleft$\left\vert \beta\right\vert
=p$\textquotedblright\ is equivalent to \textquotedblleft$\left\vert
\beta+\overline{\sigma}\right\vert =\left\vert \xi\right\vert $%
\textquotedblright. Hence, we have the following equality of summation signs:%
\begin{align*}
\sum_{\substack{\beta\in\left\{  0,1\right\}  ^{n};\\\left\vert \beta
\right\vert =p}}  &  =\sum_{\substack{\beta\in\left\{  0,1\right\}
^{n};\\\left\vert \beta+\overline{\sigma}\right\vert =\left\vert
\xi\right\vert }}=\sum_{\substack{\beta\in\mathbb{Z}^{n};\\\beta\in\left\{
0,1\right\}  ^{n};\\\left\vert \beta+\overline{\sigma}\right\vert =\left\vert
\xi\right\vert }}\ \ \ \ \ \ \ \ \ \ \left(  \text{since }\left\{
0,1\right\}  ^{n}\subseteq\mathbb{Z}^{n}\right) \\
&  =\sum_{\substack{\beta\in\mathbb{Z}^{n};\\\left(  \beta+\overline{\sigma
}\right)  -\overline{\sigma}\in\left\{  0,1\right\}  ^{n};\\\left\vert
\beta+\overline{\sigma}\right\vert =\left\vert \xi\right\vert }%
}\ \ \ \ \ \ \ \ \ \ \left(  \text{since }\beta=\left(  \beta+\overline
{\sigma}\right)  -\overline{\sigma}\right)  .
\end{align*}
Thus,%
\[
\underbrace{\sum_{\substack{\beta\in\left\{  0,1\right\}  ^{n};\\\left\vert
\beta\right\vert =p}}}_{\substack{=\sum_{\substack{\beta\in\mathbb{Z}%
^{n};\\\left(  \beta+\overline{\sigma}\right)  -\overline{\sigma}\in\left\{
0,1\right\}  ^{n};\\\left\vert \beta+\overline{\sigma}\right\vert =\left\vert
\xi\right\vert }}}}\ \ \prod_{i=1}^{n}h_{\alpha_{i}+\left(  \beta
+\overline{\sigma}\right)  _{i},\ i}=\sum_{\substack{\beta\in\mathbb{Z}%
^{n};\\\left(  \beta+\overline{\sigma}\right)  -\overline{\sigma}\in\left\{
0,1\right\}  ^{n};\\\left\vert \beta+\overline{\sigma}\right\vert =\left\vert
\xi\right\vert }}\ \ \prod_{i=1}^{n}h_{\alpha_{i}+\left(  \beta+\overline
{\sigma}\right)  _{i},\ i}.
\]

However, $\mathbb{Z}^{n}$ is a group (under addition). Hence, for any $\rho
\in\mathbb{Z}^{n}$, the map%
\begin{align*}
\mathbb{Z}^{n}  &  \rightarrow\mathbb{Z}^{n},\\
\beta &  \mapsto\beta+\rho
\end{align*}
is a bijection (whose inverse is the map $\mathbb{Z}^{n}\rightarrow
\mathbb{Z}^{n},\ \nu\mapsto\nu-\rho$). Applying this to $\rho=\overline
{\sigma}$, we conclude that the map
\begin{align*}
\mathbb{Z}^{n}  &  \rightarrow\mathbb{Z}^{n},\\
\beta &  \mapsto\beta+\overline{\sigma}%
\end{align*}
is a bijection (whose inverse is the map $\mathbb{Z}^{n}\rightarrow
\mathbb{Z}^{n},\ \nu\mapsto\nu-\overline{\sigma}$). Therefore, we can
substitute $\nu$ for $\beta+\overline{\sigma}$ in the sum $\sum
_{\substack{\beta\in\mathbb{Z}^{n};\\\left(  \beta+\overline{\sigma}\right)
-\overline{\sigma}\in\left\{  0,1\right\}  ^{n};\\\left\vert \beta
+\overline{\sigma}\right\vert =\left\vert \xi\right\vert }}\ \ \prod_{i=1}%
^{n}h_{\alpha_{i}+\left(  \beta+\overline{\sigma}\right)  _{i},\ i}$. We thus
obtain%
\[
\sum_{\substack{\beta\in\mathbb{Z}^{n};\\\left(  \beta+\overline{\sigma
}\right)  -\overline{\sigma}\in\left\{  0,1\right\}  ^{n};\\\left\vert
\beta+\overline{\sigma}\right\vert =\left\vert \xi\right\vert }}\ \ \prod
_{i=1}^{n}h_{\alpha_{i}+\left(  \beta+\overline{\sigma}\right)  _{i},\ i}%
=\sum_{\substack{\nu\in\mathbb{Z}^{n};\\\nu-\overline{\sigma}\in\left\{
0,1\right\}  ^{n};\\\left\vert \nu\right\vert =\left\vert \xi\right\vert
}}\ \ \prod_{i=1}^{n}h_{\alpha_{i}+\nu_{i},\ i}.
\]
Hence, our above computation becomes%
\begin{align}
\sum_{\substack{\beta\in\left\{  0,1\right\}  ^{n};\\\left\vert \beta
\right\vert =p}}\ \ \prod_{i=1}^{n}h_{\alpha_{i}+\left(  \beta+\overline
{\sigma}\right)  _{i},\ i}  &  =\sum_{\substack{\beta\in\mathbb{Z}%
^{n};\\\left(  \beta+\overline{\sigma}\right)  -\overline{\sigma}\in\left\{
0,1\right\}  ^{n};\\\left\vert \beta+\overline{\sigma}\right\vert =\left\vert
\xi\right\vert }}\ \ \prod_{i=1}^{n}h_{\alpha_{i}+\left(  \beta+\overline
{\sigma}\right)  _{i},\ i}\nonumber\\
&  =\sum_{\substack{\nu\in\mathbb{Z}^{n};\\\nu-\overline{\sigma}\in\left\{
0,1\right\}  ^{n};\\\left\vert \nu\right\vert =\left\vert \xi\right\vert
}}\ \ \prod_{i=1}^{n}h_{\alpha_{i}+\nu_{i},\ i}.
\label{pf.thm.pre-pieri2.innersum}%
\end{align}

Forget that we fixed $\sigma$. We thus have proved
(\ref{pf.thm.pre-pieri2.innersum}) for each permutation $\sigma\in S_{n}$.
Now, (\ref{pf.thm.pre-pieri2.sumt=1}) becomes%
\begin{align}
\sum_{\substack{\beta\in\left\{  0,1\right\}  ^{n};\\\left\vert \beta
\right\vert =p}}t_{\alpha+\beta}  &  =\sum_{\sigma\in S_{n}}\left(  -1\right)
^{\sigma}\underbrace{\sum_{\substack{\beta\in\left\{  0,1\right\}
^{n};\\\left\vert \beta\right\vert =p}}\ \ \prod_{i=1}^{n}h_{\alpha
_{i}+\left(  \beta+\overline{\sigma}\right)  _{i},\ i}}_{\substack{=\sum
_{\substack{\nu\in\mathbb{Z}^{n};\\\nu-\overline{\sigma}\in\left\{
0,1\right\}  ^{n};\\\left\vert \nu\right\vert =\left\vert \xi\right\vert
}}\ \ \prod_{i=1}^{n}h_{\alpha_{i}+\nu_{i},\ i}\\\text{(by
(\ref{pf.thm.pre-pieri2.innersum}))}}}\nonumber\\
&  =\sum_{\sigma\in S_{n}}\left(  -1\right)  ^{\sigma}\sum_{\substack{\nu
\in\mathbb{Z}^{n};\\\nu-\overline{\sigma}\in\left\{  0,1\right\}
^{n};\\\left\vert \nu\right\vert =\left\vert \xi\right\vert }}\ \ \prod
_{i=1}^{n}h_{\alpha_{i}+\nu_{i},\ i}\nonumber\\
&  =\underbrace{\sum_{\sigma\in S_{n}}\ \ \sum_{\substack{\nu\in\mathbb{Z}%
^{n};\\\nu-\overline{\sigma}\in\left\{  0,1\right\}  ^{n};\\\left\vert
\nu\right\vert =\left\vert \xi\right\vert }}}_{=\sum_{\substack{\nu
\in\mathbb{Z}^{n};\\\left\vert \nu\right\vert =\left\vert \xi\right\vert
}}\ \ \sum_{\substack{\sigma\in S_{n};\\\nu-\overline{\sigma}\in\left\{
0,1\right\}  ^{n}}}}\left(  -1\right)  ^{\sigma}\prod_{i=1}^{n}h_{\alpha
_{i}+\nu_{i},\ i}\nonumber\\
&  =\sum_{\substack{\nu\in\mathbb{Z}^{n};\\\left\vert \nu\right\vert
=\left\vert \xi\right\vert }}\ \ \sum_{\substack{\sigma\in S_{n}%
;\\\nu-\overline{\sigma}\in\left\{  0,1\right\}  ^{n}}}\left(  -1\right)
^{\sigma}\prod_{i=1}^{n}h_{\alpha_{i}+\nu_{i},\ i}\nonumber\\
&  =\sum_{\substack{\nu\in\mathbb{Z}^{n};\\\left\vert \nu\right\vert
=\left\vert \xi\right\vert }}\left(  \sum_{\substack{\sigma\in S_{n}%
;\\\nu-\overline{\sigma}\in\left\{  0,1\right\}  ^{n}}}\left(  -1\right)
^{\sigma}\right)  \prod_{i=1}^{n}h_{\alpha_{i}+\nu_{i},\ i}.
\label{pf.thm.pre-pieri2.2}%
\end{align}

Now, if $\nu\in\mathbb{Z}^{n}$ is an $n$-tuple satisfying $\left\vert
\nu\right\vert =\left\vert \xi\right\vert $, then%
\begin{align}
\sum_{\substack{\sigma\in S_{n};\\\nu-\overline{\sigma}\in\left\{
0,1\right\}  ^{n}}}\left(  -1\right)  ^{\sigma}  &  =\det\left(  \left(
\left[  \nu_{i}\trianglerighteq j\right]  \right)  _{i,j\in\left[  n\right]
}\right)  \ \ \ \ \ \ \ \ \ \ \left(  \text{by Lemma \ref{lem.xi-det2}}\right)
\nonumber\\
&  =\sum_{\substack{\sigma\in S_{n};\\\nu=\xi\circ\sigma}}\left(  -1\right)
^{\sigma}\ \ \ \ \ \ \ \ \ \ \left(  \text{by Lemma \ref{lem.xi-det}}\right)
. \label{pf.thm.pre-pieri2.sum=sum}%
\end{align}

Thus, (\ref{pf.thm.pre-pieri2.2}) becomes%
\begin{align}
\sum_{\substack{\beta\in\left\{  0,1\right\}  ^{n};\\\left\vert \beta
\right\vert =p}}t_{\alpha+\beta}  &  =\sum_{\substack{\nu\in\mathbb{Z}%
^{n};\\\left\vert \nu\right\vert =\left\vert \xi\right\vert }%
}\underbrace{\left(  \sum_{\substack{\sigma\in S_{n};\\\nu-\overline{\sigma
}\in\left\{  0,1\right\}  ^{n}}}\left(  -1\right)  ^{\sigma}\right)
}_{\substack{=\sum_{\substack{\sigma\in S_{n};\\\nu=\xi\circ\sigma}}\left(
-1\right)  ^{\sigma}\\\text{(by (\ref{pf.thm.pre-pieri2.sum=sum}))}}%
}\prod_{i=1}^{n}h_{\alpha_{i}+\nu_{i},\ i}\nonumber\\
&  =\sum_{\substack{\nu\in\mathbb{Z}^{n};\\\left\vert \nu\right\vert
=\left\vert \xi\right\vert }}\left(  \sum_{\substack{\sigma\in S_{n};\\\nu
=\xi\circ\sigma}}\left(  -1\right)  ^{\sigma}\right)  \prod_{i=1}^{n}%
h_{\alpha_{i}+\nu_{i},\ i}\nonumber\\
&  =\sum_{\substack{\nu\in\mathbb{Z}^{n};\\\left\vert \nu\right\vert
=\left\vert \xi\right\vert }}\ \ \sum_{\substack{\sigma\in S_{n};\\\nu
=\xi\circ\sigma}}\left(  -1\right)  ^{\sigma}\prod_{i=1}^{n}h_{\alpha_{i}%
+\nu_{i},\ i}. \label{pf.thm.pre-pieri2.2b}%
\end{align}

On the other hand,%
\begin{equation}
\operatorname*{rowdet}\left(  \left(  h_{\alpha_{i}+\xi_{j},\ i}\right)
_{i,j\in\left[  n\right]  }\right)  =\sum_{\sigma\in S_{n}}\left(  -1\right)
^{\sigma}\prod_{i=1}^{n}h_{\alpha_{i}+\xi_{\sigma\left(  i\right)  },\ i}
\label{pf.thm.pre-pieri2.3a}%
\end{equation}
(by (\ref{pf.thm.pre-pieri2.rowdetA=}), applied to $a_{i,j}=h_{\alpha_{i}%
+\xi_{j},\ i}$). However, for each $\sigma\in S_{n}$, the $n$-tuple $\xi
\circ\sigma$ belongs to $\mathbb{Z}^{n}$ and satisfies $\left\vert \xi
\circ\sigma\right\vert =\left\vert \xi\right\vert $ (by Proposition
\ref{prop.etapi.len}). In other words, for each $\sigma\in S_{n}$, the
$n$-tuple $\xi\circ\sigma$ is a $\nu\in\mathbb{Z}^{n}$ satisfying $\left\vert
\nu\right\vert =\left\vert \xi\right\vert $. Hence, any sum ranging over all
$\sigma\in S_{n}$ can be split according to the value of $\xi\circ\sigma$. In
other words, we have the following equality of summation signs:%
\[
\sum_{\sigma\in S_{n}}=\sum_{\substack{\nu\in\mathbb{Z}^{n};\\\left\vert
\nu\right\vert =\left\vert \xi\right\vert }}\ \ \underbrace{\sum
_{\substack{\sigma\in S_{n};\\\xi\circ\sigma=\nu}}}_{=\sum_{\substack{\sigma
\in S_{n};\\\nu=\xi\circ\sigma}}}=\sum_{\substack{\nu\in\mathbb{Z}%
^{n};\\\left\vert \nu\right\vert =\left\vert \xi\right\vert }}\ \ \sum
_{\substack{\sigma\in S_{n};\\\nu=\xi\circ\sigma}}.
\]
Thus, (\ref{pf.thm.pre-pieri2.3a}) becomes%
\begin{align*}
\operatorname*{rowdet}\left(  \left(  h_{\alpha_{i}+\xi_{j},\ i}\right)
_{i,j\in\left[  n\right]  }\right)   &  =\underbrace{\sum_{\sigma\in S_{n}}%
}_{=\sum_{\substack{\nu\in\mathbb{Z}^{n};\\\left\vert \nu\right\vert
=\left\vert \xi\right\vert }}\ \ \sum_{\substack{\sigma\in S_{n};\\\nu
=\xi\circ\sigma}}}\left(  -1\right)  ^{\sigma}\prod_{i=1}^{n}h_{\alpha_{i}%
+\xi_{\sigma\left(  i\right)  },\ i}\\
&  =\sum_{\substack{\nu\in\mathbb{Z}^{n};\\\left\vert \nu\right\vert
=\left\vert \xi\right\vert }}\ \ \sum_{\substack{\sigma\in S_{n};\\\nu
=\xi\circ\sigma}}\left(  -1\right)  ^{\sigma}\prod_{i=1}^{n}%
\underbrace{h_{\alpha_{i}+\xi_{\sigma\left(  i\right)  },\ i}}%
_{\substack{=h_{\alpha_{i}+\nu_{i},\ i}\\\text{(since }\xi_{\sigma\left(
i\right)  }=\nu_{i}\\\text{(because }\nu=\xi\circ\sigma=\left(  \xi
_{\sigma\left(  1\right)  },\xi_{\sigma\left(  2\right)  },\ldots,\xi
_{\sigma\left(  n\right)  }\right)  \\\text{(by Definition \ref{def.etapi})
and therefore }\nu_{i}=\xi_{\sigma\left(  i\right)  }\text{))}}}\\
&  =\sum_{\substack{\nu\in\mathbb{Z}^{n};\\\left\vert \nu\right\vert
=\left\vert \xi\right\vert }}\ \ \sum_{\substack{\sigma\in S_{n};\\\nu
=\xi\circ\sigma}}\left(  -1\right)  ^{\sigma}\prod_{i=1}^{n}h_{\alpha_{i}%
+\nu_{i},\ i}.
\end{align*}
Comparing this with (\ref{pf.thm.pre-pieri2.2b}), we obtain%
\[
\sum_{\substack{\beta\in\left\{  0,1\right\}  ^{n};\\\left\vert \beta
\right\vert =p}}t_{\alpha+\beta}=\operatorname*{rowdet}\left(  \left(
h_{\alpha_{i}+\xi_{j},\ i}\right)  _{i,j\in\left[  n\right]  }\right)  .
\]
This proves Theorem \ref{thm.pre-pieri2}.
\end{proof}
\end{verlong}

\subsection{Corollaries}

Several corollaries can be obtained from Theorem \ref{thm.pre-pieri2} in the
same way as we did above with Theorem \ref{thm.pre-pieri}. Here is an analogue
of Corollary \ref{cor.pre-pieri.2}:

\begin{corollary}
\label{cor.pre-pieri2.2}Let $n\in\mathbb{N}$ and $p\in\left\{  0,1,\ldots
,n\right\}  $. Let $q=n-p$. Let $h_{k,\ i}$ be an element of $R$ for all
$k\in\mathbb{Z}$ and $i\in\left[  n\right]  $. Assume that%
\begin{equation}
h_{k,\ i}=0\ \ \ \ \ \ \ \ \ \ \text{for all }k<0\text{ and }i>q.
\label{eq.cor.pre-pieri2.2.ass}%
\end{equation}

For any $\alpha\in\mathbb{Z}^{n}$, we define%
\[
t_{\alpha}:=\operatorname*{rowdet}\left(  \left(  h_{\alpha_{i}+j,\ i}\right)
_{i,j\in\left[  n\right]  }\right)  \in R.
\]

Let $\alpha\in\mathbb{Z}^{n}$. Assume that%
\begin{equation}
\alpha_{i}<-q\ \ \ \ \ \ \ \ \ \ \text{for each }i>q.
\label{eq.cor.pre-pieri2.2.ass-minus}%
\end{equation}
Then,%
\[
\sum_{\substack{\beta\in\left\{  0,1\right\}  ^{n};\\\left\vert \beta
\right\vert =p}}t_{\alpha+\beta}=\operatorname*{rowdet}\left(  \left(
h_{\alpha_{i}+j,\ i}\right)  _{i,j\in\left[  q\right]  }\right)
\cdot\operatorname*{rowdet}\left(  \left(  h_{\alpha_{q+i}+q+j+1,\ q+i}%
\right)  _{i,j\in\left[  p\right]  }\right)  .
\]

\end{corollary}

This can be derived from Theorem \ref{thm.pre-pieri2} using the following
lemma (which generalizes both our Lemma \ref{lem.laplace.pre} and
\cite[Exercise 6.29]{detnotes}, although it is stated in a rather different way):

\begin{lemma}
\label{lem.block-tria-rowdet}Let $n\in\mathbb{N}$. Let $A=\left(
a_{i,j}\right)  _{i,j\in\left[  n\right]  }\in R^{n\times n}$ be an $n\times
n$-matrix. Let $k\in\left\{  0,1,\ldots,n\right\}  $. Assume that%
\[
a_{i,j}=0\ \ \ \ \ \ \ \ \ \ \text{for every }i\in\left\{  k+1,k+2,\ldots
,n\right\}  \text{ and }j\in\left\{  1,2,\ldots,k\right\}  .
\]
Then,
\[
\operatorname*{rowdet}A=\operatorname*{rowdet}\left(  \left(  a_{i,j}\right)
_{i,j\in\left[  k\right]  }\right)  \cdot\operatorname*{rowdet}\left(  \left(
a_{k+i,k+j}\right)  _{i,j\in\left[  n-k\right]  }\right)  .
\]

\end{lemma}

\begin{example}
For $n=4$ and $k=2$, Lemma \ref{lem.block-tria-rowdet} is saying that%
\[
\operatorname*{rowdet}\left(
\begin{array}
[c]{cccc}%
a_{1,1} & a_{1,2} & a_{1,3} & a_{1,4}\\
a_{2,1} & a_{2,2} & a_{2,3} & a_{2,4}\\
0 & 0 & a_{3,3} & a_{3,4}\\
0 & 0 & a_{4,3} & a_{4,4}%
\end{array}
\right)  =\operatorname*{rowdet}\left(
\begin{array}
[c]{cc}%
a_{1,1} & a_{1,2}\\
a_{2,1} & a_{2,2}%
\end{array}
\right)  \cdot\operatorname*{rowdet}\left(
\begin{array}
[c]{cc}%
a_{3,3} & a_{3,4}\\
a_{4,3} & a_{4,4}%
\end{array}
\right)  .
\]

\end{example}

\begin{vershort}
\begin{proof}
[Proof of Lemma \ref{lem.block-tria-rowdet}.]This is a straightforward
generalization of \cite[Exercise 6.29]{detnotes}, and can be proved in the
same way (as long as the requisite attention is paid to the order of factors
in products, since the ring $R$ is not required to be
commutative\footnote{This means, in particular, that some products in
\cite[solution to Exercise 6.29]{detnotes} need to be reordered (and the
finite product notation $\prod_{i\in I}a_{i}$ needs to be understood as the
product of the $a_{i}$ in the order of increasing $i$). Furthermore, instead
of using \cite[Theorem 6.82 \textbf{(a)}]{detnotes}, we need to use the fact
that any $n\times n$-matrix $A=\left(  a_{i,j}\right)  _{i,j\in\left[
n\right]  }$ with $n>0$ satisfies%
\[
\operatorname*{rowdet}A=\sum_{q=1}^{n}\left(  -1\right)  ^{n+q}%
\operatorname*{rowdet}\left(  A_{\sim n,\sim q}\right)  \cdot a_{n,q}.
\]
(This generalizes the $p=n$ case of \cite[Theorem 6.82 \textbf{(a)}]{detnotes}
to the case of noncommutative $R$. The general case of \cite[Theorem 6.82
\textbf{(a)}]{detnotes} cannot be generalized to noncommutative $R$, but
fortunately we only need the $p=n$ case in our proof.)}). We leave the details
to the reader.
\end{proof}
\end{vershort}

\begin{verlong}
Lemma \ref{lem.block-tria-rowdet} is a straightforward generalization of
\cite[Exercise 6.29]{detnotes} to the case of an arbitrary ring $R$ (that is
not necessarily commutative)\footnote{This might not be immediately obvious
from the statements of Lemma \ref{lem.block-tria-rowdet} and \cite[Exercise
6.29]{detnotes}, since the notations used are rather different. However, when
the ring $R$ is commutative, Lemma \ref{lem.block-tria-rowdet} and
\cite[Exercise 6.29]{detnotes} are the same fact viewed from different angles.
Indeed, if $n$, $A$, $a_{i,j}$ and $k$ are as in Lemma
\ref{lem.block-tria-rowdet}, then $A=\left(
\begin{array}
[c]{cc}%
U & V\\
0_{\left(  n-k\right)  \times k} & W
\end{array}
\right)  $ for the three matrices%
\[
U=\left(  a_{i,j}\right)  _{i,j\in\left[  k\right]  }%
,\ \ \ \ \ \ \ \ \ \ V=\left(  a_{i,k+j}\right)  _{i\in\left[  k\right]
,\ j\in\left[  n-k\right]  }\ \ \ \ \ \ \ \ \ \ \text{and}%
\ \ \ \ \ \ \ \ \ \ W=\left(  a_{k+i,k+j}\right)  _{i,j\in\left[  n-k\right]
}.
\]
Thus, applying \cite[Exercise 6.29]{detnotes} to $k$, $n-k$, $U$, $V$ and $W$
instead of $n$, $m$, $A$, $B$ and $D$ yields that%
\[
\det A=\det U\cdot\det W=\det\left(  \left(  a_{i,j}\right)  _{i,j\in\left[
k\right]  }\right)  \cdot\det\left(  \left(  a_{k+i,k+j}\right)
_{i,j\in\left[  n-k\right]  }\right)
\]
in this case. Since the determinant of a matrix over a commutative ring $R$
equals its row-determinant, we can rewrite this as%
\[
\operatorname*{rowdet}A=\operatorname*{rowdet}\left(  \left(  a_{i,j}\right)
_{i,j\in\left[  k\right]  }\right)  \cdot\operatorname*{rowdet}\left(  \left(
a_{k+i,k+j}\right)  _{i,j\in\left[  n-k\right]  }\right)  .
\]
Thus, our Lemma \ref{lem.block-tria-rowdet} follows from \cite[Exercise
6.29]{detnotes} in the case when $R$ is commutative. The converse implication
is equally easy to see.}. Unsurprisingly, it can be proved in more or less the
same way as \cite[Exercise 6.29]{detnotes}, as long as the requisite attention
is paid to the order of factors in products\footnote{This means, in
particular, that some products in \cite[solution to Exercise 6.29]{detnotes}
need to be reordered (and the finite product notation $\prod_{i\in I}a_{i}$
needs to be understood as the product of the $a_{i}$ in the order of
increasing $i$). Furthermore, instead of using \cite[Theorem 6.82
\textbf{(a)}]{detnotes}, we need to use the fact that any $n\times n$-matrix
$A=\left(  a_{i,j}\right)  _{i,j\in\left[  n\right]  }$ with $n>0$ satisfies%
\[
\operatorname*{rowdet}A=\sum_{q=1}^{n}\left(  -1\right)  ^{n+q}%
\operatorname*{rowdet}\left(  A_{\sim n,\sim q}\right)  \cdot a_{n,q}.
\]
(This generalizes the $p=n$ case of \cite[Theorem 6.82 \textbf{(a)}]{detnotes}
to the case of noncommutative $R$. The general case of \cite[Theorem 6.82
\textbf{(a)}]{detnotes} cannot be generalized to noncommutative $R$, but
fortunately we only need the $p=n$ case in our proof.)}. A reader could easily
perform the necessary modifications to obtain a proof of Lemma
\ref{lem.block-tria-rowdet}. We shall, however, give a different proof of
Lemma \ref{lem.block-tria-rowdet}, which proceeds from scratch using a
combinatorial lemma. (This will obviously yield a new solution to
\cite[Exercise 6.29]{detnotes}.)

Before we state this lemma, we need to recall two notations concerning permutations:

\begin{itemize}
\item If $n\in\mathbb{N}$ and if $\sigma\in S_{n}$, then an \emph{inversion}
of $\sigma$ means a pair $\left(  i,j\right)  $ of integers satisfying $1\leq
i<j\leq n$ and $\sigma\left(  i\right)  >\sigma\left(  j\right)  $.

\item If $n\in\mathbb{N}$ and if $\sigma\in S_{n}$, then the \emph{length} of
$\sigma$ is defined to be the number of inversions of $\sigma$. This length is
denoted by $\ell\left(  \sigma\right)  $.
\end{itemize}

Recall that the sign $\left(  -1\right)  ^{\sigma}$ of a permutation
$\sigma\in S_{n}$ is defined to be $\left(  -1\right)  ^{\ell\left(
\sigma\right)  }$.

Now, we can state our combinatorial lemma:

\begin{lemma}
\label{lem.lalbe}Let $n\in\mathbb{N}$. Let $I$ be a subset of $\left\{
1,2,\ldots,n\right\}  $. Let $k=\left\vert I\right\vert $. Let $\left(
a_{1},a_{2},\ldots,a_{k}\right)  $ be the list of all elements of $I$ in
increasing order (with no repetitions). Let $\left(  b_{1},b_{2}%
,\ldots,b_{n-k}\right)  $ be the list of all elements of $\left\{
1,2,\ldots,n\right\}  \setminus I$ in increasing order (with no repetitions).
Let $\alpha\in S_{k}$ and $\beta\in S_{n-k}$. Then:

\textbf{(a)} There exists a unique $\sigma\in S_{n}$ satisfying%
\[
\left(  \sigma\left(  1\right)  ,\sigma\left(  2\right)  ,\ldots,\sigma\left(
n\right)  \right)  =\left(  a_{\alpha\left(  1\right)  },a_{\alpha\left(
2\right)  },\ldots,a_{\alpha\left(  k\right)  },b_{\beta\left(  1\right)
},b_{\beta\left(  2\right)  },\ldots,b_{\beta\left(  n-k\right)  }\right)  .
\]
Denote this $\sigma$ by $\sigma_{I,\alpha,\beta}$.

\textbf{(b)} Let $\sum I$ denote the sum of all elements of $I$. (Thus, $\sum
I=\sum_{i\in I}i$.) We have%
\[
\ell\left(  \sigma_{I,\alpha,\beta}\right)  =\ell\left(  \alpha\right)
+\ell\left(  \beta\right)  +\sum I-\left(  1+2+\cdots+k\right)
\]
and%
\[
\left(  -1\right)  ^{\sigma_{I,\alpha,\beta}}=\left(  -1\right)  ^{\alpha
}\cdot\left(  -1\right)  ^{\beta}\cdot\left(  -1\right)  ^{\sum I-\left(
1+2+\cdots+k\right)  }.
\]

\textbf{(c)} Forget that we fixed $\alpha$ and $\beta$. We thus have defined
an element $\sigma_{I,\alpha,\beta}\in S_{n}$ for every $\alpha\in S_{k}$ and
every $\beta\in S_{n-k}$. The map%
\begin{align*}
S_{k}\times S_{n-k}  &  \rightarrow\left\{  \tau\in S_{n}\ \mid\ \tau\left(
\left\{  1,2,\ldots,k\right\}  \right)  =I\right\}  ,\\
\left(  \alpha,\beta\right)   &  \mapsto\sigma_{I,\alpha,\beta}%
\end{align*}
is well-defined and a bijection.
\end{lemma}

\begin{proof}
[Proof of Lemma \ref{lem.lalbe}.]Lemma \ref{lem.lalbe} is precisely
\cite[Exercise 5.14]{detnotes}.
\end{proof}

We will use a particular case of Lemma \ref{lem.lalbe}, which has a slightly
simpler form:

\begin{lemma}
\label{lem.lalbe-k}Let $n\in\mathbb{N}$. Let $k\in\left\{  0,1,\ldots
,n\right\}  $. Let $\alpha\in S_{k}$ and $\beta\in S_{n-k}$. Then:

\textbf{(a)} There exists a unique $\sigma\in S_{n}$ satisfying%
\begin{align*}
&  \left(  \sigma\left(  1\right)  ,\sigma\left(  2\right)  ,\ldots
,\sigma\left(  n\right)  \right) \\
&  =\left(  \alpha\left(  1\right)  ,\ \alpha\left(  2\right)  ,\ \ldots
,\ \alpha\left(  k\right)  ,\ k+\beta\left(  1\right)  ,\ k+\beta\left(
2\right)  ,\ \ldots,\ k+\beta\left(  n-k\right)  \right)  .
\end{align*}
Denote this $\sigma$ by $\alpha\oplus\beta$.

\textbf{(b)} We have%
\[
\left(  -1\right)  ^{\alpha\oplus\beta}=\left(  -1\right)  ^{\alpha}%
\cdot\left(  -1\right)  ^{\beta}.
\]

\textbf{(c)} Forget that we fixed $\alpha$ and $\beta$. We thus have defined
an element $\alpha\oplus\beta\in S_{n}$ for every $\alpha\in S_{k}$ and every
$\beta\in S_{n-k}$. The map%
\begin{align*}
S_{k}\times S_{n-k}  &  \rightarrow\left\{  \tau\in S_{n}\ \mid\ \tau\left(
\left[  k\right]  \right)  =\left[  k\right]  \right\}  ,\\
\left(  \alpha,\beta\right)   &  \mapsto\alpha\oplus\beta
\end{align*}
is well-defined and a bijection.
\end{lemma}

\begin{proof}
[Proof of Lemma \ref{lem.lalbe-k}.]From $k\in\left\{  0,1,\ldots,n\right\}  $,
we obtain $0\leq k\leq n$. The definition of $\left[  n\right]  $ yields
$\left[  n\right]  =\left\{  1,2,\ldots,n\right\}  $. The definition of
$\left[  k\right]  $ yields $\left[  k\right]  =\left\{  1,2,\ldots,k\right\}
$.

Let $I$ be the set $\left[  k\right]  $. Thus, $I=\left[  k\right]  =\left\{
1,2,\ldots,k\right\}  \subseteq\left\{  1,2,\ldots,n\right\}  $ (since $k\leq
n$). In other words, $I$ is a subset of $\left\{  1,2,\ldots,n\right\}  $.
Moreover, from $I=\left[  k\right]  $, we obtain $\left\vert I\right\vert
=\left\vert \left[  k\right]  \right\vert =k$ (since $\left[  k\right]  $ is a
$k$-element set), so that $k=\left\vert I\right\vert $.

Define a $k$-tuple $\left(  a_{1},a_{2},\ldots,a_{k}\right)  \in\mathbb{Z}%
^{k}$ by $\left(  a_{1},a_{2},\ldots,a_{k}\right)  =\left(  1,2,\ldots
,k\right)  $. Thus,%
\begin{equation}
a_{i}=i\ \ \ \ \ \ \ \ \ \ \text{for each }i\in\left[  k\right]  .
\label{pf.lem.lalbe-k.ai=}%
\end{equation}

Define an $\left(  n-k\right)  $-tuple $\left(  b_{1},b_{2},\ldots
,b_{n-k}\right)  \in\mathbb{Z}^{n-k}$ by \newline$\left(  b_{1},b_{2}%
,\ldots,b_{n-k}\right)  =\left(  k+1,k+2,\ldots,n\right)  $. Thus,
\[
\left(  b_{1},b_{2},\ldots,b_{n-k}\right)  =\left(  k+1,k+2,\ldots,n\right)
=\left(  k+1,k+2,\ldots,k+\left(  n-k\right)  \right)
\]
(since $n=k+\left(  n-k\right)  $). Hence,%
\begin{equation}
b_{i}=k+i\ \ \ \ \ \ \ \ \ \ \text{for each }i\in\left[  n-k\right]  .
\label{pf.lem.lalbe-k.bi=}%
\end{equation}

We have $I=\left\{  1,2,\ldots,k\right\}  $ and $1<2<\cdots<k$. Hence,
$\left(  1,2,\ldots,k\right)  $ is the list of all elements of $I$ in
increasing order (with no repetitions). In other words, $\left(  a_{1}%
,a_{2},\ldots,a_{k}\right)  $ is the list of all elements of $I$ in increasing
order (with no repetitions) (since $\left(  a_{1},a_{2},\ldots,a_{k}\right)
=\left(  1,2,\ldots,k\right)  $).

Moreover, we have%
\begin{align*}
\left\{  1,2,\ldots,n\right\}  \setminus\underbrace{I}_{=\left\{
1,2,\ldots,k\right\}  }  &  =\left\{  1,2,\ldots,n\right\}  \setminus\left\{
1,2,\ldots,k\right\} \\
&  =\left\{  k+1,k+2,\ldots,n\right\}  \ \ \ \ \ \ \ \ \ \ \left(  \text{since
}0\leq k\leq n\right)
\end{align*}
and $k+1<k+2<\cdots<n$. Hence, $\left(  k+1,k+2,\ldots,n\right)  $ is the list
of all elements of $\left\{  1,2,\ldots,n\right\}  \setminus I$ in increasing
order (with no repetitions). In other words, $\left(  b_{1},b_{2}%
,\ldots,b_{n-k}\right)  $ is the list of all elements of $\left\{
1,2,\ldots,n\right\}  \setminus I$ in increasing order (with no repetitions)
(since $\left(  b_{1},b_{2},\ldots,b_{n-k}\right)  =\left(  k+1,k+2,\ldots
,n\right)  $). \medskip

\textbf{(a)} Let $\alpha\in S_{k}$ and $\beta\in S_{n-k}$.

We have $\alpha\in S_{k}$; in other words, $\alpha$ is a permutation of
$\left[  k\right]  $ (since $S_{k}$ is the set of all permutations of $\left[
k\right]  $). Hence, for each $j\in\left[  k\right]  $, we have $\alpha\left(
j\right)  \in\left[  k\right]  $ and thus $a_{\alpha\left(  j\right)  }%
=\alpha\left(  j\right)  $ (by (\ref{pf.lem.lalbe-k.ai=}), applied to
$i=\alpha\left(  j\right)  $). In other words, we have%
\[
\left(  a_{\alpha\left(  1\right)  },a_{\alpha\left(  2\right)  }%
,\ldots,a_{\alpha\left(  k\right)  }\right)  =\left(  \alpha\left(  1\right)
,\ \alpha\left(  2\right)  ,\ \ldots,\ \alpha\left(  k\right)  \right)  .
\]

We have $\beta\in S_{n-k}$; in other words, $\beta$ is a permutation of
$\left[  n-k\right]  $ (since $S_{n-k}$ is the set of all permutations of
$\left[  n-k\right]  $). Hence, for each $j\in\left[  n-k\right]  $, we have
$\beta\left(  j\right)  \in\left[  n-k\right]  $ and thus $b_{\beta\left(
j\right)  }=k+\beta\left(  j\right)  $ (by (\ref{pf.lem.lalbe-k.bi=}), applied
to $i=\beta\left(  j\right)  $). In other words, we have%
\[
\left(  b_{\beta\left(  1\right)  },b_{\beta\left(  2\right)  },\ldots
,b_{\beta\left(  n-k\right)  }\right)  =\left(  k+\beta\left(  1\right)
,\ k+\beta\left(  2\right)  ,\ \ldots,\ k+\beta\left(  n-k\right)  \right)  .
\]

Now, let us use the notation $u\ast v$ for the concatenation of two finite
lists $u$ and $v$. This is defined as follows: If $u=\left(  u_{1}%
,u_{2},\ldots,u_{p}\right)  $ and $v=\left(  v_{1},v_{2},\ldots,v_{q}\right)
$ are two finite lists, then their \emph{concatenation} $u\ast v$ is defined
to be the finite list $\left(  u_{1},u_{2},\ldots,u_{p},v_{1},v_{2}%
,\ldots,v_{q}\right)  $.

Clearly,%
\begin{align}
&  \left(  a_{\alpha\left(  1\right)  },a_{\alpha\left(  2\right)  }%
,\ldots,a_{\alpha\left(  k\right)  },b_{\beta\left(  1\right)  }%
,b_{\beta\left(  2\right)  },\ldots,b_{\beta\left(  n-k\right)  }\right)
\nonumber\\
&  =\underbrace{\left(  a_{\alpha\left(  1\right)  },a_{\alpha\left(
2\right)  },\ldots,a_{\alpha\left(  k\right)  }\right)  }_{=\left(
\alpha\left(  1\right)  ,\ \alpha\left(  2\right)  ,\ \ldots,\ \alpha\left(
k\right)  \right)  }\ast\underbrace{\left(  b_{\beta\left(  1\right)
},b_{\beta\left(  2\right)  },\ldots,b_{\beta\left(  n-k\right)  }\right)
}_{=\left(  k+\beta\left(  1\right)  ,\ k+\beta\left(  2\right)
,\ \ldots,\ k+\beta\left(  n-k\right)  \right)  }\nonumber\\
&  =\left(  \alpha\left(  1\right)  ,\ \alpha\left(  2\right)  ,\ \ldots
,\ \alpha\left(  k\right)  \right)  \ast\left(  k+\beta\left(  1\right)
,\ k+\beta\left(  2\right)  ,\ \ldots,\ k+\beta\left(  n-k\right)  \right)
\nonumber\\
&  =\left(  \alpha\left(  1\right)  ,\ \alpha\left(  2\right)  ,\ \ldots
,\ \alpha\left(  k\right)  ,\ k+\beta\left(  1\right)  ,\ k+\beta\left(
2\right)  ,\ \ldots,\ k+\beta\left(  n-k\right)  \right)  .
\label{pf.lem.lalbe-k.a.concats}%
\end{align}

Now, recall that $\left(  a_{1},a_{2},\ldots,a_{k}\right)  $ is the list of
all elements of $I$ in increasing order (with no repetitions), whereas
$\left(  b_{1},b_{2},\ldots,b_{n-k}\right)  $ is the list of all elements of
$\left\{  1,2,\ldots,n\right\}  \setminus I$ in increasing order (with no
repetitions). Thus, we can apply Lemma \ref{lem.lalbe} \textbf{(a)}. We thus
conclude that there exists a unique $\sigma\in S_{n}$ satisfying%
\[
\left(  \sigma\left(  1\right)  ,\sigma\left(  2\right)  ,\ldots,\sigma\left(
n\right)  \right)  =\left(  a_{\alpha\left(  1\right)  },a_{\alpha\left(
2\right)  },\ldots,a_{\alpha\left(  k\right)  },b_{\beta\left(  1\right)
},b_{\beta\left(  2\right)  },\ldots,b_{\beta\left(  n-k\right)  }\right)  .
\]
In view of (\ref{pf.lem.lalbe-k.a.concats}), we can rewrite this as follows:
There exists a unique $\sigma\in S_{n}$ satisfying%
\begin{align*}
&  \left(  \sigma\left(  1\right)  ,\sigma\left(  2\right)  ,\ldots
,\sigma\left(  n\right)  \right) \\
&  =\left(  \alpha\left(  1\right)  ,\ \alpha\left(  2\right)  ,\ \ldots
,\ \alpha\left(  k\right)  ,\ k+\beta\left(  1\right)  ,\ k+\beta\left(
2\right)  ,\ \ldots,\ k+\beta\left(  n-k\right)  \right)  .
\end{align*}
This proves Lemma \ref{lem.lalbe-k} \textbf{(a)}. \medskip

\textbf{(b)} Let $\alpha\in S_{k}$ and $\beta\in S_{n-k}$. Then, we have%
\begin{align}
&  \left(  a_{\alpha\left(  1\right)  },a_{\alpha\left(  2\right)  }%
,\ldots,a_{\alpha\left(  k\right)  },b_{\beta\left(  1\right)  }%
,b_{\beta\left(  2\right)  },\ldots,b_{\beta\left(  n-k\right)  }\right)
\nonumber\\
&  =\left(  \alpha\left(  1\right)  ,\ \alpha\left(  2\right)  ,\ \ldots
,\ \alpha\left(  k\right)  ,\ k+\beta\left(  1\right)  ,\ k+\beta\left(
2\right)  ,\ \ldots,\ k+\beta\left(  n-k\right)  \right)  .
\label{pf.lem.lalbe-k.b.concats}%
\end{align}
(Indeed, this was already shown in our above proof of Lemma \ref{lem.lalbe-k}
\textbf{(a)}.)

The permutation $\alpha\oplus\beta$ is the unique $\sigma\in S_{n}$ satisfying%
\begin{align*}
&  \left(  \sigma\left(  1\right)  ,\sigma\left(  2\right)  ,\ldots
,\sigma\left(  n\right)  \right) \\
&  =\left(  \alpha\left(  1\right)  ,\ \alpha\left(  2\right)  ,\ \ldots
,\ \alpha\left(  k\right)  ,\ k+\beta\left(  1\right)  ,\ k+\beta\left(
2\right)  ,\ \ldots,\ k+\beta\left(  n-k\right)  \right)
\end{align*}
(by the definition of $\alpha\oplus\beta$). In view of
(\ref{pf.lem.lalbe-k.b.concats}), we can rewrite this as follows: The
permutation $\alpha\oplus\beta$ is the unique $\sigma\in S_{n}$ satisfying%
\[
\left(  \sigma\left(  1\right)  ,\sigma\left(  2\right)  ,\ldots,\sigma\left(
n\right)  \right)  =\left(  a_{\alpha\left(  1\right)  },a_{\alpha\left(
2\right)  },\ldots,a_{\alpha\left(  k\right)  },b_{\beta\left(  1\right)
},b_{\beta\left(  2\right)  },\ldots,b_{\beta\left(  n-k\right)  }\right)  .
\]

However, recall that $\left(  a_{1},a_{2},\ldots,a_{k}\right)  $ is the list
of all elements of $I$ in increasing order (with no repetitions), whereas
$\left(  b_{1},b_{2},\ldots,b_{n-k}\right)  $ is the list of all elements of
$\left\{  1,2,\ldots,n\right\}  \setminus I$ in increasing order (with no
repetitions). Thus, we can apply Lemma \ref{lem.lalbe}. In particular, the
permutation $\sigma_{I,\alpha,\beta}$ from Lemma \ref{lem.lalbe} \textbf{(a)}
is well-defined. Consider this permutation $\sigma_{I,\alpha,\beta}$.
According to its definition, this permutation $\sigma_{I,\alpha,\beta}$ is the
unique $\sigma\in S_{n}$ satisfying%
\[
\left(  \sigma\left(  1\right)  ,\sigma\left(  2\right)  ,\ldots,\sigma\left(
n\right)  \right)  =\left(  a_{\alpha\left(  1\right)  },a_{\alpha\left(
2\right)  },\ldots,a_{\alpha\left(  k\right)  },b_{\beta\left(  1\right)
},b_{\beta\left(  2\right)  },\ldots,b_{\beta\left(  n-k\right)  }\right)  .
\]
However, we have already seen that the permutation $\alpha\oplus\beta$ is the
unique $\sigma\in S_{n}$ satisfying%
\[
\left(  \sigma\left(  1\right)  ,\sigma\left(  2\right)  ,\ldots,\sigma\left(
n\right)  \right)  =\left(  a_{\alpha\left(  1\right)  },a_{\alpha\left(
2\right)  },\ldots,a_{\alpha\left(  k\right)  },b_{\beta\left(  1\right)
},b_{\beta\left(  2\right)  },\ldots,b_{\beta\left(  n-k\right)  }\right)  .
\]
Comparing the preceding two sentences, we thus conclude that $\sigma
_{I,\alpha,\beta}$ and $\alpha\oplus\beta$ are equal (since each of them is
the unique $\sigma\in S_{n}$ satisfying%
\[
\left(  \sigma\left(  1\right)  ,\sigma\left(  2\right)  ,\ldots,\sigma\left(
n\right)  \right)  =\left(  a_{\alpha\left(  1\right)  },a_{\alpha\left(
2\right)  },\ldots,a_{\alpha\left(  k\right)  },b_{\beta\left(  1\right)
},b_{\beta\left(  2\right)  },\ldots,b_{\beta\left(  n-k\right)  }\right)
\]
). In other words,
\begin{equation}
\sigma_{I,\alpha,\beta}=\alpha\oplus\beta. \label{pf.lem.lalbe-k.apb=}%
\end{equation}

Now, let $\sum I$ denote the sum of all elements of the set $I$. Thus,
\begin{align*}
\sum I  &  =\sum_{i\in I}i=\sum_{i\in\left\{  1,2,\ldots,k\right\}
}i\ \ \ \ \ \ \ \ \ \ \left(  \text{since }I=\left\{  1,2,\ldots,k\right\}
\right) \\
&  =1+2+\cdots+k,
\end{align*}
so that $\sum I-\left(  1+2+\cdots+k\right)  =0$. Thus,
\begin{equation}
\left(  -1\right)  ^{\sum I-\left(  1+2+\cdots+k\right)  }=\left(  -1\right)
^{0}=1. \label{pf.lem.lalbe-k.sumI-}%
\end{equation}
However, Lemma \ref{lem.lalbe} \textbf{(b)} yields that we have
\[
\ell\left(  \sigma_{I,\alpha,\beta}\right)  =\ell\left(  \alpha\right)
+\ell\left(  \beta\right)  +\sum I-\left(  1+2+\cdots+k\right)
\]
and%
\[
\left(  -1\right)  ^{\sigma_{I,\alpha,\beta}}=\left(  -1\right)  ^{\alpha
}\cdot\left(  -1\right)  ^{\beta}\cdot\left(  -1\right)  ^{\sum I-\left(
1+2+\cdots+k\right)  }.
\]
The second of these two equalities becomes%
\[
\left(  -1\right)  ^{\sigma_{I,\alpha,\beta}}=\left(  -1\right)  ^{\alpha
}\cdot\left(  -1\right)  ^{\beta}\cdot\underbrace{\left(  -1\right)  ^{\sum
I-\left(  1+2+\cdots+k\right)  }}_{\substack{=1\\\text{(by
(\ref{pf.lem.lalbe-k.sumI-}))}}}=\left(  -1\right)  ^{\alpha}\cdot\left(
-1\right)  ^{\beta}.
\]
In view of $\sigma_{I,\alpha,\beta}=\alpha\oplus\beta$, we can rewrite this as
$\left(  -1\right)  ^{\alpha\oplus\beta}=\left(  -1\right)  ^{\alpha}%
\cdot\left(  -1\right)  ^{\beta}$. This proves Lemma \ref{lem.lalbe-k}
\textbf{(b)}. \medskip

\textbf{(c)} Recall that $\left(  a_{1},a_{2},\ldots,a_{k}\right)  $ is the
list of all elements of $I$ in increasing order (with no repetitions), whereas
$\left(  b_{1},b_{2},\ldots,b_{n-k}\right)  $ is the list of all elements of
$\left\{  1,2,\ldots,n\right\}  \setminus I$ in increasing order (with no repetitions).

For every $\alpha\in S_{k}$ and $\beta\in S_{n-k}$, we define a permutation
$\sigma_{I,\alpha,\beta}\in S_{n}$ as in Lemma \ref{lem.lalbe} \textbf{(a)}.
In our above proof of Lemma \ref{lem.lalbe-k} \textbf{(b)}, we have shown that
this permutation $\sigma_{I,\alpha,\beta}$ satisfies
(\ref{pf.lem.lalbe-k.apb=}). In other words, we have $\sigma_{I,\alpha,\beta
}=\alpha\oplus\beta$ for every $\alpha\in S_{k}$ and $\beta\in S_{n-k}$. In
other words, we have $\sigma_{I,\alpha,\beta}=\alpha\oplus\beta$ for every
$\left(  \alpha,\beta\right)  \in S_{k}\times S_{n-k}$.

Lemma \ref{lem.lalbe} \textbf{(c)} yields that the map
\begin{align*}
S_{k}\times S_{n-k}  &  \rightarrow\left\{  \tau\in S_{n}\ \mid\ \tau\left(
\left\{  1,2,\ldots,k\right\}  \right)  =I\right\}  ,\\
\left(  \alpha,\beta\right)   &  \mapsto\sigma_{I,\alpha,\beta}%
\end{align*}
is well-defined and a bijection. In other words, the map%
\begin{align*}
S_{k}\times S_{n-k}  &  \rightarrow\left\{  \tau\in S_{n}\ \mid\ \tau\left(
\left\{  1,2,\ldots,k\right\}  \right)  =I\right\}  ,\\
\left(  \alpha,\beta\right)   &  \mapsto\alpha\oplus\beta
\end{align*}
is well-defined and a bijection (because we have $\sigma_{I,\alpha,\beta
}=\alpha\oplus\beta$ for every $\left(  \alpha,\beta\right)  \in S_{k}\times
S_{n-k}$). In view of $\left\{  1,2,\ldots,k\right\}  =\left[  k\right]  $ and
$I=\left[  k\right]  $, we can rewrite this as follows: The map%
\begin{align*}
S_{k}\times S_{n-k}  &  \rightarrow\left\{  \tau\in S_{n}\ \mid\ \tau\left(
\left[  k\right]  \right)  =\left[  k\right]  \right\}  ,\\
\left(  \alpha,\beta\right)   &  \mapsto\alpha\oplus\beta
\end{align*}
is well-defined and a bijection. This proves Lemma \ref{lem.lalbe-k}
\textbf{(c)}.
\end{proof}

Our proof of Lemma \ref{lem.block-tria-rowdet} will furthermore use a piece of
notation that we have already employed in some of our previous proofs:

\begin{convention}
\label{conv.prod-noncomm}Let us set $\prod_{i=1}^{n}a_{i}:=a_{1}a_{2}\cdots
a_{n}$ for any $a_{1},a_{2},\ldots,a_{n}\in R$. (This is, of course, the usual
meaning of the notation $\prod_{i=1}^{n}a_{i}$ when the ring $R$ is
commutative; however, we are now extending it to the case of arbitrary $R$.)
\end{convention}

Using this notation, we can rewrite the definition of a row-determinant in a
slicker form: If $\left(  a_{i,j}\right)  _{i,j\in\left[  n\right]  }\in
R^{n\times n}$ is any $n\times n$-matrix over $R$ (where $n$ is any
nonnegative integer), then%
\begin{align}
\operatorname*{rowdet}\left(  \left(  a_{i,j}\right)  _{i,j\in\left[
n\right]  }\right)   &  =\sum_{\sigma\in S_{n}}\left(  -1\right)  ^{\sigma
}\underbrace{a_{1,\sigma\left(  1\right)  }a_{2,\sigma\left(  2\right)
}\cdots a_{n,\sigma\left(  n\right)  }}_{\substack{=\prod_{i=1}^{n}%
a_{i,\sigma\left(  i\right)  }\\\text{(since }\prod_{i=1}^{n}a_{i,\sigma
\left(  i\right)  }\text{ is defined to be }a_{1,\sigma\left(  1\right)
}a_{2,\sigma\left(  2\right)  }\cdots a_{n,\sigma\left(  n\right)  }\text{)}%
}}\nonumber\\
&  \ \ \ \ \ \ \ \ \ \ \ \ \ \ \ \ \ \ \ \ \left(  \text{by the definition of
a row-determinant}\right) \nonumber\\
&  =\sum_{\sigma\in S_{n}}\left(  -1\right)  ^{\sigma}\prod_{i=1}%
^{n}a_{i,\sigma\left(  i\right)  }. \label{pf.lem.block-tria-rowdet.rowdet=}%
\end{align}

\begin{proof}
[Proof of Lemma \ref{lem.block-tria-rowdet}.]From $k\in\left\{  0,1,\ldots
,n\right\}  $, we obtain $0\leq k\leq n$. The definition of $\left[  k\right]
$ yields $\left[  k\right]  =\left\{  1,2,\ldots,k\right\}  \subseteq\left\{
1,2,\ldots,n\right\}  $ (since $k\leq n$). In other words, $\left[  k\right]
\subseteq\left[  n\right]  $ (since $\left[  n\right]  =\left\{
1,2,\ldots,n\right\}  $). In other words, $\left[  k\right]  $ is a subset of
$\left[  n\right]  $.

Lemma \ref{lem.lalbe-k} \textbf{(a)} defines a permutation $\alpha\oplus
\beta\in S_{n}$ for each $\alpha\in S_{k}$ and $\beta\in S_{n-k}$. Lemma
\ref{lem.lalbe-k} \textbf{(c)} shows that the map
\begin{align*}
S_{k}\times S_{n-k}  &  \rightarrow\left\{  \tau\in S_{n}\ \mid\ \tau\left(
\left[  k\right]  \right)  =\left[  k\right]  \right\}  ,\\
\left(  \alpha,\beta\right)   &  \mapsto\alpha\oplus\beta
\end{align*}
is well-defined and a bijection.

Now, we shall show the following:

\begin{statement}
\textit{Claim 1:} For each $\alpha\in S_{k}$ and each $\beta\in S_{n-k}$, we
have%
\begin{equation}
\prod_{i=1}^{n}a_{i,\left(  \alpha\oplus\beta\right)  \left(  i\right)
}=\left(  \prod_{i=1}^{k}a_{i,\alpha\left(  i\right)  }\right)  \left(
\prod_{i=1}^{n-k}a_{k+i,k+\beta\left(  i\right)  }\right)  .
\label{pf.lem.block-tria-rowdet.prod=prodprod}%
\end{equation}

\end{statement}

[\textit{Proof of Claim 1:} Let $\alpha\in S_{k}$ and $\beta\in S_{n-k}$. Let
$\gamma=\alpha\oplus\beta$.

Recall that $\alpha\oplus\beta$ is the unique $\sigma\in S_{n}$ satisfying%
\begin{align*}
&  \left(  \sigma\left(  1\right)  ,\sigma\left(  2\right)  ,\ldots
,\sigma\left(  n\right)  \right) \\
&  =\left(  \alpha\left(  1\right)  ,\ \alpha\left(  2\right)  ,\ \ldots
,\ \alpha\left(  k\right)  ,\ k+\beta\left(  1\right)  ,\ k+\beta\left(
2\right)  ,\ \ldots,\ k+\beta\left(  n-k\right)  \right)
\end{align*}
(because this is how $\alpha\oplus\beta$ was defined in Lemma
\ref{lem.lalbe-k} \textbf{(a)}). In other words, $\gamma$ is the unique
$\sigma\in S_{n}$ satisfying%
\begin{align*}
&  \left(  \sigma\left(  1\right)  ,\sigma\left(  2\right)  ,\ldots
,\sigma\left(  n\right)  \right) \\
&  =\left(  \alpha\left(  1\right)  ,\ \alpha\left(  2\right)  ,\ \ldots
,\ \alpha\left(  k\right)  ,\ k+\beta\left(  1\right)  ,\ k+\beta\left(
2\right)  ,\ \ldots,\ k+\beta\left(  n-k\right)  \right)
\end{align*}
(since $\gamma=\alpha\oplus\beta$). Hence, $\gamma$ is an element of $S_{n}$
and satisfies
\begin{align*}
&  \left(  \gamma\left(  1\right)  ,\ \gamma\left(  2\right)  ,\ \ldots
,\ \gamma\left(  n\right)  \right) \\
&  =\left(  \alpha\left(  1\right)  ,\ \alpha\left(  2\right)  ,\ \ldots
,\ \alpha\left(  k\right)  ,\ k+\beta\left(  1\right)  ,\ k+\beta\left(
2\right)  ,\ \ldots,\ k+\beta\left(  n-k\right)  \right)  .
\end{align*}

Thus, in particular, we have%
\begin{align*}
&  \left(  \gamma\left(  1\right)  ,\ \gamma\left(  2\right)  ,\ \ldots
,\ \gamma\left(  n\right)  \right) \\
&  =\left(  \alpha\left(  1\right)  ,\ \alpha\left(  2\right)  ,\ \ldots
,\ \alpha\left(  k\right)  ,\ k+\beta\left(  1\right)  ,\ k+\beta\left(
2\right)  ,\ \ldots,\ k+\beta\left(  n-k\right)  \right)  .
\end{align*}
In other words, we have%
\begin{equation}
\left(  \gamma\left(  i\right)  =\alpha\left(  i\right)
\ \ \ \ \ \ \ \ \ \ \text{for each }i\in\left\{  1,2,\ldots,k\right\}
\right)  \label{pf.lem.block-tria-rowdet.prod=prodprod.pf.1}%
\end{equation}
and%
\begin{equation}
\left(  \gamma\left(  i\right)  =k+\beta\left(  i-k\right)
\ \ \ \ \ \ \ \ \ \ \text{for each }i\in\left\{  k+1,k+2,\ldots,n\right\}
\right)  . \label{pf.lem.block-tria-rowdet.prod=prodprod.pf.2}%
\end{equation}

Now, let $i\in\left\{  1,2,\ldots,n-k\right\}  $. Thus, $k+i\in\left\{
k+1,k+2,\ldots,k+\left(  n-k\right)  \right\}  =\left\{  k+1,k+2,\ldots
,n\right\}  $ (since $k+\left(  n-k\right)  =n$). Hence,
(\ref{pf.lem.block-tria-rowdet.prod=prodprod.pf.2}) (applied to $k+i$ instead
of $i$) yields
\begin{equation}
\gamma\left(  k+i\right)  =k+\beta\left(  \underbrace{\left(  k+i\right)
-k}_{=i}\right)  =k+\beta\left(  i\right)  .
\label{pf.lem.block-tria-rowdet.prod=prodprod.pf.3}%
\end{equation}

Forget that we fixed $i$. We thus have proved that
(\ref{pf.lem.block-tria-rowdet.prod=prodprod.pf.3}) holds for each
$i\in\left\{  1,2,\ldots,n-k\right\}  $.

Now, Convention \ref{conv.prod-noncomm} yields%
\begin{align*}
\prod_{i=1}^{n}a_{i,\gamma\left(  i\right)  }  &  =a_{1,\gamma\left(
1\right)  }a_{2,\gamma\left(  2\right)  }\cdots a_{n,\gamma\left(  n\right)
}\\
&  =\left(  a_{1,\gamma\left(  1\right)  }a_{2,\gamma\left(  2\right)  }\cdots
a_{k,\gamma\left(  k\right)  }\right)  \left(  a_{k+1,\gamma\left(
k+1\right)  }a_{k+2,\gamma\left(  k+2\right)  }\cdots a_{n,\gamma\left(
n\right)  }\right)
\end{align*}
(since $0\leq k\leq n$). In view of%
\begin{align*}
a_{1,\gamma\left(  1\right)  }a_{2,\gamma\left(  2\right)  }\cdots
a_{k,\gamma\left(  k\right)  }  &  =\prod_{i=1}^{k}\underbrace{a_{i,\gamma
\left(  i\right)  }}_{\substack{=a_{i,\alpha\left(  i\right)  }\\\text{(since
(\ref{pf.lem.block-tria-rowdet.prod=prodprod.pf.1}) yields }\gamma\left(
i\right)  =\alpha\left(  i\right)  \text{)}}}\ \ \ \ \ \ \ \ \ \ \left(
\text{by Convention \ref{conv.prod-noncomm}}\right) \\
&  =\prod_{i=1}^{k}a_{i,\alpha\left(  i\right)  }%
\end{align*}
and%
\begin{align*}
a_{k+1,\gamma\left(  k+1\right)  }a_{k+2,\gamma\left(  k+2\right)  }\cdots
a_{n,\gamma\left(  n\right)  }  &  =a_{k+1,\gamma\left(  k+1\right)
}a_{k+2,\gamma\left(  k+2\right)  }\cdots a_{k+\left(  n-k\right)
,\gamma\left(  k+\left(  n-k\right)  \right)  }\\
&  \ \ \ \ \ \ \ \ \ \ \ \ \ \ \ \ \ \ \ \ \left(  \text{since }n=k+\left(
n-k\right)  \right) \\
&  =\prod_{i=1}^{n-k}\underbrace{a_{k+i,\gamma\left(  k+i\right)  }%
}_{\substack{=a_{k+i,k+\beta\left(  i\right)  }\\\text{(since
(\ref{pf.lem.block-tria-rowdet.prod=prodprod.pf.3}) yields }\gamma\left(
k+i\right)  =k+\beta\left(  i\right)  \text{)}}}\\
&  \ \ \ \ \ \ \ \ \ \ \ \ \ \ \ \ \ \ \ \ \left(  \text{by Convention
\ref{conv.prod-noncomm}}\right) \\
&  =\prod_{i=1}^{n-k}a_{k+i,k+\beta\left(  i\right)  },
\end{align*}
we can rewrite this as
\[
\prod_{i=1}^{n}a_{i,\gamma\left(  i\right)  }=\left(  \prod_{i=1}%
^{k}a_{i,\alpha\left(  i\right)  }\right)  \left(  \prod_{i=1}^{n-k}%
a_{k+i,k+\beta\left(  i\right)  }\right)  .
\]
In other words,%
\[
\prod_{i=1}^{n}a_{i,\left(  \alpha\oplus\beta\right)  \left(  i\right)
}=\left(  \prod_{i=1}^{k}a_{i,\alpha\left(  i\right)  }\right)  \left(
\prod_{i=1}^{n-k}a_{k+i,k+\beta\left(  i\right)  }\right)
\]
(since $\gamma=\alpha\oplus\beta$). This proves Claim 1.] \medskip

Next, we observe the following:

\begin{statement}
\textit{Claim 2:} Let $\sigma\in S_{n}$ satisfy $\sigma\left(  \left[
k\right]  \right)  \neq\left[  k\right]  $. Then,
\[
\prod_{i=1}^{n}a_{i,\sigma\left(  i\right)  }=0.
\]

\end{statement}

[\textit{Proof of Claim 2:} We have $\sigma\in S_{n}$. In other words,
$\sigma$ is a permutation of $\left[  n\right]  $ (since $S_{n}$ is the set of
all permutations of $\left[  n\right]  $). In other words, $\sigma$ is a
bijective map from $\left[  n\right]  $ to $\left[  n\right]  $. Thus, the map
$\sigma$ is bijective and therefore injective.

Hence, we have $\left\vert \sigma\left(  T\right)  \right\vert =\left\vert
T\right\vert $ for any subset $T$ of $\left[  n\right]  $%
\ \ \ \ \footnote{\textit{Proof.} Recall the following general fact from set
theory: If $X$ and $Y$ are two finite sets, and if $f:X\rightarrow Y$ is an
injective map, then we have $\left\vert f\left(  T\right)  \right\vert
=\left\vert T\right\vert $ for any subset $T$ of $X$. Applying this to
$X=\left[  n\right]  $ and $Y=\left[  n\right]  $ and $f=\sigma$, we conclude
that $\left\vert \sigma\left(  T\right)  \right\vert =\left\vert T\right\vert
$ for any subset $T$ of $\left[  n\right]  $ (since $\sigma$ is injective).}.
Applying this to $T=\left[  k\right]  $, we obtain $\left\vert \sigma\left(
\left[  k\right]  \right)  \right\vert =\left\vert \left[  k\right]
\right\vert $ (since $\left[  k\right]  $ is a subset of $\left[  n\right]
$). In other words, $\left\vert \left[  k\right]  \right\vert =\left\vert
\sigma\left(  \left[  k\right]  \right)  \right\vert $. Also, $\left[
k\right]  \neq\sigma\left(  \left[  k\right]  \right)  $ (since $\sigma\left(
\left[  k\right]  \right)  \neq\left[  k\right]  $). Moreover, the sets
$\left[  k\right]  $ and $\sigma\left(  \left[  k\right]  \right)  $ are
subsets of $\left[  n\right]  $, and thus are finite (since $\left[  n\right]
$ is finite).

We recall the following elementary fact from set theory: If $X$ and $Y$ are
two finite sets satisfying $\left\vert X\right\vert =\left\vert Y\right\vert $
and $X\neq Y$, then $X\not \subseteq Y$\ \ \ \ \footnote{\textit{Proof.} Let
$X$ and $Y$ be two finite sets satisfying $\left\vert X\right\vert =\left\vert
Y\right\vert $ and $X\neq Y$. We must prove that $X\not \subseteq Y$.
\par
Assume the contrary. Thus, $X\subseteq Y$. Hence, $\left\vert Y\setminus
X\right\vert =\left\vert Y\right\vert -\underbrace{\left\vert X\right\vert
}_{=\left\vert Y\right\vert }=\left\vert Y\right\vert -\left\vert Y\right\vert
=0$, so that $Y\setminus X=\varnothing$. In other words, $Y\subseteq X$.
Combining this with $X\subseteq Y$, we obtain $X=Y$. However, this contradicts
$X\neq Y$.
\par
This contradiction shows that our assumption was false. Hence, we have
$X\not \subseteq Y$, qed.}. We can apply this to $X=\left[  k\right]  $ and
$Y=\sigma\left(  \left[  k\right]  \right)  $ (since $\left\vert \left[
k\right]  \right\vert =\left\vert \sigma\left(  \left[  k\right]  \right)
\right\vert $ and $\left[  k\right]  \neq\sigma\left(  \left[  k\right]
\right)  $), and conclude that $\left[  k\right]  \not \subseteq \sigma\left(
\left[  k\right]  \right)  $. Hence, there exists some $u\in\left[  k\right]
$ such that $u\notin\sigma\left(  \left[  k\right]  \right)  $. Consider this
$u$.

We have $u\in\left[  k\right]  \subseteq\left[  n\right]  $. Thus,
$\sigma^{-1}\left(  u\right)  $ is well-defined (since $\sigma$ is a bijective
map from $\left[  n\right]  $ to $\left[  n\right]  $). Let $v=\sigma
^{-1}\left(  u\right)  $. Then, $v=\sigma^{-1}\left(  u\right)  \in\left[
n\right]  $ and $\sigma\left(  v\right)  =u$ (since $v=\sigma^{-1}\left(
u\right)  $). If we had $v\in\left[  k\right]  $, then we would thus have
$u=\sigma\left(  \underbrace{v}_{\in\left[  k\right]  }\right)  \in
\sigma\left(  \left[  k\right]  \right)  $, which would contradict
$u\notin\sigma\left(  \left[  k\right]  \right)  $. Hence, we cannot have
$v\in\left[  k\right]  $. Thus, we have $v\notin\left[  k\right]  $. Combining
$v\in\left[  n\right]  $ with $v\notin\left[  k\right]  $, we obtain%
\[
v\in\underbrace{\left[  n\right]  }_{=\left\{  1,2,\ldots,n\right\}
}\setminus\underbrace{\left[  k\right]  }_{=\left\{  1,2,\ldots,k\right\}
}=\left\{  1,2,\ldots,n\right\}  \setminus\left\{  1,2,\ldots,k\right\}
=\left\{  k+1,k+2,\ldots,n\right\}  .
\]
Moreover, $u\in\left[  k\right]  =\left\{  1,2,\ldots,k\right\}  $.

However, one of the assumptions of Lemma \ref{lem.block-tria-rowdet} says that%
\[
a_{i,j}=0\ \ \ \ \ \ \ \ \ \ \text{for every }i\in\left\{  k+1,k+2,\ldots
,n\right\}  \text{ and }j\in\left\{  1,2,\ldots,k\right\}  .
\]
Applying this to $i=v$ and $j=u$, we obtain $a_{v,u}=0$ (since $v\in\left\{
k+1,k+2,\ldots,n\right\}  $ and $u\in\left\{  1,2,\ldots,k\right\}  $). In
view of $\sigma\left(  v\right)  =u$, we can rewrite this as $a_{v,\sigma
\left(  v\right)  }=0$.

However, $v\in\left[  n\right]  =\left\{  1,2,\ldots,n\right\}  $. Thus,
$a_{v,\sigma\left(  v\right)  }$ is one of the factors of the product
$\prod_{i=1}^{n}a_{i,\sigma\left(  i\right)  }$ (namely, the factor for
$i=v$). Since we have just shown that $a_{v,\sigma\left(  v\right)  }=0$, we
thus conclude that one of the factors of this product is $0$. Thus, the whole
product must be $0$. In other words, we have $\prod_{i=1}^{n}a_{i,\sigma
\left(  i\right)  }=0$. This proves Claim 2.] \medskip

Now, from $A=\left(  a_{i,j}\right)  _{i,j\in\left[  n\right]  }$, we obtain%
\begin{align*}
\operatorname*{rowdet}A  &  =\operatorname*{rowdet}\left(  \left(
a_{i,j}\right)  _{i,j\in\left[  n\right]  }\right)  =\sum_{\sigma\in S_{n}%
}\left(  -1\right)  ^{\sigma}\prod_{i=1}^{n}a_{i,\sigma\left(  i\right)
}\ \ \ \ \ \ \ \ \ \ \left(  \text{by (\ref{pf.lem.block-tria-rowdet.rowdet=}%
)}\right) \\
&  =\underbrace{\sum_{\substack{\sigma\in S_{n};\\\sigma\left(  \left[
k\right]  \right)  =\left[  k\right]  }}}_{=\sum_{\sigma\in\left\{  \tau\in
S_{n}\ \mid\ \tau\left(  \left[  k\right]  \right)  =\left[  k\right]
\right\}  }}\left(  -1\right)  ^{\sigma}\prod_{i=1}^{n}a_{i,\sigma\left(
i\right)  }+\sum_{\substack{\sigma\in S_{n};\\\sigma\left(  \left[  k\right]
\right)  \neq\left[  k\right]  }}\left(  -1\right)  ^{\sigma}\underbrace{\prod
_{i=1}^{n}a_{i,\sigma\left(  i\right)  }}_{\substack{=0\\\text{(by Claim 2)}%
}}\\
&  \ \ \ \ \ \ \ \ \ \ \ \ \ \ \ \ \ \ \ \ \left(
\begin{array}
[c]{c}%
\text{since each }\sigma\in S_{n}\text{ satisfies either }\sigma\left(
\left[  k\right]  \right)  =\left[  k\right] \\
\text{or }\sigma\left(  \left[  k\right]  \right)  \neq\left[  k\right]
\text{ (but not both at the same time)}%
\end{array}
\right) \\
&  =\sum_{\sigma\in\left\{  \tau\in S_{n}\ \mid\ \tau\left(  \left[  k\right]
\right)  =\left[  k\right]  \right\}  }\left(  -1\right)  ^{\sigma}\prod
_{i=1}^{n}a_{i,\sigma\left(  i\right)  }+\underbrace{\sum_{\substack{\sigma\in
S_{n};\\\sigma\left(  \left[  k\right]  \right)  \neq\left[  k\right]
}}\left(  -1\right)  ^{\sigma}\cdot0}_{=0}\\
&  =\sum_{\sigma\in\left\{  \tau\in S_{n}\ \mid\ \tau\left(  \left[  k\right]
\right)  =\left[  k\right]  \right\}  }\left(  -1\right)  ^{\sigma}\prod
_{i=1}^{n}a_{i,\sigma\left(  i\right)  }\\
&  =\underbrace{\sum_{\left(  \alpha,\beta\right)  \in S_{k}\times S_{n-k}}%
}_{=\sum_{\alpha\in S_{k}}\ \ \sum_{\beta\in S_{n-k}}}\left(  -1\right)
^{\alpha\oplus\beta}\prod_{i=1}^{n}a_{i,\left(  \alpha\oplus\beta\right)
\left(  i\right)  }\\
&  \ \ \ \ \ \ \ \ \ \ \ \ \ \ \ \ \ \ \ \ \left(
\begin{array}
[c]{c}%
\text{here, we have substituted }\alpha\oplus\beta\text{ for }\sigma\text{ in
the sum,}\\
\text{since the map }S_{k}\times S_{n-k}\rightarrow\left\{  \tau\in
S_{n}\ \mid\ \tau\left(  \left[  k\right]  \right)  =\left[  k\right]
\right\}  ,\\
\left(  \alpha,\beta\right)  \mapsto\alpha\oplus\beta\text{ is a bijection}%
\end{array}
\right) \\
&  =\sum_{\alpha\in S_{k}}\ \ \sum_{\beta\in S_{n-k}}\underbrace{\left(
-1\right)  ^{\alpha\oplus\beta}}_{\substack{=\left(  -1\right)  ^{\alpha}%
\cdot\left(  -1\right)  ^{\beta}\\\text{(by Lemma \ref{lem.lalbe-k}
\textbf{(b)})}}}\ \ \underbrace{\prod_{i=1}^{n}a_{i,\left(  \alpha\oplus
\beta\right)  \left(  i\right)  }}_{\substack{=\left(  \prod_{i=1}%
^{k}a_{i,\alpha\left(  i\right)  }\right)  \left(  \prod_{i=1}^{n-k}%
a_{k+i,k+\beta\left(  i\right)  }\right)  \\\text{(by Claim 1)}}}\\
&  =\sum_{\alpha\in S_{k}}\ \ \sum_{\beta\in S_{n-k}}\left(  -1\right)
^{\alpha}\cdot\left(  -1\right)  ^{\beta}\left(  \prod_{i=1}^{k}%
a_{i,\alpha\left(  i\right)  }\right)  \left(  \prod_{i=1}^{n-k}%
a_{k+i,k+\beta\left(  i\right)  }\right)  .
\end{align*}
Comparing this with%
\begin{align*}
&  \underbrace{\operatorname*{rowdet}\left(  \left(  a_{i,j}\right)
_{i,j\in\left[  k\right]  }\right)  }_{\substack{=\sum_{\sigma\in S_{k}%
}\left(  -1\right)  ^{\sigma}\prod_{i=1}^{k}a_{i,\sigma\left(  i\right)
}\\\text{(by (\ref{pf.lem.block-tria-rowdet.rowdet=}), applied to
}k\\\text{instead of }n\text{)}}}\cdot\underbrace{\operatorname*{rowdet}%
\left(  \left(  a_{k+i,k+j}\right)  _{i,j\in\left[  n-k\right]  }\right)
}_{\substack{=\sum_{\sigma\in S_{n-k}}\left(  -1\right)  ^{\sigma}\prod
_{i=1}^{n-k}a_{k+i,k+\sigma\left(  i\right)  }\\\text{(by
(\ref{pf.lem.block-tria-rowdet.rowdet=}), applied to }n-k\\\text{and
}a_{k+i,k+j}\text{ instead of }n\text{ and }a_{i,j}\text{)}}}\\
&  =\underbrace{\left(  \sum_{\sigma\in S_{k}}\left(  -1\right)  ^{\sigma
}\prod_{i=1}^{k}a_{i,\sigma\left(  i\right)  }\right)  }_{\substack{=\sum
_{\alpha\in S_{k}}\left(  -1\right)  ^{\alpha}\prod_{i=1}^{k}a_{i,\alpha
\left(  i\right)  }\\\text{(here, we have renamed}\\\text{the summation index
}\sigma\text{ as }\alpha\text{)}}}\cdot\underbrace{\left(  \sum_{\sigma\in
S_{n-k}}\left(  -1\right)  ^{\sigma}\prod_{i=1}^{n-k}a_{k+i,k+\sigma\left(
i\right)  }\right)  }_{\substack{=\sum_{\beta\in S_{n-k}}\left(  -1\right)
^{\beta}\prod_{i=1}^{n-k}a_{k+i,k+\beta\left(  i\right)  }\\\text{(here, we
have renamed}\\\text{the summation index }\sigma\text{ as }\beta\text{)}}}\\
&  =\left(  \sum_{\alpha\in S_{k}}\left(  -1\right)  ^{\alpha}\prod_{i=1}%
^{k}a_{i,\alpha\left(  i\right)  }\right)  \cdot\left(  \sum_{\beta\in
S_{n-k}}\left(  -1\right)  ^{\beta}\prod_{i=1}^{n-k}a_{k+i,k+\beta\left(
i\right)  }\right) \\
&  =\sum_{\alpha\in S_{k}}\ \ \sum_{\beta\in S_{n-k}}\underbrace{\left(
\left(  -1\right)  ^{\alpha}\prod_{i=1}^{k}a_{i,\alpha\left(  i\right)
}\right)  \cdot\left(  \left(  -1\right)  ^{\beta}\prod_{i=1}^{n-k}%
a_{k+i,k+\beta\left(  i\right)  }\right)  }_{\substack{=\left(  -1\right)
^{\alpha}\cdot\left(  -1\right)  ^{\beta}\left(  \prod_{i=1}^{k}%
a_{i,\alpha\left(  i\right)  }\right)  \left(  \prod_{i=1}^{n-k}%
a_{k+i,k+\beta\left(  i\right)  }\right)  \\\text{(because }\left(  -1\right)
^{\beta}\text{ belongs to }\mathbb{Z}\text{ and thus commutes with}%
\\\text{every element of }R\text{)}}}\\
&  =\sum_{\alpha\in S_{k}}\ \ \sum_{\beta\in S_{n-k}}\left(  -1\right)
^{\alpha}\cdot\left(  -1\right)  ^{\beta}\left(  \prod_{i=1}^{k}%
a_{i,\alpha\left(  i\right)  }\right)  \left(  \prod_{i=1}^{n-k}%
a_{k+i,k+\beta\left(  i\right)  }\right)  ,
\end{align*}
we obtain%
\[
\operatorname*{rowdet}A=\operatorname*{rowdet}\left(  \left(  a_{i,j}\right)
_{i,j\in\left[  k\right]  }\right)  \cdot\operatorname*{rowdet}\left(  \left(
a_{k+i,k+j}\right)  _{i,j\in\left[  n-k\right]  }\right)  .
\]
This proves Lemma \ref{lem.block-tria-rowdet}.
\end{proof}
\end{verlong}

We can now easily derive Corollary \ref{cor.pre-pieri2.2} from Theorem
\ref{thm.pre-pieri2}:

\begin{proof}
[Proof of Corollary \ref{cor.pre-pieri2.2}.]We have $q=n-p\in\left\{
0,1,\ldots,n\right\}  $ (since $p\in\left\{  0,1,\ldots,n\right\}  $). From
$q=n-p$, we obtain $n-q=p$ and $q+p=n$.

Define an $n$-tuple $\xi\in\mathbb{Z}^{n}$ as in Theorem \ref{thm.pre-pieri2}.
Then, Theorem \ref{thm.pre-pieri2} yields%
\begin{equation}
\sum_{\substack{\beta\in\left\{  0,1\right\}  ^{n};\\\left\vert \beta
\right\vert =p}}t_{\alpha+\beta}=\operatorname*{rowdet}\left(  \left(
h_{\alpha_{i}+\xi_{j},\ i}\right)  _{i,j\in\left[  n\right]  }\right)  .
\label{pf.cor.pre-pieri2.2.1}%
\end{equation}

\begin{vershort}
However, the definition of $\xi$ yields%
\begin{align*}
\xi &  =\left(  1,2,\ldots,n-p,n-p+2,n-p+3,\ldots,n+1\right) \\
&  =\left(  1,2,\ldots,q,q+2,q+3,\ldots,n+1\right)
\ \ \ \ \ \ \ \ \ \ \left(  \text{since }n-p=q\right)  .
\end{align*}
Hence,
\begin{equation}
\xi_{k}=k\ \ \ \ \ \ \ \ \ \ \text{for each }k\in\left\{  1,2,\ldots
,q\right\}  \label{pf.cor.pre-pieri2.2.short.fn2.1}%
\end{equation}
and%
\begin{equation}
\xi_{k}=k+1\ \ \ \ \ \ \ \ \ \ \text{for each }k\in\left\{  q+1,q+2,\ldots
,n\right\}  . \label{pf.cor.pre-pieri2.2.short.fn2.2}%
\end{equation}

\end{vershort}

\begin{verlong}
However, the definition of $\xi$ yields%
\begin{align*}
\xi &  =\left(  1,2,\ldots,n\right)  +\left(  \underbrace{0,0,\ldots
,0}_{n-p\text{ zeroes}},\underbrace{1,1,\ldots,1}_{p\text{ ones}}\right) \\
&  =\left(  1,2,\ldots,n-p,n-p+2,n-p+3,\ldots,n+1\right)  .
\end{align*}
Thus,
\begin{align*}
\left(  \xi_{1},\xi_{2},\ldots,\xi_{n}\right)   &  =\xi=\left(  1,2,\ldots
,n-p,n-p+2,n-p+3,\ldots,n+1\right) \\
&  =\left(  1,2,\ldots,q,q+2,q+3,\ldots,n+1\right)
\ \ \ \ \ \ \ \ \ \ \left(  \text{since }n-p=q\right)  .
\end{align*}
In other words, we have%
\begin{equation}
\xi_{k}=k\ \ \ \ \ \ \ \ \ \ \text{for each }k\in\left\{  1,2,\ldots
,q\right\}  \label{pf.cor.pre-pieri2.2.fn2.1}%
\end{equation}
and%
\begin{equation}
\xi_{k}=k+1\ \ \ \ \ \ \ \ \ \ \text{for each }k\in\left\{  q+1,q+2,\ldots
,n\right\}  . \label{pf.cor.pre-pieri2.2.fn2.2}%
\end{equation}

\end{verlong}

\begin{vershort}
Now, it is easy to see that
\[
h_{\alpha_{i}+\xi_{j},\ i}=0\ \ \ \ \ \ \ \ \ \ \text{for every }i\in\left\{
q+1,q+2,\ldots,n\right\}  \text{ and }j\in\left\{  1,2,\ldots,q\right\}
\]
(because if $i\in\left\{  q+1,q+2,\ldots,n\right\}  $ and $j\in\left\{
1,2,\ldots,q\right\}  $, then (\ref{pf.cor.pre-pieri2.2.short.fn2.1}) yields
$\xi_{j}=j\leq q$, whereas (\ref{eq.cor.pre-pieri2.2.ass-minus}) yields
$\alpha_{i}<-q$, so that $\underbrace{\alpha_{i}}_{<-q}+\underbrace{\xi_{j}%
}_{\leq q}<-q+q=0$, and therefore (\ref{eq.cor.pre-pieri2.2.ass}) (applied to
$k=\alpha_{i}+\xi_{j}$) yields $h_{\alpha_{i}+\xi_{j},\ i}=0$). Therefore,
Lemma \ref{lem.block-tria-rowdet} (applied to $a_{i,j}=h_{\alpha_{i}+\xi
_{j},\ i}$ and $A=\left(  h_{\alpha_{i}+\xi_{j},\ i}\right)  _{i,j\in\left[
n\right]  }$ and $k=q$) yields%
\begin{align*}
&  \operatorname*{rowdet}\left(  \left(  h_{\alpha_{i}+\xi_{j},\ i}\right)
_{i,j\in\left[  n\right]  }\right) \\
&  =\operatorname*{rowdet}\underbrace{\left(  \left(  h_{\alpha_{i}+\xi
_{j},\ i}\right)  _{i,j\in\left[  q\right]  }\right)  }_{\substack{=\left(
h_{\alpha_{i}+j,\ i}\right)  _{i,j\in\left[  q\right]  }\\\text{(since
(\ref{pf.cor.pre-pieri2.2.short.fn2.1}) yields }\xi_{j}=j\\\text{for }%
j\in\left[  q\right]  \text{)}}}\cdot\operatorname*{rowdet}\underbrace{\left(
\left(  h_{\alpha_{q+i}+\xi_{q+j},\ q+i}\right)  _{i,j\in\left[  n-q\right]
}\right)  }_{\substack{=\left(  h_{\alpha_{q+i}+q+j+1,\ q+i}\right)
_{i,j\in\left[  n-q\right]  }\\\text{(since
(\ref{pf.cor.pre-pieri2.2.short.fn2.2}) yields }\xi_{q+j}=q+j+1\\\text{for
}j\in\left[  n-q\right]  \text{)}}}\\
&  =\operatorname*{rowdet}\left(  \left(  h_{\alpha_{i}+j,\ i}\right)
_{i,j\in\left[  q\right]  }\right)  \cdot\operatorname*{rowdet}\left(  \left(
h_{\alpha_{q+i}+q+j+1,\ q+i}\right)  _{i,j\in\left[  n-q\right]  }\right) \\
&  =\operatorname*{rowdet}\left(  \left(  h_{\alpha_{i}+j,\ i}\right)
_{i,j\in\left[  q\right]  }\right)  \cdot\operatorname*{rowdet}\left(  \left(
h_{\alpha_{q+i}+q+j+1,\ q+i}\right)  _{i,j\in\left[  p\right]  }\right)
\end{align*}
(since $n-q=p$). Combining this with (\ref{pf.cor.pre-pieri2.2.1}), we obtain%
\begin{align*}
\sum_{\substack{\beta\in\left\{  0,1\right\}  ^{n};\\\left\vert \beta
\right\vert =p}}t_{\alpha+\beta}  &  =\operatorname*{rowdet}\left(  \left(
h_{\alpha_{i}+\xi_{j},\ i}\right)  _{i,j\in\left[  n\right]  }\right) \\
&  =\operatorname*{rowdet}\left(  \left(  h_{\alpha_{i}+j,\ i}\right)
_{i,j\in\left[  q\right]  }\right)  \cdot\operatorname*{rowdet}\left(  \left(
h_{\alpha_{q+i}+q+j+1,\ q+i}\right)  _{i,j\in\left[  p\right]  }\right)  .
\end{align*}

\end{vershort}

\begin{verlong}
Now, it is easy to see that
\[
h_{\alpha_{i}+\xi_{j},\ i}=0\ \ \ \ \ \ \ \ \ \ \text{for every }i\in\left\{
q+1,q+2,\ldots,n\right\}  \text{ and }j\in\left\{  1,2,\ldots,q\right\}
\]
\footnote{\textit{Proof.} Let $i\in\left\{  q+1,q+2,\ldots,n\right\}  $ and
$j\in\left\{  1,2,\ldots,q\right\}  $. From $i\in\left\{  q+1,q+2,\ldots
,n\right\}  $, we obtain $i\geq q+1>q$. Hence,
(\ref{eq.cor.pre-pieri2.2.ass-minus}) yields $\alpha_{i}<-q$. On the other
hand, $j\in\left\{  1,2,\ldots,q\right\}  $. Hence, applying
(\ref{pf.cor.pre-pieri2.2.fn2.1}) to $k=j$, we obtain $\xi_{j}=j\leq q$ (since
$j\in\left\{  1,2,\ldots,q\right\}  $). Thus, $\underbrace{\alpha_{i}}%
_{<-q}+\underbrace{\xi_{j}}_{\leq q}<-q+q=0$. Therefore,
(\ref{eq.cor.pre-pieri2.2.ass}) (applied to $k=\alpha_{i}+\xi_{j}$) yields
$h_{\alpha_{i}+\xi_{j},\ i}=0$ (since $i>q$). Qed.}. Therefore, Lemma
\ref{lem.block-tria-rowdet} (applied to $a_{i,j}=h_{\alpha_{i}+\xi_{j},\ i}$
and $A=\left(  h_{\alpha_{i}+\xi_{j},\ i}\right)  _{i,j\in\left[  n\right]  }$
and $k=q$) yields%
\begin{align}
&  \operatorname*{rowdet}\left(  \left(  h_{\alpha_{i}+\xi_{j},\ i}\right)
_{i,j\in\left[  n\right]  }\right) \nonumber\\
&  =\operatorname*{rowdet}\left(  \left(  h_{\alpha_{i}+\xi_{j},\ i}\right)
_{i,j\in\left[  q\right]  }\right)  \cdot\operatorname*{rowdet}\left(  \left(
h_{\alpha_{q+i}+\xi_{q+j},\ q+i}\right)  _{i,j\in\left[  n-q\right]  }\right)
\nonumber\\
&  =\operatorname*{rowdet}\left(  \left(  h_{\alpha_{i}+\xi_{j},\ i}\right)
_{i,j\in\left[  q\right]  }\right)  \cdot\operatorname*{rowdet}\left(  \left(
h_{\alpha_{q+i}+\xi_{q+j},\ q+i}\right)  _{i,j\in\left[  p\right]  }\right)
\label{pf.cor.pre-pieri2.2.3}%
\end{align}
(since $n-q=p$).

However, for each $i\in\left[  q\right]  $ and $j\in\left[  q\right]  $, we
have $h_{\alpha_{i}+\xi_{j},\ i}=h_{\alpha_{i}+j,\ i}$%
\ \ \ \ \footnote{\textit{Proof.} Let $i\in\left[  q\right]  $ and
$j\in\left[  q\right]  $. Then, $j\in\left[  q\right]  =\left\{
1,2,\ldots,q\right\}  $. Hence, (\ref{pf.cor.pre-pieri2.2.fn2.1}) (applied to
$k=j$) yields $\xi_{j}=j$. Thus, $h_{\alpha_{i}+\xi_{j},\ i}=h_{\alpha
_{i}+j,\ i}$. Qed.}. Hence,
\begin{equation}
\left(  h_{\alpha_{i}+\xi_{j},\ i}\right)  _{i,j\in\left[  q\right]  }=\left(
h_{\alpha_{i}+j,\ i}\right)  _{i,j\in\left[  q\right]  }.
\label{pf.cor.pre-pieri2.2.4}%
\end{equation}

Furthermore, for each $i\in\left[  p\right]  $ and $j\in\left[  p\right]  $,
we have $h_{\alpha_{q+i}+\xi_{q+j},\ q+i}=h_{\alpha_{q+i}+q+j+1,\ q+i}%
$\ \ \ \ \footnote{\textit{Proof.} Let $i\in\left[  p\right]  $ and
$j\in\left[  p\right]  $. Then, $j\in\left[  p\right]  =\left\{
1,2,\ldots,p\right\}  $, so that $q+j\in\left\{  q+1,q+2,\ldots,q+p\right\}
=\left\{  q+1,q+2,\ldots,n\right\}  $ (since $q+p=n$). Thus,
(\ref{pf.cor.pre-pieri2.2.fn2.2}) (applied to $k=q+j$) yields $\xi
_{q+j}=q+j+1$. Hence, $h_{\alpha_{q+i}+\xi_{q+j},\ q+i}=h_{\alpha
_{q+i}+q+j+1,\ q+i}$. Qed.}. Hence,%
\begin{equation}
\left(  h_{\alpha_{q+i}+\xi_{q+j},\ q+i}\right)  _{i,j\in\left[  p\right]
}=\left(  h_{\alpha_{q+i}+q+j+1,\ q+i}\right)  _{i,j\in\left[  p\right]  }.
\label{pf.cor.pre-pieri2.2.5}%
\end{equation}

In light of (\ref{pf.cor.pre-pieri2.2.4}) and (\ref{pf.cor.pre-pieri2.2.5}),
we can rewrite (\ref{pf.cor.pre-pieri2.2.3}) as%
\begin{align*}
&  \operatorname*{rowdet}\left(  \left(  h_{\alpha_{i}+\xi_{j},\ i}\right)
_{i,j\in\left[  n\right]  }\right) \\
&  =\operatorname*{rowdet}\left(  \left(  h_{\alpha_{i}+j,\ i}\right)
_{i,j\in\left[  q\right]  }\right)  \cdot\operatorname*{rowdet}\left(  \left(
h_{\alpha_{q+i}+q+j+1,\ q+i}\right)  _{i,j\in\left[  p\right]  }\right)  .
\end{align*}
Hence, (\ref{pf.cor.pre-pieri2.2.1}) yields%
\begin{align*}
\sum_{\substack{\beta\in\left\{  0,1\right\}  ^{n};\\\left\vert \beta
\right\vert =p}}t_{\alpha+\beta}  &  =\operatorname*{rowdet}\left(  \left(
h_{\alpha_{i}+\xi_{j},\ i}\right)  _{i,j\in\left[  n\right]  }\right) \\
&  =\operatorname*{rowdet}\left(  \left(  h_{\alpha_{i}+j,\ i}\right)
_{i,j\in\left[  q\right]  }\right)  \cdot\operatorname*{rowdet}\left(  \left(
h_{\alpha_{q+i}+q+j+1,\ q+i}\right)  _{i,j\in\left[  p\right]  }\right)  .
\end{align*}

\end{verlong}

\noindent This proves Corollary \ref{cor.pre-pieri2.2}.
\end{proof}

Next, let us state a counterpart to Corollary \ref{cor.pre-pieri.4}:

\begin{corollary}
\label{cor.pre-pieri2.4}Let $n\in\mathbb{N}$ and $p\in\left\{  0,1,\ldots
,n\right\}  $. Let $q=n-p$. Let $h_{k,\ i}$ be an element of $R$ for all
$k\in\mathbb{Z}$ and $i\in\left[  n\right]  $. Assume that%
\begin{equation}
h_{k,\ i}=0\ \ \ \ \ \ \ \ \ \ \text{for all }k<0\text{ and }i>q.
\end{equation}

For any $m\in\left\{  0,1,\ldots,n\right\}  $ and any $\lambda\in
\mathbb{Z}^{m}$, we define%
\[
s_{\lambda}:=\operatorname*{rowdet}\left(  \left(  h_{\lambda_{i}%
+j-i,\ i}\right)  _{i,j\in\left[  m\right]  }\right)  \in R.
\]

Set%
\[
e_{p,\ q}:=\operatorname*{rowdet}\left(  \left(  h_{1+j-i,\ q+i}\right)
_{i,j\in\left[  p\right]  }\right)  \in R.
\]

Fix an $n$-tuple $\mu\in\mathbb{Z}^{n}$. Assume that%
\begin{equation}
\mu_{i}=0\ \ \ \ \ \ \ \ \ \ \text{for all }i>q.
\label{eq.cor.pre-pieri2.4.mui=0}%
\end{equation}

Let $\overline{\mu}=\left(  \mu_{1},\mu_{2},\ldots,\mu_{q}\right)  $. Then,%
\[
s_{\overline{\mu}}\cdot e_{p,\ q}=\sum_{\substack{\beta\in\left\{
0,1\right\}  ^{n};\\\left\vert \beta\right\vert =p}}s_{\mu+\beta}.
\]

\end{corollary}

\begin{vershort}
\begin{proof}
[Proof of Corollary \ref{cor.pre-pieri2.4}.]This follows from Corollary
\ref{cor.pre-pieri2.2} in the same way as Corollary \ref{cor.pre-pieri.4}
follows from Corollary \ref{cor.pre-pieri.2} (i.e., by setting $\alpha=\left(
\mu_{1}-1,\ \mu_{2}-2,\ \ldots,\ \mu_{n}-n\right)  $ and rewriting all
determinants involved).
\end{proof}
\end{vershort}

\begin{verlong}
\begin{proof}
[Proof of Corollary \ref{cor.pre-pieri2.4}.]Define $t_{\alpha}\in R$ for each
$\alpha\in\mathbb{Z}^{n}$ as in Corollary \ref{cor.pre-pieri2.2}.

Define an $n$-tuple $\alpha\in\mathbb{Z}^{n}$ by%
\[
\alpha=\left(  \mu_{1}-1,\ \mu_{2}-2,\ \ldots,\ \mu_{n}-n\right)  .
\]
Thus,%
\begin{equation}
\alpha_{i}=\mu_{i}-i\ \ \ \ \ \ \ \ \ \ \text{for each }i\in\left[  n\right]
. \label{pf.cor.pre-pieri2.4.ali=}%
\end{equation}
Hence, for each $i\in\left[  n\right]  $ satisfying $i>q$, we have%
\[
\alpha_{i}=\underbrace{\mu_{i}}_{\substack{=0\\\text{(by
(\ref{eq.cor.pre-pieri2.4.mui=0}))}}}-i=0-i=-i<-q\ \ \ \ \ \ \ \ \ \ \left(
\text{since }i>q\right)  .
\]
Hence, Corollary \ref{cor.pre-pieri2.2} yields
\begin{align}
&  \sum_{\substack{\beta\in\left\{  0,1\right\}  ^{n};\\\left\vert
\beta\right\vert =p}}t_{\alpha+\beta}\nonumber\\
&  =\operatorname*{rowdet}\left(  \left(  h_{\alpha_{i}+j,\ i}\right)
_{i,j\in\left[  q\right]  }\right)  \cdot\operatorname*{rowdet}\left(  \left(
h_{\alpha_{q+i}+q+j+1,\ q+i}\right)  _{i,j\in\left[  p\right]  }\right)  .
\label{pf.cor.pre-pieri2.4.old}%
\end{align}

However, each $\beta\in\mathbb{N}^{n}$ satisfies
\begin{equation}
t_{\alpha+\beta}=s_{\mu+\beta} \label{pf.cor.pre-pieri2.4.3}%
\end{equation}
\footnote{\textit{Proof.} Let $\beta\in\mathbb{N}^{n}$. For each
$i,j\in\left[  n\right]  $, we have%
\[
\underbrace{\left(  \alpha+\beta\right)  _{i}}_{=\alpha_{i}+\beta_{i}%
}+j=\underbrace{\alpha_{i}}_{\substack{=\mu_{i}-i\\\text{(by
(\ref{pf.cor.pre-pieri2.4.ali=}))}}}+\beta_{i}+j=\mu_{i}-i+\beta
_{i}+j=\underbrace{\mu_{i}+\beta_{i}}_{=\left(  \mu+\beta\right)  _{i}%
}+j-i=\left(  \mu+\beta\right)  _{i}+j-i
\]
and therefore
\[
h_{\left(  \alpha+\beta\right)  _{i}+j,\ i}=h_{\left(  \mu+\beta\right)
_{i}+j-i,\ i}.
\]
In other words, we have%
\[
\left(  h_{\left(  \alpha+\beta\right)  _{i}+j,\ i}\right)  _{i,j\in\left[
n\right]  }=\left(  h_{\left(  \mu+\beta\right)  _{i}+j-i,\ i}\right)
_{i,j\in\left[  n\right]  }.
\]
\par
The definition of $t_{\alpha+\beta}$ yields
\[
t_{\alpha+\beta}=\operatorname*{rowdet}\left(  \underbrace{\left(  h_{\left(
\alpha+\beta\right)  _{i}+j,\ i}\right)  _{i,j\in\left[  n\right]  }%
}_{=\left(  h_{\left(  \mu+\beta\right)  _{i}+j-i,\ i}\right)  _{i,j\in\left[
n\right]  }}\right)  =\operatorname*{rowdet}\left(  \left(  h_{\left(
\mu+\beta\right)  _{i}+j-i,\ i}\right)  _{i,j\in\left[  n\right]  }\right)
=s_{\mu+\beta}%
\]
(since the definition of $s_{\mu+\beta}$ yields $s_{\mu+\beta}%
=\operatorname*{rowdet}\left(  \left(  h_{\left(  \mu+\beta\right)
_{i}+j-i,\ i}\right)  _{i,j\in\left[  n\right]  }\right)  $). Qed.}
Furthermore, we have
\begin{equation}
\operatorname*{rowdet}\left(  \left(  h_{\alpha_{i}+j,\ i}\right)
_{i,j\in\left[  q\right]  }\right)  =s_{\overline{\mu}}
\label{pf.cor.pre-pieri2.4.4}%
\end{equation}
\footnote{\textit{Proof.} We have $\overline{\mu}=\left(  \mu_{1},\mu
_{2},\ldots,\mu_{q}\right)  $; thus,
\begin{equation}
\overline{\mu}_{i}=\mu_{i}\ \ \ \ \ \ \ \ \ \ \text{for each }i\in\left[
q\right]  . \label{pf.cor.pre-pieri2.4.4.pf.1}%
\end{equation}
\par
The definition of $s_{\overline{\mu}}$ yields%
\begin{equation}
s_{\overline{\mu}}=\operatorname*{rowdet}\left(  \left(  h_{\overline{\mu}%
_{i}+j-i,\ i}\right)  _{i,j\in\left[  q\right]  }\right)  .
\label{pf.cor.pre-pieri2.4.4.pf.2}%
\end{equation}
\par
We have $q=n-p\in\left\{  0,1,\ldots,n\right\}  $ (since $p\in\left\{
0,1,\ldots,n\right\}  $). Thus, $q\leq n$. Therefore, $\left\{  1,2,\ldots
,q\right\}  \subseteq\left\{  1,2,\ldots,n\right\}  $. In other words,
$\left[  q\right]  \subseteq\left[  n\right]  $ (since $\left[  q\right]
=\left\{  1,2,\ldots,q\right\}  $ and $\left[  n\right]  =\left\{
1,2,\ldots,n\right\}  $).
\par
For each $i,j\in\left[  q\right]  $, we have $i\in\left[  q\right]
\subseteq\left[  n\right]  $ and thus%
\[
\underbrace{\alpha_{i}}_{\substack{=\mu_{i}-i\\\text{(by
(\ref{pf.cor.pre-pieri2.4.ali=}))}}}+j=\underbrace{\mu_{i}}%
_{\substack{=\overline{\mu}_{i}\\\text{(by (\ref{pf.cor.pre-pieri2.4.4.pf.1}%
))}}}-i+j=\overline{\mu}_{i}-i+j=\overline{\mu}_{i}+j-i
\]
and therefore $h_{\alpha_{i}+j,\ i}=h_{\overline{\mu}_{i}+j-i,\ i}$. In other
words, we have
\[
\left(  h_{\alpha_{i}+j,\ i}\right)  _{i,j\in\left[  q\right]  }=\left(
h_{\overline{\mu}_{i}+j-i,\ i}\right)  _{i,j\in\left[  q\right]  }.
\]
Hence,%
\[
\operatorname*{rowdet}\left(  \underbrace{\left(  h_{\alpha_{i}+j,\ i}\right)
_{i,j\in\left[  q\right]  }}_{=\left(  h_{\overline{\mu}_{i}+j-i,\ i}\right)
_{i,j\in\left[  q\right]  }}\right)  =\operatorname*{rowdet}\left(  \left(
h_{\overline{\mu}_{i}+j-i,\ i}\right)  _{i,j\in\left[  q\right]  }\right)
=s_{\overline{\mu}}%
\]
(by (\ref{pf.cor.pre-pieri2.4.4.pf.2})).}. Finally, we have%
\begin{equation}
\operatorname*{rowdet}\left(  \left(  h_{\alpha_{q+i}+q+j+1,\ q+i}\right)
_{i,j\in\left[  p\right]  }\right)  =e_{p,\ q} \label{pf.cor.pre-pieri2.4.5}%
\end{equation}
\footnote{\textit{Proof of (\ref{pf.cor.pre-pieri2.4.5}):} Let $i\in\left[
p\right]  $ and $j\in\left[  p\right]  $. Then, $i\in\left[  p\right]
=\left\{  1,2,\ldots,p\right\}  $, so that $q+i\in\left\{  q+1,q+2,\ldots
,q+p\right\}  =\left\{  q+1,q+2,\ldots,n\right\}  $ (since $q+p=n$ (because
$q=n-p$)). Thus, $q+i\geq q+1>q$. Also, $q+i\in\left\{  q+1,q+2,\ldots
,n\right\}  \subseteq\left\{  1,2,\ldots,n\right\}  =\left[  n\right]  $. Now,
(\ref{eq.cor.pre-pieri2.4.mui=0}) (applied to $q+i$ instead of $i$) yields
$\mu_{q+i}=0$ (since $q+i>q$). However, (\ref{pf.cor.pre-pieri2.4.ali=})
(applied to $q+i$ instead of $i$) yields $\alpha_{q+i}=\underbrace{\mu_{q+i}%
}_{=0}-\left(  q+i\right)  =-\left(  q+i\right)  $. Thus,%
\[
\underbrace{\alpha_{q+i}}_{=-\left(  q+i\right)  }+q+j+1=-\left(  q+i\right)
+q+j+1=1+j-i.
\]
Thus, $h_{\alpha_{q+i}+q+j+1,\ q+i}=h_{1+j-i,\ q+i}$.
\par
Forget that we fixed $i$ and $j$. We thus have shown that $h_{\alpha
_{q+i}+q+j+1,\ q+i}=h_{1+j-i,\ q+i}$ for all $i\in\left[  p\right]  $ and
$j\in\left[  p\right]  $. Thus,%
\[
\left(  h_{\alpha_{q+i}+q+j+1,\ q+i}\right)  _{i,j\in\left[  p\right]
}=\left(  h_{1+j-i,\ q+i}\right)  _{i,j\in\left[  p\right]  }.
\]
Hence,%
\[
\operatorname*{rowdet}\left(  \left(  h_{\alpha_{q+i}+q+j+1,\ q+i}\right)
_{i,j\in\left[  p\right]  }\right)  =\operatorname*{rowdet}\left(  \left(
h_{1+j-i,\ q+i}\right)  _{i,j\in\left[  p\right]  }\right)  .
\]
On the other hand,
\[
e_{p,\ q}=\operatorname*{rowdet}\left(  \left(  h_{1+j-i,\ q+i}\right)
_{i,j\in\left[  p\right]  }\right)
\]
(by the definition of $e_{p,\ q}$). Comparing these two equalities, we obtain
$\operatorname*{rowdet}\left(  \left(  h_{\alpha_{q+i}+q+j+1,\ q+i}\right)
_{i,j\in\left[  p\right]  }\right)  =e_{p,\ q}$. This proves
(\ref{pf.cor.pre-pieri2.4.5}).}.

Using (\ref{pf.cor.pre-pieri2.4.3}), (\ref{pf.cor.pre-pieri2.4.4}) and
(\ref{pf.cor.pre-pieri2.4.5}), we can rewrite (\ref{pf.cor.pre-pieri2.4.old})
as%
\[
\sum_{\substack{\beta\in\left\{  0,1\right\}  ^{n};\\\left\vert \beta
\right\vert =p}}s_{\mu+\beta}=s_{\overline{\mu}}\cdot e_{p,\ q}.
\]

\noindent This proves Corollary \ref{cor.pre-pieri2.4}.
\end{proof}
\end{verlong}

Counterparts to Corollary \ref{cor.pre-pieri.2comm} and Corollary
\ref{cor.pre-pieri.4comm} can be stated as well, but we omit them to save
space. (They are trivial consequences of Corollary \ref{cor.pre-pieri2.2} and
Corollary \ref{cor.pre-pieri2.4}.)

We note that (\ref{eq.intro.pieri4b}) is the particular case of Corollary
\ref{cor.pre-pieri2.4} for $R=\Lambda$, $h_{k,\ i}=h_{k}$ and $\mu=\lambda$.

\section{A pre-LR rule?}

We have now proved two fairly general determinantal identities -- Theorem
\ref{thm.pre-pieri} and Theorem \ref{thm.pre-pieri2} -- and seen some of their
consequences. Anyone familiar with symmetric functions will likely view these
two identities as two \textquotedblleft antipodal\textquotedblright%
\ statements, in the sense in which the complete homogeneous symmetric
functions are \textquotedblleft antipodal\textquotedblright\ to the elementary
symmetric functions.\footnote{The analogy strikes the eye from many
directions: The summation signs $\sum_{\substack{\beta\in\mathbb{N}%
^{n};\\\left\vert \beta\right\vert =p}}$ and $\sum_{\substack{\beta\in\left\{
0,1\right\}  ^{n};\\\left\vert \beta\right\vert =p}}$ in Theorem
\ref{thm.pre-pieri} and Theorem \ref{thm.pre-pieri2} are precisely the ones
that appear in the definitions of the respective symmetric functions; the
determinant $e_{p,\ q}$ in Corollary \ref{cor.pre-pieri2.4} is a rather
straightforward generalization of the Jacobi--Trudi determinant for the $p$-th
elementary symmetric function; and so on.} The latter \textquotedblleft
antipodality\textquotedblright\ can be understood particularly well by viewing
both families of symmetric functions as corner cases of \emph{Schur functions}
(see, e.g., \cite[Chapter 7]{Stanley-EC2} or \cite[\S I.3]{Macdon95}). It is
thus natural to ask whether our determinantal identities can be viewed as
corner cases of something more general, too:

\begin{question}
\label{quest.pre-LR}Is there a common generalization of Theorem
\ref{thm.pre-pieri} and Theorem \ref{thm.pre-pieri2}?
\end{question}

Such a generalization might resemble (perhaps even generalize) the
\textquotedblleft immaculate Littlewood--Richardson rule\textquotedblright\ of
Berg, Bergeron, Saliola, Serrano and Zabrocki (\cite[Theorem 7.3]{BBSSZ15}).
Indeed, as we have seen above (in our proof of Proposition \ref{prop.immac}),
the \textquotedblleft right-Pieri rule\textquotedblright\ \cite[Theorem
3.5]{BBSSZ13} is a particular case of our Theorem \ref{thm.pre-pieri}; one can
likewise derive a \textquotedblleft second right-Pieri rule\textquotedblright%
\ (with $E_{s}=\mathfrak{S}_{\left(  1^{s}\right)  }$ taking the place of
$H_{s}$) from our Theorem \ref{thm.pre-pieri2}. Both of these
\textquotedblleft right-Pieri rules\textquotedblright\ are particular cases of
the \textquotedblleft immaculate Littlewood--Richardson rule\textquotedblright%
. Thus, it is not too outlandish to suspect that the latter rule can, too, be
viewed as a particular case of a (noncommutative) determinantal identity.

A step in the general direction of such a generalization appears to be the
following proposition:

\begin{proposition}
\label{prop.pre-LR.symm}Let $n\in\mathbb{N}$. Let $h_{k,\ i}$ be an element of
$R$ for all $k\in\mathbb{Z}$ and $i\in\left[  n\right]  $.

Let $B$ be a finite set of $n$-tuples $\beta\in\mathbb{Z}^{n}$. Assume that
this set $B$ is invariant under the right $S_{n}$-action on $\mathbb{Z}^{n}$.
(This $S_{n}$-action was introduced in Definition \ref{def.etapi}.)

For any $\alpha\in\mathbb{Z}^{n}$, we define%
\[
t_{\alpha}:=\operatorname*{rowdet}\left(  \left(  h_{\alpha_{i}+j,\ i}\right)
_{i,j\in\left[  n\right]  }\right)  \in R.
\]

Let $\alpha\in\mathbb{Z}^{n}$. Then, there exists a family $\left(
\lambda_{\gamma}\right)  _{\gamma\in\mathbb{Z}^{n}}$ of coefficients
$\lambda_{\gamma}\in\mathbb{Z}$ such that all but finitely many $\gamma
\in\mathbb{Z}^{n}$ satisfy $\lambda_{\gamma}=0$, and such that%
\begin{equation}
\sum_{\beta\in B}t_{\alpha+\beta}=\sum_{\gamma\in\mathbb{Z}^{n}}%
\lambda_{\gamma}\operatorname*{rowdet}\left(  \left(  h_{\alpha_{i}+\gamma
_{j},\ i}\right)  _{i,j\in\left[  n\right]  }\right)  .
\label{eq.prop.pre-LR.symm.claim}%
\end{equation}

\end{proposition}

\begin{proof}
[Proof of Proposition \ref{prop.pre-LR.symm} (sketched).]Define the ring%
\[
N:=\mathbb{Z}\left\langle X_{k,\ i}\ \mid\ k\in\mathbb{Z}\text{ and }%
i\in\left[  n\right]  \right\rangle .
\]
This is the ring of noncommutative polynomials over $\mathbb{Z}$ in the
variables $X_{k,\ i}$ for all $k\in\mathbb{Z}$ and $i\in\left[  n\right]  $.

A noncommutative monomial in $N$ will be called \emph{multilinear} if it has
the form $X_{p_{1},\ 1}X_{p_{2},\ 2}\cdots X_{p_{n},\ n}$ for some
$p\in\mathbb{Z}^{n}$. Let $N^{\operatorname*{mult}}$ denote the $\mathbb{Z}%
$-linear span of all multilinear monomials in $N$. The symmetric group $S_{n}$
acts $\mathbb{Z}$-linearly on this $\mathbb{Z}$-submodule
$N^{\operatorname*{mult}}$ from the right according to the rule%
\begin{equation}
\left(  X_{p_{1},\ 1}X_{p_{2},\ 2}\cdots X_{p_{n},\ n}\right)  \cdot
\tau=X_{p_{\tau\left(  1\right)  },\ 1}X_{p_{\tau\left(  2\right)  }%
,\ 2}\cdots X_{p_{\tau\left(  n\right)  },\ n}
\label{pf.prop.pre-LR.symm.mon-acn}%
\end{equation}
(for all multilinear monomials $X_{p_{1},\ 1}X_{p_{2},\ 2}\cdots X_{p_{n}%
,\ n}$ and all $\tau\in S_{n}$).

An element $p\in N^{\operatorname*{mult}}$ will be called \emph{antisymmetric}
if each $\tau\in S_{n}$ satisfies $p\cdot\tau=\left(  -1\right)  ^{\tau}p$.
Let $N^{\operatorname*{sign}}$ denote the set of all antisymmetric elements
$p\in N^{\operatorname*{mult}}$; this is a $\mathbb{Z}$-submodule of
$N^{\operatorname*{mult}}$.

For each $\alpha\in\mathbb{Z}^{n}$, we define%
\[
T_{\alpha}:=\operatorname*{rowdet}\left(  \left(  X_{\alpha_{i}+j,\ i}\right)
_{i,j\in\left[  n\right]  }\right)  \in N.
\]
It is easy to see that $T_{\alpha}\in N^{\operatorname*{mult}}$ for each
$\alpha\in\mathbb{Z}^{n}$. Hence, $\sum_{\beta\in B}T_{\beta}\in
N^{\operatorname*{mult}}$. Set%
\begin{equation}
T_{B}:=\sum_{\beta\in B}T_{\beta}. \label{pf.prop.pre-LR.symm.TB=}%
\end{equation}
We shall now show that $T_{B}\in N^{\operatorname*{sign}}$.

[\textit{Proof:} We have $T_{B}=\sum_{\beta\in B}T_{\beta}\in
N^{\operatorname*{mult}}$. It thus remains to show that $T_{B}$ is
antisymmetric, i.e., that each $\tau\in S_{n}$ satisfies $T_{B}\cdot
\tau=\left(  -1\right)  ^{\tau}T_{B}$.

Let $\tau\in S_{n}$ be arbitrary. For each $\beta\in\mathbb{Z}^{n}$, we have%
\begin{align}
T_{\beta}  &  =\operatorname*{rowdet}\left(  \left(  X_{\beta_{i}%
+j,\ i}\right)  _{i,j\in\left[  n\right]  }\right)
\ \ \ \ \ \ \ \ \ \ \left(  \text{by the definition of }T_{\beta}\right)
\nonumber\\
&  =\sum_{\sigma\in S_{n}}\left(  -1\right)  ^{\sigma}X_{\beta_{1}%
+\sigma\left(  1\right)  ,\ 1}X_{\beta_{2}+\sigma\left(  2\right)  ,\ 2}\cdots
X_{\beta_{n}+\sigma\left(  n\right)  ,\ n} \label{pf.prop.pre-LR.symm.Tbet=}%
\end{align}
and thus%
\begin{align}
&  T_{\beta}\cdot\tau\nonumber\\
&  =\left(  \sum_{\sigma\in S_{n}}\left(  -1\right)  ^{\sigma}X_{\beta
_{1}+\sigma\left(  1\right)  ,\ 1}X_{\beta_{2}+\sigma\left(  2\right)
,\ 2}\cdots X_{\beta_{n}+\sigma\left(  n\right)  ,\ n}\right)  \cdot
\tau\nonumber\\
&  =\sum_{\sigma\in S_{n}}\left(  -1\right)  ^{\sigma}\underbrace{\left(
X_{\beta_{1}+\sigma\left(  1\right)  ,\ 1}X_{\beta_{2}+\sigma\left(  2\right)
,\ 2}\cdots X_{\beta_{n}+\sigma\left(  n\right)  ,\ n}\right)  \cdot\tau
}_{\substack{=X_{\beta_{\tau\left(  1\right)  }+\sigma\left(  \tau\left(
1\right)  \right)  ,\ 1}X_{\beta_{\tau\left(  2\right)  }+\sigma\left(
\tau\left(  2\right)  \right)  ,\ 2}\cdots X_{\beta_{\tau\left(  n\right)
}+\sigma\left(  \tau\left(  n\right)  \right)  ,\ n}\\\text{(by
(\ref{pf.prop.pre-LR.symm.mon-acn}))}}}\nonumber\\
&  =\sum_{\sigma\in S_{n}}\left(  -1\right)  ^{\sigma}X_{\beta_{\tau\left(
1\right)  }+\sigma\left(  \tau\left(  1\right)  \right)  ,\ 1}X_{\beta
_{\tau\left(  2\right)  }+\sigma\left(  \tau\left(  2\right)  \right)
,\ 2}\cdots X_{\beta_{\tau\left(  n\right)  }+\sigma\left(  \tau\left(
n\right)  \right)  ,\ n}\nonumber\\
&  =\sum_{\sigma\in S_{n}}\left(  -1\right)  ^{\sigma}X_{\left(  \beta
\cdot\tau\right)  _{1}+\left(  \sigma\circ\tau\right)  \left(  1\right)
,\ 1}X_{\left(  \beta\cdot\tau\right)  _{2}+\left(  \sigma\circ\tau\right)
\left(  2\right)  ,\ 2}\cdots X_{\left(  \beta\cdot\tau\right)  _{n}+\left(
\sigma\circ\tau\right)  \left(  n\right)  ,\ n}\nonumber\\
&  \ \ \ \ \ \ \ \ \ \ \ \ \ \ \ \ \ \ \ \ \left(  \text{since }\beta
_{\tau\left(  i\right)  }=\left(  \beta\cdot\tau\right)  _{i}\text{ and
}\sigma\left(  \tau\left(  i\right)  \right)  =\left(  \sigma\circ\tau\right)
\left(  i\right)  \text{ for each }i\right) \nonumber\\
&  =\sum_{\sigma\in S_{n}}\underbrace{\left(  -1\right)  ^{\sigma\circ
\tau^{-1}}}_{=\left(  -1\right)  ^{\sigma}\left(  -1\right)  ^{\tau}%
}X_{\left(  \beta\cdot\tau\right)  _{1}+\sigma\left(  1\right)  ,\ 1}%
X_{\left(  \beta\cdot\tau\right)  _{2}+\sigma\left(  2\right)  ,\ 2}\cdots
X_{\left(  \beta\cdot\tau\right)  _{n}+\sigma\left(  n\right)  ,\ n}%
\nonumber\\
&  \ \ \ \ \ \ \ \ \ \ \ \ \ \ \ \ \ \ \ \ \left(  \text{here, we have
substituted }\sigma\text{ for }\sigma\circ\tau\text{ in the sum}\right)
\nonumber\\
&  =\left(  -1\right)  ^{\tau}\underbrace{\sum_{\sigma\in S_{n}}\left(
-1\right)  ^{\sigma}X_{\left(  \beta\cdot\tau\right)  _{1}+\sigma\left(
1\right)  ,\ 1}X_{\left(  \beta\cdot\tau\right)  _{2}+\sigma\left(  2\right)
,\ 2}\cdots X_{\left(  \beta\cdot\tau\right)  _{n}+\sigma\left(  n\right)
,\ n}}_{\substack{=T_{\beta\cdot\tau}\\\text{(by
(\ref{pf.prop.pre-LR.symm.Tbet=}), applied to }\beta\cdot\tau\text{ instead of
}\beta\text{)}}}\nonumber\\
&  =\left(  -1\right)  ^{\tau}\cdot T_{\beta\cdot\tau}.
\label{pf.prop.pre-LR.symm.Tbetatau}%
\end{align}
Now, from (\ref{pf.prop.pre-LR.symm.TB=}), we obtain%
\begin{align*}
T_{B}\cdot\tau &  =\left(  \sum_{\beta\in B}T_{\beta}\right)  \cdot\tau
=\sum_{\beta\in B}\underbrace{T_{\beta}\cdot\tau}_{\substack{=\left(
-1\right)  ^{\tau}\cdot T_{\beta\cdot\tau}\\\text{(by
(\ref{pf.prop.pre-LR.symm.Tbetatau}))}}}=\sum_{\beta\in B}\left(  -1\right)
^{\tau}\cdot T_{\beta\cdot\tau}=\left(  -1\right)  ^{\tau}\cdot\sum_{\beta\in
B}T_{\beta\cdot\tau}\\
&  =\left(  -1\right)  ^{\tau}\cdot\underbrace{\sum_{\beta\in B}T_{\beta}%
}_{=T_{B}}\ \ \ \ \ \ \ \ \ \ \left(
\begin{array}
[c]{c}%
\text{here, we have substituted }\beta\text{ for }\beta\cdot\tau\text{ in the
sum,}\\
\text{since the map }B\rightarrow B,\ \beta\mapsto\beta\cdot\tau\text{ is a
bijection}\\
\text{(because }B\text{ is invariant under the }S_{n}\text{-action)}%
\end{array}
\right) \\
&  =\left(  -1\right)  ^{\tau}T_{B}.
\end{align*}
This completes our proof of $T_{B}\in N^{\operatorname*{sign}}$.] \medskip

On the other hand, we claim the following:

\begin{statement}
\textit{Claim 1:} Each $p\in N^{\operatorname*{sign}}$ is a $\mathbb{Z}%
$-linear combination of the row-determinants $\operatorname*{rowdet}\left(
\left(  X_{\gamma_{j},\ i}\right)  _{i,j\in\left[  n\right]  }\right)  $ with
$\gamma\in\mathbb{Z}^{n}$.
\end{statement}

Before we prove this, let us note that all these row-determinants
\newline$\operatorname*{rowdet}\left(  \left(  X_{\gamma_{j},\ i}\right)
_{i,j\in\left[  n\right]  }\right)  $ actually belong to
$N^{\operatorname*{sign}}$ (since each $\gamma\in\mathbb{Z}^{n}$ satisfies%
\begin{align*}
\operatorname*{rowdet}\left(  \left(  X_{\gamma_{j},\ i}\right)
_{i,j\in\left[  n\right]  }\right)   &  =\sum_{\sigma\in S_{n}}\left(
-1\right)  ^{\sigma}\underbrace{X_{\gamma_{\sigma\left(  1\right)  }%
,\ 1}X_{\gamma_{\sigma\left(  2\right)  },\ 2}\cdots X_{\gamma_{\sigma\left(
n\right)  },\ n}}_{\substack{=\left(  X_{\gamma_{1},\ 1}X_{\gamma_{2}%
,\ 2}\cdots X_{\gamma_{n},\ n}\right)  \cdot\sigma\\\text{(by
(\ref{pf.prop.pre-LR.symm.mon-acn}))}}}\\
&  =\sum_{\sigma\in S_{n}}\left(  -1\right)  ^{\sigma}\left(  X_{\gamma
_{1},\ 1}X_{\gamma_{2},\ 2}\cdots X_{\gamma_{n},\ n}\right)  \cdot\sigma,
\end{align*}
which easily yields that $\operatorname*{rowdet}\left(  \left(  X_{\gamma
_{j},\ i}\right)  _{i,j\in\left[  n\right]  }\right)  \in
N^{\operatorname*{sign}}$); thus, Claim 1 shows that these row-determinants
span the $\mathbb{Z}$-module $N^{\operatorname*{sign}}$. However, we will not
need this. \medskip

[\textit{Proof of Claim 1:} Let $p\in N^{\operatorname*{sign}}$. We shall show
that $p$ is a $\mathbb{Z}$-linear combination of the row-determinants
$\operatorname*{rowdet}\left(  \left(  X_{\gamma_{j},\ i}\right)
_{i,j\in\left[  n\right]  }\right)  $.

We know that $p$ is antisymmetric (since $p\in N^{\operatorname*{sign}}$). In
other words, each $\sigma\in S_{n}$ satisfies%
\begin{equation}
p\cdot\sigma=\left(  -1\right)  ^{\sigma}p. \label{pf.prop.pre-LR.symm.asy}%
\end{equation}

On the other hand, $p\in N^{\operatorname*{sign}}\subseteq
N^{\operatorname*{mult}}$. Hence, we can write $p$ as a $\mathbb{Z}$-linear
combination of multilinear monomials (by the definition of
$N^{\operatorname*{mult}}$). In other words,%
\begin{equation}
p=\sum_{\alpha\in\mathbb{Z}^{n}}\mu_{\alpha}X_{\alpha_{1},\ 1}X_{\alpha
_{2},\ 2}\cdots X_{\alpha_{n},\ n} \label{pf.prop.pre-LR.symm.p=sum}%
\end{equation}
for some scalars $\mu_{\alpha}\in\mathbb{Z}$ (almost all of them $0$). These
scalars $\mu_{\alpha}$ must furthermore satisfy%
\begin{equation}
\mu_{\alpha\cdot\sigma}=\left(  -1\right)  ^{\sigma}\cdot\mu_{\alpha
}\ \ \ \ \ \ \ \ \ \ \text{for all }\alpha\in\mathbb{Z}^{n}\text{ and all
}\sigma\in S_{n} \label{pf.prop.pre-LR.symm.mu-eq}%
\end{equation}
(by comparing coefficients in (\ref{pf.prop.pre-LR.symm.asy})). Hence, we have
$\mu_{\alpha}=0$ for any $n$-tuple $\alpha\in\mathbb{Z}^{n}$ that has at least
two equal entries (because for any such $\alpha$, there exists some
transposition $\sigma\in S_{n}$ such that $\alpha\cdot\sigma=\alpha$, so that
the equality (\ref{pf.prop.pre-LR.symm.mu-eq}) simplifies to $\mu_{\alpha
}=\underbrace{\left(  -1\right)  ^{\sigma}}_{=-1}\cdot\mu_{\alpha}%
=-\mu_{\alpha}$, and therefore we have $\mu_{\alpha}=0$). Therefore,
(\ref{pf.prop.pre-LR.symm.p=sum}) simplifies to%
\begin{align*}
p  &  =\sum_{\substack{\alpha\in\mathbb{Z}^{n};\\\text{all entries of }%
\alpha\text{ are distinct}}}\mu_{\alpha}X_{\alpha_{1},\ 1}X_{\alpha_{2}%
,\ 2}\cdots X_{\alpha_{n},\ n}\\
&  =\sum_{\substack{\gamma\in\mathbb{Z}^{n};\\\gamma_{1}<\gamma_{2}%
<\cdots<\gamma_{n}}}\ \ \sum_{\sigma\in S_{n}}\underbrace{\mu_{\gamma
\cdot\sigma}}_{\substack{=\left(  -1\right)  ^{\sigma}\mu_{\gamma}\\\text{(by
(\ref{pf.prop.pre-LR.symm.mu-eq}))}}}X_{\gamma_{\sigma\left(  1\right)  }%
,\ 1}X_{\gamma_{\sigma\left(  2\right)  },\ 2}\cdots X_{\gamma_{\sigma\left(
n\right)  },\ n}\\
&  \ \ \ \ \ \ \ \ \ \ \ \ \ \ \ \ \ \ \ \ \left(
\begin{array}
[c]{c}%
\text{here, we have split the sum according to the }n\text{-tuple }\gamma\\
\text{obtained by sorting the }n\text{-tuple }\alpha\text{ in increasing
order}%
\end{array}
\right) \\
&  =\sum_{\substack{\gamma\in\mathbb{Z}^{n};\\\gamma_{1}<\gamma_{2}%
<\cdots<\gamma_{n}}}\ \ \sum_{\sigma\in S_{n}}\left(  -1\right)  ^{\sigma}%
\mu_{\gamma}X_{\gamma_{\sigma\left(  1\right)  },\ 1}X_{\gamma_{\sigma\left(
2\right)  },\ 2}\cdots X_{\gamma_{\sigma\left(  n\right)  },\ n}\\
&  =\sum_{\substack{\gamma\in\mathbb{Z}^{n};\\\gamma_{1}<\gamma_{2}%
<\cdots<\gamma_{n}}}\mu_{\gamma}\underbrace{\sum_{\sigma\in S_{n}}\left(
-1\right)  ^{\sigma}X_{\gamma_{\sigma\left(  1\right)  },\ 1}X_{\gamma
_{\sigma\left(  2\right)  },\ 2}\cdots X_{\gamma_{\sigma\left(  n\right)
},\ n}}_{=\operatorname*{rowdet}\left(  \left(  X_{\gamma_{j},\ i}\right)
_{i,j\in\left[  n\right]  }\right)  }\\
&  =\sum_{\substack{\gamma\in\mathbb{Z}^{n};\\\gamma_{1}<\gamma_{2}%
<\cdots<\gamma_{n}}}\mu_{\gamma}\operatorname*{rowdet}\left(  \left(
X_{\gamma_{j},\ i}\right)  _{i,j\in\left[  n\right]  }\right)  .
\end{align*}
This equality shows that $p$ is a $\mathbb{Z}$-linear combination of the
row-determinants $\operatorname*{rowdet}\left(  \left(  X_{\gamma_{j}%
,\ i}\right)  _{i,j\in\left[  n\right]  }\right)  $. This completes our proof
of Claim 1.] \medskip

Now, recall that $T_{B}\in N^{\operatorname*{sign}}$. Hence, Claim 1 shows
that $T_{B}$ is a $\mathbb{Z}$-linear combination of the row-determinants
$\operatorname*{rowdet}\left(  \left(  X_{\gamma_{j},\ i}\right)
_{i,j\in\left[  n\right]  }\right)  $ with $\gamma\in\mathbb{Z}^{n}$. In other
words, there exists a family $\left(  \lambda_{\gamma}\right)  _{\gamma
\in\mathbb{Z}^{n}}$ of coefficients $\lambda_{\gamma}\in\mathbb{Z}$ such that
all but finitely many $\gamma\in\mathbb{Z}^{n}$ satisfy $\lambda_{\gamma}=0$,
and such that%
\begin{equation}
T_{B}=\sum_{\gamma\in\mathbb{Z}^{n}}\lambda_{\gamma}\operatorname*{rowdet}%
\left(  \left(  X_{\gamma_{j},\ i}\right)  _{i,j\in\left[  n\right]  }\right)
. \label{pf.prop.pre-LR.symm.N-goal}%
\end{equation}
Consider this family $\left(  \lambda_{\gamma}\right)  _{\gamma\in
\mathbb{Z}^{n}}$. We shall now show that this family also satisfies
(\ref{eq.prop.pre-LR.symm.claim}).

Indeed, consider the ring homomorphism $f:N\rightarrow R$ that sends each
$X_{k,\ i}$ to $h_{\alpha_{i}+k,\ i}$. (This clearly exists by the universal
property of the free $\mathbb{Z}$-algebra $N$.) For each $\beta\in
\mathbb{Z}^{n}$, we have%
\begin{align}
f\left(  T_{\beta}\right)   &  =f\left(  \operatorname*{rowdet}\left(  \left(
X_{\beta_{i}+j,\ i}\right)  _{i,j\in\left[  n\right]  }\right)  \right)
\ \ \ \ \ \ \ \ \ \ \left(  \text{by the definition of }T_{\beta}\right)
\nonumber\\
&  =\operatorname*{rowdet}\left(  \left(  f\left(  X_{\beta_{i}+j,\ i}\right)
\right)  _{i,j\in\left[  n\right]  }\right)  \ \ \ \ \ \ \ \ \ \ \left(
\text{since }f\text{ is a ring homomorphism}\right) \nonumber\\
&  =\operatorname*{rowdet}\left(  \left(  h_{\left(  \alpha+\beta\right)
_{i}+j,\ i}\right)  _{i,j\in\left[  n\right]  }\right) \nonumber\\
&  \ \ \ \ \ \ \ \ \ \ \ \ \ \ \ \ \ \ \ \ \left(
\begin{array}
[c]{c}%
\text{because we have }f\left(  X_{\beta_{i}+j,\ i}\right)  =h_{\alpha
_{i}+\beta_{i}+j,\ i}=h_{\left(  \alpha+\beta\right)  _{i}+j,\ i}\\
\text{for all }i\in\left[  n\right]  \text{ and }j\in\left[  n\right]
\end{array}
\right) \nonumber\\
&  =t_{\alpha+\beta} \label{pf.prop.pre-LR.symm.fTb}%
\end{align}
(by the definition of $t_{\alpha+\beta}$). Now, applying the homomorphism $f$
to both sides of the equality (\ref{pf.prop.pre-LR.symm.TB=}), we obtain%
\begin{align*}
f\left(  T_{B}\right)   &  =f\left(  \sum_{\beta\in B}T_{\beta}\right)
=\sum_{\beta\in B}\underbrace{f\left(  T_{\beta}\right)  }%
_{\substack{=t_{\alpha+\beta}\\\text{(by (\ref{pf.prop.pre-LR.symm.fTb}))}%
}}\ \ \ \ \ \ \ \ \ \ \left(  \text{since }f\text{ is a ring homomorphism}%
\right) \\
&  =\sum_{\beta\in B}t_{\alpha+\beta}.
\end{align*}
Hence, applying the homomorphism $f$ to both sides of
(\ref{pf.prop.pre-LR.symm.N-goal}), we obtain%
\begin{align*}
\sum_{\beta\in B}t_{\alpha+\beta}  &  =\sum_{\gamma\in\mathbb{Z}^{n}}%
\lambda_{\gamma}\underbrace{f\left(  \operatorname*{rowdet}\left(  \left(
X_{\gamma_{j},\ i}\right)  _{i,j\in\left[  n\right]  }\right)  \right)
}_{\substack{=\operatorname*{rowdet}\left(  \left(  f\left(  X_{\gamma
_{j},\ i}\right)  \right)  _{i,j\in\left[  n\right]  }\right)  \\\text{(since
}f\text{ is a ring homomorphism)}}}\\
&  =\sum_{\gamma\in\mathbb{Z}^{n}}\lambda_{\gamma}\operatorname*{rowdet}%
\left(  \left(  \underbrace{f\left(  X_{\gamma_{j},\ i}\right)  }%
_{\substack{=h_{\alpha_{i}+\gamma_{j},\ i}\\\text{(by the definition of
}f\text{)}}}\right)  _{i,j\in\left[  n\right]  }\right) \\
&  =\sum_{\gamma\in\mathbb{Z}^{n}}\lambda_{\gamma}\operatorname*{rowdet}%
\left(  \left(  h_{\alpha_{i}+\gamma_{j},\ i}\right)  _{i,j\in\left[
n\right]  }\right)  .
\end{align*}
In other words, (\ref{eq.prop.pre-LR.symm.claim}) holds.

Thus, we have found a family $\left(  \lambda_{\gamma}\right)  _{\gamma
\in\mathbb{Z}^{n}}$ of coefficients $\lambda_{\gamma}\in\mathbb{Z}$ such that
all but finitely many $\gamma\in\mathbb{Z}^{n}$ satisfy $\lambda_{\gamma}=0$,
and such that (\ref{eq.prop.pre-LR.symm.claim}) holds. The proof of
Proposition \ref{prop.pre-LR.symm} is thus complete.
\end{proof}

\end{document}